\documentclass{article}
\usepackage[utf8]{inputenc}
\usepackage[english]{babel}
\usepackage{xypic}
\usepackage[all]{xy}
\usepackage{eucal}
\usepackage{latexsym, amscd,eucal,xr,makeidx,stmaryrd,color,mathrsfs,amsmath, amsfonts,amssymb,bbm,stmaryrd,amsthm,graphicx,pstricks,xypic,comment,graphics}
\usepackage{wrapfig}
\usepackage{color,epsfig,enumerate}
\usepackage{xcolor}
\usepackage{caption}
\usepackage{hyperref}
\xyoption{2cell}

\newtheorem{theorem}{Theorem}[section]
\newtheorem{thm}[theorem]{Theorem}
\newtheorem{lemma}[theorem]{Lemma}

\newtheorem{cor}[theorem]{Corollary}
\newtheorem{prop}[theorem]{Proposition}
\newtheorem{convention}[theorem]{Convention}
\newtheorem{conjecture}[theorem]{Conjecture}
\newtheorem{insight}[theorem]{Insight}

\newtheorem{defn}[theorem]{Definition}

\theoremstyle{definition}
\newtheorem{remark}[theorem]{Remark}
\newtheorem{construction}[theorem]{Construction}
\newtheorem{example}[theorem]{Example}
\newtheorem{notation}[theorem]{Notation}

\numberwithin{equation}{section}

\newcommand{\C}{\mathcal{C}}
\newcommand{\Spacessmash}{\mathcal{S}_*^{\wedge}}
\newcommand{\U}{\mathcal{U}}
\newcommand{\uniU}{\mathbb{U}}
\newcommand{\uniV}{\mathbb{V}}
\newcommand{\uniW}{\mathbb{W}}
\newcommand{\Sp}{\mathit{Sp}}
\newcommand{\derivedk}{\mathcal{D}(k)}
\newcommand{\Spmonoidal}{\mathit{Sp}^{\otimes}}

\newcommand{\D}{\mathcal{D}}
\newcommand{\V}{\mathcal{V}}
\newcommand{\E}{\mathcal{E}}

\newcommand{\Cat}{\mathit Cat}
\newcommand{\Fin}{\mathit Fin_*}
\newcommand{\nfin}{\langle n \rangle}
\newcommand{\mfin}{\langle m \rangle}
\newcommand{\twofin}{\langle 2 \rangle}
\newcommand{\onefin}{\langle 1 \rangle}

\newcommand{\sch}{\mathit{Sm}^{ft}(S)}
\newcommand{\aff}{\mathit{N(AffSm}^{ft}(k))}

\newcommand{\ssets}{\mathit{\hat{\Delta}}}
\newcommand{\ssetsmarked}{\mathit{\hat{\Delta}}_{+}}
\newcommand{\ssetsp}{\mathit{\hat{\Delta}}_*}

\newcommand{\M}{\mathcal{M}}
\newcommand{\N}{\mathcal{N}}
\newcommand{\Q}{\mathcal{Q}}
\newcommand{\Op}{\mathcal{O}}
\newcommand{\Trivmonoidal}{\mathcal{T}riv^{\otimes}}
\newcommand{\Triv}{\mathcal{T}riv}
\newcommand{\Assmonoidal}{\mathcal{A}ss^{\otimes}}
\newcommand{\Ass}{\mathcal{A}ss}
\newcommand{\Comm}{\mathcal{C}omm}
\newcommand{\Commmonoidal}{\mathcal{C}omm^{\otimes}}
\newcommand{\Opmonoidal}{\mathcal{O}^{\otimes}}
\newcommand{\Opprimemonoidal}{(\mathcal{O}')^{\otimes}}
\newcommand{\Cmonoidal}{\mathcal{C}^{\otimes}}
\newcommand{\Cmonoidalx}{\mathcal{C}^{\otimes}[X^{-1}]}

\newcommand{\Dmonoidal}{\mathcal{D}^{\otimes}}
\newcommand{\Emonoidal}{\mathcal{E}^{\otimes}}
\newcommand{\Lmonoidal}{\mathcal{L}_{(\Cmonoidal,X)}^{\otimes}}
\newcommand{\Lx}{\mathcal{L}_{(\Cmonoidal,X)}}
\newcommand{\Lpr}{\mathcal{L}_{(\Cmonoidal,X)}^{Pr}}
\newcommand{\Lpmonoidal}{\mathcal{L}_{(\Cmonoidal,X)}^{Pr,\otimes}}

\newcommand{\Spaces}{\mathcal{S}}
\newcommand{\Spacesp}{\mathcal{S}_*}

\newcommand{\Mot}{\mathcal{M}ot(S)}
\newcommand{\Motmonoidal}{\mathcal{M}ot(S)^{\otimes}}
\newcommand{\st}{\mathcal{S}\mathcal{H}(S)}
\newcommand{\stk}{\mathcal{S}\mathcal{H}(k)}

\newcommand{\stmonoidal}{\mathcal{S}\mathcal{H}(S)^{\otimes}}
\newcommand{\stmonoidalk}{\mathcal{S}\mathcal{H}(k)^{\otimes}}
\newcommand{\stnck}{\mathcal{S}\mathcal{H}_{nc}(k)}
\newcommand{\stncmonoidalk}{\mathcal{S}\mathcal{H}_{nc}(k)^{\otimes}}
\newcommand{\stncmonoidal}{\mathcal{S}\mathcal{H}_{nc}(S)^{\otimes}}
\newcommand{\stnc}{\mathcal{S}\mathcal{H}_{nc}(S)}

\newcommand{\iCatstable}{\mathit{Cat}_{\infty}^{\mathcal{E}x}}
\newcommand{\iCatstableidem}{\mathit{Cat}_{\infty}^{\mathcal{E}x, idem}}
\newcommand{\iCat}{\mathit{Cat}_{\infty}}
\newcommand{\iCatbig}{\mathit{ Cat}_{\infty}^{big}}

\newcommand{\Prl}{\mathcal{P}r^L}
\newcommand{\Prr}{\mathcal{P}r^R}

\newcommand{\Prlmonoidal}{\mathcal{P}r^{L, \otimes}}
\newcommand{\Prlkmonoidal}{\mathcal{P}r_k^{L, \otimes}}
\newcommand{\Prlstable}{\mathcal{P}r^L_{Stb}}

\newcommand{\hsch}{\mathcal{H}(S)}
\newcommand{\hschp}{\mathcal{H}(S)_*}

\newcommand{\hnck}{\mathcal{H}_{nc}(k)}
\newcommand{\nc}{\mathcal{N}cS}
\newcommand{\nck}{\mathcal{N}cS(k)}
\newcommand{\nckmonoidal}{\mathcal{N}cS(k)^{\otimes}}

\newcommand{\dg}{\mathcal{D}g(k)}

\newcommand{\dglp}{\mathcal{D}g^{lp}(k)}
\newcommand{\dgcc}{\mathcal{D}g^{cc}(k)}
\newcommand{\dgc}{\mathcal{D}g^{c}(k)}

\newcommand{\dgloccof}{\mathcal{D}g^{loc-cof}(k)}
\newcommand{\dgloccofmonoidal}{\mathcal{D}g^{loc-cof}(k)^{\otimes}}

\newcommand{\icategories}{(\infty,1)-categories}

\DeclareMathOperator{\Mod}{\mathit{Mod}}

\newcommand{\A}{\mathcal A}
\newcommand{\B}{\mathcal B}

\newcommand{\W}{\mathcal W}
\newcommand{\X}{\mathcal X}
\newcommand{\Y}{\mathcal Y}
\newcommand{\UU}{\mathcal U}
\newcommand{\Z}{\mathcal Z}

\begin{document}

\title{ Noncommutative Motives I: \\ \LARGE{From Commutative to Noncommutative Motives}}
\author{Marco Robalo\footnote{This project author was financially supported  by the Portuguese Foundation for Science and Technology - FCT Doctoral Grant -  SFRH / BD / 68868 / 2010  and its scientific contents are part of a wider framework developed within the project  ANR-09-BLAN-0151 }\\
\small{marco.robalo@math.univ-montp2.fr}\\
\small{I3M} \\  \small{Universit\'e de Montpellier2}}
\date{June, 2013}

\maketitle

\begin{abstract}
Let $\V$ be a symmetric monoidal model category and let $X$ be an object in $\V$. Following the results of \cite{hovey-spectraandsymmetricspectra} it is possible to construct a new symmetric monoidal model category $Sp^{\Sigma}(\V,X)$ of symmetric spectrum objects in $\V$ with respect to $X$, together with a left Quillen monoidal map $\V\to Sp^{\Sigma}(\V,X)$ sending $X$ to an invertible object.

In this paper we use the recent developments in the subject of higher algebra \cite{lurie-ha} to extend the results of \cite{hovey-spectraandsymmetricspectra}. Every symmetric monoidal model category has an underlying symmetric monoidal $(\infty,1)$-category and the first notion should be understood as a mere "presentation" of the second. Our main result is the characterization of the underlying symmetric monoidal $\infty$-category of $Sp^{\Sigma}(\V,X)$, by means of a universal property inside the world of symmetric monoidal $(\infty,1)$-categories. In the process we also extend the results of \cite{hovey-spectraandsymmetricspectra} relating the construction of ordinary spectra to the one of symmetric spectra. As a corollary, we obtain a precise universal characterization for the motivic stable homotopy theory of schemes of Morel-Voevodsky, with its symmetric monoidal structure. This characterization trivializes the problem of finding motivic monoidal realizations. 

As an application we provide a new approach to the theory of noncommutative motives by constructing a stable motivic homotopy theory for the noncommutative spaces of Kontsevich \cite{kontsevich3, kontsevich1, kontsevich2}. For that we introduce an analogue of the Nisnevich topology in the noncommutative setting. Our universal property for the classical theory for schemes provides a canonical monoidal map from the classical stable motivic theory towards these new noncommutative motives and allows us to compare the two theories.
\end{abstract}

\setcounter{tocdepth}{2}
\tableofcontents

\section{Introduction}

\subsection{Motivation}

This paper is the first part of a research project whose main goal is to compare classical algebraic geometry with the new noncommutative algebraic geometry in the sense of Kontsevich \cite{kontsevich1}. More precisely, we want to compare the motivic levels of both theories.

\subsubsection{Motives}
In the sixties, and following the works of Weil and Serre, Grothendieck constructed an "arithmetically flavored" cohomology theory for algebraic varieties. More precisely, he found a whole family of different cohomology theories, each one reflecting different arithmetic properties. The subject of motives started exactly as quest for an abelian category whose objects (the "motives") would be the possible values of a conjectural universal theory of such kind. 

At that time, cohomology theories were formulated in a rather artificial way using abelian categories as the basic input. The notion of triangulated categories appeared as an attemptive to provided a new, more natural setting for cohomology theories. Of course, the subject of motives followed these innovations \cite{MR923131} and finally, in the 90's, V. Voevodsky \cite{voevodsky-triangulatedmotives} constructed what became known as "motivic cohomology". Many good introductory references to the subject are now available \cite{motivesSeattle,MR2115000, MR2242284}, together with the  historical background given in the introduction of  \cite{cisinski-tcmm}.  \\

In the late 90s, V.Voevodsky and F. Morel brought the ideas of homotopy theory to the context of algebraic geometry \cite{voevodsky-icm, voevodsky-morel,MR1693330}. Their main goal was to mimic the fruitful techniques of algebraic topology in the algebraic context. In particular, they aimed to have something resembling the stable homotopy theory of spaces - a new setting where all cohomology theories for schemes become representable. In particular, this would allow easier definitions for the motivic cohomology, the algebraic $K$-theory, algebraic cobordism, etc, by merely providing their representing spectrums. Their construction has two main steps: the first part mimics the homotopy theory of spaces and its stabilization; the second part forces the "Tate motive" to become invertible with respect to the monoidal multiplication. The final result is known as the \emph{motivic stable homotopy theory of schemes}. Our main goal in this paper is the formulation of a precise universal property for their construction.\\

\subsubsection{Noncommutative Algebraic Geometry} 
In Algebraic Geometry, and specially after the works of Serre and Grothendieck, it became a common practice to study a scheme $X$ via its abelian category of quasi-coherent sheaves $Qcoh(X)$. The reason for this is in fact purely technical for at that time, abelian categories were the only formal background to formulate cohomology theories. In fact, the object $Qcoh(X)$ turns out to be a very good replacement for the geometrical object $X$: thanks to \cite{gabrielthesis, MR1615928} we know that $X$ can be reconstructed from $Qcoh(X)$. However, it happens that abelian categories do not provide a very natural framework for homological algebra. It was Grothendieck who first noticed that this natural framework would be, what we nowadays understand as, the homotopy theory of complexes in the abelian category. At that time, the standard way to deal with homotopy theories was to consider their homotopy categories - the formal strict inversion of the weak-equivalences. This is how we obtain the derived category of the scheme $\mathit{D}(X)$. For many reasons, it was clear that the passage from the whole homotopy theory of complexes to the derived category loses to much information. The answer to this problem appeared from two different directions. First, from the theory of dg-categories \cite{bondal-kapranov3,bondal-kapranov2,bondal-kapranov}. More recently, an ultimately, with the theory of $\infty$-categories \cite{DimitriAra,bergner-survey,lurie-htt, lurie-ha, simpson-book,toen-hab}. The first subject become very popular specially with all the advances in \cite{bondal-orlov,bondal-vandenbergh,drinfeld2, drinfeld1,kaledin,tabuada-quillen,Toen-homotopytheorydgcatsandderivedmoritaequivalences}. The second, although initiated in the 80s with the famous manuscript \cite{pursuingstacks}, only in the last ten years, and specially due to the tremendous efforts of \cite{lurie-htt, lurie-ha}, has reached a state where its full potential can be explored. Both subjects provide an appropriate way to encode the homotopy theory of complexes of quasi-coherent sheaves. In fact, the two approaches are related and, for our purposes, should give equivalent answers (see our Section \ref{linkdgspectrastable}). Every scheme $X$ (over a ring $k$) gives birth to a $k$-dg-category $L_{qcoh}(X)$ - \emph{the dg-derived category of $X$} - whose "zero level" recovers the classical derived category of $X$. For reasonable schemes, this dg-category has an essential property deeply related to its geometrical origin - it has a compact generator and the compact objects are the perfect complexes (see \cite{bondal-vandenbergh} and \cite{thomasonalgebraic}). It follows that the smaller sub-dg-category $L_{pe}(X)$ spanned by the compact objects is "affine", and enough to recover the whole $L_{qcoh}(X)$.

In his works \cite{kontsevich3, kontsevich1, kontsevich2}, Kontsevich initiated a systematic study of the dg-categories with the same formal properties of $L_{pe}(X)$, with the observation that many different examples of such objects exist in nature: if $A$ is an associative algebra then $A$ can be considered as a dg-category with a single object and we consider $L(A)$ the dg derived category of complexes of $A$-modules and take its compact objects. The same works with a differential graded algebra. The \emph{Fukaya} category of a sympletic manifold is another example \cite{kontsevich-sympletic}. There are also examples coming from complex geometry \cite{francoispetit}, representation theory, matrix factorizations (see \cite{Tobias}), and also from the techniques of deformation quantization. This variety of examples with completely different origins motivated the understanding of dg-categories as natural \emph{noncommutative spaces}. The study of these dg-categories can be systematized and the assignment $X\mapsto L_{pe}(X)$ can be properly arranged as a functor

\begin{equation}
\label{fnc}
\xymatrix{
L_{pe}:\text{Classical Schemes}/k \ar[r]& \text{Noncommutative Spaces}/k
}
\end{equation}

In fact, the functor $L_{pe}$ is defined not only for schemes but for a more general class of geometrical objects, so called \emph{derived stacks} (see \cite{toen-vaquie, Be-Fr-Na,lurie-tannakian}). They are the natural geometric objects in the theory of derived algebraic geometry of \cite{toen-vezzosi-hag1, toen-vezzosi-hag2, lurie-thesis}. For the purposes of noncommutative geometry, this fact is crucial: thanks to the results of Toën-Vaquié in \cite{toen-vaquie}, at the level of derived stacks, $L_{pe}$ admits a right adjoint, providing a canonical mechanism to construct a geometric object out of a noncommutative one.\\

Kontsevich proposes also that similarly to schemes, these noncommutative spaces should also admit a motivic theory. Our second goal in this paper is to provide a natural candidate for this theory, that extends in a natural way the theory of Morel-Voevodsky. The brigde between the two theories is a canonical extension of the map $L_{pe}$ given by our universal characterization of the theory for schemes.

\subsubsection{Our Work}
The motivic construction of Morel-Voevodsky was performed using the techniques of model category theory. Nowadays we know that a model category is a mere strict presentation of a more fundamental object - an $(\infty,1)$-category. Every model category has an underlying $(\infty,1)$-category and the last is what really matters. 
It is important to say that the need for this passage overcomes the philosophical reasons. Thanks to the techniques of \cite{lurie-htt, lurie-ha} we have the ways to do and prove things which would remain out of range only with the highly restrictive techniques of model categories.\\

The first part of our quest concerns the universal characterization of the $(\infty,1)$-category underlying the stable motivic homotopy theory of schemes, as constructed by Voevodsky and Morel, with its symmetric monoidal structure. The characterization becomes meaningful if we want to compare the motives of schemes with other theories. In our case, the goal is to conceive a theory of motives for the noncommutative spaces and to relate it to the theory of Voevodsky-Morel. The universal property proved in the first part, ensures, for free, the existence of a (monoidal) dotted arrow at the motivic level

\begin{equation}
\xymatrix{
\text{Classical Schemes}/k \ar[r]\ar[d]& \text{NC-Spaces}/k\ar[d] \\
\text{Stable Motivic Homotopy}/k\ar@{-->}[r]&\text{NC-Stable Motivic Homotopy}/k
}
\end{equation}

Let us also say that other different types of motivic theories for dg-categories are already known (see \cite{tabuada-higherktheory,tabuada-cisinski,tabuada-classicalvsnoncommutative} and \cite{tabuada-notes} for a pedagogical overview). These approaches are essentially of "`cohomological nature"'. Our method should be said "`homological" and follows the spirit of stable homotopy theory. Its main advantage is the canonical way in which we extract the dotted monoidal map. In general, this kind of monoidal maps are extremely hard to obtain only by constructive methods and the techniques of model category theory. Other important advantage is that it allows us to work over any base scheme, not necessarily a field.\\

To achieve the universal characterization we will need to rewrite the constructions of Morel-Voevodsky in the setting of $\infty$-categories. The dictionary between the two worlds is given by the techniques of \cite{lurie-htt} and \cite{lurie-ha}. In fact, \cite{lurie-htt} already contains all the necessary results for the characterization of the \emph{$\mathbb{A}^1$-homotopy theory of schemes} and its \emph{stable non-motivic} version. The problem concerns the  description of the stable motivic world with its symmetric monoidal structure. This is our main contribution. The key ingredient is the following:\\

\begin{insight}(see the Theorem \ref{maintheorem} for the precise formulation):\\
Let $\V$ be a combinatorial simplicial symmetric monoidal model category with a cofibrant unit and let $\Cmonoidal$ denote its underlying symmetric monoidal $\infty$-category. Let $X$ be a cofibrant object in $\V$ satisfying the following condition:\\

(*) the cyclic permutation of factors $\sigma:X\otimes X\otimes X\to X\otimes X\otimes X$ is equal to the identity map in the homotopy category of $\V$.\footnote{More precisely we demand the existence of an homotopy in $\V$ between the cyclic permutation and the identity.}\\

Then the underlying symmetric monoidal $\infty$-category of $Sp^{\Sigma}(\V,X)$ is the universal symmetric monoidal $(\infty,1)$-category equipped with a monoidal map from $\Cmonoidal$, sending $X$ to an invertible object.\\
\end{insight}

This extra assumption on $X$ is not new. It is already present in the works of Voevodsky (\cite{voevodsky-icm}) and it also appears in \cite{hovey-spectraandsymmetricspectra}. We must point out that we believe our result to be true even without this extra assumption on $X$. We will explain this in the Remark \ref{withoutassumption}.\\

\begin{cor}(Corollary \ref{universalpropertymotives})
Let $S$ be a base scheme and let $\sch$ denote the category of smooth separated schemes of finite type over $S$. The $(\infty,1)$-category $\st$ underlying the stable motivic homotopy theory of schemes is stable, presentable and admits a canonical symmetric monoidal structure $\stmonoidal$. Moreover, the construction of Morel-Voevodsky provides a functor $\sch^{\times}\to \stmonoidal$ monoidal with respect to the cartesian product of schemes, and endowed with the following universal property: \\

$(*)$ for any pointed presentable symmetric monoidal $(\infty,1)$-category $\Dmonoidal$, the composition map \footnote{see the notations in \ref{notations}}

\begin{equation}
Fun^{\otimes, L}(\stmonoidal, \Dmonoidal)\to Fun^{\otimes}(\sch^{\times}, \Dmonoidal)
\end{equation}

\noindent is fully faithful and its image consists of those monoidal functors $\sch^{\times}\to \Dmonoidal$ satisfying Nisnevich descent, $\mathbb{A}^1$-invariance and such that the cofiber of the image of the point at $\infty$, $\xymatrix{S\ar[r]^{\infty} &\mathbb{P}^1}$ is an invertible object in $\Dmonoidal$. Moreover, any pointed presentable symmetric monoidal $(\infty,1)$-category $\Dmonoidal$ admitting a monoidal map in this image, is necessarily stable.
\end{cor}

This result trivializes the problem of finding motivic monoidal realizations.

\begin{example}
Let $S=Spec(k)$ be field of characteristic zero. The assignment $X\mapsto \Sigma^{\infty}(X(\mathbb{C}))$ provides a functor $\sch\to \Sp$ with $\Sp$ the $(\infty,1)$-category of spectra (see below). This map is known to be monoidal, to satisfy all the descent conditions in the previous corollary and to invert $\mathbb{P}^1$ in the required sense.  Therefore, it extends in a essential unique way to a monoidal map of stable presentable symmetric monoidal $(\infty,1)$-categories $\stmonoidal\to \Spmonoidal$;
\end{example}

\begin{example}
Again, let $S=Spec(k)$. Another immediate example of a monoidal motivic realization is the Hodge realization. Properly constructed, the map $X\mapsto C_{DR}(X)$ sending a scheme to its De Rham complex provides a functor  $\sch\to \derivedk$ with $\derivedk$ the $(\infty,1)$-derived category of $k$. This map is known to be monoidal with respect to the cartesian product of schemes (Kunneth formula), satisfies all the descent conditions and inverts $\mathbb{P}^1$ in the sense above. Because of the universal characterization, it extends in a essential unique way to a monoidal motivic Hodge Realization $\stmonoidal\to \mathcal{D}(k)^{\otimes}$ (where on the left we have the monoidal structure induced by the derived tensor product of complexes). See \cite{brad-thesis} for the proper construction of this map.
\end{example}

As a main application, we systematize the comparison between the commutative and noncommutative worlds. After some preliminairs on dg-categories, we introduce the $(\infty,1)$-category of smooth noncommutative spaces $\nck$ as the opposite of the $(\infty,1)$-category of idempotent dg-categories of finite type $\dg^{ft}\subseteq \dg^{idem}$ introduced by Toën-Vaquié in \cite{toen-vaquie}. By introducing an appropriate analogue for the Nisnevich topology, we construct a new stable presentable symmetric monoidal $(\infty,1)$-category $\stncmonoidal$ encoding a stable motivic homotopy theory for these noncommutative spaces. To conclude, we explain how to encode the map $X\mapsto L_{pe}(X)$ as a functor $L_{pe}$ towards $\nck$ and how our universal characterization of the stable motivic homotopy theory of schemes allows us to extend it to a monoidal colimit preserving functor

\begin{equation}
\stmonoidalk\to \stncmonoidalk
\end{equation}

\subsubsection{Further Applications}
To conclude let us provide another application for our work. Thanks to the famous theorem HKR, the \emph{Periodic Cyclic Homology $HP_{\bullet}(X)$} provides the correct noncommutative analogue of the classical de Rham cohomology. In \cite{kontsevich-pantev-katzarkov}, the authors introduced the notion of a \emph{noncommutative Hodge Structure}. They formulate the following conjecture:

\begin{enumerate}[(*)]
\item If $X$ is a "good enough" noncommutative space then $HP_{\bullet}(X)$ carries a noncommutative Hodge-Structure;
\end{enumerate}

Said in a different way, $HP_{\bullet}$ should provide a functor from noncommutative spaces to noncommutative Hodge-structures. We should then expect this functor to factor through our new noncommutative version of the motivic stable homotopy theory because of its universal property. More generally, we expect our main commutative diagram to fit in a larger one

\begin{equation}
\xymatrix{
\text{Classical Schemes}/k \ar[r]^{L_{pe}}\ar[d]\ar@/^/[rdd]^(.3){H_{DR}(-)}& \text{NC-Schemes}/k\ar[d] \ar@/^/[rdd]^(.3){HP_{\bullet}(-)}&\\
\text{Stable Motivic Homotopy}/k\ar@{-->}[r]\ar@{-->}[dr]^{univ}_{prop.}&\text{NC-Stable Motivic Homotopy}/k\ar@{-->}[dr]^{univ}_{prop.}& \\
&\text{Classical Hodge-Structures}\ar[r]&\text{NC-Hodge Structures}
}
\end{equation}

\noindent where the map from the classical to the noncommutative Hodge structures was introduced in \cite{kontsevich-pantev-katzarkov}\footnote{Of course we should only expected the part of the diagram concerning the Hodge Theory to work if we restrict to a good class of schemes over $k$}. The diagonal maps are known as the \emph{Hodge-realizations functors}: the commutative case is known to the experts (see \cite{riou-realizationfunctors} for a survey of the main results); the noncommutative case is given by the conjecture $(*)$. This conjecture can be divided in two parts: the first concerns the de Rham part (see \cite{MR2435845}) and the second is related to the Betti part. Some recent progress towards this last part is now available in \cite{Anthony-thesis}.

\subsection{Acknowledgments and Credits}
This paper is the first part of my ongoing Doctoral Thesis at the Université de Montpellier under the direction of Bertrand Toën. I want to express and emphazise my sincere admiration, gratitude and mathematical debt to him. 
For accepting me as his student, for proposing me such an amazing quest, and for, so kindly, sharing and discussing his ideas and beautiful visions with me. Moreover, I want to thank him for all the comments and suggestions improving the text here presented.\\

The story of the ideas and motivations for this work started in the spring of 2010 when they were discussed in the first ANR-Meeting - "GAD1" - in  Montpellier (ANR-09-BLAN-0151) . The discussion continued later in the summer of 2010, when a whole research project was envisioned and discussed by Bertrand, together with Gabriele Vezzosi, Michel Vaquié and Anthony Blanc. I am grateful to all of them for allowing me to dive into this amazing vision and to pursue the subject.\\

I also want to acknowledge the deep influence of the colossal works of Jacob Lurie in the subject of higher algebra. I have learned a lot from his writings and of course, this work depends heavily and continuously on his results and techniques.\\

I'm also grateful to Denis-Charles Cisinski for a very helpful conversation on the subject of motives and for explaining me his ongoing joint work with Gonçalo Tabuada.\\

I also wish to express my deep gratitude to Dimitri Ara and Georges Maltsiniotis. For all the mathematical discussions from which I have learned so much and for the friendly reception in Paris. It has also been a
pleasure to have Anthony Blanc and Benjamin Hennion as comrades in this Quest, and I want to thank them
for the all discussions along the way. Moreover, I want to acknowledge the friendly environment I found in
the Université de Montpellier.\\

Finally, for the continuous support, encouragement and love, I want to dedicate this work to Maria Ieshchenko.

\subsection{Outline of the paper}

In Section \ref{section1} we set our notations and review the main notions and tools from Higher Category Theory, together with the mechanism to pass from the world of model categories to $(\infty,1)$-categories. These tools will be used all along the paper. Section \ref{section3} is dedicated to the subject of higher algebra: following \cite{lurie-ha}, we summarize the theory of symmetric monoidal $(\infty,1)$-categories, their algebra objects and the associated theories of modules. We also add some original preliminary results and remarks needed in the following sections. The reader familiar with the subject can skip this section and consult these results later on. Section \ref{section4} forms the core of our paper. In \ref{section4-1} and following some ideas of \cite{toen-vezzosi-hag2}, we deal with the formal inversion of an object $X$ in a symmetric monoidal $(\infty,1)$-category. First we treat the case of \emph{small symmetric monoidal $(\infty,1)$-categories} and then we extend the results to the \emph{presentable} setting. In \ref{section4-2} we recall the classical notion of (ordinary) spectra, which can be defined either via a limit kind of construction or via a colimit. When applied to a presentable $(\infty,1)$-category both methods coincide. Still in this section, we recall a classical theorem (see \cite{voevodsky-icm}) which says that, under a certain symmetric condition on $X$, the formal inversion of an object in a symmetric monoidal category is equivalent to the (ordinary) category of spectra with respect to $X$. In the Corollary \ref{main5} we prove that this results also holds in the $\infty$-setting.
In \ref{section4-3} we use the results of \cite{hovey-spectraandsymmetricspectra} to compare our formal inversion
 to the more familiar notion of symmetric spectra and we prove (Theorem \ref{maintheorem}) that, again under the same extra assumption on $X$, both notions coincide. 
In Section \ref{section5} we use the results of \ref{section4}, together with the techniques of \cite{lurie-htt, lurie-ha}, to completely characterize the $\mathbb{A}^1$-homotopy theory of schemes and its associated motivic stabilization by means of a universal property inside the world of symmetric monoidal $(\infty,1)$-categories.
Finally in the Section \ref{section6}, we construct a motivic $\mathbb{A}^1$-homotopy theory for the noncommutative spaces of Kontsevich. As a corollary of the universal characterization we have, for free, a canonical monoidal map comparing the classical theory for schemes and the new noncommutative motivic side.

\section{Preliminaries I: Higher Category Theory}
\label{section1}

\subsection{Notations and Categorical Preliminaries}
\label{notations}

\subsubsection{Quasi-Categories} 
The theory of $(\infty,1)$-categories has been deeply explored over the last years and we now have many different models to access them. In this article we follow the approach of \cite{lurie-htt,lurie-ha}, using the model provided by Joyal's theory of Quasi-Categories \cite{joyal-article}. In this sense, the two notions of \emph{quasi-category} and \emph{$(\infty,1)$-category} will be identified throughout this paper. Recall that the Joyal's model structure is a combinatorial, left proper, cartesian closed model structure in the category of simplicial sets $\ssets$, for which the cofibrations are the monomorphisms and the fibrant objects are the quasi-categories - by definition, the simplicial sets $\C$ with the lifting property

\begin{equation}
\xymatrix{
\Lambda^k[n]\ar@{^{(}->}[d]\ar[r]^f& \C\\
\Delta[n]\ar@{-->}[ru]&
}
\end{equation}

\noindent for any inclusion of an inner-horn $\Lambda^k[n]\subseteq \Delta[n]$ with $0<k<n$ and any map $f$.\\

For a quasi-category $\C$, we will follow \cite{lurie-htt} and write $Obj(\C)$ for the set zero-simplexes of $\C$; given two objects $X, Y \in Obj(\C)$ we let $Map_{\C}(X,Y)$ denote the simplicial set \emph{Mapping Space between them} and finally we let $h(\C)$ denote the homotopy category of $\C$. Moreover, as in \cite{joyal-article, lurie-htt} the term \emph{categorical equivalence} will refer to a weak-equivalence of simplicial sets for the Joyal's model structure.\\ 

\subsubsection{Universes}
In order to deal with the set-theoretical issues we will follow the approach of Universes (our main reference being the Appendix by Nicolas Bourbaki in \cite{sga4}). We will denote them as $\uniU$, $\uniV$, $\uniW$, etc. Moreover, we  adopt a model for set theory where every set is \emph{artinian}. In this case, for every strongly inaccessible cardinal $\kappa$\footnote{Recall that a cardinal $\kappa$ is called strongly inaccessible if it is regular (meaning, the sum $\sum_{i}\alpha_i$ of strictly smaller cardinals $\alpha_i<\kappa$ with $i\in I$ and $card(I)<\kappa$, is again strictly smaller than $\kappa$, which is the same as saying that $\kappa$ is not the sum of cardinals smaller than $\kappa$) and if for any strictly smaller cardinal $\alpha<\kappa$, we have $2^{\alpha}<\kappa$}, the collection $\uniU(\kappa)$ of all sets of rank $< \kappa$\footnote{Recall that a set $X$ is said to have rank smaller than $\kappa$ if the cardinal of $X$ is smaller than $\kappa$ and for any succession of memberships $X_n\in X_{n-1}\in ... \in X_0= X$, every $X_i$ has cardinal smaller than $\kappa$.} is a set and satisfies the axioms of a Universe. The correspondence $\kappa\mapsto \uniU(\kappa)$ establishes a bijection between strongly inaccessible cardinals and Universes, with inverse given by $\uniU\mapsto card(\uniU)$. We adopt the \emph{axiom of Universes} which allows us to consider every set as a member of a certain universe (equivalently, every cardinal can be strictly upper bounded by a strongly inaccessible cardinal). We will also adopt the \emph{axiom of infinity}, meaning that all our universes will contain the natural numbers $\mathbb{N}$ and therefore $\mathbb{Z}$, $\mathbb{Q}$, $\mathbb{R}$ and $\mathbb{C}$.  Whenever necessary we will feel free to enlarge the universe $\uniU\in \uniV$. This is possible by the axiom of Universes. \\

Let $\uniU$ be an universe. As in \cite{sga4} we say that a mathematical object $T$ is $\uniU$-small (or simply, \emph{small}) if all the data defining $T$ is collected by sets isomorphic to elements of $\uniU$. For instance, a set is $\uniU$-small if it is isomorphic to a set in $\uniU$; a category is $\uniU$-small if both its collection of objects and morphisms are isomorphic to sets in $\uniU$; a simplicial set $X$ is $\uniU$-small if all its level sets $X_i$ are isomorphic to elements in $\uniU$, etc. A mathematical object $T$ is called \emph{essentially small} if is equivalent (in a context to be specified) to a $\uniU$-small object. A category $\C$ is called  \emph{locally $\uniU$-small} (resp. locally essentially $\uniU$-small) if its hom-sets between any two objects are $\uniU$-small (resp. essentially small).\footnote{Notice that this definition is not demanding any smallness condition on the collection of objects and therefore a locally small category does not need to be small} We define the category of $\uniU$-sets as follows: the collection of objects is $\uniU$ and the morphisms are the functions between the sets in $\uniU$. It is locally small. Another example is the category of $\uniU$-small categories $\Cat_{\uniU}$ whose objects are the $\uniU$-small categories and functors between them. Another important example is given by $\ssets_{\uniU}$ the category of $\uniU$-small simplicial sets. Again, it is locally small and, together with the Joyal's model structure (see \cite{joyal-article}) it forms a $\uniU$-combinatorial model category (in the sense of \cite{toen-vezzosi-hag1}) and its cofibrant-fibrant objects are the $\uniU$-small $(\infty,1)$-categories.

Consider now an enlargement of universes $\uniU\in \uniV$. In this case, it follows from the axiomatics that every $\uniU$-small object is also $\uniV$-small. With a convenient choice for $\uniV$, the collection of all $\uniU$-small $(\infty,1)$-categories can be organized as a $\uniV$-small $(\infty,1)$-category, $\iCat^{\uniU}$ (See \cite{lurie-htt}-Chapter 3 for the details). We have a canonical inclusion $\ssets_{\uniU} \subseteq \ssets_{\uniV}$ which is compatible with the Joyal's Model structure. Again, through a convenient enlargement of the universes $\uniU\in \uniV\in \uniW$, we have an inclusion of $\uniW$-small $(\infty,1)$-categories $\iCat^{\uniU}\subseteq \iCat^{\uniV}$\footnote{$\iCat^{\uniU}$ is $\uniV$-small and so it is also $\uniW$-small}. We say that a $\uniV$-small $(\infty,1)$-category is \emph{essentially $\uniU$-small} if it is weak-equivalent in $\ssets_{\uniV}$ to a $\uniU$-small simplicial set. Thanks to \cite[5.4.1.2]{lurie-htt}, the following conditions are equivalent for a $\uniV$-small $(\infty,1)$-category $\C$: (i) $\C$ is essentially $\uniU$-small; $(ii)$  $card(\pi_0(\C))< card(\uniU)$ and $\C$ is locally small, which means that for any two objects $X$ and $Y$ in $\C$, we have $card(\pi_i(Map_{\C}(X,Y)))< card(\uniU)$; $(iii)$ $\C$ is a $card(\uniU)$-compact object in $\iCat^{\uniV}$ (see \ref{filteredcompact} below).\\ 

Some constructions require us to control "how small" our objects are. Given a cardinal $\tau$ in the universe $\uniU$, we will say that a small simplicial set $K$ is $\tau$-small if a fibrant-replacement $\C$ of $K$ satisfies the conditions above, replacing $card(\uniU)$ by $\tau$.\\

The category of $\uniU$-small simplicial sets can also be endowed with the standard Quillen model structure (see \cite{hovey-modelcategories}) and it forms a $\uniU$-combinatorial simplicial model category in which the fibrant-cofibrant objects are the $\uniU$-small \emph{Kan-complexes}. They provide models for the homotopy types of $\uniU$-small spaces and following the ideas of the Section \ref{section1-2} we can collected them in a new $(\infty,1)$-category $\Spaces_{\uniU}$. Again we can enlarge the universe $\uniU\in\uniV$ and produce inclusions $\Spaces_{\uniU}\subseteq \Spaces_{\uniV}$.\\

Throughout this paper we will fix three universes $\uniU\in \uniV\in\uniW$ with $\uniV$ chosen conveniently large and $\uniW$, very large. In general, we will work with the $\uniV$-small simplicial sets and the $\uniU$-small objects will be refered to simply as \emph{small}. In order to simplify the notations we write $\iCat$ (resp. $\Spaces$) to denote the $(\infty,1)$-category of small $(\infty,1)$-categories (resp. spaces). With our convenient choice for $\uniV$, both of them are $\uniV$-small. The third universe $\uniW$ is assumed to be sufficiently large so that we have $\uniW$-small simplicial sets  $\iCatbig$ (resp. $\Spaces^{big}$) to encode the $(\infty,1)$-category of all the $\uniV$-small $(\infty,1)$-categories (resp. spaces).\\

\subsubsection{Fibrations of Simplicial Sets}

Let $p:X\to Y$ map of simplicial sets. We say $p$ is a \emph{trivial fibration} if it has the right-lifting property with respect to every monomorphism of simplicial sets. We say $p$ is a categorical fibration if it is a fibration for the Joyal's model structure. We say $p$ is an $inner fibration$ if it has the right-lifting property with respect to every inclusion $\Lambda^k[n]\subseteq \Delta[n]$, with $0<k<n$. We have

\begin{equation}
\{\text{trivial fibrations}\}\subseteq \{\text{categorical fibrations}\}\subseteq \{\text{inner fibrations}\}
\end{equation}

\subsubsection{Categories of Functors}
 The Joyal's model structure is cartesian closed (see \cite{joyal-article} or the Corollary 2.3.2.4 of \cite{lurie-htt}). In particular, if $\C$ and $\D$ are $(\infty,1)$-categories in a certain universe, the internal-hom in $\ssets$, $Fun(\C,\D):=\underline{Hom}_{\ssets}(\C,\D)$ is again an $(\infty,1)$-category in the same universe (See Prop. 1.2.7.3 of \cite{lurie-htt}). It provides the good notion of $(\infty,1)$-category of functors between $\C$ and $\D$;

\subsubsection{Diagrams}
Let $\C$ be an $(\infty,1)$-category and let $K$ be a simplicial set. A diagram in $\C$ indexed by $K$ is a map of simplicial sets $d:K\to \C$. We denote by $K^{\rhd}$ (resp. $K^{\lhd}$) the simplicial set $K*\Delta[1]$ (resp. $\Delta[1]*K$) where $*$ is the join operation of simplicial sets (see \cite{lurie-htt}-Section 1.2.8).\\

If $\C$ is an $(\infty,1)$-category, a commutative square in $\C$ is a diagram $d:\Delta[1]\times \Delta[1]\to \C$. This is the same as a map $\Lambda^0[2]^{\rhd}\to \C$. It is easy to check that $\Delta[1]\times \Delta[1]$ has four $0$-simplexes $A$,$B$,$C$, $D$; five non-degenerated $1$-simplexes $f$,$g$,$h$,$u$,$v$ and four non-degenerated $2$-simplexes $\alpha$, $\beta$, $\gamma$, $\sigma$, which we can picture together as

\begin{equation}
\xymatrix{
A\ar@{=}[r]^{id} \ar[dr]|*{}="N" \ar[d]_u & \ar@{=>};"N"^\gamma A \ar[dr]|*{}="M" \ar[r]^f \ar[d] & \ar@{=>};"M"^\sigma B \ar[d]^g\\
D\ar[r]^v \ar@{=>};"N"^\alpha & C \ar@{=>};"M"_\beta \ar@{=}[r]_{id} & C
}
\end{equation}

\noindent where all the inners $1$-simplexes are given by $h$. 
Since $\C$ is an $(\infty,1)$-category, we can use the lifting property to show that the data of a commutative diagram in $\C$ is equivalent to the data of two triangles

\begin{equation}
\xymatrix{ 
A \ar[d]_{u} \ar[dr]^{v\circ u}|*{}="N" & &&A \ar[r]^{f} \ar[dr]_{g\circ f}|*{}="M" & B\ar[d]^g  \\ 
D\ar[r]^v \ar@{=>};"N"^\beta & C && & C \ar@{=>}[u];"M"^\alpha                          }
\end{equation}

\noindent together with a map $r:A\to C$ and two-cells providing homotopies between  $g\circ f\simeq r\simeq v\circ u$.\\ 

\subsubsection{Comma-Categories}
If $\C$ is an $(\infty,1)$-category and $X$ is an object in $\C$, there are $(\infty,1)$-categories $\C_{/X}$ and $\C_{X/}$ where the objects are, respectively, the morphisms $A\to X$  and $X\to A$. More generally, if $p:K\to \C$ is a diagram in $\C$ indexed by a simplicial set $K$, there are $(\infty,1)$-categories $\C_{/p}$ and $\C_{p/}$ of cones (resp cocones) over the diagram. These $(\infty,1)$-categories are characterized by an universal property - see for instance \cite{lurie-htt}-Prop.1.2.9.2.

\subsubsection{Limits and Colimits}
Let $\C$ be an $(\infty,1)$-category. An object $E:\Delta[0]\to \C$ is said to be \emph{initial} (resp. \emph{final}) if for every object $Y$ in $\C$ the mapping space $Map_{\C}(E,Y)$ (resp. $Map_{\C}(Y,E)$) is weakly contractible (see the Definitions 1.2.12.1 and 1.2.12.3 and the Corollary 1.2.12.5 of \cite{lurie-htt}).\\

Let $\C$ be an $(\infty,1)$-category and let $K\to \C$ be a diagram in $\C$. A \emph{colimit} (resp. \emph{limit}) for a diagram $d:K\to \C$ is an initial (resp. final) object in the category $\C_{p/}$ (resp. $\C_{/p}$). By the universal property defining the comma-categories, this corresponds to the data of a new diagram $\bar{d}:K^{\rhd}\to \C$ (resp. $K^{\lhd}\to \C$) extending $d$ and satisfying the universal property of \cite[1.2.13.5]{lurie-htt}. Whenever appropriate, we will also use the relative notions of limits and colimits (see \cite[4.3.1.1]{lurie-htt}).\\

Following \cite[4.1.1.1, 4.1.1.8]{lurie-htt}, we say that a map of simplicial sets $\phi:K'\to K$ is \emph{cofinal} if for every $(\infty,1)$-category $\C$ and every colimit diagram $K^{\lhd}\to \C$, the composition with $\phi$, $(K')^{\lhd}\to \C$ remains a colimit diagram.\\

We will say that an $(\infty,1)$-category \emph{has all small colimits} (resp. \emph{limits}) if every diagram in $\C$ indexed by a small simplicial set has a colimit (resp. limit) in $\C$. As in the classical situation, $\C$ has all $\kappa$-small colimits (resp. limits) if and only if it has all $\kappa$-small coproducts and all pushouts exist \cite[4.4.2.6]{lurie-htt}. In particular, it has an initial (resp. final) object.\\

If $\C$ is an $(\infty,1)$-category having colimits of a certain kind, then for any simplicial set $S$, the $(\infty,1)$-category $Fun(S,\C)$ has colimits of the same kind and they can computed objectwise in $\C$ \cite[5.1.2.3]{lurie-htt}.\\

If $\C$ and $\D$ are $(\infty,1)$-categories with colimits we will denote by $Fun^L(\C,\D)$ the full subcategory of $Fun(\C,\D)$ spanned by those functors which commute with colimits.\\

We say that an ($\infty,1)$-category is \emph{pointed} it it admits an object which is simultaneously initial and final. Given an arbitrary $(\infty,1)$-category with a final object $*$, we consider the comma-category $\C_*:=\C_{*/}$. This is pointed. Moreover there is a canonical forgetful morphism $\C_*\to \C$ which commutes with limits.

\subsubsection{Subcategories}
 If $\C$ is an $(\infty,1)$-category, and $\mathcal{O}$ is a subset of objects and $\mathcal{F}$ is a subset of edges between the objects in $\mathcal{O}$, the \emph{subcategory of $\C$} spanned by the objects in $\mathcal{O}$ together with the edges in $\mathcal{F}$ is the new $(\infty,1)$-category $\C_{\mathcal{O}, \mathcal{F}}$ obtained as the pullback of the diagram

\begin{equation}
\xymatrix{
\C_{\mathcal{O}, \mathcal{F}}\ar[r]^i\ar[d]& \C\ar[d] \\
N(h(\C)_{\mathcal{O}, \mathcal{F}})\ar[r]&N(h(\C))
}
\end{equation}

\noindent where the lower map is the nerve of the inclusion of the subcategory $h(\C)_{\mathcal{O}, \mathcal{F}}$ of $h(\C)$, spanned by the objects in $\mathcal{O}$ together with the morphisms in $h(\C)$ represented by the edges in $\mathcal{F}$. The right-vertical map is the unit of the adjunction $(h, N)$. It follows immediately from the definition that $\C_{\mathcal{O}, \mathcal{F}}$ will also be an $(\infty,1)$-category.

\subsubsection{Grothendieck Construction}
We recall the existence of a \emph{Grothendieck Construction} for $(\infty,1)$-categories (See \cite{lurie-htt}-Chapter 3). Thanks to this, we can present a functor between two $(\infty,1)$-categories $\C\to \D$ as a \emph{cocartesian fibration} (see \cite{lurie-htt}-Def. 2.4.2.1) $p:\M\to \Delta[1]$ with $p^{-1}(\{0\})= \C$ and $p^{-1}(\{1\})=\D$. Using this machinery, the data of an adjunction between $\C$ and $\D$ corresponds to a \emph{bifibration} $\M\to \Delta[1]$ (see the proof of \cite[5.2.1.4]{lurie-htt} to understand how to extract a the pair of functors our a bifibration, using the model structure on marked simplicial sets);

\subsubsection{Localizations}
\label{locinfinity}
There is a theory of localizations for $(\infty,1)$-categories. If $(\C,W)$ is an $(\infty,1)$-category together with a class of morphisms $W$ we can produce a new $(\infty,1)$-category $\C[W^{-1}]$ together with a map $\C\to \C[W^-1]$ with the universal property of sending the edges in $W$ to equivalences. To construct this localization we can make use of the model structure on the marked simplicial sets of \cite{lurie-htt}-Chapter 3. Recall that every marked simplicial sets is cofibrant and the fibrant ones are precisely the pairs $\C^{\sharp}:=(\C, eq)$ with $\C$ a quasi-category and $eq$ the collection of all equivalences in $\C$. Therefore, $\C[W^{-1}]$ can be obtained as a fibrant-replacement of the pair $(\C,W)$. We recover the desired universal property from the fact that the marked structure is simplicial. Following the Construction 4.1.3.1 of \cite{lurie-ha}, this procedure can be presented in more robust terms. It is possible to construct an $(\infty,1)$-category $\mathcal{W}\iCat$ where the objects are the pairs $(\C,W)$ with $\C$ a quasi-category and $W$ a class of morphisms in $\C$. Moreover, the mapping $\C\mapsto \C^{\sharp}$ provides a fully faithful functor 

\begin{equation}
\iCat\subseteq \mathcal{W}\iCat
\end{equation}

\noindent and the upper localization procedure $(\C,W)\mapsto \C[W^{-1}]$ provides a left adjoint to this inclusion (see the Proposition 4.1.3.2 of \cite{lurie-ha}.\\

Let $\C$ be an $(\infty,1)$-category and let $\C_0$ be a full subcategory of $\C$. We say that $\C_0$ is a \emph{reflexive localization of $\C$} if the fully faithful inclusion $\C_0\subseteq \C$ admits a left adjoint $L:\C\to \C_0$. A reflexive localization is a particular instance of the notion in the previous item, with $W$ the class of edges in $\C$ which are sent to equivalences through $L$ (see the Proposition 5.2.7.12 of \cite{lurie-ha});

\subsubsection{Presheaves}
 If $\C$ is a small $(\infty,1)$-category, we define the $(\infty,1)$-category of $\infty$-presheaves over $\C$ as $\mathcal{P}(\C):=Fun(\C^{op}, \Spaces)$. It is not small anymore because $\Spaces$ is not small. It comes naturally equipped with a fully faithful analogue of the classical Yoneda's map $\C\to \mathcal{P}(\C)$, endowed with the following universal property: for any $(\infty,1)$-category $\D$ having all colimits indexed by small simplicial sets, the composition 

\begin{equation}
Fun^L(\mathcal{P}(\C), D)\to Fun(\C, \D)
\end{equation}

\noindent induces an equivalence of $(\infty,1)$-categories, where the left-side denotes the full-subcategory of all colimit preserving functors (see Theorem 5.1.5.6 of \cite{lurie-htt}).

\subsubsection{$\kappa$-filtered categories and $\kappa$-compact objects}\label{filteredcompact}
 Let $\kappa$ be a small cardinal. A simplicial set $S$ is called $\kappa$-filtered if there is an $(\infty,1)$-category $\C$ together with a categorical equivalence $\C\to S$, such that for any $\kappa$-small simplicial set $K$, any diagram $K\to \C$ admits a cocone $K^{\triangleright}\to \C$ (see the Notation 1.2.8.4 of \cite{lurie-htt}). We use the terminology \emph{filtered} when $\kappa=\omega$. Notice that if $\kappa\leq \kappa'$ and $\C$ is $\kappa'$-filtered then it is also $\kappa$-filtered.\\
 
It follows from \cite[4.2.3.11]{lurie-htt} that an $(\infty,1)$-category $\C$ has all small colimits iff there exists a regular cardinal $\kappa$ such that $\C$  has all $\kappa$-small colimits together with all $\kappa$-filtered colimits.\\

Let object $X$ in a big $(\infty,1)$-category $\C$. We say that $X$ is \emph{completely compact} if the associated map $Map_{\C}(X,-):\C\to \Spaces^{big}$ commutes with all small colimits. We say that $X$ is $\kappa$-compact (for $\kappa$ a small regular cardinal) if $Map_{\C}(X,-)$ commutes with colimits indexed by $\kappa$-filtered simplicial sets. We denote by $\C^{\kappa}$ the full subcategory of $\C$ spanned by the $\kappa$-compact objects in $\C$. We use the terminology \emph{compact} when $\kappa=\omega$. Notice that if $\kappa\leq \kappa'$ and $X$ is $\kappa$-compact it is also $\kappa'$-compact.

\subsubsection{Ind-Completion}
 Let $\C$ be a small $(\infty,1)$-category and choose a regular cardinal $\kappa$ with $\kappa<card(\uniU)$. Following the results of \cite{lurie-htt}- Section 5.3.5, it is possible to formally complete $\C$ with all small colimits indexed by small $\kappa$-filtered simplicial sets. More precisely, we can construct a new $(\infty,1)$-category $Ind_{\kappa}(\C)$ (which is not small anymore), together with a canonical map $\C\to Ind_{\kappa}(\C)$ having the following universal property: for any $(\infty,1)$-category $\D$ having all colimits indexed by a small $\kappa$-filtered simplicial set, the composition 

\begin{equation}
Fun^{\kappa}(Ind_{\kappa}(\C), D)\to Fun(\C, \D)
\end{equation}

\noindent induces an equivalence of $(\infty,1)$-categories, where the left-side denotes the full-subcategory spanned by the functors commuting with colimits indexed by a $\kappa$-filtered small simplicial set (see Theorem 5.3.5.10 of \cite{lurie-htt}). In the case $\kappa=\omega$ we write $Ind(\C):= Ind_{\omega}(\C)$.

\subsubsection{Completion with colimits}
\label{completionwithcolimits}
Following the ideas of \cite{lurie-htt}-Section 5.3.6, given an arbitrary $(\infty,1)$-category $\C$ together with a collection $\mathcal{K}$ of arbitrary simplicial sets and a collection of diagrams $\mathcal{R}=\{p_i:K_i\to \C\}$ with each $K_i\in \mathcal{K}$, we can form a new $(\infty,1)$-category $\mathcal{P}^{\mathcal{K}}_{\mathcal{R}}(\C)$ together with a canonical map $\C\to \mathcal{P}^{\mathcal{K}}_{\mathcal{R}}(\C)$ such that for any $(\infty,1)$-category $\D$, the composition map

\begin{equation}
\label{formulaaddingcolimits}
Fun_{\mathcal{K}}(\mathcal{P}^{\mathcal{K}}_{\mathcal{R}}(\C),\D)\to Fun_{\mathcal{R}}(\C,\D)
\end{equation}

\noindent is an equivalence of $(\infty,1)$-categories, where the left-side denotes the full subcategory of $\mathcal{K}$-colimit preserving functors and the right-side denotes the full-subcategory of functors sending diagrams in the collection $\mathcal{R}$ to colimit diagrams in $\D$. This allows us to formally adjoint colimits of a given type to a certain $(\infty,1)$-category. We denote by $\iCatbig(\mathcal{K})$ the (non-full) subcategory of $\iCatbig$ spanned by the $(\infty,1)$-categories which admit all the colimits of diagrams indexed by simplicial sets in $\mathcal{K}$, together with the $\mathcal{K}$-colimit preserving functors. The intersection $\iCatbig(\mathcal{K})\cap \iCat$ is denoted as $\iCat(\mathcal{K})$. In the particular case when $\mathcal{K}$ is the collection of $\kappa$-small simplicial sets, we will use the notation
$\iCat(\kappa)$.\\

If $\mathcal{K}\subseteq \mathcal{K}'$ are two collections of arbitrary simplicial sets and $\C$ is an arbitrary $(\infty,1)$-category having all $\mathcal{K}$-indexed colimits, we can let $\mathcal{R}$ be the collection of all $\mathcal{K}$-colimit diagrams in $\C$. The result of the previous paragraph $\mathcal{P}^{\mathcal{K}}_{\mathcal{R}}(\C)$ will in this particular case, be denoted as $\mathcal{P}^{\mathcal{K}'}_{\mathcal{K}}(\C)$. By ignoring the set-theoretical aspects, the universal property defining $\mathcal{P}^{\mathcal{K}'}_{\mathcal{K}}(\C)$ allows us to understand the formula $\C\mapsto \mathcal{P}^{\mathcal{K}'}_{\mathcal{K}}(\C)$ as an informal left adjoint to the canonical (non-full) inclusion of the collection of $(\infty,1)$-categories with all the $\mathcal{K}'$-indexed colimits together with the $\mathcal{K}'$-colimit preserving functors between them, into the collection of $(\infty,1)$-categories with all the $\mathcal{K}$-indexed colimits together with the $\mathcal{K}$-colimit preserving functors.

By combining the universal properties, we find that if $\mathcal{K}$ is the empty collection and $\mathcal{K}'$ is the collection of all small simplicial sets, $\mathcal{P}^{\mathcal{K}'}_{\mathcal{K}}(\C)$ is simply given by $\mathcal{P}(\C)$. In the case $\mathcal{K}$ is the empty collection and $\mathcal{K}'$ is the collection of all $\kappa$-small filtered simplicial sets (for some small cardinal $\kappa$), we obtain an equivalence $\mathcal{P}^{\mathcal{K}'}_{\mathcal{K}}(\C)\simeq Ind_{\kappa}(\C)$. Another important example is when $\mathcal{K}$ is again the empty collection and $\mathcal{K}'$ is the collection of $\kappa$-small simplicial sets. In this case we have a canonical equivalence $\mathcal{P}^{\mathcal{K}'}_{\mathcal{K}}(\C)\simeq \mathcal{P}(\C)^{\kappa}$. Following the fact that an $(\infty,1)$-category has all small colimits if and only it has $\kappa$-small colimits and $\kappa$-filtered colimits, we find a canonical equivalence between $\mathcal{P}(\C)$ and $Ind_{\kappa}(\mathcal{P}(\C)^{\kappa})$.\\

\subsubsection{Sifted Colimits and Geometric Realizations}
\label{siftedcolimits}
Following \cite[5.5.8.1]{lurie-htt}, a simplicial set $K$ is said to be \emph{sifted} if it is nonempty and if the diagonal map $K\to K\times K$ is cofinal. The main examples are given by filtered simplicial sets and by the simplicial set $N(\Delta)^{op}$ - the opposite of the nerve of the category $\Delta$ (see \cite[5.5.8.4]{lurie-htt}). \\

A simplicial object in an $(\infty,1)$-category $\C$ is, by definition, a diagram $N(\Delta)^{op}\to \C$. We say that $\C$ admits \emph{geometric realizations of simplicial objects} if every simplicial object in $\C$ has a colimit.\\

For a small $(\infty,1)$-category $\C$, we let $\mathcal{P}_{\Sigma}(\C)$ denote the formal completion of $\C$ under sifted colimits (as in the previous section). Thanks to \cite[5.5.8.14]{lurie-htt}, if $\C$ has finite coproducts, the formal completion $\mathcal{P}_{\Sigma}(\C)$ is equivalent to the completion of $Ind(\C)$ under geometric realizations of simplicial objects. Moreover, by the Corollary \cite[5.5.8.17]{lurie-htt}, if $\C$ has small colimits, a functor $\C\to \D$ commutes with sifted colimits if and only if it commutes with filtered colimits and geometric realizations.

\subsubsection{Accessibility}
\label{nc1accessibility}
 Sometimes an arbitrary $(\infty,1)$-category $\C$ is not small but it is completely determined by small information. Let $\kappa$ be a small regular cardinal. We say that a big\footnote{We can also define the notion of accessibility for the small $(\infty,1)$-categories. In this case, by the Corollary 5.4.3.6 of \cite{lurie-htt}, a small $(\infty,1)$-category is accessible iff it is idempotent complete.} $(\infty,1)$-category $\C$ is $\kappa$\emph{-accessible} if there exists a small $(\infty,1)$-category $\C^0$ together with an equivalence

\begin{equation}
Ind_{\kappa}(\C^0)\to \C
\end{equation}

By the Proposition 5.4.2.2 of \cite{lurie-htt} a big $(\infty,1)$-category is $\kappa$-accessible if and only if it is locally small, admits small $\kappa$-filtered  colimits, $\C^{\kappa}$ is essentially small and generates $\C$ under small $\kappa$-filtered colimits. In this case, by the Proposition 5.4.2.4 of loc.cit, $\C^{\kappa}$ is the idempotent completion of $\C^0$.

We say that a big $(\infty,1)$-category is \emph{accessible} if it is $\kappa$-accessible for some small regular cardinal $\kappa$. Given two small cardinals $\kappa< \kappa'$, a $\kappa$-accessible $(\infty,1)$-category is not necessarily $\kappa'$-accessible. However, by the Proposition 5.4.2.11 of \cite{lurie-htt}, this is the case if $\kappa'$ satisfies the following condition: for any cardinals $\tau<\kappa$ and $\pi<\kappa'$, we have $\pi^{\tau}<\kappa'$.\\

An important example of accessibility comes from the theory of presheaves: if $\C$ is a small $(\infty,1)$-category, $\mathcal{P}(\C)$ is accessible (see Proposition 5.3.5.12 of \cite{lurie-htt}).\\

The natural morphisms between the accessible $(\infty,1)$-categories are the functors $f:\C\to \D$ which are again determined by the small data. More precisely, if $\C=Ind_{\kappa}(\C^0$) and $\D=Ind_{\kappa}(\D^0)$ for the same cardinal $\kappa$, a functor $f$ is called $\kappa$-accessible if it preserves small $\kappa$-filtered colimits and sends $\kappa$-compact objects in $\C$ to $\kappa$-compact objects in $\D$. The crucial result is that the information of the restriction $f|_{\kappa}: \C^{\kappa}\to \D^{\kappa}$ determines $f$ in a essentially unique way (see Prop. 5.3.5.10 of \cite{lurie-htt}).

\subsubsection{Idempotent Complete $(\infty,1)$-categories}
\label{idempotentcompletecategories}
Let $\C$ be a classical category and let $X$ be an object in $\C$. A morphism $f:X\to X$ is said to be an idempotent if $f\circ f= f$. If we want to extend this notion to the setting of higher category theory, we need to specify a $2$-cell $\sigma$ rendering the diagram 

\begin{equation}
\xymatrix{ 
X \ar[d]_{f} \ar[dr]^{f}|*{}="N"&\\ 
X\ar[r]^f \ar@{=>};"N"^\sigma  & X                         }
\end{equation}

\noindent commutative. Moreover, we should be able to glue together different copies of $\sigma$ to built up a $3$-cell encoding the relation $f\circ f\circ f\simeq f$. This continues for every positive $n$. In \cite[Section 4.4.5]{lurie-htt} the author introduces a simplicial set $Idem$ suitable to encode all this kinds of coherences. It has a unique nondegenerate cell on each dimension $n\geq 0$. To give a diagram $Idem\to \C$ is equivalent to the data of an object $X \in \C$, together with a morphism $f:X\to X$ and all the expected coherences that make $f$ an idempotent.

Recall now that an object $Y\in \C$ is said to be a \emph{retract} of an object $X\in \C$ if the identity of $Y$ factors as a composition $Y\to X\to Y$. Every decomposition likes this provides a morphism $f:X\to Y\to X$ which by  \cite[4.4.5.7]{lurie-htt}, can be extended to a diagram $Idem\to \C$. It follows that if $d$ has a colimit in $\C$, this colimit is canonically equivalent to $Y$ \cite[4.4.5.14]{lurie-htt}. Following this, $\C$ is said to be \emph{idempotent complete} if every diagram $d:Idem\to \C$ has a colimit. In this case, there is a bijective correspondence between retracts and idempotents. In particular, every functor $\C\to \D$ between idempotent complete $(\infty,1)$-categories preserves colimits indexed by the simplicial set $Idem$, because the functoriality will send retracts to retracts.\footnote{We can rewrite this definition in more simpler terms. Since $\mathcal{P}(\C)$ has all colimits, we can easily see that an $(\infty,1)$-category $\C$ is idempotent complete if and only if the image of the Yoneda embedding $\C\to \mathcal{P}(\C)$ is stable under retracts.}\\

\begin{remark}
\label{kfiltrantimplyidempotent}
Since the simplicial set $Idem$ is not finite, the fact that an $(\infty,1)$-category $\C$ has all finite colimits does not imply that $\C$ is idempotent complete. However, even though $Idem$ is not filtrant, if $\kappa$ is a regular cardinal and $\C$ admits small $\kappa$-filtered colimits then $\C$ is idempotent complete \cite[4.4.5.16]{lurie-htt}.
\end{remark}

 We denote by $\iCat^{idem}$ the full subcategory of $\iCat$ spanned by the small $(\infty,1)$-categories which are idempotent. By the Proposition 5.1.4.2 of \cite{lurie-htt} every $(\infty,1)$-category $\C$ admits an idempotent completion $Idem(\C)$ given by the full subcategory of $\mathcal{P}(\C)$ spanned by the completely compact objects (which by the Prop. 5.1.6.8 of \cite{lurie-htt} are exactly the retracts of objects in the image of the Yoneda embedding). The formula $\C\mapsto Idem(\C)$ provides a left adjoint to the full inclusion 
 
\begin{equation}
\iCat^{idem}\subseteq \iCat
\end{equation}

Following the discussion in \ref{completionwithcolimits}, we can also identify $\iCat^{idem}$ with $\iCat(\mathcal{K})$ where $\mathcal{K}=\{Idem\}$. Moreover, we have a canonical equivalence of functors $\mathcal{P}^{\{Idem\}}(-)\simeq Idem(-)$. By the Lemma 5.4.2.4 of \cite{lurie-htt}, $Idem(\C)$ can also be identified with $Ind_{\kappa}(\C)^{\kappa}$, the full subcategory of $\kappa$-compact objects in $Ind_{\kappa}(\C)$, for any small regular cardinal $\kappa$.\\

Let now $\C$ be a small $(\infty,1)$-category and let $\C\to \C'$ be an idempotent completion of $\C$. Then, by the Lemma 5.5.1.3 of \cite{lurie-htt}, for any regular cardinal $\kappa$, the induced morphism $Ind_{\kappa}(\C)\to Ind_{\kappa}(\C')$ is an equivalence of $(\infty,1)$-categories. Thus, if $\C$ is a $\kappa$-accessible $(\infty,1)$-category, with $\C\simeq Ind_{\kappa}(\C_0)$ for some small $(\infty,1)$-category $\C_0$ and some regular cardinal $\kappa$, then, since $\C_0\to Ind_{\kappa}(\C_0)^{\kappa}\simeq \C^{\kappa}$ is an idempotent completion of $\C$, the canonical morphism $Ind_{\kappa}(\C^{\kappa})\to \C$ is an equivalence. The converse is immediate by definition.

\subsubsection{Presentable $(\infty,1)$-categories}

We say that an $(\infty,1)$-category $\C$ is \emph{presentable} if it is accessible and admits all colimits indexed by small simplicial sets. Again, we have a good criterium to understand if an $(\infty,1)$-category $\C$ is presentable. By the  Theorem 5.5.1.1 of \cite{lurie-htt}), the following are equivalent: $(i)$ $\C$ is presentable;  $(ii)$ there exists a small $(\infty,1)$-category $\D$ such that $\C$ is an accessible reflexive localization of $\mathcal{P}(\D)$\footnote{The reflexive localization $\C\subseteq \mathcal{P}(\D)$ is accessible if $\C$ is accessible for some cardinal. Using the universal property of the ind-completion (see the Proposition 5.5.1.2 in \cite{lurie-htt}) this is equivalent to ask for the composition $\mathcal{P}(\D)\to \C\subseteq \mathcal{P}(\D)$ to be $\kappa$-accessible for some small regular cardinal $\kappa$}; $(iii)$ $\C$ is locally small, admits small colimits and there exists a small regular cardinal $\kappa$ and a small $S$ set of $\kappa$-compact objects in $\C$ such that every object of $\C$ is a colimit of a small diagram with values in the full subcategory of $\C$ spanned by $S$.

The natural morphisms between the presentable $(\infty,1)$-categories are the colimit preserving functors. We let $\Prl$ (resp. $\mathcal{P}r^R$) denote the (non full!) subcategory of $\iCatbig$ spanned by \emph{presentable} $(\infty,1)$-categories together with colimit (resp. limit) preserving functors. As $\iCatbig$, $\Prl$ is only a $\uniW$-small $(\infty,1)$-category. By the \emph{ Adjoint Functor Theorem} (see Corollary 5.5.2.9 of \cite{lurie-htt}) a functor between presentable $(\infty,1)$-categories commutes with colimits (resp. limits) if and only if it admits a right (resp. left) adjoint and therefore we have a canonical equivalence $\Prl\simeq (\mathcal{P}r^R)^{op}$. By the Propositions 5.5.3.13 and 5.5.3.18 of \cite{lurie-htt} we know that both $\Prl$ and $\mathcal{P}r^R$ admit all small limits and the inclusions $\Prl, \mathcal{P}r^R\subseteq \iCatbig$ preserve them. In particular, colimits in $\Prl$ are computed as limits in $\mathcal{P}r^R$ using the natural equivalence $\Prl\simeq (\mathcal{P}r^R)^{\circ}$.\\

\subsubsection{$\kappa$-compactly generated $(\infty,1)$-categories}
\label{compactlygenerated1}
Although each presentable $(\infty,1)$-category is determined by small information, not all the information in the study of $\Prl$ is determined by small data. This is mainly because the morphisms in $\Prl$ are all kinds of colimit preserving functors without necessarily a compatibility condition with the small information. Again, as in the accessible setting, if we want to isolate what is determined by small information, we consider for each small regular cardinal $\kappa$, the (non-full!) subcategory $\Prl_{\kappa}\subseteq \Prl$ spanned by the presentable $\kappa$-accessible $(\infty,1)$-categories together with the colimit preserving functors that preserve $\kappa$-compact objects. By definition, we will say that an $(\infty,1)$-category is $\kappa$-compactly generated if it is an object of $\Prl_{\kappa}$. The idea that $\kappa$-compactly generated $(\infty,1)$-categories are determined by smaller information can now be made precise: by the Propositions \cite[5.5.7.8, 5.5.7.10]{lurie-htt}, the correspondence $\C\mapsto \C^{\kappa}$ sending a $\kappa$-compactly generated $(\infty,1)$-category to the full subcategory $\C^{\kappa}\subseteq \C$ spanned by the $\kappa$-compact objects, determines a fully faithful map of $(\infty,1)$-categories $\Prl_{\kappa}\to \iCatbig(\kappa)$ whose image is the full subcategory $\Cat_{\infty}(\kappa)^{idem}$ of $\iCatbig(\kappa)$ spanned by those big $(\infty,1)$-categories $\C$ with all $\kappa$-small colimits, which are essentially small and idempotent complete. Following the discussion in \ref{idempotentcompletecategories}, $\Cat_{\infty}(\kappa)^{idem}$ can be identified with $\Cat_{\infty}(\mathcal{K})$ with $\mathcal{K}$ the collection of all $\kappa$-small simplicial sets together with the simplicial set $Idem$. The construction $Ind_{\kappa}: \Cat_{\infty}(\mathcal{K})\to \Prl_{\kappa}$ provides an inverse to this map. \footnote{See also the Proposition \cite[6.3.7.9]{lurie-ha} for a direct proof of this result}. Moreover, and following the discussion in \ref{kfiltrantimplyidempotent}, in case $\kappa\geq \omega$ we can drop the idempotent considerations because the full inclusion $\Cat_{\infty}(\kappa)^{idem}\subseteq \Cat_{\infty}(\kappa)$ is an equivalence.

We will be particularly interested in $\Prl_{\omega}$, the study of the presentable $(\infty,1)$-categories of the form $Ind(\C_0)$ with $\C_0$ having all finite colimits. These are called \emph{compactly generated}.

\subsubsection{Localizations of Presentable $(\infty,1)$-categories}
The theory of presentable $(\infty,1)$-categories admits a very friendly internal theory of localizations. By the Proposition 5.5.4.15 of \cite{lurie-htt}, if $\C$ is a presentable $(\infty,1)$-category and $W$ is strongly satured class of morphisms in $\C$ generated by a set $S$ (as in \cite[5.5.4.5]{lurie-htt}), then the localization $\C[W^{-1}]$ is again a presentable $(\infty,1)$-category equivalent to the full subcategory of $\C$ spanned by the $S$-local objects and the localization map is a left adjoint to this inclusion;

\subsubsection{Postnikov Towers}
\label{nc1postnikovtowers}
 Recall that a space $X\in \Spaces$ is said to be $n$-truncated if the homotopy groups $\pi_i(X,x)$ are all trivial for $i\geq n$. It is said to be $n$-connective if all the homotopy groups $\pi_i(X,x)$ are trivial for $i< n$. If $\C$ is an $(\infty,1)$-category we say that an object $X\in \C$ is $n$-truncated if for every object $Y$ in $\C$ the mapping spaces $Map_{\C}(Y,X)$ are $n$-truncated. This notion agrees with the previous definition when $\C=\Spaces$. Let $\tau_{\leq n}\C$ denote the full subcategory of $\C$ spanned by the $n$-truncated objects. A morphism $f:X\to X'$ in $\C$ is said to exhibit $X'$ as an $n$-truncation of $X$ if for every $n$-truncated object $Y$ in $\C$ the composition with $f$ induces an homotopy equivalence $Map_{\C}(Y,X)\simeq Map_{\C}(Y,X')$. By definition a \emph{Postnikov tower} in $\C$ is a diagram $X:(N(\mathbb{Z}_{\geq 0})^{op})^{\triangleleft}\to \C$ such that for every $n\leq m$ the map $X_{\ast}\to X_n$ exhibits $X_n$ as an $n$-truncation of $X_{\ast}$. In particular, this implies that $X_n\to X_m$ exhibits $X_m$ as a $m$-truncation of $X_n$. We say that Postnikov towers converge in $\C$ if the forgetful map $Fun((N(\mathbb{Z}_{\geq 0})^{op})^{\triangleleft},\C)\to Fun(N(\mathbb{Z}_{\geq 0}),\C)$ induces an equivalence when restrict to the full subcategory spanned by the Postnikov towers. In particular, if $\C$ admits all limits, Postnikov towers converge in $\C$ if and only if every Postnikov tower is a limit diagram \cite[5.5.6.26]{lurie-htt}.\\

If $\C$ is presentable, the inclusions $\tau_{\leq n}\C\subseteq \C$ admits a left adjoint for every $n\geq 0$. This follows from the Adjoint functor theorem together with the fact that this inclusion commutes with all limits \cite[5.5.6.5]{lurie-htt}. In this case, we can find a sequence of functors

\begin{equation}
... \to \tau_{\leq 2}\C\to \tau_{\leq 1}\C\to \tau_{\leq 0}\C
\end{equation}

\noindent and Postnikov towers converge in $\C$ if and only if $\C$ is the limit of this sequence \cite[3.3.3.1]{lurie-htt}.

\subsubsection{Stable $(\infty,1)$-categories}
\label{stableinfinitycategories}
 We now discuss another important topic. In the classical setting, the notion of \emph{triangulated category} seems to be of fundamental importance. \emph{Stable $\infty$-categories} are the proper providers of triangulated structures - for any stable $\infty$-category $\C$ the homotopy category $h(\C)$ carries a natural triangulated structure, where the exact triangles rise from the fiber sequences and the shift functor is given by the suspension (see \cite[1.1.2.14]{lurie-ha}). Most of the triangulated categories are of this form. Grosso modo, a stable $\infty$-category is an $\infty$-category with a zero object, finite limits and colimits, satisfying the stronger condition that every pushout square is a pullback square and vice-versa (see Definition 1.1.1.9 and Prop. 1.1.3.4  \cite{lurie-ha}). In particular this implies that finite sums are equivalent to finite products \cite[1.1.2.9]{lurie-ha}. If $\C$ is a pointed $(\infty,1)$-category with finite colimits, one equivalent way to formulate the notion of stability is to ask for the suspension functor $X\mapsto \Sigma(X):=*\coprod_X *$ and its adjoint $Y\mapsto \Omega(Y):=*\times_X *$ to form an equivalence $\C\to \C$ (see \cite{lurie-ha}-Cor. 1.4.2.27). It is important to remark that stability is a property rather than an additional structure. The canonical example of a stable $(\infty,1)$-category is the $(\infty,1)$-category of spectra $\Sp$ which can be defined as the underlying $(\infty,1)$-category encoding the homotopy theory of $\infty$-loop spaces. It admits many other equivalent definitions (see \cite{lurie-ha}-Section 1.4.3). The appropriate maps between stable $\infty$-categories are the functors commuting with finite limits (or equivalently, with finite colimits - see Prop. 1.1.4.1 of \cite{lurie-ha}). The collection of small stable $\infty$-categories together with these functors (so called \emph{exact}) can be organized in a new $\infty$-category $\iCatstable$. Thanks to \cite[1.1.4.4]{lurie-ha} $\iCatstable$ has all small limits and the inclusion in $\iCat$ preserves them. Moreover, if $K$ is a simplicial set and $\C$ is stable then $Fun(K,\C)$ remains stable \cite[1.1.3.1]{lurie-ha}. \\

 Also important  is that any stable $(\infty,1)$-category  $\C$ comes with a natural \emph{enrichment over spectra}. More precisely the mapping spaces $Map_{\C}(X,Y)$ have a natural structure of an $\infty$-loop space. To see this we can use the fact the suspension and loop functors in $\C$ are equivalences, so that we can find a new object $X'$ with $X\simeq \Sigma (X')$ so that $Map_{\C}(X,Y)\simeq Map_{\C}(\Sigma(X'), Y)\simeq \Omega Map_{\C}(X,Y)$. Another way to make this precise is to use a universal property of the stabilization which tells us that the  composition with $\Omega^{\infty}:\Sp\to \Spaces$ induces an equivalence of $(\infty,1)$-categories $Exc_{\ast}(\C, \Sp)\simeq Exc_{\ast}(\C, \Spaces)$ ( see \cite[1.4.2.22]{lurie-ha}). In particular, this provides for  any object $X$  an essentially unique factorization of the functor $Map_{\C}(X,-):\C\to \Spaces$ as

\begin{equation}
\xymatrix{
\C\ar[rr]^{Map_{\C}(X,-)}\ar@{-->}[d]_{Map_{\C}^{Sp}(X,-)}&& \Spaces\\
\Sp \ar[urr]_{\Omega^{\infty}}&
}
\end{equation}

\noindent such that for any object $Y$,  the spectra $Map_{\C}^{Sp}(X,Y)$ can be identified with the collection of spaces  $\{Map_{\C}(X, \Sigma^n Y)\}_{n\in \mathbb{Z}}$. Moreover, and since $\Omega$ is an equivalence, it is equivalent to the family $\{Map_{\C}(\Omega^n X,Y)\}_{n\in \mathbb{Z}}$. The \emph{Ext groups} $Ext_{\C}^{i}(X,Y)$ are defined as $\pi_0(Map_{\C}(\Omega^n X,Y))$. If $i\leq 0$ these groups correspond to the homotopy groups of the mapping space $Map_{\C}(X,Y)$.\\

We can now isolate the full subcategory $\Prl_{Stb}$ of $\Prl$ spanned by those presentable $(\infty,1)$-categories which are stable (every morphism of presentable $(\infty,1)$-categories which are stable is exact). Again by \cite[1.1.4.4]{lurie-ha} and the results in the presentable setting, $\Prl_{Stb}$ has all small limits and the inclusion $\Prl_{Stb}\subseteq \Prl$ preserves them. 

We discuss now an adapted version of the Proposition \cite[1.4.4.2]{lurie-ha} that provides a very helpful characterization of presentable stable $(\infty,1)$-categories. First we introduce some terminology. Let $\C$ be an $(\infty,1)$-category with a zero object. We say that a collection $\mathcal{E}$ of objects in $\C$ \emph{generates $\C$} if the full subcategory $\mathcal{E}^{\bot}\subseteq \C$ of all objects $A$ in $\C$ such that $Map_{\C}(E,A)=\ast$ consists only of zero objects in $\C$. Let now $\kappa$ be a regular cardinal. We say that $\mathcal{E}$ is a \emph{family of $\kappa$-compact generators of $\C$ } if $\mathcal{E}$ generates $\C$ in the previous sense and each object $E\in \mathcal{E}$ is $\kappa$-compact. In particular, we will say that an object $X$ in $\C$ is a \emph{$\kappa$-compact generator of $\C$} if the family $\mathcal{E}=\{X\}$ is a family of $\kappa$-compact generators of $\C$. 

\begin{prop}
\label{cg}
Let $\C$ be a stable $(\infty,1)$-category. Then, $\C$ is presentable if and only if the following conditions are satisfied:
\begin{enumerate}[(i)]
\item $\C$ has arbitrary small coproducts \footnote{Since $\C$ is stable this is equivalent to ask for all small colimits};
\item the triangulated category $h(\C)$ is locally small; 
\item there exists a regular cardinal $\kappa$ and a family $\mathcal{E}$ (indexed by a small set) of $\kappa$-compact generators in $\C$ .

\end{enumerate}

Moreover, if $\C$ admits a family of  $\kappa$-compact generators then it is $\kappa$-compactly generated in the sense of \ref{compactlygenerated1}. 

\begin{proof}

We follow essentially the same arguments of \cite[1.4.4.2]{lurie-ha}.  For the "only if" part, by definition, there is a small $(\infty,1)$-category $\D$ and a regular cardinal $\tau$, together with an equivalence $\C\simeq Ind_{\tau}(\D)$. The formal completion of $\D$ with $\tau$-small colimits is given by $\D\to \D'=\mathcal{P}(\D)^{\tau}$. Passing to the ind-completions we obtain a map

\begin{equation}
\C\simeq Ind_{\kappa}(\D)\to Ind_{\kappa}(\mathcal{P}(\D)^{\kappa})\simeq \mathcal{P}(\D)
\end{equation}

\noindent commuting with $\tau$-filtered colimits. From the proof of \cite[5.5.1.1]{lurie-htt} we know that this map  has a left adjoint $L$ that establishes $\C$ as $\tau$-accessible reflexive localization of $\mathcal{P}(\D)$. The items $(i)$ and $(ii)$ follow immediately from this. Moreover, the composition functor

\begin{equation}
\label{mafic}
\mathcal{P}(\D)\to \C\subseteq \mathcal{P}(\D)
\end{equation}

\noindent preserves $\tau$-filtered colimits.  To prove $(iii)$ we consider the family $\mathcal{E}$ of all objects of the form $L(j(d))$  in $\C$ with $j$ the Yoneda's embedding $j:\D\to \mathcal{P}(\D)$ and $d\in \D$. It follows immediately from the Yoneda's lemma, from the fact the composition (\ref{mafic}) is $\tau$-accessible and from the fact the right adjoint of $L$ is fully-faithful that $\mathcal{E}$ is a family of $\tau$-compact generators in $\C$. The family is indexed by a small set because $\D$ is small.

For the "if" part, we consider the full subcategory $\C_{\mathcal{E}}$ of $\C$ spanned by the objects in $\mathcal{E}$, their suspensions and loopings. Inductively, we consider the successive closures under $\kappa$-small colimits. As a result we find a full subcategory $\C_{\mathcal{E}}(\kappa)$ of $\C$ closed under $\kappa$-small colimits, suspensions and loopings. Since each $E\in \mathcal{E}$ is $\kappa$-compact and the suspensions of $\kappa$-compact objects are again $\kappa$-compact \footnote{If $I$ is a $\kappa$-filtered simplicial set and $d:I\to \C$ is a diagram, we have $Map(\Sigma X, colim_I d_i)\simeq Map(X, \Omega (colim_k d))$ and since $\C$ is stable (which implies that $\Omega$ is an equivalence and therefore commutes with colimits) and $X$ is $\kappa$-compact, we find that the last space is homotopy equivalent to $Map(X, colim_I(\Omega(d_i))\simeq colim_I Map(X,\Omega(d_i))\simeq colim_I Map(\Sigma X, d_i)$} and $\kappa$-compact objects are closed under $\kappa$-small colimits, we find that $\C_{\mathcal{E}}(\kappa)$ is made of $\kappa$-compact objects and closed under $\kappa$-small colimits. It follows that the inclusion $\C_{\mathcal{E}}(\kappa)\subseteq \C$ extends to a fully-faithfull functor $F:Ind_{\kappa}(\C_{\mathcal{E}}(\kappa))\to \C$ that commutes with $\kappa$-filtrant colimits. Since $Ind_{\kappa}(\C_{\mathcal{E}}(\kappa))$ has all $\kappa$-small colimits and all $\kappa$-filtrant colimits, it has all colimits and $F$ commutes with all colimits. By the hypothesis (ii) and the Remark \cite[5.5.2.10]{lurie-htt} $F$ has a right adjoint $G$ and the fully-faithfulness implies $G\circ F\simeq Id$. We are reduce to show that for every $Y\in \C$, the adjunction map $F\circ G(Y)\to Y$ is an equivalence. For that, we consider its fiber $Z$. Since $F$ is fully-faithful, and $G$ preserves limits, we have $G(Z)\simeq \ast$ and by adjunction we find that for every object $D\in Ind_{\kappa}(\C_{\mathcal{E}}(\kappa))$ we have have $Map(F(D),Z)\simeq Map(D,G(Z))\simeq \ast$. In particular, the formula holds for any $D=E\in \mathcal{E}$ and by the definition of generating family we find that $Z$ is a zero object in $\C$ so that the counit map is an equivalence. In particular, $\C$ is a stable $\kappa$-compactly generated $(\infty,1)$-category. In the case  $\mathcal{E}$ is a family of $\omega$-compact generators, $\C$ is compactly generated and its full subcategory of compact objects is equivalent to $Idem(\C_{\mathcal{E}}(\kappa))$.
\end{proof}
\end{prop}

\begin{remark}
\label{presenseofNeeman}
The condition $(iii)$ in the Proposition \ref{cg} is equivalent to the existence of an $\alpha$-compact generator for some regular cardinal $\alpha$, not necessarily the same as $\kappa$. For the if direction, by definition, if $\C$ has a $\kappa$-compact generator, then it provides a $\kappa$-generating family with a single element. Conversely, if $\mathcal{E}=\{E_i\}_{i\in I}$ is a  $\kappa$-generating family with multiple objects, by the hypothesis $(i)$, the sum $\coprod_{i\in I}E_i$ exists in $\C$ and is an $\alpha$-compact generator of $\C$ for some $\alpha$ a regular cardinal  (let $\gamma= max\{\kappa, card(I)\}$ and choose $\alpha$ satifying the condition described in \ref{nc1accessibility}).
\end{remark}

\begin{remark}
\label{senseofNeeman}

The statement given  in \cite[1.4.4.2]{lurie-ha} is somewhat different from ours, namely because the notion of compact generator therein is stronger. More precisely, an object $X$ there is said to be a $\kappa$-compact generator if it is $\kappa$-compact and such that for any $Y\in \C$, if $\pi_0(Map(X,Y))=0$ then $Y$ is a zero object. Of course, if $X$ verifies this condition, the family $\mathcal{E}=\{X\}$ verifies our condition $(iii)$. However, the converse is not necessarily true for the same $X$ and the same cardinal. If $X$ is a $\kappa$-generator in our sense, then the infinite coproduct $\coprod_{n\in \mathbb{Z}} \Sigma^n X$ is a generator in the sense of \cite[1.4.4.2]{lurie-ha} but a priori it will only be $\kappa'$-compact for some cardinal $\kappa'\geq \kappa$. 

Our version is needed to match the familiar results coming from the classical theory of compact generators in triangulated categories. Following Neeman \cite{Neeman-triangulatedcategories}, we recall that an object $X$ in a triangulated category $T$ is said to be \emph{compact} if it commutes with infinite coproducts. Moreover, a collection of objects $\mathcal{E}$ in $T$ is said to \emph{generate} $T$ if its right-orthogonal complement $\mathcal{E}^{\bot}:=\{ A \in Ob(T): Hom_T(E[n],A)=0 ,\forall n\in \mathbb{Z}, \forall E\in \mathcal{E}\}$ consists only of zero objects in $T$. Again, $\mathcal{E}$ is said to be a \emph{family of compact generators of $T$ in the sense of Neeman} if it generates $T$ and each  $E\in \mathcal{E}$ is compact in the sense of triangulated categories. Finally, an object $X$ is said to be a \emph{compact generator of $T$} if it is compact and the set $\mathcal{E}=\{X\}$ generates $T$.

Let now $\C$ be a stable $(\infty,1)$-category and let $\mathcal{E}$ be a collection of objects in $\C$. It follows that $\mathcal{E}$ is a family of compact generators of $h(\C)$ in the sense of Neeman if and only it satisfies the condition $(iii)$ for $\kappa=\omega$. Indeed, the two notions of generator agree because $\pi_0Map_{\C}(\Sigma^n E, A)\simeq \pi_nMap_{\C}(E,A)$. Therefore, it is enough to see that an object $X$ is compact in the triangulated category $h(\C)$ if and only it is $\omega$-compact in $\C$. This follows from \cite[1.4.4.1-(3)]{lurie-ha} and from the fact that coproducts in $\C$ are the same as coproducts in $h(\C)$: if $\{X_i\}_{i\in I}$ is a collection of objects in $\C$, its coproduct $\coprod_i X_i$ in $\C$ is a coproduct in $h(\C)$  because the functor $\pi_0$ commutes with homotopy products; conversely, if $\coprod_i X_i$ is a coproduct in $\C$, by definition, this means that $\pi_0Map_{\C}(\coprod_i X_i, Z)\simeq \oplus_i \pi_{0}Map_{\C}(X_i,Z)$ holds for any $Z\in h(\C)$. In particular, this holds for all the loopings $\Omega^{n}Z$ so that the formula holds for all $\pi_n$ and $\coprod_i X_i$ is a coproduct in $\C$.
\end{remark}

This characterization allows us to detect the property of a stable $(\infty,1)$-category being compactly generated simply by studying its homotopy category. The following example is crucial to algebraic geometry and will play a fundamental role in the last part of our paper:

\begin{example}
\label{derivedinfinitycategoryofascheme}
Let $X$ be a quasi-compact and separated scheme. Its underlying $(\infty,1)$-category $\mathcal{D}(X)$ (see below) is stable  and its homotopy category  $h(\mathcal{D}(X))$ recovers the classical derived category of $X$. As proved in the Corollary 5.5 of \cite{Neeman-HomotopyLimitsTriangulated}, when $X$ is quasi-compact and separated, $h(\mathcal{D}(X))$ is equivalent to the derived category of $\mathcal{O}_X$-modules with quasi-coherent cohomology sheaves. Together with the Theorem 3.1.1 of \cite{bondal-vandenbergh}, this tells us that $\mathcal{D}(X)$ has a compact generator in the sense of Neeman and that the compact objects are the perfect complexes. Thus, by the previous discussion, $\mathcal{D}(X)$ is a stable compactly generated $(\infty,1)$-category. 
\end{example}

 We will write $\Prl_{\omega,Stb}$ to denote the full subcategory of $\Prl_{\omega}$ spanned by the stable $(\infty,1)$-categories that are compactly generated, together with the compact preserving morphisms. The equivalence $\Prl_{\omega}\to \iCat(\omega)^{idem}$ of \ref{compactlygenerated1} restricts to an equivalence $\Prl_{\omega, Stb}\to \iCatstableidem$ where the last denotes the (non full) subcategory of $\iCat^{idem}$ spanned by the small stable $\infty$-categories which are idempotent complete, together with the exact functors. This follows from the fact that the idempotent completion of a stable $(\infty,1)$-category remains stable \cite[1.1.3.7]{lurie-ha}, together with the observation that stable $(\infty,1)$-categories have all finite colimits and that exact functors preserve them.\\ 
 
\begin{remark}
In \cite{tabuada-gepner} the authors identify the subject of \emph{Topological Morita theory} with the study of the $(\infty,1)$-category $\Prl_{\omega, Stb}$. We will come back to this in \ref{linkdgspectrastable}.\\
\end{remark} 

To conclude this section, we give a useful result that will be necessary for many of the future applications we have in mind:

\begin{prop}
\label{equivalencecompactgenerators}
Let $f:\C\to \D$ be colimit preserving functor between stable presentable $(\infty,1)$-categories. Assume that
\begin{enumerate}[(i)]
\item The $(\infty,1)$-category $\C$ has a family of $\omega$-compact generators $\mathcal{E}$ in the sense of the Proposition \ref{cg}  (here we assume, without loss of generality, that $\mathcal{E}$ is closed under suspensions and loopings\footnote{We can always assume this because, as discussed in the previous footnote, suspensions of compact objects are compact.}) and $f$ is fully-faithful when restricted to the objects in this collection;
\item for any object $E\in \mathcal{E}$, the object $f(E)$ is $\omega$-compact in $\D$;
\end{enumerate} 

Then, $f$ is fully-faithful. Moreover, if the image of the collection $\mathcal{E}$ in $\D$ is a family of $\omega$-compact generators, then $f$ is an equivalence.

\begin{proof}
To start with, we observe that to assume $\mathcal{E}$ to be closed under suspensions and loopings and $f$ to be fully-faithful when restricted to the objects in $\mathcal{E}$ produces the same effects as droping the condition of stability under suspensions and loopings and asking for the naturally induced maps of spectra

\begin{equation}
Map^{Sp}_{\C}(X,Y)\to Map^{Sp}_{\D}(f(X), f(Y))
\end{equation} 

\noindent to be equivalences in $\Sp$ for any $X$ and $Y$ in $\mathcal{E}$.

Let us now explain the proof. Using the same notations of the Proposition \ref{cg},  we have $\C\simeq Ind(\C_{\mathcal{E}}(\omega))$. To deduce fully-faithfulness we prove that the restriction of $f$ to $\C_{\mathcal{E}}(\omega)$ is fully-faithful so that, by the hypothesis $(ii)$ together with \cite[5.3.5.11]{lurie-htt} we conclude that $f$ is fully-faithful. To see this, it is enough to check that $f$ is fully-faithful when restricted to each one of the subcategories in the inductive construction of $\C_{\mathcal{E}}(\omega)$ (see the proof in \cite[1.4.4.2]{lurie-ha} for the precise inductive step). Using induction, and since $\C$ is stable, it is enough to check that $f$ is fully-faithful when restricted to finite direct sums and cofibers of objects in the collection $\mathcal{E}$. For direct  sums this is immediate. Suppose now we have a cofiber sequence $X\to Y\to Z$ in $\C$ with $X$ and $Y$ in $\mathcal{E}$ and let $A$ be another object in $\mathcal{E}$. In this case, and since the functors $Map^{Sp}_{\C}(A,-)$ are exact by construction, we obtain a cofiber sequence in $\Sp$

\begin{equation}
Map^{\Sp}_{\C}(A,X)\to Map^{\Sp}_{\C}(A,Y)\to Map^{\Sp}_{\C}(A,Z)
\end{equation} 

Since $f$ commutes with small colimits, the induced sequence $f(X)\to f(Y)\to f(Z)$ is a cofiber sequence and we get a canonical diagram of cofiber sequences in $\Sp$

\begin{equation}
\xymatrix{
Map^{\Sp}_{\C}(A,X)\ar[r]\ar[d]^{\sim}& \ar[r]\ar[d]^{\sim}Map^{\Sp}_{\C}(A,Y)&\ar[d] Map^{\Sp}_{\C}(A,Z)\\
Map^{\Sp}_{\D}(f(A),f(X))\ar[r]& \ar[r]Map^{\Sp}_{\D}(f(A),f(Y))&Map^{\Sp}_{\D}(f(A),f(Z))
}
\end{equation}

\noindent where the two first vertical maps are equivalences by hypothesis. We conclude the vertical map on the right is also an equivalence. Finally, for any other cofiber sequence $U\to V\to W$ in $\C$, we conclude using the universal property of the cofiber that $Map^{\Sp}_{\C}(f(W),f(Z))\simeq Map^{\Sp}_{\C}(W,Z)$. \\

To conclude the additional statement we use the definition of generating familiy and the consequences of the Prop. \ref{cg} to reduce everything  to prove that the induced restriction  $\C_{\mathcal{E}}(\omega)\to\D_{f(\mathcal{E})}(\omega)$ is an equivalence. This follows because $f$ commutes with colimits.

\end{proof}

\end{prop}
 
\subsubsection{Localizations of Stable $(\infty,1)$-categories and Exact Sequences}
\label{exactsequencepresentable}

Our goal in this section is to prove the Proposition \ref{bondalcontext} below. Let us start by reviewing some standard terminology for triangulated categories. Let $C$ be triangulated category and let $A$ be a triangulated subcategory. We say that $A$ is thick in $C$ (also said \emph{epaisse}), if it is closed under direct summands. Moreover generally, we say that a triangulated functor $A\to C$ is cofinal if the image of $A$ is thick in $C$. Recall also that a sequence of triangulated categories $A\to C\to D$ is said to be exact if the composition is zero, the first map is fully-faithful and the inclusion from the Verdier quotient $C/A\hookrightarrow D$ is cofinal, meaning that every object in $D$ is a direct summand of an object in $B/A$.\\
 
Following \cite{tabuada-gepner}, we say that a sequence in $\Prl_{Stb}$ 

\begin{equation}
\mathcal{A}\to \C\to \D
\end{equation}

\noindent is exact if the composition is zero, the first map is fully-faithful and the diagram

\begin{equation}
\xymatrix{
\mathcal{A} \ar[d]\ar@{^{(}->}[r]& \C\ar[d]\\
\ast \ar[r] &\D
}
\end{equation}

\noindent is a pushout. Here we denote by $\ast$ the final object in $\Prl_{Stb}$. As proved in \cite[Prop. 4.5, Prop. 4.6]{tabuada-gepner}, this notion of exact sequence can be reformulated using the language of localizations: if $\phi:\mathcal{A}\hookrightarrow \C$ is a fully-faithful functor, the cofiber of $\phi$ can be identified with the accessible reflexive localization  

\begin{equation}
\xymatrix{
\D\ar@{^{(}->}[r]& \ar@/_1pc/[l]\C
}
\end{equation}

\noindent with local equivalences given by the class of edges $f$ in $\C$ with cofiber in the essential image of $\phi$. In particular, an object $x \in \C$ is in $\D$ if and only if for every object $a\in \mathcal{A}$ we have $Map_{\C}(a,x)\simeq \ast$.

\begin{remark}
\label{compactgeneratorisenough}
Let $\mathcal{A}\hookrightarrow \C\to \D$ be an exact sequence of presentable stable $(\infty,1)$-categories as above. If the homotopy category $h(\A)$ has a compact generator in the sense of Neeman, say $k\in \mathcal{A}$, then for an object $x \in \C$ to be in $\D$ it is enough to have $Map_{\C}(k,a)\simeq \ast$. This follows from the arguments in the proof of tne Proposition \ref{cg}: every object in $\mathcal{A}$ can be obtained as a colimit of suspensions of $k$.
\end{remark}

Thanks to \cite[Prop. 5.9]{tabuada-gepner} and to the arguments in the proof of \cite[Prop. 5.13]{tabuada-gepner}, this notion of exact sequence extends the notion given by Verdier in \cite{Verdier}: a sequence $\mathcal{A}\hookrightarrow \C\to \D$ in $\Prl$ is exact if and only if the sequence of triangulated functors $h(\mathcal{A})\hookrightarrow h(\C)\to h(\D)$ is exact sequence in the classical sense and the inclusion $h(\C)/h(\mathcal{A})\hookrightarrow h(\D)$ is an equivalence of triangulated categories. In the compactly generated case we have the following

\begin{prop}
\label{gepnerexact}Let $\mathcal{A}\hookrightarrow \C\to \D$  be a sequence in $\Prl_{\omega,Stb}$. The following are equivalent:
\begin{enumerate}
\item the sequence is exact;
\item the induced sequence of triangulated functors $h(\mathcal{A})\hookrightarrow h(\C)\to h(\D)$ is exact in the classical sense and the inclusion  $h(\C)/h(\mathcal{A})\hookrightarrow h(\D)$ is an equivalence;
\item the sequence of triangulated functors induced between the homotopy categories of the associated stable subcategories of compact objects $h(\mathcal{A}^{\omega})\hookrightarrow h(\C^{\omega})\to h(\D^{\omega})$ is exact in the classical sense.
\end{enumerate}
\begin{proof}

The equivalence between $1)$ and $2)$ follows from the results of \cite{tabuada-gepner} discussed above. The equivalence between $2)$ and $3)$ follows from the results of B.Keller \cite[Section 4.12, Corollary]{keller-exact} and the fact that for any compactly generated stable $(\infty,1)$-category $\C$ we can identify $h(\C^{\omega})$ with the full subcategory of compact objects (in the sense of Neeman) in $h(\C)$ (see \ref{senseofNeeman}).

\end{proof}
\end{prop}

The following result will become important in the last section of our work:

\begin{prop}
\label{bondalcontext}
Let 

\begin{equation}
\xymatrix{
& \D\ar[d]^f\\
\C\ar[r]^L& \C_0
}
\end{equation}

be a diagram in $\Prl_{\omega, Stb}$ such that

\begin{itemize}
\item The homotopy triangulated category $h(\D)$ has a compact generator in the sense of Neeman;
\item The map $L:\C\to \C_0$ is an accessible reflexive localization of $\C$ obtained by killing a stable subcategory $\mathcal{A}\subseteq \C$ such that $h(\mathcal{A})$ has a compact generator (in the sense of Neeman) and the inclusion $\A\subseteq \C$ is a map in $\Prl_{\omega, Stb}$.
\end{itemize}

Then:
\begin{enumerate}
\item the diagram admits a limit $\sigma=$

\begin{equation}
\xymatrix{
\mathcal{T} \ar[r]\ar[d] & \D\ar[d]^f\\
\C\ar[r]^L& \C_0
}
\end{equation}

\noindent in $\Prl_{\omega, Stb}$;

\item the diagram $\sigma$ remains a pullback after the (non-full) inclusion $\Prl_{\omega, Stb}\subseteq \Prl_{Stb}$;
\item the homotopy category $h(\mathcal{T})$ has compact generator in the sense of Neeman.
\end{enumerate}
\begin{proof}

We start by noticing that $\Prl_{\omega, Stb}\subseteq \Prl_{Stb}$ preserves colimits (see \cite[5.5.7.6, 5.5.7.7]{lurie-ha}). Therefore, the map $L:\C\to \C_0$ remains a Bousfield localization in the sense discussed above. We recall also that all pullbacks exists in $\Prl_{Stb}$ and thanks to \cite[1.1.4.4]{lurie-ha} and to \cite[5.5.3.13]{lurie-htt} they can be computed in $Cat_{\infty}^{big}$. In this case, let 

\begin{equation}
\xymatrix{
\V \ar[r]^p\ar[d]^q & \D\ar[d]^f\\
\C\ar[r]^L& \C_0
}
\end{equation}

\noindent be a pullback for the diagram given by $(f,L)$ in $\Prl$. Of course, we can assume that $f$ is a categorical fibration and nothing will change up to categorical equivalence (see our discussion about homotopy pullbacks in \ref{section4-2}). With this, we can actually describe $\V$ as the strict pullback $\D\times_{\C_0}\C$. It follows from the description of compact-objects in the pullback \cite[5.4.5.7]{lurie-htt} that both maps $p$ and $q$ preserve compact objects. Therefore, to achieve the proof we are reduced to show that $\V$ is $\omega$-accessible. Indeed, if this is the case, it follows from the universal property of the pullback in $\Prl_{\omega, Stb}$ and in $\Prl_{Stb}$ that $\V$ is canonically equivalent to $\mathcal{T}$. To prove that $\V$ is $\omega$-accessible we can make use of the Proposition \ref{cg}: it suffices to show that the homotopy category of $\V$ has a compact generator in the sense of Neeman. We can construct one using exactly the same arguments of \cite[Prop. 3.9]{toen-azumaya}, itself inspired in the arguments of the famous theorem of Bondal - Van den Bergh \cite[Thm 3.1.1]{bondal-vandenbergh}:\\

Let $d$ be a compact generator in $\D$, which exists as part of our assumptions.  As $f$ and $L$ are functors in $\Prl_{\omega,Stb}$, they preserve compact objects and therefore $f(d)$ is compact. As $L$ is a Bousfield localization of compactly generated stable $(\infty,1)$-categories we can use the famous result of Neeman-Thomason  \cite[Theorem 2.1]{Neeman-thomason} to deduce the existence of a compact object $c\in \C$ whose image in $\C_{0}$ is $f(d)\oplus (f(d)[1])$. The new object $d'=d\oplus (d[1])$ is again a compact generator in $h(\D)$ and since $f$ preserves colimits we conclude the existence of an object $v \in \V$ such that $p(v)=d'$ and $q(v)=c$. The Lemma 5.4.5.7 of \cite{lurie-htt} implies that $v$ is a compact object in $\V$.\\

At the same time, we use our second assumption that $\mathcal{A}$ has a compact generator $k$. Since $k$ is in $\mathcal{A}$, $L(k)$ is a zero object in $\C_{0}$. Therefore, it lifts to an object $\tilde{k}\in \V$ with $q(\tilde{k})=k$ and $p(\tilde{k})=0\in \D$. To deduce that $\tilde{k}$ is a compact object in $\V$ we observe that for any $z$ in $\V$, the mapping space in the pullback is given by the formula

\begin{eqnarray}
\label{nc1chikititas}
Map_{\V}(\tilde{k}, z)\simeq Map_{\C}(k,q(z))\times_{Map_{\C_0}(L(k),L(q(z)))} Map_{\D}(p(\tilde{k}),p(z))\\
\simeq Map_{\C}(k,q(z))\times_{Map_{\C_0}(0,L(q(z)))} Map_{\D}(0,p(z))\\
\simeq Map_{\C}(k,q(z))\times_{\ast}\ast \simeq Map_{\C}(k,q(z))
\end{eqnarray}

\noindent so that, since $q$ commutes with colimits, $\tilde{k}$ is compact in $\V$ if and only if $k$ is compact in $\C$.  The last is true because of our hypothesis that the inclusion $\A\subseteq \C$ preserves compact objects.\\

We claim that the sum $v\oplus \tilde{k}$ is a compact generator of $\V$. Obviously, as a finite sum of compacts, it is compact. We are left to check that it is a generator of $h(\V)$. In other words, we have to prove that for an arbitrary object $z$ in $\V$, if $z$ is right-orthogonal to the sum $v\oplus \tilde{k}$ in $h(\V$), then it is a zero object. Notice that $z$ is right-orthogonal to the sum if and only if it is right-orthogonal to $v$ and $\tilde{k}$ at the same time. In particular, the formula (\ref{nc1chikititas}) implies that $z$ is right-orthogonal to $\tilde{k} \in h(\V)$ if and only if $q(z)$ is right-orthogonal to $k$ in $\C$. Since $k$ is a compact generator of $h(\A)$ (by assumption), it follows from the Remark \ref{compactgeneratorisenough} that $q(z)$ is right-orthogonal to $k$ if and only if $q(z)$ is $L$-local, meaning that it is in $\C_0$ and we have $i\circ L(q(z))\simeq q(z)$, where $i$ is the fully faithfull right adjoint of $L$. Let us assume that $z$ is right-orthogonal to $\tilde{k}$. Then, this discussion implies that

\begin{equation}
Map_{\V}(v,z)\simeq Map_{\C}(c,q(z))\times_{Map_{\C_0}(f(d'),f(q(z))} Map_{\D}(d',p(z))
\end{equation}

\noindent and using the fact that $q(z)\simeq i\circ L(q(z))$, it becomes

\begin{equation}
\simeq Map_{\C}(L(c),L(q(z))\times_{Map_{\C_0}(f(d'),f(q(z))} Map_{\D}(d',p(z))\simeq Map_{\D}(d', p(z))
\end{equation}

We conclude that if $z$ is orthogonal to $\tilde{k}$ and $v$ at the same time, then $p(z)$ is orthogonal to d'. However, by construction, d$'$ is again a compact generator of $h(\D)$ so that $p(z)$ is zero in $\D$. Since we have $q(z)\simeq i\circ L (q(z))\simeq i\circ f\circ p(z)$, this implies that $q(z)$ is also zero in $\C$. Using the Lemma 5.4.5.5 of \cite{lurie-htt}, we find that $z$ is a zero object in $\D$. This concludes the proof.

\end{proof}
\end{prop}

\begin{remark}
\label{bondalcontextforfamilycompactgenerators} The proof of the Proposition \ref{bondalcontext} works mutatis-mutandis if we replace the hypothesis of solo compact generators in $\A$ and $\D$ by the existence of compact generating families. More precisely, and using the same arguments and notations, if $\mathcal{E}_{\D}=\{d_i\}_{i\in I}$ and $\mathcal{E}_{\A}=\{k_j\}_{j\in J}$ are families of compact generators respectively in $\D$ and in $\A$, we can prove that the family $\{\tilde{k}_j\oplus v_i\}_{(i,j)\in I\times J)}$ is a family of compact generators in $\mathcal{T}$.
\end{remark}

In particular, we have the following immediate corollary:

\begin{cor}
\label{exactsequencesarestrict}
Let $\sigma=$

\begin{equation}
\xymatrix{
\mathcal{A} \ar[d]\ar@{^{(}->}[r]& \C\ar[d]\\
\ast \ar[r] &\C_0
}
\end{equation}

\noindent be an exact sequence in $\Prl_{\omega, Stb}$ such that $h(\A)$ admits a family of compact generators in the sense of Neeman. Then, the diagram $\sigma$ is a pullback in $\Prl_{\omega, Stb}$

\begin{proof}
This is the degenerated case of \ref{bondalcontext} (together with the Remark \ref{bondalcontextforfamilycompactgenerators}) where $\D=0$. The inclusion $\A\subseteq \C$ admits a canonical factorization through the pullback, which, by the arguments in  \ref{bondalcontext} and \ref{bondalcontextforfamilycompactgenerators}, sends the generating family of $\A$ to a generating family.  The conclusion now follows from the Proposition \ref{equivalencecompactgenerators}.
\end{proof}
\end{cor}

\subsubsection{$t$-structures} 
\label{Tstructures}
\emph{$t$-structures} are an important tool in the study of triangulated categories. Following  \cite[Section 1.2.1]{lurie-ha} they extend in a natural way to the setting of stable $(\infty,1)$-categories: A $t$-structure in a stable $(\infty,1)$-category $\C$ is the data of a $t$-structure in the homotopy category $h(\C)$. Given a $t$-structure $(h(\C)_{\leq 0}, h(\C)_{\geq 0})$ in $h(\C)$, we denote by $\C_{\leq 0}$ (resp. $\C_{\geq 0}$) the full subcategory of $\C$ spanned by the objects in $h(\C)_{\leq 0}$ (resp. $h(\C)_{\geq 0}$). Moreover, we will write $\C_{\leq n}$ (resp. $\C_{\geq n}$) to denote the image of $\C_{\leq 0}$ (resp. $\C_{\geq 0}$) under the functor $\Sigma^n$. Recall also that an object $X \in \C$ is said to be \emph{connective with respect to the $t$-structure} if it belongs to $\C_{\geq 0}$.\\

It follows from the axioms for a $t$-structure that for every $n\in \mathbb{Z}$ the inclusion $\C_{\leq n}\subseteq \C$ admits a left adjoint $\tau_{\leq n}$ \cite[1.2.1.5]{lurie-ha} and the inclusion $\C_{\geq n} \subseteq \C$ admits a right adjoint $\tau_{\geq n}$ and these two adjoints are related by the existence of a cofiber/fiber sequence

\begin{equation} 
\label{cofibertruncation}
\xymatrix{
\tau_{\geq n}(X)\ar[r]\ar[d]& X\ar[d]\\
\ast \ar[r]& \tau_{\leq n-1}(X)
}
\end{equation}

Moreover, for every $m,n\in \mathbb{Z}$ they are related by a natural equivalence 

\begin{equation}\tau_{\leq m}\circ \tau_{\geq n}\simeq\tau_{\geq n}\circ \tau_{\leq m}\end{equation}

\noindent (see \cite[1.2.1.10]{lurie-ha}).

\begin{remark}
\label{truncationsandtstructures}
The two notations $\tau_{\leq n}:\C\to \tau_{\leq n}\C$ (after \ref{nc1postnikovtowers}) and $\tau_{\leq n}:\C\to \C_{\leq n}$ are not compatible. However, they are compatible with restricted to $\C_{\geq 0}$ and we have $\tau_{\leq n}(\C_{\geq 0})\simeq \C_{\geq 0}\cap \C_{\leq n}$ (See \cite[1.2.1.9]{lurie-ha}).
\end{remark}

By definition, the \emph{heart} of the $t$-structure is the full subcategory $\C^{\heartsuit}$ spanned by the objects in the interesection $\C_{\leq o}\cap \C_{\geq 0}$. It follows from the axioms that $\C^{\heartsuit}$ is equivalent to the nerve of $h(\C^{\heartsuit})$. Given an object $X \in \C$ we denote by $\mathbb{H}_n(X)$ the object of $\C^{\heartsuit}$ obtained by shifting the object $\tau_{\leq n}\tau_{\geq n}(X) \in \C_{\leq n}\cap \C_{\geq n}$.\\

\begin{remark}
\label{fibertruncation}
The cofiber/fiber sequence (\ref{cofibertruncation}) implies that if $X$ is already in $\C_{\leq n}$ (which means that $X\simeq \tau_{\leq n}(X))$ we a have pushout diagram

\begin{equation} 
\label{cofibertruncation2}
\xymatrix{
\Sigma^n \mathbb{H}_n(X)=\tau_{\geq n}\tau_{\leq n}(X)\ar[d]\ar[r]& X\ar[d]\\
\ast \ar[r]& \tau_{\leq n-1}(X)
}
\end{equation}
\end{remark}

The data of a $t$-structure in stable $(\infty,1)$-category $\C$ is completely caracterized by the data of the reflexive localization $\C_{\leq 1}\subseteq \C$ \cite[1.2.1.16]{lurie-ha}. Following this, if $\C$ is an accessible $(\infty,1)$-category we say that the $t$-structure is \emph{accessible} if this localization is accessible. Moreover, we say that the $t$-structure is compatible with filtered colimits if the inclusion $\C_{\leq 0}\subseteq \C$ also commutes with filtered colimits.\\

If $\C$ and $\C'$ are stable $(\infty,1)$-categories carrying $t$-structures, we say that a functor $f:\C\to \C'$ is \emph{right $t$-exact} if it is exact and carries $\C_{\leq 0}$ to $(\C')_{\leq 0}$. Respectively, we say that $f$ is \emph{left $t$-exact} if it is exact and carries $\C_{\geq 0}$ to $(\C')_{\geq 0}$.\\

To conclude this section we recall the notions of left and right completeness. A $t$-structure in $\C$ is said to be \emph{left-complete} if the canonical map from $\C$ to the homotopy limit $\widehat{C}:=lim_{n}\C_{\leq n}$ of the diagram

\begin{equation}
\xymatrix{
... \ar[r]& \C_{\leq 2}\ar[r]^{\tau_{\leq 1}}& \C_{\leq 1}\ar[r]^{\tau_{\leq 0}}&\C_{\leq 0}\ar[r]^{\tau_{\leq -1}}&...
}
\end{equation}

\noindent is an equivalence. A dual definition gives the notion of \emph{right-completeness}. In general this limit is again a stable $(\infty,1)$-category and its objects can be identify with families $X=\{X_i\}_{i\in \mathbb{Z}}$ such that $X_i \in \C_{\leq i}$ and $\tau_{\leq n}X_i\simeq X_n$ for every $n\leq i$. It admits a natural $t$-structure where $X$ is in the positive subcategory if each $X_i$ is in $\C_{\geq 0}$. This $t$-structure makes the canonical map $\C\to lim_{n}\C_{\leq n}$ both left and right $t$-exact. Moreover, the restriction $\C_{\leq 0}\to (\widehat{\C})_{\leq 0}$ is an equivalence \cite[1.2.1.17]{lurie-ha}. In general the difference between $\C$ and $\widehat{\C}$ lays exactly in the connective part. This difference disappears if the $t$-structure is left-complete: the restriction $\C_{\geq 0}\to \widehat{\C}_{\geq 0}\simeq lim_n (\C_{\leq n}\cap \C_{\geq 0})$ is an equivalence. Thanks to the Proposition 1.2.1.19 in \cite{lurie-ha}, a $t$-structure is known to be left-complete if and only if the subcategory $\cap_{n}\C_{\geq n}\subseteq \C$ consists only of zero objects.

\begin{remark}
\label{postnikovtowersconvergewithtstructure}
If $\C$ is a stable $(\infty,1)$-category with a left-complete $t$-structure then Postnikov towers converge in $\C_{\geq 0}$. This follows from the definition of left-completeness and from the Remark \ref{truncationsandtstructures}.
\end{remark}

Again a classical example of a stable $(\infty,1)$-category with a $t$-structure is the $(\infty,1)$-category of spectra $\Sp$ \cite[1.4.3.4, 1.4.3.5, 1.4.3.6]{lurie-ha} where $\Sp_{\geq 0}$ is the full subcategory spanned by the spectrum objects $X$ such that $\pi_n(X)=0, \forall n\leq 0$. It is both right and left complete and its heart is equivalent to the nerve of the category of abelian groups.\\

\subsubsection{Homological Algebra}
 The subject of homological algebra can be properly formulated using the language of stable $(\infty,1)$-categories. If $A$ is a Grothendieck $k$-abelian category, we can obtain the classical unbounded derived category of $A$ has the homotopy category of an $(\infty,1)$-category $\mathcal{D}(A)$. By the main result of \cite{hovey-modelstructureonsheaves} the category of unbouded chain complexes $Ch(A)$ admits a model structure for which the weak-equivalences are the quasi-isomorphisms of complexes and the cofibrations are the monomorphisms (this is usually called the \emph{injective model structure}). We define $\mathcal{D}(A)$ as $(\infty,1)$-category \emph{underlying} this model structure (see Section \ref{section1-2} below). It is stable \cite[Prop. 1.3.5.9]{lurie-ha} and the pair of full subcategories $(\mathcal{D}(A)_{\leq 0}, \mathcal{D}(A)_{\geq 0})$ respectively spanned by the objects whose homology groups vanish in positive degree (resp. negative), determines a right-complete $t$-structure \cite[1.3.5.21]{lurie-ha}. This $t$-structure is not left-complete in general.

If $X$ is a scheme, we know from \cite{har66} that $A=Qcoh(X)$ is Grothendieck abelian. The $(\infty,1)$-category of the Example \ref{derivedinfinitycategoryofascheme} is $\mathcal{D}(A)$.\\

In \cite[Section 1.3]{lurie-ha} the author describes several alternative approaches to access this $(\infty,1)$-category and its subcategory spanned by the right-bounded complexes. We will not review these results here.

\subsection{From Model Categories to $(\infty,1)$-categories}
\label{section1-2}

\subsubsection{Model categories and $\infty$-categories}
Model categories were invented (see \cite{quillen}) as axiomatic structures suitable to perform the classical notions of homotopy theory. They have been extensively used and developed (see \cite{hovey-modelcategories,hirschhorn} for an introduction) and still form the canonical way to introduce/present homotopical studies. A typical example is the homotopy theory of schemes which provides the motivation for this paper.
The primitive ultimate object associated to a model category $\M$ is its homotopy category $h(\M)$ which can be obtained as a localization of $\M$ with respect to the class $W$ of weak-equivalences. This localization should be taken in the world of categories. The problems start when we understand that $h(\M)$ lacks some of the interesting homotopical information contained in $\M$ up to such a point that it is possible to have two model categories which are not equivalent but their homotopy categories are equivalent (see \cite{duggerexample}). This tells us that $h(\M)$ is not an ultimate invariant and that in order to do homotopy theory we should not abandon the setting of model categories. But this brings some troubles. To start with, the theory of model categories is not "closed"  meaning that, in general, the collection of morphisms between two model categories does not provide a new model category. Moreover, the theory is not suitable adapted to the consideration of general homotopy theories with monoidal structures, their associated theories of homotopy algebra-objects and modules over them. 

The quest to solve these problems is one of the possible motivations for the subject of $(\infty,1)$-categories. Every model category should have an associated $(\infty,1)$-category which should work as an ultimate container for the homotopical information in $\M$. In particular, the information about the homotopy category. The original motivation for the subject had its origins in the famous manuscript of A. Grothendieck \cite{pursuingstacks}. In the last few years there were amazing developments and the reader has now available many good references for the different directions \cite{ bergner-survey, simpson-book, lurie-htt,DimitriAra, rezk-homotopytheoryofhomotopytheories}.\\
Back to our discussion, the key idea is that every model category $\M$ hides an $(\infty,1)$-category and this $(\infty,1)$-category encodes all the "`homotopical information"' contained in $\M$. The key idea dates to the works of Dwyer-Kan \cite{dwyer-kan-simpliciallocalizationofcategories,dwyer-kan-homotopyfunctioncomplexes} who found out that by performing the "`simplicial localization of $\M$"' - meaning a localization in the world of simplicial categories - instead of the usual localization in the setting of ordinary categories, the resulting object would contain all the interesting homotopical information and, in particular, the classical homotopy category of $\M$ appearing in the "ground" level of this localization. The meaning of the preceding technique became clear once it was understood that simplicial categories are simply one amongst many other possible models for the theory of $(\infty,1)$-categories. Another possible model is provided by the theory of Joyal's quasi-categories, which was extensively developed in the recent years \cite{lurie-htt}. The method to assign an $(\infty,1)$-category to a model category $\M$ reproduces the original idea of Dwyer and Kan - Start from $\M$, see $\M$ as a trivial $(\infty,1)$-category and perform the localization of $\M$ with respect to the weak-equivalences - not in the world of usual categories - but in the world of $(\infty,1)$-categories. The resulting object will be refer to as the \emph{underlying $\infty$-category of the model category } $\M$. For a more detailed exposition on this subject we redirect the reader to the exposition in \cite{toen-hab}.

For our purposes we need to understand that the nerve functor $N:\Cat\to \ssets$ provides the way to see a category as a trivial quasi-category. By definition, if $\M$ is model category with a class of weak-equivalences $W$, the underlying $(\infty,1)$-category of $\M$ is the localization $N(\M)[W^{-1}]$ obtained in the setting of $(\infty,1)$-categories using the process described in \ref{locinfinity}. Moreover, the universal property of this new object implies that its associated homotopy category $h(N(\M)[W^{-1}])$ recovers the classical localization. In particular, $N(\M)[W^{-1}]$ and $N(\M)$ have essentially the same objects. The main technical result which was originally discovered by Dwyer and Kan is

\begin{prop}\label{dwyer-kan}(\cite{lurie-ha}-Prop. 1.3.4.20  )\\
Let $\M$ be a simplicial model category\footnote{Assume the existence of functorial factorizations}. Then there exists an equivalence of $(\infty,1)$-categories between the underlying $\infty$-category of $\M$ and the $(\infty,1)$-category $N_ {\Delta}(\M^{\circ})$ where $N_{\Delta}$ is the simplicial nerve construction (see Def. 1.1.5.5 of \cite{lurie-htt}) and $\M^{\circ}$ denotes the full simplicial subcategory of $\M$ of cofibrant-fibrant objects. In other words we have an isomorphism

\begin{equation}
N(\M)[W^{-1}]\simeq N_ {\Delta}(\M^{\circ})
\end{equation}

\noindent in the homotopy category of simplicial sets with the Joyal model structure \cite{joyal-article}.
\end{prop}

This statement provides an $\infty$-generalization of the fundamental result by Quillen (see \cite{quillen}) telling us that the localization $Ho(\M)$ is equivalent to the naive homotopy theory of cofibrant-fibrant objects.
By combining this result with the Theorem 4.2.4.1 of \cite{lurie-htt}, we find a dictionary between the classical notions of homotopy limits and colimit in $\M$ (with $\M$ simplicial) and limits and colimits in the underlying $(\infty,1)$-category of $\M$.\\

\subsubsection{Combinatorial Model Categories}
\label{combinatorialmodelcategories}

The theory of combinatorial model categories and that of presentable $(\infty,1)$-categories are equivalent. Morever, this equivalence is compatible with Bousfield localizations:

\begin{prop}(\cite{lurie-htt}-Propositions A 3.7.4 and A.3.7.6)
\label{mainforcombi}
Let $\C$ be a big $(\infty,1)$-category. Then, $\C$ is presentable if and only if there exists a big $\uniU$-combinatorial simplicial model category $\M$ such that $\C$ is the underlying $(\infty,1)$-category of $\M$. Moreover, if $\M$ is left-proper, Bousfield localizations of $\M$ \footnote{with respect to a class of morphisms of small generation} correspond bijectively to accessible reflexive localizations of $\C$ (see our Notations).
\end{prop}

This has many important consequences. To start with, if $\xymatrix{\M\ar@<+.7ex>[r]^{F}& \ar@<+.7ex>[l]^{G} \N}$ is a Quillen adjunction between combinatorial model categories then it induces an adjunction between the underlying $(\infty,1)$-categories. To see this, remember from our preliminairs that the localization $N(\M)[W^{-1}]$ can be obtained as a fibrant-replacement for the pair $(N(\M), W)$ in the model category of marked simplicial sets. Under the combinatorial hypothesis, $\M$ admits cofibrant and fibrant replacement functors and of course, they preserve weak-equivalences. If we let $\M^{c}$ denote the full subcategory of $\M$ spanned by the cofibrant and $W_c$ the weak-equivalences between them, we will have an inclusion of marked simplicial sets $(N(\M^c),W_c)\subseteq (N(\M),W)$ together with a map in the inverse direction provided by the cofibrant-replacement functor (the same applies for the subcategories of fibrant, resp. cofibrant-fibrant, objects). By the universal property of the localization, these two maps provide an equivalence of $(\infty,1)$-categories $N(\M^c)[W_c^{-1}]\simeq N(\M)[W^{-1}]$. Back to the Quillen adjunction $(F,G)$, the Ken Brown's lemma provides a well-defined map of marked simplicial sets

\begin{equation}
(N(\M^c), W_c)\to (N(\N^c),W_c')
\end{equation}

\noindent and therefore, a new one $N(\M^c)[W_c^{-1}]\to N(\N_c)[W_c'^{-1}]$ through the choice of fibrant-replacements in the model category of marked simplicial sets. It is the content of
 \cite[1.3.4.21]{lurie-ha} that if the initial Quillen adjunction is an equivalence then this map is will also be.

Thanks to the results of \cite{dugger-combinatorial} we know that every combinatorial model category is Quillen equivalent (by a zig-zag) to a simplicial combinatorial model category. The previous discussion implies that the underlying $(\infty,1)$-category of a combinatorial model category is always presentable. In particular, it admits all limits and colimits which, again by the results  of \cite{dugger-combinatorial} together with the Theorem 4.2.4.1 of \cite{lurie-htt}, can be computed as homotopy limits and homotopy colimits in $\M$, namely, an object $X\in \M$ is an homotopy limit (resp. colimit) of a diagram $I\to \M$ if and only if it is a limit (resp. colimit) in $N(\M)[W^{-1}]$ of the composition $N(I)\to N(\M)\to N(\M)[W^{-1}]$, now in the sense of $(\infty,1)$-categories (see \cite[1.3.4.23 and 1.3.4.24]{lurie-ha}).

 Moreover, combining the Theorem 4.2.4.4 of \cite{lurie-htt} again with the main result of \cite{dugger-combinatorial} we find that for any combinatorial model category $\M$ and small category $I$, there is an equivalence  

\begin{equation}
N(\M^I)[W_{levelwise}^{-1}]\simeq Fun(N(I),N(\M)[W^{-1}])
\end{equation}

In particular, for a left Quillen map between combinatorial model categories, the map induced between the underlying $(\infty,1)$-categories (as above) commutes with colimits. The presentability, together with the adjoint functor theorem (Corollary 5.5.2.9 of \cite{lurie-ha}) implies the existence of a right adjoint $N(\M^c)[W_c^{-1}]\leftarrow N(\N_c)[W_c'^{-1}]$ which we can describe explicitly as the composition 

\begin{equation}
\xymatrix{
N(\N^c)\ar[r]^P& N(\N^{cf})\ar[r]^G & N(\M^{f})\ar[r]^Q & N(\M^{cf})\ar@{^{(}->}[r]& N(\M^{c})  
}
\end{equation}

where $P$ is a fibrant replacement functor in $\N$ and $Q$ is a cofibrant replacement functor in $\M$. 

In the simplicial case the underlying adjunction can be obtained with simpler technology (see the Proposition 5.2.4.6 of \cite{lurie-htt} defining $\bar{F}(X)$ as a fibrant replacement of $F(X)$ and $\bar{G}(Y)$ via a cofibrant replacement of $G(Y)$.\\

\subsubsection{Compactly Generated Model Categories}
\label{compactlygeneratedmodelcategories}

The following discussion will be usefull in the last part of this work. Let $\M$ be a model category. Recall that an object $X$ in $\M$ is said to be \emph{homotopically finitely presented} if the mapping space functor $Map(X,-)$ commutes with homotopy filtered colimits. Recall also that if $\M$ is cofibrantly generated with $I$ a set of generating cofibrations, then $X$ is said to be a strict finite $I$-cell if there exists a finite sequence of morphisms in $\M$

\begin{equation}
X_0=\emptyset \to X_1\to ...\to X_n=X
\end{equation}

\noindent such that for any $i$, we have a pushout square

\begin{equation}
\xymatrix{
X_i\ar[r]& X_{i+1}\\
\ar[u] A \ar[r]^s& \ar[u]B
}
\end{equation}

\noindent with $s\in I$. Recall also that $\M$ is said to be \emph{compactly generated } if it is cellular and there is a set of generating cofibrations (resp. trivial cofibrations) $I$ (resp. $J$) whose domains and codomains are cofibrant and (strictly) $\omega$-compact and (strictly) $\omega$-small with respect to the whole category $\M$. We have the following result

\begin{prop}(\cite{toen-vaquie} Prop. 2.2)
\label{compactlygeneratedmodelcategoriesprop}
Let $\M$ be a compactly generated model category. Then any object is equivalent to a filtered colimit of strict finite $I$-cell objects. Moreover, if the (strict) filtered colimits in $\M$ are exact, an object $X$ is homotopically finitely presented if and only if it is a retract of a strict finite $I$-cells object.
\end{prop}

This proposition, together with the results of \cite{lurie-ha} described in the last section, implies that if $\M$ is a combinatorial compactly generated model category where (strict) filtered colimits are exact, then the compact objects in the presentable $(\infty,1)$-category $N(\M)[W^{-1}]$ are exactly the homotopicaly finitely presented objects in $\M$. Moreover, we have a canonical equivalence $N(\M)[W^{-1}]\simeq Ind((N(\M)[W^{-1}])^{\omega})$ (consult our Notations).

\section{ Preliminaries II - Higher Algebra}
\label{section3}

Our goal in this section is to review the fundaments of the subject of \emph{higher algebra} as developed in the works of J. Lurie in \cite{lurie-ha}. We collect the main notions and results and provide some new observations and tools needed in the later sections of the paper.

\subsection{$\infty$-Operads and Symmetric Monoidal $(\infty,1)$-categories}
\label{section3-1}

The (technical) starting point of higher algebra is the definition of a symmetric monoidal structure on a $(\infty,1)$-category (see Section \ref{link2} for the philosophical motivations). The guiding principle is that a symmetric monoidal $(\infty,1)$-category is the data of an $(\infty,1)$-category together with an operation $\C\times \C\to \C$, a unit object $\Delta[0]\to \C$ and a collection of commutative diagrams providing associative and commutative restrains. There are three main reasons why a precise definition is difficult using brute force: $(i)$ we don't know how to precise explicitly the whole list of diagrams; $(ii)$ these diagrams are expected to be interrelated; $(iii)$ in higher category theory the data of a commutative diagram is not a mere collection of vertices and edges: commutativity is defined by the existence of higher cells. The first and second problem exist already in the classical setting. The third problem makes the higher setting even more complicated for now to give $(i)$ is to specify higher cells in $\iCat$ and $(ii)$ is to write down explicit relations between them.\\

In this work we will follow the approach of \cite{lurie-ha} where a symmetric monoidal $(\infty,1)$-category is a particular instance of the notion of $\infty$-operad. In order to understand the idea, we recall that both classical operads and classical symmetrical monoidal categories can be seen as particular instances of the more general notion of colored operad (also known as "`multicategory"'). At the same time, classical symmetric monoidal categories can be understood as certain types of diagrams of categories indexed by the category $\Fin$ of pointed finite sets. Using the \emph{Grothendieck-Construction}, we can encode this diagramatic definition of a symmetric monoidal category in the form of a category cofibered over $\Fin$ with an additional property - the fiber over a finite set $\nfin$ is equivalent to the $n$-th power of the fiber over $\onefin$ (follow the notations below). Moreover, by weakening this form, it is possible to reproduce the notion of a coloured operad in this context. This way - operads, symmetric monoidal structures and coloured operads - are brought to the same setting: everything can be written in the world of "things over $\Fin$". The book \cite{leinster-book} provides a good introduction to these ideas. \\

In \cite{dendroidal1, dendroidal2} the authors explore another approach to the theory of $(\infty,1)$-operads. The key observation is that the theory of simplicial sets is not enough to capture the structure of a multicategory. To correct this failure they propose the notion of dendroidal sets, for which they prove the existence of an appropriate homotopy theory. More recently in \cite{dendroidal3} the authors established an equivalence between this approach and the theory developed by J. Lurie in \cite{lurie-ha}. \\

\subsubsection{$\infty$-operads}
In order to provide the formal definitions we need to recall some of the terminology introduced in \cite{lurie-ha}. First,we write $\nfin\in N(\Fin)$ to denote the finite set $\{0,1,..., n\}$ with $0$ as the base point and $\nfin^+$ to denote its subset of non-zero elements. A morphism $f:\nfin\to \mfin$ will be called \emph{inert} if for each $i\in \mfin^+$, $f^{-1}(\{i\})$ has exactly one element. Alternatively, a map $f$ is inert iff it is surjective and the induced map $\nfin-f^{-1}(\{0\})\to \mfin^+$ is a bijection. Notice that the canonical maps $\nfin\to \langle 0 \rangle$ are inert. Moreover, for each $i\in \nfin^+$, we write $\rho^i:\nfin \to \onefin$ for the inert map sending $i$ to $1$ and everything else to $0$. We say that $f$ is \emph{active} if $f^{-1}(\{0\})=\{0\}$.

\begin{defn}(\cite{lurie-ha}- Definition 2.1.1.10)\\
An $\infty$-operad is an $\infty$-category $\Opmonoidal$ together with a map $p:\Opmonoidal\to N(\Fin)$ satisfying the following list of properties:
\begin{enumerate}
\item For every inert morphism $f:\mfin\to \nfin$ and every object $C$ in the fiber $\Opmonoidal_{\mfin}:=p^{-1}(\{\mfin\})$, there exists a $p$-coCartesian morphism (see Definition of \cite{lurie-htt}) $\bar{f}:C\to \bar{C}$ lifting $f$. In particular, $f$ induces a functor $f_{!}:\Opmonoidal_{\mfin}\to \Opmonoidal_{\nfin}$;
\item Given objects $C$ in $\Opmonoidal_{\mfin}$ and $C'$ in $\Opmonoidal_{\nfin}$ and a morphism $f:\mfin \to \nfin$ in $\Fin$, we write $Map_{\Opmonoidal}^f(C,C')$ for the disjoint union of those connected components of $Map_{\Opmonoidal}(C,C')$ which lie over $f\in Map_{N(\Fin)}(\mfin,\nfin):= Hom_{Fin}(\mfin,\nfin)$.

We demand the following condition: whenever we choose $p$-coCartesian morphisms $C' \to C'_i$ lifting the inert morphisms $\rho^i:\nfin\to \onefin$ for $1\leq i\leq n$ (these liftings exists by $(1)$), the induced map

\begin{equation}
Map_{\Opmonoidal}^f(C,C')\to \prod Map_{\Opmonoidal}^{\rho^i\circ f}(C,C'_i)
\end{equation}

is an homotopy equivalence of spaces;

\item For each $n\geq 0$, the functors $\rho_{!}^i:\Opmonoidal_{\nfin}\to \Op$ (where $\Opmonoidal$ denotes the fiber over $\onefin$) induced by the inert maps $\rho^i$ through condition $(1)$, induce an equivalence of $(\infty,1)$-categories $\Opmonoidal_{\nfin}\to \Op^n$. In particular, for $n=0$ we have $\Opmonoidal_{\langle 0 \rangle}\simeq \Delta[0]$.
\end{enumerate}

\end{defn}

Notice that with this definition any $\infty$-operad $\Opmonoidal\to N(\Fin)$ is a categorical fibration. From now, we will make an abuse of notation and write $\Opmonoidal$ for an $\infty$-operad $p:\Opmonoidal\to N(\Fin)$, ommiting the structure map to $N(Fin)$. We denote the fiber over $\onefin$ by $\Op$ and refer to it as the \emph{underlying $\infty$-category} of $\Opmonoidal$. The objects of $\Op$ are called the \emph{colours} or \emph{objects} of the $\infty$-operad $\Opmonoidal$. To illustrate the definition, condition $(3)$ tells us that any object $C\in \Opmonoidal$ living over $\nfin$ can be identified with a unique (up to equivalence) collection $(X_1,X_2,...,X_n)$ where each $X_i$ is an object in $\Op$. Moreover, if $C\to C'$ is a coCartesian morphism in $\Opmonoidal$ lifting an inert morphism $\nfin\to \mfin$ and if $C=(X_1,X_2,...,X_n)$ then $C'$ corresponds (up to equivalence) to the collection $(X_{f^{-1}(\{1\})},...,X_{f^{-1}(\{m\})})$. In other words, coCartesian liftings of inert morphisms $C=(X_1,X_2,...,X_n)\to C'$ in $\Opmonoidal$ correspond to the selection of $m$ colours (without repetition) out of the $n$ presented in $C$.
Finally, if $C=(X_i)_{1\leq i\leq n}$ and $C'= (X'_j)_{1\leq i\leq m}$ are objects in $\Opmonoidal$, condition $(2)$ tells us that 

\begin{equation}
Map_{\Opmonoidal}((X_i)_{1\leq i\leq n}, (X'_j)_{1\leq i\leq m})\simeq \prod_j Map_{\Opmonoidal}((X_i)_{1\leq i\leq n}, X'_j)
\end{equation}

Let $p:\Opmonoidal\to N(\Fin)$ be an $\infty$-operad. We say that a morphism $f$ in $\Opmonoidal$ is \emph{inert} if its image in $N(\Fin)$ is inert and $f$ is $p$-coCartesian. We say that $f$ is active if $p(f)$ is active. By the Proposition 2.1.2.4 of \cite{lurie-ha}, the collections $(\{\text{inert morphism}\}, \{\text{active morphisms}\})$ form a (\emph{strong}) \emph{factorization system} in $\Opmonoidal$ (Definition 5.2.8.8 of \cite{lurie-htt}).\\

The simplest example of an $\infty$-operad is the identity map $N(\Fin)\to N(\Fin)$. Its underlying $(\infty,1)$-category corresponds to $\Delta[0]$. It is called the \emph{commutative $\infty$-operad} and we use the notation $\Commmonoidal=N(\Fin)$. Another simple example is the \emph{trivial} $\infty$-operad $\Trivmonoidal$. By definition, it is given by the full subcategory of $N(\Fin)$ of all objects $\nfin$ together with the inert morphisms.\\

More generally, there is a mechanism - the so-called \emph{operadic nerve} $N^{\otimes}(-)$ - to produce an $\infty$-operad out of \emph{simplicial coloured operad whose mapping spaces are Kan-complexes}.

\begin{construction}
\label{monoidaltilde}
If $\A$ is a simplicial coloured operad, we construct a new simplicial category $\tilde{\A}$ as follows: the objects of $\tilde{\A}$ are the pairs $(\nfin, (X_1,..., X_n))$ where $\nfin$ is an object in $\Fin$ and $(X_1,..., X_n)$ is a sequence of colours in $\A$. The mapping spaces are defined by the formula

\begin{equation}
Map_{\tilde{\A}}((X_1,..., X_n), (Y_1,..., Y_m)):= \coprod_{f:\nfin\to \mfin} \prod_{i=1}^{m} Map_{\A}((X_{\alpha})_{\alpha\in f^{-1}(\{i\})}, Y_i)
\end{equation}

If $\A$ is enriched over Kan complexes, it is immediate that $\tilde{\A}$ is a fibrant simplicial category. Following the Definition 2.1.1.23 of \cite{lurie-ha} we set $N^{\otimes}(\A):= N_{\Delta}(\tilde{\A})$. In this case (see \cite[2.1.1.27]{lurie-ha}) the canonical projection $\pi:N^{\otimes}(\A)\to N(\Fin)$ is an $(\infty,1)$-operad. In particular, this mechanism works using a classical operad as input.\\

\end{construction}

\begin{example}
\label{associativeoperad}
This mechanism can be used to construct the \emph{associative operad} $\Assmonoidal$. Following the Definition 4.1.1.3 of \cite{lurie-ha}), we let $\textbf{Ass}$ be the multicategory with one color $\textbf{a}$ and having as set of operations $Hom(\{\textbf{a}\}_{I}, \textbf{a})$ the set of total order relations on $I$. In other words, an operation 

\begin{equation}
\underbrace{(\textbf{a},..., \textbf{a})}_{n}\to \textbf{a}
\end{equation}

\noindent consists of a choice of a permutation of the n-factors. We can now understand $\textbf{Ass}$ as enriched over constant Kan-complexes and applying the Construction \ref{monoidaltilde} we find a fibrant simplicial category $\widetilde{\textbf{Ass}}$ whose simplicial nerve is by definition, the associative $\infty$-operad $\Assmonoidal$. 

Explicitly, the objects of $\Assmonoidal$ can be identified with the objects of $N(\Fin)$. Morphisms $f:\nfin\to \mfin$ are given by the choice of a morphisms in $N(\Fin)$, $\nfin\to \mfin$ together with the choice of a total order on each $f^{-1}(\{j\})$ for each $j\in \mfin^{\circ}$. With this description,it is obvious that $\Assmonoidal$ comes equipped with a map towards $N(\Fin)$ obtain by forgetting the total orderings.

\end{example}

\begin{example}
\label{enoperads}
The associative operad represents the first element in a distinguished family of $\infty$-operads: for any natural number $n\in \mathbb{N}$, we can construct a fibrant simplicial colored operad \cite[Def. 5.1.0.2]{lurie-ha} whose simplicial nerve $\mathbb{E}_n^{\otimes}$ is called the $\infty$-operad of \emph{little $n$-cubes}. For every $n\geq 0$ the objects of $\mathbb{E}_n^{\otimes}$ are the same objects of $N(\Fin)$ and in particular it only has one color. When $n=1$, there is an equivalence $\mathbb{E}_1^{\otimes}\simeq \Ass$. For every $n\in \mathbb{N}$, there is a natural map of $\infty$-operads $\mathbb{E}_n^{\otimes}\to \mathbb{E}_{n+1}^{\otimes}$ and by \cite[Cor. 5.1.1.5]{lurie-ha},
the colimit of the sequence (in the $(\infty,1)$-category of operads described in the next section)

\begin{equation}
\mathbb{E}_0^{\otimes}\to \mathbb{E}_{1}^{\otimes}\to \mathbb{E}_2^{\otimes}\to ...
\end{equation}

\noindent is the commutative operad $\Commmonoidal$.
\end{example}

\subsubsection{The $(\infty,1)$-category of $\infty$-operads}
By definition, a map of $\infty$-operads is a map of simplicial sets $\Opmonoidal\to \Opprimemonoidal$ over $N(\Fin)$, sending inert morphisms to inert morphisms. Following \cite{lurie-ha}, we write $Alg_{\Op}(\Op')$ to denote the full subcategory of $Fun_{N(\Fin)}(\Opmonoidal, \Opprimemonoidal)$ spanned by the maps of $\infty$-operads.\\

The collection of $\infty$-operads can be organized in a new $(\infty,1)$-category $\mathit{Op}_{\infty}$ which can be obtained as the simplicial nerve of the fibrant simplicial category whose objects are the $(\infty,1)$-operads and the mapping spaces are the maximal Kan-complexes inside $Alg_{\Op}(\Op')$. According to the Proposition 2.1.4.6 of \cite{lurie-ha}, there is a model structure in the category of marked simplicial sets over $N(\Fin)$ which has $\mathit{Op}_{\infty}$ as its underlying $(\infty,1)$-category.

\subsubsection{Symmetric Monoidal $(\infty,1)$-categories}
We say that a map of $\infty$-operads $q:\Cmonoidal\to \Opmonoidal$ is a \emph{fibration of $\infty$-operads} (respectively \emph{coCartesian fibration of $\infty$-operads}) if it is a categorical fibration (resp. coCartesian fibration) of simplicial sets (see Definition 2.4.2.1  of \cite{lurie-htt}).

\begin{defn}
Let $\Opmonoidal$ be an $\infty$-operad. An \emph{$\Op$-monoidal $(\infty,1)$-category} is the data of an $\infty$-operad $\Cmonoidal\to N(\Fin)$ together with a coCartesian fibration of $\infty$-operads $\Cmonoidal\to \Opmonoidal$. A \emph{symmetric monoidal $(\infty,1)$-category} is a $\Comm$-monoidal $(\infty,1)$-category.
\end{defn}

Let $p:\Cmonoidal\to N(\Fin)$ be a symmetric monoidal $(\infty,1)$-category. As for general $\infty$-operads, we  denote by $\C$ the fiber of $p$ over $\onefin$ and refer to it as the \emph{underlying $(\infty,1)$-category} of $\Cmonoidal$. To understand how this definition encodes the usual way to see the monoidal operation, we observe that if $f:\nfin\to \mfin$ is an active morphism in $N(\Fin)$ and $C=(X_1,..., X_n)$ is an object in the fiber of $\nfin$  (notation: $\Cmonoidal_{\nfin}$), by the definition of a coCartesian fibration, there exists a $p$-coCartesian lift of $f$, $\tilde{f}:C\to C'$ where we can identify $C'$ with a collection $(Y_1,..., Y_m)$, with each $Y_i$ an object in $\C$. The coCartesian property motivates the identification

\begin{equation}
Y_i= \bigotimes_{\alpha\in f^{-1}(\{i\})}X_{\alpha}
\end{equation}

\noindent where the equality should be understood only in the philosophical sense. When applied to the active morphisms $\langle 0 \rangle \to \onefin$ and $\twofin\to \onefin$ we obtain functors $\mathit{1}:\Delta[0]\to \C$ and $\otimes:\C\times \C\to \C$. We will refer to the first as the \emph{unit} of the monoidal structure. The second recovers the usual multiplication. By playing with the other active morphisms we  recover the usual data defining a symmetric monoidal structure. The coherences will appear out of the properties characterizing the coCartesian lifts.\\

It is an important observation that these operations endow the homotopy category of $\C$ with symmetric monoidal structure in the classical sense. \\

\begin{example}
\label{classicalsymmetricmonoidal}
Let $\C$ be a classical symmetric monoidal category. By regarding it as a trivial simplicial coloured operad and using the Construction \ref{monoidaltilde} we obtain an $\infty$-operad $N^{\otimes}(\C)\to N(\Fin)$ which is a symmetric monoidal $(\infty,1)$-category whose underlying $(\infty,1)$-category is equivalent to the nerve $N(\C)$.
\end{example}
 
We recommend the reader to follow the highly pedagogical exposition in \cite{lurie-ha}, which is also a good introduction to the classical theory. \\

\subsubsection{Monoidal Functors} Let $p:\Cmonoidal\to \Opmonoidal$ and $q:\Dmonoidal\to \Opmonoidal$ be $\Op$-monoidal $\icategories$ and let $F:\Cmonoidal\to \Dmonoidal$ be a map of $\infty$-operads over $\Opmonoidal$. To simplify things let us consider $\Op=\Comm$. For any object object $C=(X,Y)$ in the fiber over $\twofin$, the definitions allow us to extract a natural morphism in $\D$

\begin{equation}F(X)\otimes F(Y)\to F(X\otimes Y)\end{equation}

\noindent which in general have no need to be an isomorphism. In other words, operadic maps correspond to \emph{lax} monoidal functors. In the general situation, the full compatibility between the monoidal structures is equivalent to ask for $F$ to send $p$-coCartesian morphisms in $\Cmonoidal$ to $q$-coCartesian morphisms in $\Dmonoidal$. These are called \emph{$\Op$-monoidal functors} and we write $Fun^{\otimes}_{\Op}(\C,\D)$ to denote their full subcategory inside

\begin{equation} Fun_{\Opmonoidal}(\Cmonoidal,\Opmonoidal):= \Opmonoidal\times_{Fun(\Cmonoidal,\Opmonoidal)}Fun(\Cmonoidal,\Dmonoidal)
\end{equation}

Following the Remark 2.1.3.8 of \cite{lurie-ha}, an $\Op$-monoidal functor $\Cmonoidal\to \Dmonoidal$ is an equivalence of $\infty$-categories if and only if the map induced between the underlying $\infty$-categories $\C\to \D$ is an equivalence.\\

\subsubsection{Objectwise product on diagram categories}

 Let $p:\Cmonoidal\to N(\Fin)$ be a symmetric monoidal $(\infty,1)$-category. Given an arbitrary simplicial set $K$, and similarly to the classical case, we can hope for the existence of a monoidal structure in $Fun(K,\C)$ defined objectwise, meaning that the product of two functors $f$, $g$ at an object $k\in K$ should be given by the product of $f$ and $g$ at $\kappa$, in $\Cmonoidal$. Indeed, there exists such a structure $Fun(K, \C)^{\otimes}$, defined as the homotopy pullback of the diagram of $\infty$-categories

\begin{equation}
\xymatrix{
&Fun(K, \Cmonoidal)\ar[d]\\
N(\Fin)\ar[r]^{\delta}&Fun(K, N(\Fin))
}
\end{equation}

\noindent where the vertical map corresponds to the composition with $p$ and the map $\delta$ sends an object $\nfin$ to the constant diagram in $N(\Fin)$ with value $\nfin$. By the Proposition 3.1.2.1 of \cite{lurie-htt}, the composition map $Fun(K, \Cmonoidal) \to Fun(K, N(\Fin))$ is also a coCartesian fibration and therefore since object in the diagram is fibrant, the homotopy pullback is given by the strict pullback. Moreover, since the natural map $Fun(K, \C)^{\otimes}\to N(\Fin)$ is a cocartesian fibration because it is the pull-back of a cocartesian fibration. Notice that the underlying $(\infty,1)$-category of $Fun(K, \C)^{\otimes}$ is equivalent to $Fun(K, \C)$ by the formulas

\begin{eqnarray}
Fun(K,\C)^{\otimes}\times_{N(\Fin)}^h \Delta[0]\simeq Fun(K,\Cmonoidal)\times^h_{Fun(K,N(\Fin))} N(\Fin)\times^h_{N(\Fin)}\Delta[0]\simeq \\
\simeq Fun(K, \Cmonoidal)\times^h_{Fun(K, N(\Fin))} \Delta[0]\\
\simeq Fun(K,\C)
\end{eqnarray}

In fact this constructions holds if we consider $\Cmonoidal\to \Opmonoidal$ to be any coCartesian fibration of operads (see the Remark 2.1.3.4 of \cite{lurie-ha}).

\subsubsection{Subcategories closed under the monoidal product}
\label{subcategoriesclosedunderproduct}
If $p:\Cmonoidal\to N(\Fin)$ is a symmetric monoidal $(\infty,1)$-category with underlying category $\C$, whenever we have $\C_0\subseteq \C$ a full subcategory of $\C$ we can ask if the monoidal structure $\Cmonoidal$ can be restricted to a new one $(\C_0)^{\otimes}$ in $\C_0$. By the Proposition 2.2.1.1 and the Remark 2.2.1.2 of \cite{lurie-ha}, if $\C_0$ is stable under equivalences (meaning that if $X$ is an object in $\C_0$ and $X\to Y$ (or $Y\to X$) is an equivalence in $\C$, then $Y$ is in $\C_0$) and if $\C_0$ is closed under the tensor product $\C\times \C\to \C$ and contains the unit object, then we have that restriction of $p$ to the full subcategory $(\C_0)^{\otimes}\subseteq \Cmonoidal$ spanned by the objects $X=(X_1,..., X_n)$ in $\Cmonoidal$ where each $X_i$ is in $\C_0$, is again a coCartesian fibration. Of course, the inclusion $\C_0^{\otimes}\subseteq \Cmonoidal$ is a monoidal functor. Moreover, if the inclusion $\C_0\subseteq \C$ admits a right adjoint $\tau$, it can be naturally extended to a map of $\infty$-operads $\tau^{\otimes}: \Cmonoidal\to \C_{0}^{\otimes}$. In particular, for any $\infty$-operad $\Opmonoidal$, $\tau^{\otimes}$ gives a right adjoint to the canonical inclusion

\begin{equation}Alg_{\Op}(\C_0)\hookrightarrow Alg_{\Op}(\C)\end{equation}

(see \ref{section3-2} below).

\subsubsection{Monoidal Reflexive Localizations}
\label{mfl}
Let again $p:\Cmonoidal\to N(\Fin)$ be a symmetric monoidal $(\infty,1)$-category. In the sequence of the previous topic, we can find situations in which a full subcategory $\C_0\subseteq \C$ is not stable under the product in $\C$ but we can still define a monoidal structure in $\C_0$. We say that a subcategory $\Dmonoidal\subseteq \Cmonoidal$ is a \emph{monoidal reflexive localization of $\Cmonoidal$} if the inclusion $\Dmonoidal\subseteq \Cmonoidal$ admits a left adjoint map of $\infty$-operads $L^{\otimes}:\Cmonoidal\to\Dmonoidal $ with $L^{\otimes}$ a monoidal map. By the Proposition 2.2.1.9 of \cite{lurie-ha}, if $\C_0$ is a reflexive localization of $\C$ and the localization  satisfies the condition: \\

$(*)$ for every $L$-equivalence $f:X\to Y$ in $\C$ (meaning that $L(f)$ is an equivalence) and every object $Z$ in $\C$, the induced map $X\otimes Z\to Y\otimes Z$ is again a $L$-equivalence (see the Definition 2.2.1.6, Remark 2.2.1.7 in \cite{lurie-ha}). \\

\noindent then the full subcategory $\C_0^{\otimes}$ of $\Cmonoidal$ defined in the previous topic, becomes a monoidal reflexive localization of $\Cmonoidal$. However, and contrary to the previous situation, the inclusion $\C_0^{\otimes} \subseteq \Cmonoidal$ will only be lax monoidal.

\begin{remark}
\label{reflexivelocalizationalgebras}
If $\C_{0}^{\otimes}\subseteq \Cmonoidal$ is a monoidal reflexive localization, then for any $\infty$-operad $\Opmonoidal$ the category of algebras $Alg_{\Op}(\C_0)$ is a reflexive localization of $Alg_{\Op}(\C)$. (see \ref{section3-2} below for the theory of algebras).
\end{remark}

\subsubsection{Cartesian and Cocartesian Symmetric Monoidal Structures}

We recall the analogues of two classical situations. If $\C$ is a category with finite products and a final object then the operation (-$\times$-) gives birth to a symmetric monoidal structure in $\C$. In \cite[Section 2.4.1]{lurie-ha} the author provides a mechanism that allows us to extend this classical situation to the $\infty$-setting. For any $(\infty,1)$-category $\C$ we can construct a new $(\infty,1)$-category $\C^{\times}$ equipped with a map to $N(\Fin)$ \cite[2.4.1.4]{lurie-ha} being this map a symmetric monoidal structure if and only if $\C$ admits finite products \cite[2.4.1.5]{lurie-ha}. More generally, and thanks to the results of \cite[2.4.1.6, 2.4.1.7 and 2.4.1.8]{lurie-ha} a symmetric monoidal $(\infty,1)$-category $\Cmonoidal$ is said to be \emph{Cartesian} if the underlying $\infty$-category of $\C$ admits finite products and we have a monoidal equivalence $\Cmonoidal\simeq \C^{\times}$ which is the identity on $\C$. Moreover, the construction $\C^{\times}$ is characterized by a universal property related to the preservation of products: if $\C$ and $\D$ are $(\infty,1)$-categories with finite products then the space of monoidal maps $\C^{\times}\to \D^{\times}$ is homotopy equivalent to the space of functors $\C\to \D$ that preserve products. \\

The second classical situation is that of a category with finite sums and an initial object. In the $\infty$-categorical setting we can apply the preceding argument to the opposite category of $\C$ to deduce the existence of a monoidal structure induced by the disjoint sums in $\C$. In \cite[Section 2.4.3]{lurie-ha} the author provides an independent description of this monoidal structure. Again, from any $(\infty,1)$-category $\C$ we can construct (see \cite[2.4.3.1]{lurie-ha}) a simplicial set $\C^{\coprod}$ together with a map to $N(\Fin)$ which we can prove to be always an $\infty$-operad \cite[2.4.3.3]{lurie-ha}. Finally, and as explained in the Remark \cite[2.4.3.4]{lurie-ha} this $\infty$-operad is a symmetric monoidal $(\infty,1)$-category if and only if and $\C$ has finite sums and an initial object.  With this, we say that an $\infty$-operad is \emph{cocartesian} if it is equivalent to one of the form $\C^{\coprod}$ for some $(\infty,1)$-category $\C$. The assignemt $\C\mapsto \C^{\coprod}$ has a universal property \cite[Thm 2.4.3.18]{lurie-ha}: for any symmetric monoidal $(\infty,1)$-category $\D^{\otimes}$ any map $\C\to CAlg(\D)$ can be lifted in a essentially unique way to a lax monoidal functor $\C^{\coprod}\to \D^{\otimes}$.

Finally, if $\C$ is an $(\infty,1)$-category with direct sums and a zero object, the cartesian and cocartesian symmetric monoidal structures are canonically equivalent by means of the description in \cite[2.4.3.19]{lurie-ha}.\\

In the next section (Remarks \ref{algebrascartesian} and \ref{algebrascocartesian}) we will review how the theory of algebras in a cartesian/ cocartesian structure admits a much more simpler description than in the general case.\\

\subsubsection{Monoid objects}
If $\C^{\times}\to N(\Fin)$ is a cartesian symmetric monoidal $(\infty,1)$-category, for each $\infty$-operad $p:\Opmonoidal\to N(\Fin)$, an \emph{$\Op$-monoid object in $\C$} is a functor $F:\Opmonoidal\to \C$ satisfying the usual Segal condition: for each object $C=(x_1,..., x_n)$ in $\Opmonoidal$ with $x_1,..., x_n$ in $\Op$ and given $p$-coCartesian liftings $\tilde{\rho}^i:(x_1,..., x_n)\to x_i$ for the inert morphisms $\rho^i$ in $N(\Fin)$, the induced product map $F(C)\to \prod_i F(X_i)$ is an equivalence in $\C$. The collection of $\Op$-monoid objects in $\C$ can be organized in a new $(\infty,1)$-category $Mon_{\Op}(\C)$.

\subsection{Algebra Objects}
\label{section3-2}

\subsubsection{Algebras over an $(\infty,1)$-operad}
Let

\begin{equation}
\xymatrix{
&\Cmonoidal\ar[d]^p\\
\Op'^{\otimes}\ar[r]^f&\Opmonoidal
}
\end{equation}

\noindent be a diagram of $\infty$-operads with $p$ a fibration of $\infty$-operads. We denote by $Fun_{\Opmonoidal}(\Op'^{\otimes},\Cmonoidal)$ the strict pullback

\begin{equation}
\xymatrix{
&Fun(\Opmonoidal, \Cmonoidal)\ar[d]\\
\Delta[0]\ar[r]^(0.4){Id_{\Opmonoidal}}&Fun(\Opmonoidal, \Opmonoidal)
}
\end{equation}

\noindent whose vertices correspond to the dotted maps rendering the diagram commutative

\begin{equation}
\xymatrix{
&\Cmonoidal\ar[d]^p\\
\Op'^{\otimes}\ar[r]^f\ar@{-->}[ru]&\Opmonoidal
}
\end{equation}

By construction, $Fun_{\Opmonoidal}(\Opmonoidal,\Cmonoidal)$ is an $(\infty,1)$-category and following the Definition 2.1.3.1 of \cite{lurie-ha}, we denote by $Alg_{\Op'/\Op}(\C)$ its full subcategory spanned by the maps of $\infty$-operads defined over $\Opmonoidal$. We refer to it as the \emph{$\infty$-category of $\Op'$-algebras of $\C$}. In the special case when $f=Id$, we will simply write $Alg_{/\Op}(\C)$ to denote this construction. In the particular case  $\Opmonoidal=N(\Fin)$, this construction recovers the $\infty$-category of maps of $\infty$-operads $Alg_{\Op'}(\C)$ defined in the previous section. If both $\Op'^{\otimes}=\Opmonoidal=N(\Fin)$, the $\infty$-category $Alg_{\Comm}(\C)$ can be identified with the $\infty$-category of sections $s: N(Fin)\to \Cmonoidal$ of the structure map $p:\Cmonoidal\to N(\Fin)$, which send inert morphisms to inert morphisms. This condition forces every $s(\nfin)$ to be of the form $(X,X,...,X)$ for some object $X$ in $\C$. Moreover, the image of the active morphisms in $N(\Fin)$ will produce maps $X\otimes X\to X$ and $1\to X$ endowing $X$ in $\C$ with the structure of a commutative algebra. The cocartesian property is the machine that produces coherence diagrams. As an example, to extract the first associative restrain we consider the image through $s$ of the diagram

\begin{equation}
\xymatrix{
\langle 3 \rangle \ar[r] \ar[d]&\twofin\ar[d]\\
\twofin\ar[r]& \onefin
}
\end{equation}

\noindent of active maps in $N(\Fin)$. Since $s(\onefin)$ lives in the fiber over $\onefin$, the cocartesian property will ensure the existence of a uniquely determined (up to homotopy) new commutative square in $\C$

\begin{equation}
\xymatrix{
X\otimes X\otimes X\ar@{-->}[r]\ar@{-->}[d] \ar@{-->}[dr]|*{}="N"& \ar@{=>};"N" X\otimes X\ar@{-->}[d]\\
X\otimes X\ar@{-->}[r]\ar@{=>};"N" & X
}
\end{equation}

The commutativity restrain follows from the commutativity of the diagram
\begin{equation}
\xymatrix{
\langle 2 \rangle \ar[r]\ar[d]& \onefin\\
\langle 2 \rangle \ar[ru]
}
\end{equation}
\noindent in $N(\Fin)$, where the vertical map is permutation.\\

These are called \emph{commutative algebra objects of } $\C$ and we write $CAlg(\C):=Alg_{/\Comm}(\C)$. In particular, it follows from the description $\Commmonoidal \simeq colim_k E_k^{\otimes}$ that $CAlg(\C)$ is equivalent to $lim_k Alg_{E_k}(\C)$.\\

If $\Op'^{\otimes}=\Assmonoidal\to \Opmonoidal=N(\Fin)$ is the associative operad, the associated algebra-objects in $\Cmonoidal$ can be identified with the data of an object $X$ in $\C$ together with a unit and a multiplication satisfying the usual associative coherences which are extracted as explained in the previous discussion. The main difference is that the permutation of factors is no longer a map in $\Assmonoidal$ so that the commutativity restrains disappears. We set the notation $Alg(\C):= Alg_{\Ass}(\C)$. It follows that the composition with $\Assmonoidal\to N(\Fin)$ produces a forgetful map $CAlg(C)\to Alg(\C)$.\\

\begin{example}
\label{classicalalgebras}
Let $\C$ be a classical symmetric monoidal category. As explained in the Example \ref{classicalsymmetricmonoidal}, the nerve $N(\C)$ adquires the structure of a symmetric monoidal $(\infty,1)$-category. It follows that $CALg(N(\C))$ and $Alg(N(\C))$ can be identified, respectively, with the nerves of the classical categories of strictly commutative (resp. associative) algebra objects in $\C$ in the classical sense. 
\end{example}

Another important situation is the case when $\Op'^{\otimes}=\Trivmonoidal$ the trivial operad for which, as expected, we have a canonical equivalence $Alg_{/\Triv}(\C)\simeq \C$ (see \cite[2.1.3.5]{lurie-ha}).\\

The theory of algebras becomes much simpler in the case of cartesian and cocartesian monoidal structures. The following two remarks collect some of these aspects:

\begin{remark}
\label{algebrascartesian}
If $\Cmonoidal$ is cartesian symmetric monoidal $(\infty,1)$-category, we have a canonical map relating the theory of algebras with the theory of monoids described in the previous section

\begin{equation}Alg_{/\Op}(\C)\to Mon_{\Op}(\C)\end{equation}

By \cite[2.4.2.5]{lurie-htt} this map is an equivalence. We will use this in the next section.
\end{remark}

\begin{remark}
\label{algebrascocartesian}
As in the classical situation if $\C$ is a category with finite sums $\coprod$ then every object  $X$ in $\C$ carries admits a unique structure of commutative algebra, where the codiagonal map $X\coprod X\to X$ is the multiplication. In the $\infty$-setting this situation has its analoge for any Cocartesian $\infty$-operad as a consequence of the fact that Cocartesian $\infty$-operads are determined by their underlying $(\infty,1)$-categories in a very strong sense. More precisely (see \cite[2.4.3.16]{lurie-ha}), for any unital generalized $\infty$-operad $\Opmonoidal$ and any Cocartesian $\infty$-operad $\C^{\coprod}$ the restriction map $Alg_{\Op}(\C)\to Fun(\Op, \C)$ is an equivalence of $(\infty,1)$-categories. In particular, when $\Opmonoidal$ is the commutative or the associative operad\footnote{These are both unital operads}, the evaluation functors $Alg(\C)\to \C$ and $CAlg(\C)\to \C$ are equivalences so that the forgetful map

\begin{equation}
\xymatrix{
CAlg(\C)\ar[rr]\ar[dr]^{\sim}&&Alg(\C)\ar[dl]_{\sim}\\
&\C&
}
\end{equation}

\noindent is also an equivalence. In particular, by choosing an inverse to $CAlg(\C)\to \C$ we find a precise way to reproduce the classical situation.
\end{remark}

\subsubsection{Symmetric Monoidal $(\infty,1)$-categories as commutative algebras in $\iCat$ and Monoidal Localizations}
\label{remarkcat}
Let us consider the $(\infty,1)$-category of small $(\infty,1)$-categories $\iCat$ (see Chapter 3 of \cite{lurie-htt}). The cartesian product endows $\iCat$ with a symmetric monoidal structure $\iCat^{\otimes}$ which can be obtained as the operadic nerve of the combinatorial simplicial model category of marked simplicial sets with the cartesian model structure. The objects of $\iCat^{\otimes}$ are the finite sequences of $(\infty,1)$-categories $(\C_1,..., \C_n)$ and the morphisms $(\C_1,..., \C_n)\to (\D_1,..., \D_m)$ over a map $f:\nfin\to \mfin$ are given by families of maps 

\begin{equation}\prod_{j\in f^{-1}(\{i\})} C_j\to D_i\end{equation}

\noindent with $1\leq i\leq m$. Using the Grothendieck-construction of Theorem 3.2.0.1 of \cite{lurie-htt}, the objects of $CAlg(\iCat)\simeq Mon_{\Comm}(\iCat)$ can be identified with the small symmetric monoidal $(\infty,1)$-categories and the maps of algebras are identified with the monoidal functors (see Remark 2.4.2.6 of \cite{lurie-htt}). The same idea works if we replace $\iCat$ by $\iCatbig$. These examples will play a vital role throughout this paper.\\

\begin{remark}
\label{monoidalenvelope}
Thanks to \cite[2.2.4.9]{lurie-ha} the forgetful functor $CAlg(\iCat)\to Op_{\infty}$ admits left adjoint $Env^{\otimes}$. Given an $\infty$-operad $\Opmonoidal$, the symmetric monoidal $(\infty,1)$-category $Env^{\otimes}(\Opmonoidal)$ is called the \emph{monoidal envelope of $\Opmonoidal$}.
\end{remark}

The theory in \ref{mfl} can be extended to localizations which are not necessarily reflexive. Recall from our preliminairs that the formula $(\C,W)\mapsto \C[W^{-1}]$ provides a left adjoint to the fully-faithful map
$\iCat\subseteq \mathcal{W}\iCat$. This makes $\iCat$ a reflexive localization of $\mathcal{W}\iCat$. The last carries a natural monoidal structure given by the cartesian product of pairs which extends the cartesian product in $\iCat$. By the Proposition 4.1.3.2 in \cite{lurie-ha} the formula $(\C,W)\mapsto \C[W^{-1}]$ commutes with products and therefore fits in the conditions of the previous item, providing a monoidal functor 

\begin{equation}
\mathcal{W}\iCat^{\times}\to \iCat^{\times}
\end{equation}

\noindent thus providing a left adjoint to the inclusion

\begin{equation}CAlg(\iCat)\subseteq CAlg(\mathcal{W}\iCat)\end{equation}

We can identify the objects in $CAlg(\mathcal{W}\iCat)$ with the pairs $(\Cmonoidal,W)$ where $\Cmonoidal$ is a symmetric monoidal $(\infty,1)$-category and $W$ is collection of edges in the underlying $\infty$-category $\C$, together with the condition that the operations $\otimes^n:\C^n \to \C$ send sequences of edges in $W$ to a new edge in $W$. The previous adjunction is telling us that every time we have a symmetric monoidal $(\infty,1)$-category $\Cmonoidal$ together with a collection of edges $W$ which is compatible with the operations, then there is natural symmetric monoidal structure $\C^{\otimes}[W^{-1}]^{\otimes}$ in the localization $\C[W^{-1}]$. Plus, the localization functor is monoidal and has the obvious universal property.

Notice that the condition $(*)$ in \ref{mfl} is exactly asking for the edges $W= L-equivalences$ to satisfy the compatibility in the present discussion.

\subsubsection{Change of Operad and Free Algebras}
Let us consider now a diagram of $\infty$-operads.

\begin{equation}
\xymatrix{
&&\Cmonoidal\ar[d]^p\\
\Op_2^{\otimes}\ar[r]^{\alpha}&\Op_1^{\otimes}\ar[r]^f&\Opmonoidal
}
\end{equation}

\noindent with $p$ a fibration. Composition with $\alpha$ produces a  forgetful functor

\begin{equation}
Alg_{\Op_1/\Op}(\C)\to Alg_{\Op_2/\Op}(\C)
\end{equation}

The main result of \cite{lurie-ha} Section 3.1.3 is that, under some mild hypothesis on $\Cmonoidal$\footnote{These mild conditions hold for any symmetric monoidal $(\infty,1)$-category \emph{compatible with all small colimits} (see below) and this will be enough for our present purposes.}, we can use the theory of \emph{operadic Kan extensions} (see \cite{lurie-ha} - Section 3.1.2) to ensure the existence of a left adjoint $F$ to this functor (Corollary 3.1.3.5 of \cite{lurie-ha}). For each algebra $X\in Alg_{\Op_2/\Op}(\C)$,  $F(X)$ can be understood as the \emph{ free $\Op_1$-Algebra generated by the  $\Op_2$ algebra $X$} \cite[Def. 3.1.3.1]{lurie-ha}: for each color $b\in \Op_2$, the value of $F(X)$ at $b$ is given by the operadic $p$-colimit of the diagram consisting of all active morphisms over $b\in \Op_2^{\otimes}$, whose source is in the image of $\alpha$.\\ 

 The Construction 3.1.3.7 and the Prop. 3.1.3.11 of \cite{lurie-ha} provide a very precise description of this construction in the case where $\Op_2$ is the trivial operad and $\Op_{1}$ is the associative or the commutative operad. Given a trivial algebra $X$ in $\C$, for the first \cite[Prop. 4.1.1.14]{lurie-ha} we obtain

\begin{equation}
F(X)(\onefin)=\coprod_{n\geq 0}X^{\otimes n}
\end{equation}

\noindent while the second (see the Example 3.1.3.12 of \cite{lurie-ha}) is obtained from the previous formula by killing the action of the permutation groups

\begin{equation}
F(X)(\onefin)=\coprod_{n\geq 0}(X^{\otimes n}/\Sigma_n)
\end{equation}

\subsubsection{Colimits of algebras}
Another important feature of the $\infty$-categories $Alg_{/\Op}(\C)$ is the existence of limits and colimits. The first exist whenever they exist in $\C$ and can be computed using the forgetful functor (Prop. 3.2.2.1 -\cite{lurie-ha}). The existence of colimits needs a more careful discussion. In order to make the colimit of algebras an algebra we need to ask for a certain compatibility of the monoidal structure with colimits in $\C$. This observation motivates the notion of an \emph{$\Op$-monoidal $(\infty,1)$-category compatible with $\mathcal{K}$-indexed colimits}, with $\mathcal{K}$ a given collection of simplicial sets (see Definition 3.1.1.18 and Variant 3.1.1.19 of \cite{lurie-ha}). The definition demands the existence of $\mathcal{K}$-colimits on each $\C_x$ (for each $x\in \Op$), and also that the multiplication maps associated to the monoidal structure preserve all colimits indexed by the simplicial sets $K\in \mathcal{K}$, separately in each variable. The main result (Corollary 3.2.3.3 of \cite{lurie-ha}) is that if $\Cmonoidal$ is an $\Op$-monoidal $(\infty,1)$-category compatible with $\mathcal{K}=\{k-\text{small simplicial sets}\}$-colimits and if $\Opmonoidal$ is an essentially $\kappa$-small $\infty$-operad, then $Alg_{\Op}(\C)$ admits $\kappa$-small colimits. However and in general, contrary to limits, colimits cannot be computed using the forgetful functor to $\C_x$ for each color $x\in \Op$. \\

\subsubsection{Tensor product of Algebras}
\label{tensorproductalgebras}

Let $q:\Cmonoidal\to N(\Fin)$ be an $\infty$-operad. Thanks to \cite[3.2.4.1, 3.2.4.3]{lurie-ha}, for any $\infty$-operad $\Opmonoidal$ the category of algebras $Alg_{\Op}(\C)$ can be endowed again with the structure of $\infty$-operad $p:Alg_{\Op}(\C)^{\otimes}\to N(\Fin)$. Moreover, a morphism $\alpha$ in $Alg_{\Op}(\C)^{\otimes}$ is $p$-cocartesian if and only if for each color $x\in \Op$ its image through the evaluation functor $e_x: Alg_{\Op}(\C)^{\otimes}\to \Cmonoidal$ is $q$-cocartersian. In particular, $Alg_{\Op}(\C)^{\otimes}$ is a symmetric monoidal $(\infty,1)$-category if and only if $\Cmonoidal$ is, and in this case the evaluation functors $e_x$ are symmetric monoidal. In other words, the category of algebras inherits a tensor product given by the tensor operation in the underlying category $\C$. In particular, for any morphism of $\infty$-operads $\mathcal{O}'^{\otimes}\to \Opmonoidal$, since the forgetful functor $f:Alg_{\Op}(\C)\to Alg_{\mathcal{O}'}(\C)$ is defined over the evaluation functors $e_x$, it extends to a monoidal map $Alg_{\Op}(\C)^{\otimes}\to Alg_{\mathcal{O}'}(\C)^{\otimes}$.\\

\begin{remark}
\label{last1}
By the universal property of the simplicial set $Alg_{\Op}(\C)^{\otimes}$ (see \cite[Const. 3.2.4.1]{lurie-ha}), any  monoidal functor $\Cmonoidal\to \Dmonoidal$  between symmetric monoidal $(\infty,1)$-categories, extends to a monoidal functor between the symmetric monoidal categories of algebras $Alg_{\Op}(\C)^{\otimes}\to Alg_{\Op}(\D)^{\otimes}$. As a corollary of the previous discussion, for every color $x\in \Op$, the evaluation maps $e_x$ provide a commutative diagram of monoidal functors 

\begin{equation}
\xymatrix{
Alg_{\Op}(\C)^{\otimes}\ar[r]\ar[d]^{e_x}& Alg_{\Op}(\D)^{\otimes}\ar[d]^{e_x}\\
\Cmonoidal \ar[r] & \Dmonoidal
}
\end{equation}
\end{remark}

As in the classical case, if $\Op=\Comm$, this monoidal structure is coCartesian (Prop. 3.2.4.7 of \cite{lurie-ha}). More generally, for $\Op=\mathbb{E}_k$ there is a formula relating this monoidal structure with coproducts in $\C$ \cite[Theorem 5.1.5.1]{lurie-ha}.

\subsubsection{Tensor Product of $\infty$-operads}

The $(\infty,1)$-category of $\infty$-operads admits a symmetric monoidal structure where the tensor product of two operads $\Opmonoidal$ and $\Opprimemonoidal$ is characterized (\cite[2.2.5.3]{lurie-ha}) by the data of a map of simplicial sets $f:\Opmonoidal\times \Opprimemonoidal \to\Opmonoidal\otimes \Opprimemonoidal$ with the following universal property: for any $\infty$-operad $\Cmonoidal$, composition with $f$ induces an equivalence 

\begin{equation}
Alg_{(\Opmonoidal\otimes \Opprimemonoidal)_{\onefin}}(\C)\simeq Alg_{\Op}(Alg_{\mathcal{O}'}(C))
\end{equation}

\noindent where $Alg_{\mathcal{O}'}(C)$ is viewed with the operadic structure of the previous section. In particular, the unit is the trivial operad.

This monoidal structure can be defined at the level of marked simplicial sets over $N(\Fin)$ and can be seen to be compatible with the model structure therein \cite[2.2.5.7, 2.2.5.13]{lurie-ha}. Moreover, it is compatible with the natural inclusion

\begin{equation}
Cat_{\infty}\hookrightarrow Op_{\infty}
\end{equation}

\noindent so that the cartesian product of $(\infty,1)$-categories is sent to this new product of operads \cite[Prop. 2.2.5.15]{lurie-ha}.\\

An important application is the description of the $\infty$-operad $\mathbb{E}_{i+j}^{\otimes}$ as the tensor product of $\mathbb{E}_{i}^{\otimes}$ with $\mathbb{E}_{j}^{\otimes}$ \cite[5.1.2.2]{lurie-ha}. In particular, this characterizes an $\mathbb{E}_{n}^{\otimes}$-algebra $X$ in a symmetric monoidal $(\infty,1)$-category $\Cmonoidal$ has a collection of $n$ different associative multiplications on $X$, monoidal with respect to each other.

\subsubsection{Transport of Algebras via Monoidal functors}
\label{transportofalgebras}

Let 

\begin{equation}
\xymatrix{
\Cmonoidal\ar[rr]^{f}\ar[rd]^p&&\ar[ld]^q \Dmonoidal\\ 
&\Opmonoidal&}
\end{equation}

\noindent be a morphism of $(\infty,1)$-operads (not necessarily monoidal) with both $p$ and $q$ given by fibrations of $(\infty,1)$-operads. In this case, for any morphism of $(\infty,1)$-operads $\Opprimemonoidal\to \Opmonoidal$, $f$ induces a composition map
\begin{equation}
f_*: Alg_{\Op'/\Op}(\C)\to Alg_{\Op'/\Op}(\D)
\end{equation}

This is because the composition of maps of $\infty$-operads is again a map of $\infty$-operads.\\

Assume now $f$ is a monoidal functor between symmetric monoidal $(\infty,1)$-categories. Following the discussion in \ref{tensorproductalgebras}, for any $\infty$-operad $\Opmonoidal$ we can use the universal property of \cite[3.2.4.1]{lurie-ha} to obtain an induced map of $\infty$-operads between symmetric monoidal $(\infty,1)$-categories $Alg_{\Op}(\C)^{\otimes}\to Alg_{\Op}(\D)^{\otimes}$. Again by the arguments exposed in \ref{tensorproductalgebras} we can easily deduce that this last map is monoidal.

\subsubsection{Symmetric Monoidal Structures and Compatibility with Colimits}
\label{compatiblewithcolimitsmonoidal}
As mentioned in \ref{remarkcat} the objects of $CAlg(\iCat)$ can be identified with the (small) symmetric monoidal $(\infty,1)$-categories. We have an analogue for the (small) symmetric monoidal $(\infty,1)$-categories compatible with $\mathcal{K}$-indexed colimits: as indicated in \ref{notations}, given an arbitrary $(\infty,1)$-category $\C$ together with a collection $\mathcal{K}$ of arbitrary simplicial sets and a collection $\mathcal{R}$ of diagrams indexed by simplicial sets in $\mathcal{K}$, we can fabricate a new $(\infty,1)$-category $\mathcal{P}^{\mathcal{K}}_{\mathcal{R}}(\C)$ with the universal property described by the formula (\ref{formulaaddingcolimits}). We can now use this mechanism to fabricate a monoidal structure in $\iCat(\mathcal{K})$ induced by the cartesian structure of $\iCat$. Given two small $(\infty,1)$-categories $\C$ and $\C'$ having all the $\mathcal{K}$-indexed colimits, we consider the collection $\mathcal{R}=\mathcal{K}\boxtimes\mathcal{K}$ of all diagrams $p:K\to \C\times \C'$ such that $K\in\mathcal{K}$ and $p$ is constant in one of the product components, and define a new $(\infty,1)$-category $\C\otimes\C':= \mathcal{P}^{\mathcal{K}}_{\mathcal{K}\boxtimes\mathcal{K}}(\C\times \C')$. By construction it admits all the $\mathcal{K}$-indexed colimits and comes equipped with a map $\C\times \C'\to \C\otimes\C'$ endowed with the following universal property: for any $(\infty,1)$-category $\D$ having all the $\mathcal{K}$-indexed colimits, the composition map

\begin{equation}
\label{equationlala}
Fun_{\mathcal{K}}(\C\otimes \C', \D)\to Fun_{\mathcal{K}\boxtimes\mathcal{K}}(\C\times \C', \D)
\end{equation}

\noindent is an equivalence. The right-side denotes the category of all $\mathcal{K}$-colimit preserving functors and the left-side denotes the category spanned by the functors preserving $\mathcal{K}-$colimits separately in each variable.

We can now use this operation to define a symmetric monoidal structure in $\iCat(\mathcal{K})$. 
For that, we start with $\iCat^{\otimes}\to N(\Fin)$ the cartesian monoidal structure in $\iCat$, and we consider the (non-full) subcategory $\iCat(\mathcal{K})^{\otimes}$ whose objects are the finite sequences $(\C_1,...,\C_n)$ where each $\C_i$ admits all $\mathcal{K}$-indexed colimits together with those morphisms $(\C_1,..,\C_n)\to (\D_1,...,\D_m)$ in $\iCat^{\otimes}$ over some $f:\nfin\to \mfin$, which correspond to a family of maps

\begin{equation}
\prod_{j\in f^{-1}(\{i\})}\C_j\to \D_i
\end{equation}

\noindent given by functors commuting with $\mathcal{K}$-indexed colimits separately in each variable. We can now use the universal property described in the previous paragraph to prove that $\iCat(\mathcal{K})^{\otimes}$ is a cocartesian fibration: given a morphism $f:\nfin\to \mfin$ and a sequence of $(\infty,1)$-categories $X=(\C_1,..., \C_n)$ having all the $\mathcal{K}$-indexed colimits, a cocartesian lifting for $f$ at $X$ in $\iCat(\mathcal{K})^{\otimes}$ is given by the family of universal maps

\begin{equation}
\prod_{j\in f^{-1}(\{i\})}\C_j\to \D_i:= \mathcal{P}^{\mathcal{K}}_{\boxtimes_{j\in f^{-1}(\{i\}}\mathcal{K}}(\prod_{j\in f^{-1}(\{i\}}\C_j)
\end{equation}

\noindent which we know commutes with $\mathcal{K}$-indexed colimits separately in each variable. Moreover, it follows from this formula that the canonical inclusion $\iCat(\mathcal{K})^{\otimes}\to \iCat^{\otimes}$ is a lax-monoidal functor (see the Proposition 6.3.1.3 and the Corollary 6.3.1.4 of \cite{lurie-ha} for the full details).\\

Finally, the objects of $CAlg(\iCat(\mathcal{K}))$ can be naturally identified with the symmetric monoidal $(\infty,1)$-categories compatible with $\mathcal{K}$-colimits.\\

More generally, given two arbitrary collections of simplicial sets $\mathcal{K}\subseteq \mathcal{K}'$, it results from the universal properties defining the monoidal structures that the inclusion

\begin{equation}
\iCatbig(\mathcal{K}') \subseteq \iCatbig(\mathcal{K})
\end{equation}

\noindent is lax monoidal and its (informal) left adjoint $\C\mapsto \mathcal{P}^{\mathcal{K}^{'}}_{\mathcal{K}}(\C)$ (see \ref{completionwithcolimits}) is monoidal. In other words, for every inclusion $\mathcal{K}\subseteq \mathcal{K}^{'}$ of collections of simplicial sets, if $\Cmonoidal$ is a symmetric monoidal $(\infty,1)$-category compatible with all $\mathcal{K}$-colimits, the $(\infty,1)$-category $\mathcal{P}^{\mathcal{K}^{'}}_{\mathcal{K}}(\C)$ inherits a canonical symmetric monoidal structure $\mathcal{P}^{\mathcal{K}^{'}}_{\mathcal{K}}(\C)^{\otimes}$ compatible with all the $\mathcal{K}'$-indexed colimits. Moreover, the canonical functor $\C\to \mathcal{P}^{\mathcal{K}^{'}}_{\mathcal{K}}(\C)$ extends to a monoidal functor $\Cmonoidal\to \mathcal{P}^{\mathcal{K}^{'}}_{\mathcal{K}}(\C)^{\otimes}$ and, again by ignoring the set-theoretical aspects, the previous adjunction extends to a new one (see the Proposition $6.3.1.10$ of \cite{lurie-ha} for the correct statement) 

\begin{equation}
\xymatrix{
CAlg(\iCat^{big}(\mathcal{K}')) \ar[r]_i & CAlg(\iCat^{big}(\mathcal{K})) \ \ar@/_/[l]
}
\end{equation}

\begin{example}
\label{monoidalstructurepresheaves}
In the particular case when $\mathcal{K}$ is empty and $\mathcal{K}'$ is the collection of all small simplicial sets, this tells us that if $\Cmonoidal$ is a small symmetric monoidal $(\infty,1)$-category, the $\infty$-category of $\infty$-presheaves on $\C$ inherits a natural symmetric monoidal structure $\mathcal{P}^{\otimes}(\C)$, commonly called the \emph{convolution product}. Moreover, the Yoneda's map is monoidal and satisfies the following universal property: for any symmetric monoidal  $(\infty,1)$-category $\Dmonoidal$ with all small colimits, the natural map given by composition

\begin{equation}Fun^{\otimes,L}(\mathcal{P}(\C),\D)\to Fun^{\otimes}(\C,\D)\end{equation}

\noindent is an equivalence.
\end{example}

\begin{example}
Following the discussion in \ref{idempotentcompletecategories}, $\iCat(\{Idem\})$ can be identified with $\iCat^{idem}$. In this case, the previous discussion endows $\iCat^{idem}$ with a symmetric monoidal structure and the idempotent-completion $Idem(-)$ is a monoidal left adjoint to the inclusion $\iCat^{idem}\subseteq \iCat$.
\end{example}
 
\begin{remark}
Let $\Cmonoidal$ be a symmetric monoidal $(\infty,1)$-category. We say that $\C$ is \emph{closed} if for each object $X \in \C$ the map $(-\otimes X):\C\to \C$ has a  right adjoint. In other words, $\Cmonoidal$ is closed if and only if for any objects $X$ and $Y$ in $\C$ there is an object $X^Y$ and a map $Y^X\otimes X\to Y$ inducing an homotopy equivalence $Map_{\C}(Z\otimes X, Y)\simeq Map_{\C} (Z,Y^X)$.
If $\Cmonoidal$ is closed symmetric monoidal $(\infty,1)$-category and its underlying $\infty$-category $\C$ has all small colimits, then $\Cmonoidal$ is a symmetric monoidal $(\infty,1)$-category compatible with all small colimits. 
An important example is the cartesian symmetric monoidal structure on $\iCat$, where the right adjoint to $(-\times X)$ is provided by the construction $Fun(X,-)$. The colimits exist in $\iCat$ because it is presentable.
\end{remark}

\subsection{Modules over an Algebra}
\label{section3-3}

We now recall the theory of module-objects over an algebra-object. We mimic the classical situation: in a symmetric monoidal category $\C$, each algebra-object $A$ has an associated theory of modules $Mod_A(\C)$ and under some nice assumptions on $\C$, this new category has a natural monoidal product. This provides an example of a more general object - a collection of $\infty$-operads indexed by the objects of a small $(\infty,1)$-category.

\subsubsection{Generalized $(\infty,1)$-operads and Operadic families}
\label{generalizedoperads}

We start with a review of the appropriate language to formulate the notion of a family of $\infty$-operads indexed by the collection of objects of an $(\infty,1)$-category $\B$. The theory of modules provides an example, with $\B=CAlg(\C)$ for a symmetric monoidal $(\infty,1)$-category $\C$.\\

In \cite{lurie-ha}, the author develops two equivalent ways to formulate this idea of family. The first is the notion of a \emph{$\B$-operadic family} (See Definition 2.3.2.10 of \cite{lurie-ha}). It consists of categorical fibration $p:\Opmonoidal\to \B\times N(\Fin)$ such that 

\begin{enumerate}
\item for each object $b\in \B$, the fiber $\Opmonoidal\times_{\B\times N(\Fin)}\{b\}\to N(\Fin)$ is an $\infty$-operad. In particular, we can identify an object $X$ in the fiber of $(b,\nfin)$ as a sequence $(b; (X_1,..., X_n))$ by choosing $p$-cocartesian liftings $X\to X_i$ of the canonical morphisms $\rho_i:\nfin\to \onefin$;

\item For any $Z=(b'; Z_1,...,Z_m)\in \Opmonoidal$, $X=(b, X_1,..., X_n)$ and every pair of morphisms $(u,f):(b',\mfin)\to (b,\nfin)$ in $\C\times N(\Fin)$, we ask for the canonical map

\begin{equation}
Map_{\Opmonoidal}^{u,f}(Z,X)\to \prod_{i=1}^{n} Map_{\Opmonoidal}^{u, \rho_i\circ }(Z, X_i)
\end{equation}

to be an equivalence. In this notation, $Map_{\Opmonoidal}^{u,f}$ denotes the connected component of $Map_{\Opmonoidal}$ of all morphisms lying over $(u,f)$.

\end{enumerate}

The second notion is that of a \emph{generalized $\infty$-operad} (\cite{lurie-ha}-Defn 2.3.2.1). It is given by the data of an $(\infty,1)$-category $\Opmonoidal$ equipped with a map $q:\Opmonoidal\to N(\Fin)$ such that: 

\begin{enumerate}

\item For any object $X$ over $\nfin$ and any inert morphism $f:\nfin\to \mfin$, there is a $q$-cocartesian lifting $\tilde{f}:X\to X'$ of $f$. In particular, these induce functors $f_{!}:\Opmonoidal_{\nfin}\to \Opmonoidal_{\mfin}$ and if $\mfin=\langle 0 \rangle$ we find a canonical map $\Opmonoidal\to \Opmonoidal_{\langle 0 \rangle}$. Let $X$ be a object over $\nfin$. Choose $X\to X_i$ a $p$-cocartesian lifting for each $\rho_i$. Moreover, choose a $p$-cocartesian lifting $X\to X_0$ for the canonical map $\nfin\to \langle 0 \rangle$ and $p$-cocartesian liftings $X_i\to X_{i,0}$ for the null map $\onefin\to \langle 0 \rangle$. Because of the cocartesian property, the diagram

\begin{equation}
\xymatrix{
 &X_1\ar[ddr]&\\
 &X_2\ar[dr]&\\
X\ar[uur] \ar[ur]\ar[dr]&\vdots& X_0\simeq X_{i,0}\\
 &X_n\ar[ur]&\\
}
\end{equation}

commutes in $\Opmonoidal$.

\item For each $\nfin$, the natural map 

\begin{equation}
\Opmonoidal_{\nfin}\to \underbrace{\Opmonoidal_{\onefin}\times_{\Opmonoidal_{\langle 0 \rangle}}.... \times_{\Opmonoidal_{\langle 0 \rangle}}\Opmonoidal_{\onefin}}_{n}
\end{equation}

induced by the morphisms $\rho_i:\nfin\to \onefin$, is an equivalence. This condition is weaker than the condition in the definition of an $\infty$-operad for it does not force $\Opmonoidal_{\langle 0 \rangle}$ to be contractible. This second axiom allows us to identify an object $X$ over $\nfin$ with a sequence of objects $(X_1,..., X_n)$ in $\Opmonoidal_{\onefin}$ living over the same (up to equivalence) object $X_0\in \Opmonoidal_{\langle 0 \rangle}$ and motivates the notation $X=(X_0; X_1,..., X_n)$.

\item Let $X=(X_0; X_1,..., X_n)$ and $Z=(Z_0; Z_1,..., Z_m)$ be objects in $\Opmonoidal$. For any $f:\mfin\to \nfin$, we ask for the canonical map 

\begin{equation}
Map^f_{\Opmonoidal}(Z,X)\to (\prod_{I=1}^n Map_{\Opmonoidal}^{\rho_i\circ f} Map_{\Opmonoidal}(Z,X_i))\times_{Map_{\Opmonoidal}^{0:\nfin\to \langle 0 \rangle}(Z, X_0) } Map_{\Opmonoidal_{\langle 0 \rangle}}(Z_0, X_0)
\end{equation}

to be an equivalence.
 
\end{enumerate}

According to the Proposition 2.3.2.11 of \cite{lurie-ha}, the two notions are equivalent: if $q:\Opmonoidal\to \B\times N(\Fin)$ is a $\B$-operadic family, the composition with the projection towards $N(\Fin)$ is a generalized $\infty$-operad and the canonical projection $\Opmonoidal_{\langle 0 \rangle}\to \B$ is a trivial Kan fibration. Conversely, if $p:\Opmonoidal\to N(\Fin)$ is a generalized $\infty$-operad, the product of $p$ with the canonical projection $\Opmonoidal\to \Opmonoidal_{\langle 0 \rangle}$ is a $\Opmonoidal_{\langle 0 \rangle}$-operadic family. These two constructions are mutually inverse. Notice also that if $\B=\Delta[0]$, we recover the notion of $\infty$-operad. \\

Let $p:\Opmonoidal\to N(\Fin)$ be a generalized $\infty$-operad. We say that a morphism in $\Opmonoidal$ is \emph{inert} if it is $p$-cocartesian and its image in $N(\Fin)$ is inert. If $\Opmonoidal$ and $\Opprimemonoidal$ and generalized $\infty$-operads, we say that a map of simplicial sets $f:\Opmonoidal\to \Opprimemonoidal$ is a map of generalized $\infty$-operads if it is defined over $N(\Fin)$ and sends inert morphisms to inert morphisms. According to the Remark 2.3.2.4 of \cite{lurie-ha}, there is a left proper, combinatorial simplicial model structure in the category of marked simplicial sets over $N(\Fin)$ having the generalized $\infty$-operads as cofibrant-fibrant objects. We denote by $\mathit{O}p^{gn}_{\infty}$ its underlying $(\infty,1)$-category. According to the Corollary 2.3.2.6 of \cite{lurie-ha}, the model structure for $\infty$-operads is a Bousfield localization of this model structure for generalized $\infty$-operads. At the level of the underlying $(\infty,1)$-categories this is the same as saying that $\mathit{O}p_{\infty}$ is a reflexive localization of $\mathit{O}p^{gn}_{\infty}$. The inclusion understands an $\infty$-operad as generalized $\infty$-operad whose indexing category is $\Delta[0]$.\\

In the language of $(\infty,1)$-categories, the relation between the two notions of operadic families can now be understood by using an adjunction: the assignment $\Opmonoidal\mapsto \Opmonoidal_{\langle 0 \rangle}$ sending a generalized $\infty$-operad to its fiber over $\langle 0 \rangle$ can be understood as a functor $F:\mathit{O}p^{gn}_{\infty}\to \iCat$ and according to the Proposition 2.3.2.9 of \cite{lurie-ha}, the map sending an $(\infty,1)$-category $\B$ to the generalized $\infty$-operad $\B\times N(\Fin)$ is a fully-faithful right adjoint of $F$. In this language, $\Opmonoidal\to \B\times N(\Fin)$ is an operadic family if and only it is a fibration of $\infty$-operads and its adjoint morphism $\Opmonoidal_{\langle 0 \rangle}\to \B$ is a trivial Kan fibration.

\subsubsection{The $(\infty,1)$-category of modules over an algebra-object}
\label{module-objects}

Let $\Cmonoidal\to \Opmonoidal$ be a fibration of $\infty$-operads. Following the results of \cite{lurie-ha} Section 3.3, by assuming a \emph{coherent} condition on the $\infty$-operad $\Opmonoidal$ (see Def. 3.3.1.9 in \cite{lurie-ha}), it is possible to construct a $Alg_{/\Op}(\C)$-operadic family $Mod^{\Op}(\C)^{\otimes}\to Alg_{/\Op}(\C)\times \Opmonoidal$ whose fiber over an algebra $A$

\begin{equation}
Mod^{\Op}_A(\C)^{\otimes}\to \Opmonoidal
\end{equation}

\noindent can be understood as a theory of $A$-modules. We will not reproduce here the details of this construction. Let us just say that if $\Opmonoidal$ is \emph{unital} (see Def. 2.3.1.1 of \cite{lurie-ha}) \footnote{In particular coherent operads are unital}, the category $Alg_{/\Op}(\C)\times \Opmonoidal$ can be described by means of a nice universal property in the homotopy theory of simplicial sets over $\Opmonoidal$ with the Joyal's model structure (See the Notation 3.3.3.5 and the Remarks 3.3.3.6 and 3.3.3.7 in \cite{lurie-ha}). By enlarging a bit this universal property we can define a new simplicial set $Mod^{\Op}(\C)^{\otimes}$ over $\Opmonoidal$ (see the Construction 3.3.3.1, and the Definition 3.3.3.8 in \cite{lurie-ha}), together with a canonical map

\begin{equation}
Mod^{\Op}(\C)^{\otimes}\to Alg_{/\Op}(\C)\times \Opmonoidal
\end{equation}

\noindent which by the Remark 3.3.3.16 of \cite{lurie-ha} is a fibration of generalized $\infty$-operads. Moreover, by the Theorem 3.3.3.9 of loc. cit, if $\Opmonoidal$ is coherent, then for each algebra $A$, the fiber $Mod_A^{\Op}(\C)^{\otimes}:=Mod^{\Op}(\C)^{\otimes}\times_{Alg_{/\Op}(\C)} \{A\}\to \Opmonoidal$ is a fibration of $\infty$-operads.

\subsubsection{Monoidal Structures in Categories of Modules}
\label{monoidalstructuremodules}
Under some extra conditions on $\Cmonoidal$, it is possible to prove that $Mod_A^{\Op}(\C)^{\otimes}$ is actually a  $\Op$-monoidal structure, with $A$ as a unit object. These sufficient conditions are already visisible in the classical situation: If $\C$ is the category of abelian groups with the usual tensor product, and $A$ is a (classical) commutative ring (a commutative algebra object in $\C$), then the tensor product of $A$-modules $M$ and $N$ is by definition, the colimit of

\begin{equation}
\label{productmodules}
\xymatrix{M\otimes A\otimes N \ar@<+.7ex>[r] \ar@<-.7ex>[r] & M\otimes N}
\end{equation}

\noindent where $\otimes$ denotes the tensor product of abelian groups and the two different arrows correspond, respectively, to the multiplication on $M$ and $N$. In order for this pushout to be a new $A$-module we need to assume that $\otimes$ commutes with certain colimits. By the combination of the Corollary 3.1.1.21 with the Theorem 3.4.4.3 of \cite{lurie-ha}, if $\Cmonoidal$ is a symmetric monoidal $(\infty,1)$-category compatible with $\kappa$-small colimits (for $\kappa$ an uncountable regular cardinal) and if $\Opmonoidal$ is a $\kappa$-small coherent $\infty$-operad, then the map $\Mod_A^{\Op}(\C)^{\otimes}\to \Opmonoidal$ is a coCartesian fibration of $\infty$-operads which is again compatible with $\kappa$-small colimits. In particular this result is valid for algebras over the $\infty$-operad $\mathbb{E}_k^{\otimes}$, for any $k\geq 0$, because it is known to be coherent (see \cite[Theorem 5.1.1.1]{lurie-ha} for the general case and \cite[Example 3.3.1.12]{lurie-ha} for the commutative operad).

\subsubsection{Limits and Colimits of Modules}
\label{limitscolimitsmodules}
Another important feature of the module-categories $\Mod_A^{\Op}(\C)^{\otimes}_{x}$ (for each $x\in \Op$) is the existence of limits, which can be computed directly on each $\C_x$ using the forgetful functor (Thm 3.4.3.1 of \cite{lurie-ha}). The existence of colimits requires again the compatibility of the monoidal structure with colimits on each $\C_x$. If $\Cmonoidal$ is compatible with $\kappa$-small colimits, then by the Corollary 3.4.4.6 of \cite{lurie-ha}, colimits in $\Mod_A^{\Op}(\C)^{\otimes}$ can be computed in the underlying categories $\C_x$ by means of the forgetful functors $\Mod_A^{\Op}(\C)^{\otimes}_x\to \C_x$, for each color $x\in \Op$.

\subsubsection{Algebra-objects in the category of modules}
We also recall another important result relating algebra objects in $\Mod_A^{\Op}(\C)^{\otimes}$ and algebras $B$ in $\C$ equipped with a map of algebras $A\to B$.

\begin{prop}\label{algmod-alg}(Corollary 3.4.1.7 of \cite{lurie-ha})
Let $\Opmonoidal$ be a coherent operad and let $\Cmonoidal\to \Opmonoidal$ be a fibration of $\infty$-operads. Let $A\in Alg_{/\Op}(\C)$ be an $\Op$-algebra object of $\C$. Then the obvious map

\begin{equation}
\label{quasiquasiquasi}
Alg_{/\Op}(\Mod_A^{\Op}(\C))\to Alg_{/\Op}(\C)_{A/}
\end{equation}

\noindent is a categorical equivalence (where $Alg_{/\Op}(\C)_{A/}$ denotes the $(\infty,1)$-category of objects $B$ in $Alg_{\Op}(\C)$ equipped with a map of algebras $A\to B$ - see \cite{joyal-article} or Prop. 1.2.9.2 of \cite{lurie-htt} ).

In particular, if $\Cmonoidal$ is coCartesian fibration compatible with all small colimits, $Alg_{/\Op}(\C)_{A/}$ inherits a monoidal structure from $\Mod_A^{\Op}(\C)^{\otimes}$ and by the discussion above this structure is cocartesian.
\end{prop}

Under the same conditions, it is also true \cite[Cor. 3.4.1.9]{lurie-ha} that for any algebra $\widetilde{B}\in Alg_{/\Op}(\Mod_A^{\Op}(\C))$ the canonical map

\begin{equation}
Mod_{\widetilde{B}}^{\Op}(Mod_A^{\Op}(\C))^{\otimes}\to Mod_{B}^{\Op}(\C)^{\otimes}.
\end{equation}
is an equivalence of $\infty$-operads, with $B$ the image of $\widetilde{B}$ through (\ref{quasiquasiquasi}). 

\subsubsection{Modules over associative algebras}
\label{modulesoverassociativealgebras}
We now review the particular situation over the $\infty$-operad $\Assmonoidal$. Let $\Cmonoidal$ be a monoidal $(\infty,1)$-category. Whenever given an associative algebra $A$ in $\Cmonoidal$, it is possible to introduce two new constructions $LMod_A(\C)$, $RMod_A(\C)$ encoding, respectively, the theories of left and right modules over $A$. Their objects are pairs $(A,M)$ where $A$ is an object in $Alg(\C)$ and $M$ is an object in $\C$ equipped with a left (resp. right) action of $A$.  This idea can be made precise with the construction of two new $\infty$-operads $\mathcal{L}\mathcal{M}^{\otimes}$ and $\mathcal{R}\mathcal{M}^{\otimes}$ fabricated to shape left-modules (see Definitions 4.2.1.7 and 4.2.1.13 of \cite{lurie-ha}), respectively, right-modules. Let us overview the mechanism for left-modules. Grosso modo, $\mathcal{L}\mathcal{M}^{\otimes}$ is the operadic nerve of a classical operad $LM$ constructed to have two colours $\textbf{a}$ and $\textbf{m}$ and a unique morphism 

\begin{equation}
\xymatrix{
(\underbrace{\textbf{a}, \textbf{a}, ..., \textbf{a}}_{n}, \textbf{m})\ar[r]&\textbf{m}.
}
\end{equation}

\noindent for each $n\in \mathbb{N}$, these being the only morphisms where the color $\textbf{m}$ appears. Moreover, the full subcategory spanned by the color $\textbf{a}$ recovers the associative operad. At the same time, the projection sending the two colors $(\textbf{a},\textbf{m})$ in $\mathcal{L}\mathcal{M}^{\otimes}$ to the unique color in $\Assmonoidal$, determines a fibration of $\infty$-operads. 

The idea now is to consider algebra-objects $s\in LMod(\C):=Alg_{\mathcal{L}\mathcal{M}/\Ass}(\C)$. From such an object $s$ we can extract an associative algebra-object in $\C$, $s|_{\Assmonoidal}\in Alg(\C) $, an object $s(\textbf{m})=M$ in $\C$ and a multiplication map $(A\otimes M\to M):=s((\textbf{a},\textbf{m})\to \textbf{a})$ which with the help of the cocartesian property of $\Cmonoidal\to \Assmonoidal$ satisfies all the coherences defining the usual module-structure (see the Example 4.2.1.18 of \cite{lurie-ha}). If we fix $A$ an associative algebra object in $\C$, we obtain $LMod_A(\C)$ - the left-modules in $\C$ over the algebra $A$ - as the fiber over $A$ of the canonical map $LMod(\C)\to Alg(\C)$ induced by the composition with the inclusion $\Assmonoidal\subseteq \mathcal{L}\mathcal{M}^{\otimes}$.\\

Given a pair of associative algebras $A$ and $B$, it is also possible to perform a third construction ${}_ABMod_B(\C)$ encoding the theory of bimodules over the pair $(A,B)$ (left-modules over $A$ and right-modules over $B$). Again, the construction starts with the fabrication of an $\infty$-operad $\mathcal{B}\mathcal{M}^{\otimes}$ whose algebras in $\C$ are identified with bimodules (see Definitions 4.3.1.6 and 4.3.1.12 of \cite{lurie-ha}). By construction there are inclusions of $\infty$-operads 

\begin{equation}\xymatrix{  
\mathcal{A}ss^{\otimes}\ar@{^{(}->}[r]^{()^{+}}&\mathcal{L}\mathcal{M}^{\otimes}\ar@{^{(}->}[r]&\mathcal{B}\mathcal{M}^{\otimes}& \ar@{_{(}->}[l] \mathcal{R}\mathcal{M}^{\otimes}&\mathcal{A}ss^{\otimes}\ar@{_{(}->}[l]_{()_{-}}
} 
\end{equation}

\noindent which implies the existence of forgetful functors

\begin{equation}
\xymatrix{LMod_A(\C)& \ar[l] {}_ABMod_B(\C)\ar[r]&RMod_{B}(\C)}
\end{equation}

\noindent which, in general, are not equivalences. By the Theorem $4.3.4.28$ of \cite{lurie-ha} this new theory of modules is related to the general theory by means of a canonical equivalence

\begin{equation}
\xymatrix{ Mod^{\Ass}_A(\C)\ar[r]^{\sim}& {}_ABMod_A(\C)}
\end{equation}

\noindent where $Mod^{\Ass}_A(\C)$ is by definition, the underlying $\infty$-category of $Mod^{\Ass}_A(\C)^{\otimes}$ (the general construction). Under some general conditions, for any triple $(A,B,C)$ of associative algebras in $\C$ it is possible to fabricate a map of $(\infty,1)$-categories

\begin{equation}
{}_ABMod_B(\C)\times {}_BBMod_C(\C)\to {}_AMod_C(\C)
\end{equation}

\noindent encoding a \emph{relative tensor product} (see Def. 4.3.5.10 and Example 4.3.5.11 of \cite{lurie-ha}). It can be understood as a generalization of the formula (\ref{productmodules}), replacing it by the geometric realization of a whole simplicial object $Bar_B(M,N)_{\bullet}$ given informally by the formula 

\begin{equation}
Bar_B(M,N)_n= M\otimes B^n\otimes N
\end{equation}

\noindent (see \cite[Construction 4.3.5.7, Theorem 4.3.5.8]{lurie-ha}). By the Proposition 4.3.6.12 of \cite{lurie-ha}, if $\C$ satisfies the condition\\

$(***)$ it admits geometric realizations of simplicial objects and the tensor product preserves geometric realizations of simplicial objects, separately in each variable;\\

\noindent then, the fibration of $\infty$-operads $Mod^{\Ass}_A(\C)^{\otimes}\to \Assmonoidal$ (obtained by the general methods) is an $\Ass$-monoidal $(\infty,1)$-category with the monoidal structure given by the relative tensor product. Moreover, if $\C$ admits all small colimits and the tensor product is compatible with them on each variable, the equivalence $Mod^{\Ass}_A(\C)\simeq {}_ABMod_A(\C)$ will send the existing abstract-nonsense-monoidal structure on $Mod^{\Ass}_A(\C)$ provided by the general discussion in \ref{monoidalstructuremodules} to the relative tensor product.\\

\begin{remark}
As mentioned before, the theory of left-modules, resp. right-modules, does not have to be equivalent to the general theory (as we will see in the next section, this is true in the commutative case). For this reason it is convenient to have a theory of limits and colimits specific for left, resp. right, modules. See the Corollaries 4.2.3.5 and 4.2.3.5 of \cite{lurie-ha}.
\end{remark}

\subsubsection{Modules over commutative algebras}
\label{modulescommutative}
Finally, if $A$ is a commutative algebra in a symmetric monoidal $(\infty,1)$-category $\Cmonoidal$,  the forgetful map $Mod_A(\C):=Mod_A^{\Comm}(\C)\to Mod_A^{\Ass}(\C)$ fits in a commutative diagram

\begin{equation}
\xymatrix{
&Mod_A(\C)\ar[d]\ar[ddr]\ar[ddl]&\\
&Mod_A^{\Ass}(\C)\ar[d]^{\sim}&\\
LMod_A(\C)&{}_ABMod_A(\C)\ar[r]\ar[l]&RMod_A(\C)
}
\end{equation} 

\noindent and by the Proposition $4.4.1.4$ and the Corollary $4.4.1.6$ of \cite{lurie-ha}, the diagonal arrows are equivalences. Moreover, by the Theorem 4.4.2.1 of \cite{lurie-ha}, if $\C$ satisfies (***), then the $\infty$-operad  $Mod_A(\C)^{\otimes}:=Mod_A^{Comm}(\C)^{\otimes}$ is a symmetric monoidal $(\infty,1)$-category and its tensor product can be identified with the composition

\begin{equation}
\xymatrix{Mod_A(\C)\times Mod_A(\C)\ar[r]& {}_ABMod_A(\C)\times {}_ABMod_A(\C)\ar[r]^{\otimes_A} & {}_ABMod_A(\C)\ar[d]\\
&& LMod_A(\C)\simeq Mod_A(\C)}
\end{equation}

Moreover, if $\C$ admits all small colimits, this monoidal structure agrees with the therefore existing abstract-nonsense structure provided by the discussion in \ref{monoidalstructuremodules}.

\subsubsection{Base Change}
\label{basechange}
We now review the procedure of base change. If $\Cmonoidal$ be symmetric monoidal $(\infty,1)$-category and  $f:A\to B$ is a morphism of commutative algebras, the pre-composition with $f$ produces a forgetful functor

\begin{equation}
f_*: Mod_B(\C)\to Mod_{\A}(\C)
\end{equation}

\noindent which in general is not a monoidal functor. Assuming $\C$ satisfies the condition $(***)$, the relative tensor product discussed in the previous section provides monoidal structures in $Mod_A(\C)$ and $Mod_B(\C)$. The Theorem \cite[Theorem 4.4.3.1]{lurie-ha} enhances this result with the additional fact that $p:Mod(\C)^{\otimes}\to CAlg(\C)\times N(\Fin)$ is a cocartesian fibration. The construction of $p$-cocartesian liftings is achieved using the relative tensor product construction: every morphism of algebras $f:A\to B$ admits a $p$-cocartesian lifting which we can informaly describe with the formula

\begin{equation}M\mapsto L_{A\to B}(M):=M\otimes_A B\end{equation}

Using the Grothendieck construction, this formula assembles to a left adjoint to the forgetful functor $f_*$. Moreover, the properties of the relative tensor product in \cite[4.3.5.9]{lurie-ha} imply that this left adjoint is monoidal.\\

It is also evident by the nature of the construction (obtained via cocartesian liftings) that for any composition $A\to B\to C$ and for any $\A$-module $\M$, there are natural equivalences $(M\otimes_A C)\simeq (M\otimes_A B)\otimes_B C$.

\subsubsection{Transport of Modules via a monoidal functor}
\label{changeofmodulesundermonoidalfunctor}

Our goal in this section is to explain how given $f:\Cmonoidal\to \Dmonoidal$ a monoidal functor between symmetric monoidal $(\infty,1)$-categories, we can associate to every commutative algebra $A\in CAlg(\C)$ a natural map

\begin{equation}
\xymatrix{Mod_A(\C)\ar[r]^{f_A}& Mod_{f(A)}(\D)}
\end{equation}

\noindent and that under some nice hypothesis on $f$, $\Cmonoidal$ and $\Dmonoidal$ this new map will again be monoidal with respect to the monoidal structure on modules described in \ref{monoidalstructuremodules} and \ref{modulescommutative}. Moreover, we want to see that if $A\to B$ is map of commutative algebras, then the diagram

\begin{equation}
\xymatrix{
Mod_A(\C)\ar[r]^{f_A}\ar[d]^{-\otimes_A B}& Mod_{f(A)}(\D)\ar[d]^{-\otimes_{f(A)} f(B)}\\
Mod_B(\C)\ar[r]^{f_B}& Mod_{f(B)}(\D)
}
\end{equation}

\noindent commutes. Here the vertical maps are the base change maps of \ref{basechange}.\\

We start with the construction of the maps $f_A$. For that, we recall that the generalized operads $Mod(\C)^{\otimes}$ and $Mod(\D)^{\otimes}$ are defined in terms of a universal property as simplicial sets over $N(\Fin)$ (See \cite[Construction 3.3.3.1, Definition 3.3.3.8]{lurie-ha}). Using this universal property we can deduce that $f$ induces a map $F:Mod(\C)^{\otimes}\to Mod(\D)^{\otimes}$ sending inert morphims to inert morphisms, and fitting in a commutative diagram

\begin{equation}
\xymatrix{
Mod(\C)^{\otimes}\ar[r]^F\ar[d]^p& Mod(\D)^{\otimes}\ar[d]^q\\
CAlg(\C)\times N(\Fin)\ar[r]^{f_*\times Id}& CAlg(\D)\times N(\Fin)
}
\end{equation}

\noindent where the map $f_*$ is the transport of algebras explained in \ref{transportofalgebras}. We obtain the maps $f_A$ as the restriction of $F$ to the fiber over $A$. In the commutative case, the \cite[Theorem 4.4.3.1]{lurie-ha} explained in the previous section tells us that if $\C$ and $\D$ both satisfy $(***)$, then both maps $p$ and $q$ are cocartesian fibrations. Our goal follows immediately from the following property

\begin{prop}
Let $\Cmonoidal$ and $\Dmonoidal$ be symmetric monoidal $(\infty,1)$-categories such that both $\C$ and $\D$ both admit geometric realizations of simplicial objects and the tensor product preserves them on each variable. Let $f:\Cmonoidal\to \Dmonoidal$ be a monoidal functor commuting with geometric realizations of simplicial objects. Then, the induced map map $F$ sends $p$-cocartesian morphisms to $q$-cocartesian morphisms.
\begin{proof}
Because of \cite[Lemma 2.4.2.8]{lurie-htt}, we are reduced to show that $F$ sends locally $p$-cocartesian morphisms in $Mod(\C)^{\otimes}$ to locally $q$-cartesian morphisms in $Mod(\D)^{\otimes}$. Thanks to the Lemma \cite[4.4.3.4]{lurie-ha}, we are reduced to show that the induced maps

\begin{equation}
\xymatrix{
Mod_A(\C)^{\otimes}\ar[dr]^{p_A}\ar[rr]&& Mod_{f(A)}(\D)^{\otimes}\ar[dl]^{q_A}\\
&N(\Fin)&
}
\end{equation}

\noindent and 

\begin{equation}
\xymatrix{
Mod(\C)\ar[d]\ar[r]& Mod(\D)\ar[d]\\
CAlg(\C)\ar[r]& CAlg(\D)
}
\end{equation}

\noindent both send locally cocartesian morphisms to locally cocartesian morphisms. By the inspection of the proofs of \cite[4.4.2.1]{lurie-ha} for the first case and \cite[4.4.3.6, 4.3.7.14]{lurie-ha} for the second, we conclude that everything is reduced to show that $f$ sends the relative tensor product in $\C$ to the relative tensor product in $\D$. By inspection of \cite[Construction 4.3.5.7 and Theorem 4.3.5.8]{lurie-ha} we conclude that this follows immediately from our assumptions on $f$.

\end{proof}
\end{prop}

\subsection{Endomorphisms Objects}
\label{endomorphismsobjects}

In this section we review the theory of endomorphism objects as developed in \cite{lurie-ha}-Sections 6.1 and 6.2. Let $\Cmonoidal$ be a monoidal $(\infty,1)$-category. As reviewed in the section \ref{modulesoverassociativealgebras}, to every object $A\in Alg(\C)$ we can associate a new $(\infty,1)$-category $LMod_A(\C)$ whose objects provide objects $m$ in $\C$ equipped with a multiplication $A\otimes m\to m$ satisfying the usual coherences for modules. We can now generalize the notion of an $A$-module to include objects $m$ belonging to any $(\infty,1)$-category $\M$ where $\C$ acts. More precisely, recall that $\C$ can be understood as an object in $Alg(\iCat)$ and therefore $\C$ itself admits a theory of left-modules $LMod_{\C}(\iCat)$ - the objects of this $(\infty,1)$-category can be understood as $(\infty,1)$-categories $\M$ equipped with an action $\bullet:\C\times \M\to \M$ satisfying the standard coherences for module-objects in $\iCat$. We generalize the notion of an $A$-module to include objects $m\in M$ endowed with a multiplication $A\bullet m\to m$ satisfying the standard coherences for modules in $\M$. This can be made precise as follows. Let $\overline{\M}$ be an object in $LMod_{\C}(\iCat)$. Explicitly, $\overline{\M}$ is a map of $\infty$-operads 

\begin{equation}
\xymatrix{
&\iCat^{\times}\ar[d]\\
\mathcal{L}\mathcal{M}^{\otimes}\ar[r]\ar[ru]^{\overline{M}}&N(\Fin)
}
\end{equation}

\noindent whose restriction to $\Assmonoidal\subseteq \mathcal{L}\mathcal{M}^{\otimes}$ is the monoidal $(\infty,1)$-category $\Cmonoidal$ and whose evaluation at the color $\textbf{m}$ is another $(\infty,1)$-category $\M$. Since $\iCat^{\times}$ is cartesian, we have an equivalence $Alg_{\mathcal{L}\mathcal{M}}(\iCat)\simeq Mon_{\mathcal{L}\mathcal{M}}(\iCat)$ and therefore we can use the Grothendieck construction to present the diagram $\overline{\M}$ in the format of a cocartesian fibration of $\infty$-operads $\Opmonoidal_{\overline{\M}}\to \mathcal{L}\mathcal{M}^{\otimes}$ where we recover $\Opmonoidal_{\overline{\M}}\times_{\mathcal{L}\mathcal{M}^{\otimes}}\Assmonoidal\simeq \Cmonoidal$ and $\Opmonoidal_{\overline{\M}}\times_{\mathcal{L}\mathcal{M}^{\otimes}}\{\textbf{m}\}\simeq \M$ and the action $\bullet:\C\times \M\to \M$ is again extracted from the cocartesian property.

In this setting, an object $\overline{m}\in LMod(\C,\M):= Alg_{/\mathcal{L}\mathcal{M}}(\Op_{\overline{M\}}})$ provides exactly the concept we seek: the restriction $\overline{m}|_{\Assmonoidal}$ is an algebra-object in $\Opmonoidal_{\overline{M}}\times_{\mathcal{L}\mathcal{M}^{\otimes}}\Assmonoidal\simeq \Cmonoidal$; the value at $\textbf{m}$ is an object $m$ in $\Opmonoidal_{\overline{M}}\times_{\mathcal{L}\mathcal{M}^{\otimes}}\{\textbf{m}\}\simeq \M$ and the image of canonical morphism $(\textbf{a}, \textbf{m})\to \textbf{m}$ provides, via the cocartesian property of $\Opmonoidal_{\overline{M}}\to \mathcal{L}\mathcal{M}^{\otimes}$, a map $A\bullet m\to m$ which, also because of the cocartesian property, will satisfies all the standard coherences we seek.\\

There are canonical projections $LMod(\C,\M)\to Alg(\C)$ and $LMod(\C,\M)\to \M$ induced, respectively, by the composition with the inclusion $\Assmonoidal\subseteq \mathcal{L}\mathcal{M}^{\otimes}$ and the inclusion $\{\textbf{m}\}\subseteq \mathcal{L}\mathcal{M}^{\otimes}$ (see Def. 4.2.1.13 of \cite{lurie-ha}). For each associative algebra $A$ in $\C$, the fiber $LMod_A(\C,\M):=LMod(\C,\M)\times_{Alg(\C)}\{A\}$ gathers the collection of objects $m$ in $\M$ endowed with a left action of $A$ satisfying the standard coherences of being a module-object. Similarly, for each object $m\in \M$, the fiber $LMod(\C,\M)\times_{\M} \{m\}$ codifies all the different ways in which the object $m$ can be endowed the structure of an $A$-module, for some associative algebra $A$ in $\C$.

\begin{remark}
If $\Cmonoidal\to \Assmonoidal$ is a monoidal $(\infty,1)$-category, the tensor operation provides $\C$ with the structure of a $\C$-module. To make this precise we recall that $\Cmonoidal$ can be understood as an algebra-object in $\iCat$ and since this last is cartesian, we can encode $\Cmonoidal$ as a diagram $\Assmonoidal\to \iCat$. 
We define $\overline{\M}$ as the composition 

\begin{equation}
\mathcal{L}\mathcal{M}^{\otimes}\to \Assmonoidal \to \iCat 
\end{equation}

\noindent where the first map is the canonical projection. Again because $\iCat$ is cartesian, this map provides an $\mathcal{L}\mathcal{M}^{\otimes}$-algebra in $\iCat$ which we can see exhibits $\C$ as a module over itself. Finally, by performing the Grothendieck construction on $\overline{\M}$ we find a canonical equivalence 

\begin{equation}
\Opmonoidal_{\overline{\M}}\simeq \Cmonoidal\times_{\Assmonoidal} \mathcal{L}\mathcal{M}^{\otimes}
\end{equation}

Using the definitions we can easily find an equivalence

\begin{equation}LMod(\C)\simeq LMod(\C,\M=\C)\end{equation}

\end{remark}

\begin{remark}
This construction uses the data of symmetric monoidal $(\infty,1)$-category $\Cmonoidal$ and a module $\M$ over it as initial ingredients. However, the defining ingredient is the data of the cartesian fibration $\Opmonoidal_{\overline{M}}\to \mathcal{L}\mathcal{M}^{\otimes}$. Dropping the cocartesian condition we can reproduce the situation with the data of fibration of $\infty$-operads $p:\Opmonoidal\to \mathcal{L}\mathcal{M}^{\otimes}$. This gives rise to what in \cite{lurie-ha}-Definition 4.2.1.12 is called a \emph{weak enrichment} of $\M:= \Opmonoidal\times_{\mathcal{L}\mathcal{M}^{\otimes}} \{\textbf{m}\}$ over $\Cmonoidal:=\Opmonoidal\times_{\mathcal{L}\mathcal{M}^{\otimes}} \Assmonoidal$.
\end{remark}

We now have the following important result:

\begin{prop}(see \cite{lurie-ha}-Corollary 6.1.2.42)
Let $\C$ be a monoidal $(\infty,1)$-category and $\overline{\M}$ be an object in $LMod_{\C}(\iCat)$ with $\overline{\M}|_{\Assmonoidal}=\Cmonoidal$. Let $m$ be an object in $\M=\overline{\M}(\textbf{m})$. Then, the canonical map $p: LMod(\C,\M)\times_{\M}\{m\}\to Alg(\C)$ is a right fibration (in particular it is a cartesian fibration and its fibers are $\infty$-groupoids).
\end{prop}

In \cite{lurie-ha}-Section 6.1.2, the author proves this result by constructing a new monoidal $(\infty,1)$-category $\C^{+}[m]$ whose objects can be identified with pairs $(X,\eta)$ where $X$ is an object in $\C$ and $\eta:X\bullet m\to m$ is a morphisms in $M$. The canonical map $Alg(\C^+[m])\to Alg(\C)$ is a right fibration (Prop. 6.1.2.39 of \cite{lurie-ha} ) and the conclusion follows from the existence of an equivalence of right-fibrations $LMod(\C,\M)\times_{\M}\simeq Alg(\C^{+}[m])$ (Corollary 6.1.2.40 of \cite{lurie-ha}).\\

In the context of the previous result, we say that the object $m$ admits a \emph{classifying object for endomorphisms} if the right fibration $p$ is representable. Because of \cite{lurie-htt}-Theorem 4.4.4.5, this amounts to the existence an algebra-object $End(m)\in Alg(\C)$ and an equivalence of right fibrations $LMod(\C,\M)\times_{\M}\{m\}\simeq Alg(\C)_{/End(m)}$ over $Alg(\C)$. In this case, for each associative algebra-object $A$ in $\C$ we have a canonical homotopy equivalence
\begin{equation}
Map_{Alg(\C)}(\A,End(m))\simeq \{A\}\times_{\Assmonoidal} LMod(\C,\M)\times_{\M}\{m\}
\end{equation}
In other words, morphisms of algebras $A\to End(m)$ correspond to $A$-module structures on $m$.\\

\begin{remark}
\label{existenceofendomorphismsobject}
Following the arguments in the proof of the previous result, and due to the Corollary 3.2.2.4 of \cite{lurie-ha}, the existence of a classifying object for endomorphisms for $m$ can be deduced from the existence of a final object in $Alg(\C^{+}[m])$.\\ 
\end{remark}

We will be mostly interested in finding classifying objects for endomorphisms in the case when $\C=\M$ is $\iCat$ with the cartesian product. In other words, we want to have, for any monoidal $(\infty,1)$-category $A\in Alg(\C=\iCat)$ and any $(\infty,1)$-category $m\in \M=\iCat$, a new monoidal $(\infty,1)$-category $End(m)$ such that the space of categories $m$ left-tensored over $A$ is homotopy equivalent to the space of monoidal functors $A\to End(m)$. As expected, $End(m)$ exists: it can be canonically identified with the $(\infty,1)$-category of endofunctors of $m$ - $Fun(m,m)$ - equipped with the \emph{strict} monoidal structure $End(M)^{\otimes}\to \Assmonoidal$ induced by the composition of functors. See the Notation 6.2.0.1 and the Proposition 6.2.0.2 of \cite{lurie-ha} for the construction of this monoidal structure. Finally, the fact that $End(m)^{\otimes}$ has the required universal property follows from the universal property of $Fun(m,m)$ as an internal-hom in $\iCat$, from the Proposition 6.1.2.39 and the Corollary 3.2.2.4 of \cite{lurie-ha} (see the Remark 6.2.0.5 of loc.cit). 

\subsection{Idempotent Algebras}
\label{idemalgebras}

In this section we review the theory of idempotents as developed in \cite{lurie-ha}- Section 6.3.2.
Let $\Cmonoidal$ be a symmetric monoidal $(\infty,1)$-category with unit $1$ and let $E$ be an object in $\C$. A morphism $e:1\to E$ is said to be \emph{idempotent} if the product morphism $id_E\otimes e:  E\otimes 1\to E\otimes E$ is an equivalence. Since $\Cmonoidal$ is symmetric this is equivalent to ask for $e\otimes id_E$ to be an equivalence. We write $(E,e)$ to denote an idempotent. The first important result concerning idempotents is that a pair $(E,e)$ is an idempotent if and only if the product functor $E\otimes-: \C\to \C$ makes  $\C_E$ - the essential image of $(E\otimes-)$ - a full reflexive subcategory of $\C$ (see the Proposition 6.3.2.4 of \cite{lurie-ha}). Notice that $\C_E$ equals the full subcategory of $\C$ spanned by those objects in $\C$ which are stable under products with $E$. By the Proposition 6.3.2.7 of loc. cit, this localization is monoidal and therefore $\C_E$ inherits a symmetric monoidal structure $\C_E^{\otimes}$ where the unit object is $E$ and the product map $(E\otimes-)$ extends to a monoidal map $\Cmonoidal \to \C_{E}^{\otimes}$. Its right adjoint (the inclusion) is lax-monoidal and therefore induces an inclusion

\begin{equation}
CAlg(\C_E)\to CAlg(\C)
\end{equation}

\noindent and since $E$ is the unit in $\C_E$ we can use this inclusion to  equip $E$ with the structure of a commutative algebra in $\C$ for which the multiplication map $E\otimes E\to E$ is an equivalence. In fact, by the Proposition 6.3.2.9 of \cite{lurie-ha}, there is a perfect matching between idempotents and commutative algebras whose multiplication map is an equivalence (these are called \emph{idempotent-algebras}). More precisely, if we denote by $CAlg^{idem}(\C)$ the full subcategory of commutative algebra objects in $\C$ whose multiplication map $A\otimes A\to A$ is an equivalence, the natural composition 

\begin{equation}
CAlg^{idem}(\C)\subseteq CAlg(\C)\simeq CAlg(\C)_{1/}\to \C_{1/}
\end{equation}

\noindent sending an commutative algebra object $A$ to its unit $1\to A$ morphism, is fully-faithful and its image consists exactly of the idempotent objects in $\C$.\\ 

The main result for idempotents can be stated as follows:

\begin{prop}(\cite{lurie-ha}-Prop.6.3.2.10)
Let $\Cmonoidal$ be a symmetric monoidal $(\infty,1)$-category and let $(E,e)$ be an idempotent which we now known can be given by the unit of a commutative algebra object $A$ in $\C$ (which is unique of to equivalence). Then, the natural forgetful map $Mod_A(\C)^{\otimes}\to \Cmonoidal$ induces an equivalence $\Mod_A(\C)^{\otimes}\to \C_E^{\otimes}$.
\end{prop}

\subsection{Presentability}
\label{section3-4}

The results in this paper depend crucially on the presentability of the closed cartesian symmetric monoidal $(\infty,1)$-category $\iCat$ (see Prop. 3.1.3.7 and Cor. 3.1.4.4  of \cite{lurie-htt}). By the Proposition 5.5.4.15 of \cite{lurie-htt}, the theory of presentable $(\infty,1)$-categories admits a very friendly theory of localizations: every localization with respect to a set of morphisms admits a description by means of local objects and, conversely, every (small) local theory is a localization. This feature will play a vital role in the proceeding sections where we shall work with presentable symmetric monoidal $(\infty,1)$-categories.\\

Let $\mathcal{K}$ be the collection of all small simplicial sets. By definition (see Def. 3.4.4.1 of \cite{lurie-ha}) a \emph{presentable $\Op$-monoidal $(\infty,1)$-category} is an $\Op$-monoidal $(\infty,1)$-category compatible with $\mathcal{K}$-colimits such that for each color $x\in \Op$, the fiber $\C_x$ is presentable. In this case, it is a corollary of the \emph{Adjoint Functor Theorem} that $\Cmonoidal$ is necessary closed.

\begin{remark}
\label{presentableremark}
Let $\Opmonoidal$ be a small coherent $(\infty,1)$-operad. Let $\Cmonoidal$ be a presentable $\Op$-monoidal $(\infty,1)$-category. By \cite[3.2.3.5]{lurie-ha} $Alg_{/\Op}(\C)$ is a presentable $(\infty,1)$-category and by \cite[3.4.4.2]{lurie-ha} $Mod_A^{\Op}(\C)^{\otimes}$ is presentable $\Op$-monoidal.
\end{remark}

The following monoidal version of the adjoint functor theorem will be useful to us in the future:

\begin{prop}
\label{monoidaladjoint}
Let $f^{\otimes}:\Cmonoidal\to \Dmonoidal$ be a monoidal functor between symmetric monoidal $(\infty,1)$-categories $\pi_1:\Cmonoidal\to N(\Fin)$ and $\pi_2: \Dmonoidal\to N(\Fin)$.  By the adjoint functor theorem $f^{\otimes}_{\onefin}$ has a right adjoint  $g:\D\to \C$. Then, $g$ can be extended to a lax-monoidal functor $g^{\otimes}:\Dmonoidal\to \Cmonoidal$.
\begin{proof}
Using the Grothendieck construction, we can encoded the data of the monoidal functor $f^{\otimes}$ as a cocartesian fibration $p:\mathcal{A}\to \Delta[1]\times N(\Fin)$ with $p^{-1}(\{0\})=\Cmonoidal$ and $p^{-1}(\{1\})=\Dmonoidal$. To extend $g$ to a monoidal functor we are reduced to show  that for every $Y=(Y_1,.... Y_n)\in \Dmonoidal_{\nfin}$ we can find a $p$-cartesian morphism $Y'\to Y$ over the edge $(0, \nfin)\to (1,\nfin)$ in $\Delta[1]\times N(\Fin)$ with $\nfin\to \nfin$ the identity map. The existence of $g$ provides these $p$-cartesian morphisms when $\nfin=\onefin$. When $n\geq 2$ we make use of the case $n=1$ using the fact that in symmetric monoidal $(\infty,1)$-category $\pi_2:\Dmonoidal\to N(\Fin)$,  an object $Y=(Y_1,.... Y_n)\in \Dmonoidal_{\nfin}$  is a $\pi_2$-limit diagram in $\Dmonoidal$ of  the objects $Y_1, ..., Y_n$ in $\Dmonoidal_{\onefin}$ (see \cite[2.1.2.11]{lurie-ha}). Using this definition we set $Y'$ as the $\pi_1$-limit diagram in $\Cmonoidal$ of $g(Y_1),..., g(Y_n)$ (which exists due to presentability) and by the same argument we can identify it with  $(g(Y_1),..., g(Y_n))\in \Cmonoidal_{\nfin}$. We are now reduced to define a $p$-cartesian morphism $\alpha:Y'\to Y$ over $0\to 1$.  By the definition of $\mathcal{A}$ and because $f$ is in particular lax monoidal,  such a morphism can be uniquely identified with a morphism $(f(g(Y_1),..., f(g(Y_n))\to (Y_1,..., Y_n)$ in $\Dmonoidal$. We define $\alpha$ as the product of the counit morphisms $f(g(Y_i))\to Y_i$ given by the adjunction $(f,g)$. Finally,  we show that our choice makes $\alpha$ $p$-cartesian. For that, let $X=(X_1,...., X_m)$ be an object in $\Cmonoidal_{\mfin}$ and let $u:\mfin\to \nfin$ be a morphism in $\N(\Fin)$. We have

\begin{eqnarray}
Map_{\Cmonoidal}^u( X,Y') \simeq \prod_{1\leq i\leq n} Map_{\Cmonoidal}^{\rho_i\circ u} ((\{X_j\}_{j\in u^{-1}(i)},g(Y_i))\simeq \\
 \prod_{1\leq i\leq n} Map_{\C} (\otimes_{j\in u^{-1}(i)} X_j, g(Y_i))\simeq  \prod_{1\leq i\leq n} Map_{\D} (f(\otimes_{j\in u^{-1}(i)} X_j), Y_i)\simeq \\
 \prod_{1\leq i\leq n} Map_{\D} (\otimes_{j\in u^{-1}(i)}f( X_j), Y_i) \simeq\prod_{1\leq i\leq n} Map_{\Dmonoidal}^{\rho_i\circ u} ((\{f(X_j)\}_{j\in u^{-1}(i)},Y_i)\simeq \\
Map_{\mathcal{A}}^{(01)}(X, Y)
\end{eqnarray}

where the first equivalence follows because $\Cmonoidal$ is an $\infty$-operad, the second because $\Cmonoidal$ is symmetric monoidal, the third 
because $(f,g)$ is an adjunction, the fourth because $f$ is  monoidal, the fifth because $\Dmonoidal$ is an $\infty$-operad and the final follows from the definition of $\mathcal{A}$.

\end{proof}
\end{prop}

\subsubsection{The Monoidal Structure in $\Prl$}
Following the notations from \cite{lurie-htt} we write $\Prl$ for full subcategory of $\iCatbig(\mathcal{K})$ (with $\mathcal{K}$ denoting the collection of all simplicial sets) spanned by the presentable $(\infty,1)$-categories together with the colimit preserving functors. By \cite[6.3.1.14]{lurie-ha}, $\Prl$ is closed under the monoidal structure in $\iCatbig(\mathcal{K})^{\otimes}$ described by the formula (\ref{equationlala}) and therefore inherits a symmetric monoidal structure $(\Prl)^{\otimes}$: if $\C_0$ and $\C_0'$ are two small $\infty$-categories, the tensor product $\mathcal{P}(\C_0)\otimes \mathcal{P}(\C_0')$ is given by $\mathcal{P}(\C_0\times \C_0')$. More generally, if $\C$ and $\C'$ are two presentable $(\infty,1)$-categories and $S$ is a small collection of morphism in $\C$, the product $(S^{-1}\C)\otimes \C'$ is the localization $T^{-1}(\C\otimes \C')$ where $T$ is the image of the collection $S\times \{id_X\}_{X\in Obj(\C')}$ via the canonical morphism $\C\times \C'\to \C\otimes \C'$. By the Theorem 5.5.1.1 of \cite{lurie-htt} this is enough to describe any product and also to conclude that the unit object is the $(\infty,1)$-category of spaces $\mathcal{S}=\mathcal{P}(*)$.

The objects in $CAlg(\Prl)$ can now be identified with the presentable symmetric monoidal $(\infty,1)$-categories. Plus, this symmetric monoidal structure is closed: for any pair of presentable $(\infty,1)$-categories $\A$ and $\B$, the $(\infty,1)$-category $Fun^L(\A,\B)$ of colimit-preserving functors $\A\to \B$ is again presentable and provides an internal-hom object in $\Prlmonoidal$ (see the Remark 6.3.1.17 of \cite{lurie-ha}). Since $\Prl$ admits all small colimits (by the combination of the Corollary 5.5.3.4 and the Theorem 5.5.3.18 of \cite{lurie-htt}), we conclude that $(\Prl)^{\otimes}$ is a symmetric monoidal structure compatible with all small colimits. \\

The following result will also be important to us:

\begin{prop}
\label{presentablehaveendomorphismobject}
The symmetric monoidal $(\infty,1)$-category $\Prlmonoidal$ admits classifying objects for endomorphisms: for each presentable $(\infty,1)$-category $M$, the $(\infty,1)$-category $End^L(\M)$ of colimit-preserving endomorphisms of $M$ is the underlying $(\infty,1)$-category of a presentable monoidal $(\infty,1)$-category $End^L(M)^{\otimes}\to \Assmonoidal$ whose monoidal operation is determined the composition of functors. Moreover, for any presentable symmetric monoidal $(\infty,1)$-category, we have a canonical homotopy equivalence

\begin{equation}
Maps_{Alg(\Prl)}(\Cmonoidal, End^L(M)^{\otimes})\simeq \{\Cmonoidal\}\times_{\Assmonoidal} LMod(\Prl, \Prl)\times_{\Prl}\{M\}
\end{equation}

\begin{proof}
For the part that concerns the monoidal structure on $End^L(M)$, we know from the Notation 6.2.0.1 and the Proposition 6.2.0.1 of \cite{lurie-ha} that $End(M)$ admits a monoidal structure $End(M)^{\otimes}\to \Assmonoidal$ where the fiber over $\nfin$ is isomorphic to $\prod_{\nfin}Fun(M,M)$. We take $End^L(M)^{\otimes}$, by definition, the full subcategory of $End(M)^{\otimes}$ spanned by those sequences $(f_1,..., f_n)$ where each $f_i$ is a colimit-preserving endofunctor of $M$. The fact that the composition $q:End^L(M)^{\otimes}\subseteq End(M)^{\otimes}\to \Assmonoidal$ is still a cocartesian fibration follows immediately from the fact that the composition of colimit-preserving functors is colimit-preserving. It follows also that this monoidal structure is strictly associative because this holds for $End(M)^{\otimes}$ (see Notation 6.2.0.1 of \cite{lurie-ha}).

To prove that this monoidal structure is presentable (see Definition 3.4.4.1 of \cite{lurie-ha}) it suffices to observe that: $(i)$ since $M$ is presentable, $End^L(M)$ is also presentable (See \cite{lurie-htt}-Prop. 5.5.3.8); $(ii)$ since the colimits in $End^L(M)$ are computed objectwise in $M$ (\cite{lurie-htt}-Cor. 5.1.2.3 ) and the objects in $End^L(M)$ are, by definition, colimit-preserving functors, the cocartesian fibration $q$ is compatible with small colimits (See \cite{lurie-ha}-Definition 3.1.1.18).

To conclude, the fact that $End^L(M)$ provides a classifying object for endomorphisms results from the same arguments as in the Remark 6.2.0.5 of \cite{lurie-ha}: since $End^L(M)$ has the property of internal-hom object in $\Prl$, it provides a final object in $(\Prl)^{+}[\Prl]$. The Corollary 3.2.2.4 \cite{lurie-ha} applied to $End^L(M)^{\otimes }$ concludes the proof.

\end{proof}
\end{prop}

\subsubsection{The Monoidal Structure in $\Prl_{\kappa}$}
Let $\kappa$ be a regular cardinal. Following \cite[6.3.7.11]{lurie-ha}), the monoidal structure in $\Prl$ restricts to a monoidal structure in the (non-full) subcategory $\Prl_{\kappa}\subset \Prl$. Moreover, if $\mathcal{K}$ denotes the collection of $\kappa$-small simplicial sets together with the simplicial set $Idem$, the equivalence 

\begin{equation}
Ind_{\kappa}:\iCat(\mathcal{K})\to \Prl_{\kappa}
\end{equation}

\noindent of the discussion in \ref{compactlygenerated1} is compatible with the monoidal structures (where on the left side we consider the monoidal structure described in \ref{compatiblewithcolimitsmonoidal}). 

To see this we use the fact the monoidal structure in $\Prl$ is the restriction of the monoidal structure described in \ref{compatiblewithcolimitsmonoidal} for the $(\infty,1)$-category of big $(\infty,1)$-categories with all colimits together with colimit preserving functors. The discussion in the same section implies also that $Ind_{\kappa}$ is monoidal, so that the product of $\kappa$-compactly generated in $\Prl$ is again compactly generated. Moreover, if $x$ is a $\kappa$-compact object in $\C$ and $y$ is a $\kappa$-compact object in $\C'$, their product $x\otimes y$ is a $\kappa$-compact object in the product $\C\otimes \C'$ and the collection of $\kappa$-compact objects in $\C\otimes \C'$ is generated by the objects of this form under $\kappa$-small colimits. This implies that if $\C$ and $\C'$ and $\D$ are $\kappa$-compactly generated, the equivalence in (\ref{equationlala}) restricts to an equivalence between the full subcategory of $Fun_{\mathcal{K}}(\C\otimes \C', \D)$ spanned by those functors which preserve $\kappa$-compact objects and the full subcategory of $Fun_{\mathcal{K}\boxtimes\mathcal{K}}(\C\times \C', \D)$ spanned by the functors which preserve $\kappa$-compact objects separately in each variable.

Let now $\Prlkmonoidal$ denote the (non-full) subcategory of $\Prlmonoidal$ spanned by the objects $(\C_1,...,\C_n)$ where each $\C_i$ is a $\kappa$-compactly generated $(\infty,1)$-category, together with the maps $(\C_1,..,\C_n)\to (\D_1,...,\D_m)$ over some $f:\nfin\to \mfin$ corresponding to those families of functors 

\begin{equation}\{u_i:\prod_{j\in f^{-1}(\{i\})} \C_j\to \D_i\}_{i\in \{1,...,m\}}\end{equation}

\noindent in $\iCat^{big}$ where each functor commutes with colimits separately in each variable and sends $\kappa$-compact objects to $\kappa$-compact objects, separately in each variable. It follows from the restriction of the equivalence in (\ref{equationlalala}) to the subcategories of compact preserving functors, that if $f:\nfin\to \mfin$ is a map in $N(\Fin)$ and $X=(\C_1,..., \C_n)$ is a sequence of $\kappa$-compactly generated $(\infty,1)$-categories, then the map in $\Prl$ corresponding to the family of universal functors

\begin{equation}
\prod_{j\in f^{-1}(\{i\})}\C_j\to \D_i:= \mathcal{P}^{\mathcal{K}}_{\boxtimes_{j\in f^{-1}(\{i\}}\mathcal{K}}(\prod_{j\in f^{-1}(\{i\}}\C_j)
\end{equation}

\noindent is in $\mathcal{P}r^{L, \otimes}_k$ (because it commutes with colimits separately in each variable and preserves compact objects separately in each variable) and provides a cocartesian lift to $f$ at $X$. It follows that the non-full inclusion $\mathcal{P}r^{L}_{\kappa}\subseteq \Prl$ is monoidal.\\

\subsection{Dualizable Objects}
\label{dualizable}

We will recall the notion of duals. If $\Cmonoidal$ is a symmetric monoidal $(\infty,1)$-category with a unit $\mathit{1}$, we say that an object $X$ is dualizable, or that it has a dual, if there exist an object $\check{X}$ together with morphisms 

\begin{equation}
\xymatrix{
\mathit{1}\ar[r]^{\alpha_X}& \check{X}\otimes X && X\otimes \check{X}\ar[r]^{\beta_X}& \mathit{1}
}
\end{equation}

\noindent such that the compositions

\begin{equation}
\xymatrix{
X\simeq X\otimes \mathit{1}\ar[rr]^{Id_X\otimes \alpha_X}&& X\otimes \check{X}\otimes X\ar[rr]^{\beta_X\otimes Id_X}&& \mathit{1}\otimes X\simeq X\\ 
\check{X}\simeq \mathit{1}\otimes \check{X}\ar[rr]^{\alpha_X\otimes Id_{\check{X}}}&& \check{X}\otimes X \otimes \check{X} \ar[rr]^{Id_{\check{X}}\otimes \beta_X}&& \check{X}\otimes \mathit{1}\simeq \check{X}
}
\end{equation}

\noindent are homotopic to the identity maps in $\C$. These restrains are equivalent to ask that for any pair of objects $Y$ and $Z$ in $\C$, the multiplication with the dual induces an homotopy equivalence

\begin{equation}
Map_{\C}(X\otimes Y, Z)\simeq Map_{\C}(Y,\check{X}\otimes Z)
\end{equation}

In particular, if $\Cmonoidal$ admits internal-hom objects and $X$ has a dual, then we have for every object $Y$ in $\C$, a canonical equivalence $Y^X\simeq \check{X}\otimes Y$.\\

\subsection{Stability}
\label{section3-5}

\subsubsection{Stable Monoidal $(\infty,1)$-categories}

Let $\iCatstable$ denote the $(\infty,1)$-category of small stable $\infty$-categories together with the exact functors. The inclusion $\iCatstable\subseteq \iCat$ preserves finite products (as a result of the Theorem 1.1.1.4 of \cite{lurie-ha}) and therefore $\iCatstable$ inherits a symmetric monoidal structure $(\iCatstable)^{\otimes}$ induced from the cartesian structure in $\iCat$. By definition (see Def. 8.3.4.1 of \cite{lurie-ha}) a \emph{stable $\Op$-monoidal $\infty$-category} is an $\Op$-monoidal $\infty$-category $q:\Cmonoidal\to \Opmonoidal$ such that for each color $X\in \Op$, the fiber $\C_{X}$ is a stable $\infty$-category and the monoidal operations are exact separately in each variable. In particular, the monoidal structure commutes with finite colimits. The small stable symmetric monoidal $\infty$-categories can be identified with commutative algebra objects in $(\iCatstable)^{\otimes}$.\\

If $\Cmonoidal\to \Opmonoidal$ is an $\Op$-monoidal $\infty$-category compatible with all colimits, then it is stable $\Op$-monoidal if and only if for each color $x\in \Op$ the fiber $\C_x$ is stable. This is obvious because, by definition, the monoidal structure preserves colimits on each variable and therefore is exact on each variable. These will be called \emph{stable presentable $\Op$-monoidal $\infty$-categories}. We know that the stable presentable $(\infty,1)$-categories form a full subcategory of $\Prlstable$ of $\Prl$ and by \cite[6.3.2.10, 6.3.2.18]{lurie-ha}  it is closed under the tensor structure in $\Prl$. Moreover, following \cite[6.3.1.17]{lurie-ha}, if $\C$ and $\D$ are stable presentabled $(\infty,1)$-categories, $Fun^L(\C,\D)$ is again stable presentable so that the monoidal structure in $\Prlstable$ is closed. 
We can identify the stable presentable symmetric monoidal $(\infty,1)$-categories with the objects in $CAlg(\Prlstable)$.\\

\begin{remark}
Let $\C$ be a stable $\Op$-monoidal $(\infty,1)$-category compatible with all colimits. Then, for any $A\in Alg_{/\Op}(\C)$ the symmetric monoidal $(\infty,1)$-categories $Mod_A^{\Op}(\C)^{\otimes}$ is stable. This follows immediately from the fact that for each colour $x\in \Op$,  pushouts and pullbacks in $Mod_A^{\Op}(\C)_x$ are computed in $\C_x$ by means of the forgetful functor  $Mod_A^{\Op}(\C)^{\otimes}_x\to \C_x$ (which is conservative). Moreover, since $Mod_A^{\Op}(\C)^{\otimes}$ is again compatible with colimits, the multiplication maps of the monoidal structure are exact on each variable. Notice however that the same is not true for algebras because colimits are not computed directly as colimits in the underlying category.
\end{remark}

The canonical example of a stable symmetric monoidal $(\infty,1)$-category is the $(\infty,1)$-category of spectra $\Sp$ with the smash product structure. One way to obtain this monoidal structure is to prove that $\Sp$ is an idempotent object in $\Prlmonoidal$ \cite[Prop. 6.3.2.18]{lurie-ha}. Our results in this paper provide an alternative way to obtain this monoidal structure. We will return to this in the Example \ref{examplemonoidalstructurespectra}.\\

To conclude this section we provide an helpful characterization of compact generators in categories of modules over a stable presentable $(\infty,1)$-category. 

\begin{prop}
\label{compactgeneratorscategoriesofmodules}
Let $\Cmonoidal$ be a stable presentable symmetric monoidal $(\infty,1)$-category.  Suppose that its underlying $(\infty,1)$-category $\C$ admits a family $\mathcal{E}=\{E_i\}_{i\in I}$ of $\kappa$-compact generators in the sense of \ref{stableinfinitycategories}. Then, for any commutative algebra object $A$ in $\C$, the family $\{A\otimes E_i\}_{i\in I}$ is a family of $\kappa$-compact generators in the $(\infty,1)$-category $Mod_{A}(\C)$ (this makes sense because by the previous remark the category of modules is stable).
\begin{proof}
By definition, $A\otimes E_i$ is the image of $E_i$ under the base-change monoidal functor $(-\otimes A):\Cmonoidal\to Mod_{A}(\C)^{\otimes}$. This functor is a left adjoint to the forgetful functor. The result follows immediately from this adjunction, together with the fact the forgetful functor is conservative and commutes with colimits (\cite[3.4.4.6]{lurie-ha}).
\end{proof}
\end{prop}

\subsubsection{Compatibility with $t$-structures}
\label{monoidalt}

Let now $\Cmonoidal$ be a stable symmetric monoidal $(\infty,1)$-category and suppose that $\C$ is equipped with a $t$-structure $((\C)_{\leq 0}, (\C)_{\geq 0})$. Following \cite[2.2.1.3]{lurie-ha} we say that the monoidal structure is compatible with the $t$-structure if the full subcategory $\C_{\geq 0}$ contains the unit object and is closed under the tensor product. In this case, it inherits a symmetric monoidal structure. Moreover, the truncation functors $\tau_{\leq n}: (\C)_{\geq 0}\to (\C)_{\geq 0}$ are monoidal \cite[2.2.1.8]{lurie-ha} and in particular, the subcategories $(\C_{\geq 0}\cap \C_{\leq n})$ are monoidal reflexive localizations of $\C_{\geq 0}$ \cite[2.2.1.10]{lurie-ha}. In particular, the heart $\C^{\heartsuit}$ inherits a symmetric monoidal structure and the zero-homology functor $\mathbb{H}_0:\C_{\geq 0}\to \C^{\heartsuit}$ is monoidal.\\

Given an $\infty$-operad $\Opmonoidal$, we write $Alg_{\Op}(\C)^{cn}$ for the full subcategory of $Alg_{\Op}(\C)$ spanned by the algebra objects whose underlying object in $\C$ is in $\C_{\geq 0}$. Since $\C_{\geq 0}$ inherits a monoidal structure, we have a fully-faithfull map $Alg_{\Op}(\C_{\geq 0})\subseteq Alg_{\Op}(\C)$ and its image can be identified with $Alg_{\Op}(\C)^{cn}$. It follows from the discussion in \ref{subcategoriesclosedunderproduct} that the right adjoint $\tau_{\geq 0}: \C\to \C_{\geq 0}$ extends to a right adjoint to the inclusion

\begin{equation}
Alg_{\Op}(\C)^{cn}\hookrightarrow Alg_{\Op}(\C)
\end{equation}

Assume now that the $t$-structure is left complete. In this case we have an equivalence 
$\C_{\geq 0}\simeq lim_n (\C_{\geq 0}\cap \C_{\leq n})$. In fact this equivalence in $\iCat$ lifts to an equivalence in $CAlg(\iCat)$ through the forgetful functor $CAlg(\iCat)\to \iCat$. Indeed, the functors $\tau_{\leq n}:\C_{\geq 0}\to \C_{\geq 0}\cap \C_{\leq n}$ are monoidal and limits in $CAlg(\iCat)$ are computed in $\iCat$ by means of the same forgetful map. In particular, since the forgetful map $CAlg(\iCat)\subseteq Op_{\infty}$ has a left adjoint (see \ref{monoidalenvelope}), it commutes with limits so that $\C_{\geq 0}^{\otimes}$ is the limit of $(\C_{\geq 0}\cap \C_{\leq n})^{\otimes}$ in $Op_{\infty}$. In particular, for any $\infty$-operad $\Opmonoidal$, we have an equivalence

\begin{equation}
\label{chule}
Alg_{\Op}(\C)^{cn}\simeq lim_n Alg_{\Op}(\C_{\geq 0}\cap \C_{\leq n})
\end{equation}

If we assume that $\C$ is presentable, then $Alg_{\Op}(\C)$ will also be presentable and in particular the subcategory of $n$-truncated objects $\tau_{\leq n}Alg_{\Op}(\C)$ is a reflexive localization of $Alg_{\Op}(\C)$. Moreover, since the truncation functor given by the $t$-structure is monoidal, it exhibits $Alg_{\Op}(\C_{\geq 0}\cap \C_{\leq n})$ also a reflexive localization of $Alg_{\Op}(\C)$ (see \ref{reflexivelocalizationalgebras}) so that the two subcategories are equivalent. Together with the equivalence (\ref{chule}) this implies that Postnikov towers converge in $Alg_{\Op}(\C)^{cn}$.\\

Again, an important example is the $(\infty,1)$-category of spectra $\Spmonoidal$ \cite[Lemma 8.1.1.7]{lurie-ha}. More generally, for any connective $\mathbb{E}_{k+1}$-algebra $R$ in $\Sp$, the category of left modules $LMor_R(\Sp)$ inherits a natural left-complete $t$-structure \cite[8.1.1.10,8.1.1.13]{lurie-ha} together with a compatible $\mathbb{E}_k$-monoidal structure \cite[8.1.2.5,8.1.3.15]{lurie-ha}.

\subsection{From symmetric monoidal model categories to symmetric monoidal $(\infty,1)$-categories}
\label{link2}

\subsubsection{The (monoidal) link}
\label{monoidalink}

The link described in the Section \ref{section1-2} can now be extend to the world of monoidal structures. Recall that a model category $\M$ equipped with a monoidal structure $\otimes$ is said to be a \emph{monoidal model category} if  the monoidal structure is closed, the tensor functor is a left-Quillen bifunctor and the unit of the monoidal structure is a cofibrant object in $\M$. The main idea is that\\

\textit{Every \emph{symmetric monoidal model category} "presents" a symmetric monoidal $(\infty,1)$-category}.\\

Following the Example $4.1.3.6$ of \cite{lurie-ha}, if $\M$ is a symmetric monoidal model category (see Definition 4.2.6 of \cite{hovey-modelcategories}) then the underlying $\infty$-category of $\M$ inherits a canonical symmetric monoidal structure which we denote here as $N(\M)[W^{-1}]^{\otimes}\to N(\Fin)$. It can be obtained as follows: first recall that in a symmetric monoidal model category, the product of cofibrant objects is again cofibrant and by the Ken Brown's Lemma, the product of weak-equivalences between cofibrant objects is again a weak-equivalence. This implies that the full subcategory of cofibrant objects in $\M$ inherits a monoidal structure and we can regard it as a simplicial coloured operad $(\M^c)^{\otimes}$ enriched over constant simplicial sets. This way,
its operadic nerve $N^{\otimes}((\M^c)^{\otimes})\to N(\Fin)$ provides a trivial $\infty$-operad which furthermore is a symmetric monoidal $(\infty,1)$-category with underlying $\infty$-category equivalent to $N(M^c)$ (see the Example \ref{classicalsymmetricmonoidal}). Since the restriction of the monoidal structure to the cofibrant objects preserves weak-equivalences, we can understand the pair $(N^{\otimes}((\M^c)^{\otimes}),W)$ as an object in $CAlg(\mathcal{W}\iCat)$ and we define the underlying symmetric monoidal $(\infty,1)$-category of $\M$ as the monoidal localization (see \ref{remarkcat})

\begin{equation}N(\M)[W^{-1}]^{\otimes}:= N^{\otimes}((\M^c)^{\otimes})[W_c^{-1}]^{\otimes}\end{equation} 

It follows from the definitions that its underlying $\infty$-category is canonically equivalent to the underlying $\infty$-category of $\M$. Moreover, it comes canonically equipped with a universal monoidal functor $N^{\otimes}((\M^c)^{\otimes})\to N(\M)[W^{-1}]^{\otimes}$.

 At the same time, if $\M$ comes equipped with a \emph{compatible simplicial enrichment}, then $\M^\circ$, although not a simplicial monoidal category (because the product of fibrant objects is not fibrant in general), can be seen as the underlying category of a simplicial coloured operad $(\M^{\circ})^{\otimes}$ where the colours are the cofibrant-fibrant objects in $\M$ and the operation space is given by 
 
\begin{equation}
Map_{(\M^{\circ})^{\otimes}}(\{X_i\}_{i\in I}, Y):= Map_{\M}(\bigotimes_{i} X_i, Y)
\end{equation}

\noindent which is a Kan-complex because $Y$ is fibrant and the product of cofibrant objects is cofibrant. With this, we consider the $\infty$-operad given by the operadic nerve $N^{\otimes}((\M^{\circ})^{\otimes})$. By the Proposition $4.1.3.10$ of \cite{lurie-ha}, this $\infty$-operad is a symmetric monoidal $(\infty,1)$-category and the product of cofibrant-fibrant objects $X, Y$ is given by the choice of a trivial cofibration $X\otimes Y\to Z$ providing a fibrant replacement for the product in $\M$. Moreover, the  Corollary 4.1.3.16 of \cite{lurie-ha} provides an $\infty$-symmetric-monoidal-generalization of the Proposition \ref{dwyer-kan}: The symmetric monoidal $(\infty,1)$-category $N(M)[W^{-1}]^{\otimes}$ is monoidal equivalent to $N^{\otimes}((\M^{\circ})^{\otimes})$.

A particular instance of this is when $\M$ is a cartesian closed combinatorial simplicial model category with a cofibrant final object. In this case, it is a symmetric monoidal model category with respect to the product and we can consider its operadic nerve $N^{\otimes}((\M^{\circ})^{\times})$. From the Example 2.4.1.10 of \cite{lurie-ha}, this is equivalent to a Cartesian structure in the underlying $\infty$-category of $\M$ -  $N_{\Delta}(\M^{\circ})^{\times}$.\\

A monoidal left-Quillen map (\cite{hovey-modelcategories}-Def. 4.2.16) between monoidal model categories
induces a monoidal functor between the underlying symmetric monoidal $(\infty,1)$-categories. This is because the monoidal localization was constructed as a functor $CAlg(\mathcal{W}\iCat)\to CAlg(\iCat)$. In the simplicial case we can provide a more explicit construction:

\begin{construction}
Let $\M\to \N$ be a monoidal left-Quillen functor between two combinatorial simplicial symmetric monoidal model categories. Let $G$ be its right adjoint. We construct a monoidal map between the associated operadic nerves

\begin{equation}
\xymatrix{
N^{\otimes}((\M^{\circ})^{\otimes})\ar[rd]\ar[rr]^{F^{\otimes}}&& N^{\otimes}((\N^{\circ})^{\otimes})\ar[dl]\\
&N(\Fin)&
}
\end{equation}

For that, we consider the simplicial category $\A$ whose objects are the triples $(i, \nfin, (X_1,...,X_n))$ with $i\in\{0,1\}$, $\nfin\in \Fin$ and $X_1$,..., $X_n$ are objects in $\M$ if $i=0$ and in $\N$ if $i=1$. The mapping spaces

\begin{equation}
Map_{\A}((0, \nfin, (X_1,...,X_n)),(j, \mfin, (Y_1,...,Y_n)))
\end{equation}

\noindent are defined to be the mapping spaces in $\tilde{M}$\footnote{consult the notation in the Construction \ref{monoidaltilde}} (resp. $\tilde{N})$  if $i,j=0$ (resp. $i,j=1$). If $i=1$ and $j=0$, we declare it to be empty and finally, if $i=0$ and $j=1$, we defined it by the formula

\begin{equation}
Map_{\A}((0, \nfin, (X_1,...,X_n)),(j, \mfin, (Y_1,...,Y_n))):=Map_{\tilde{N}}(( \nfin, (F(X)_1,...,F(X)_n)),(\mfin, (Y_1,...,Y_n)))
\end{equation}

\noindent which by the adjunction $(F,G)$ and the fact that $F$ is strictly monoidal, are the same as 

\begin{equation}
Map_{\tilde{M}}(( \nfin, (X_1,...,X_n)),(\mfin, (G(Y)_1,...,G(Y)_n)))
\end{equation}

The composition is the obvious one induced from $\M$ and $\N$.  We consider the full simplicial subcategory $\A^{\circ}$ spanned by the objects $(i, \nfin, (X_1,...,X_n))$ where each $X_i$ is cofibrant-fibrant (respectively in $\M$ or $\N$ depending on the value of $i$). It follows that $\A^{\circ}$ is enriched over Kan-complexes (because $\M$ and $\N$ are simplicial model categories) and therefore its simplicial nerve is an $(\infty,1)$-category. Moreover, it admits a canonical projection $p:N_{\Delta}(\A^{\circ})\to N(\Fin)\times \Delta[1]$ whose fibers

\begin{equation}
\{0\}\times_{N(\Fin)\times \Delta[1]} N_{\Delta}(\A^{\circ})\simeq N^{\otimes}((\M^{\circ})^{\otimes})
\end{equation}
\noindent and 
\begin{equation}
\{1\}\times_{N(\Fin)\times \Delta[1]} N_{\Delta}(\A^{\circ})\simeq N^{\otimes}((\N^{\circ})^{\otimes})
\end{equation}

\noindent recover the operadic nerves of $\M$ and $\N$, respectively.
\end{construction}

\begin{prop}
\label{inducedmonoidafunctor}
The projection $p:N_{\Delta}(\A^{\circ})\to N(\Fin)\times \Delta[1]$ is a cocartesian fibration of $(\infty,1)$-categories. 
\begin{proof}
We follow the arguments in the proof of the Proposition 4.1.3.15 in \cite{lurie-ha}. We have to prove that for any edge $u:(i,\nfin)\to (j,\mfin)$ in $N(\Fin)\times \Delta[1]$ and any object $C:=(i, \nfin, (X_1,..., X_n))$ over $(i,\nfin)$, there is a cocartesian lift $\tilde{u}$ of $u$ starting at $C$. Notice that any such morphism $u$ is consists of a pair $(i\to i', f:\nfin\to \mfin)$ with $i\to i'$ an edge in $\Delta[1]$ and $f$ a morphism in $\Fin$.

Since we already know that both fibers  $N^{\otimes}((\M^{\circ})^{\otimes})$ and $N^{\otimes}((\N^{\circ})^{\otimes})$ are symmetric monoidal $(\infty,1)$-categories, our task is reduced to the case $u:(i=0, \nfin)\to (j=1, \mfin)$ which is determined by the second componenent $f:\nfin\to \mfin$.

Given an object $C:=(0, \nfin, (X_1,..., X_n))$ over $(0,\nfin)$ we have to find a new object $C':=(1, \mfin, (\tilde{X}_1,..., \tilde{X}_m))$ together with a cocartesian morphism in $N_{\Delta}(\A^{\circ})$

\begin{equation}\tilde{u}:C\to C'\end{equation}

\noindent defined over $u$. Recall that by definition, the connected component of the mapping space 

\begin{equation}Map_{N_{\Delta}(\A^{\circ})}((i,\nfin,(X_0,..., X_n)),(j,\mfin,(Y_1,..., Y_m)))\end{equation} 

\noindent spanned by the maps which are defined over $u$ was defined to be the mapping space 

\begin{equation}\prod_{i\in \mfin}Map_{\N^{\circ}}(\otimes_{\alpha\in f^{-1}\{i\}} F(X_{\alpha}), \tilde{Y}_i)\end{equation}

With this in mind, we define $\tilde{X}_i$ to be a fibrant replacement for the product

\begin{equation}
u_i:\otimes_{\alpha\in f^{-1}\{i\}} F(X_{\alpha})\to \tilde{X}_i
\end{equation}

\noindent where each $u_i$ is the trivial cofibration that comes out from the device of the model structure providing the functorial factorizations. Finally, we take $\tilde{u}$ to be the point in $Map_{N_{\Delta}(\A^{\circ})}((0,(X_0,..., X_1)),(1,(\tilde{X}_1,..., \tilde{X}_m)))$ corresponding the product of the trivial cofibrations $u_i$. Notice that each $\tilde{X}_i$ is cofibrant-fibrant in $\N$ because the product of cofibrants is cofibrant. We are now reduced to the task of proving that $\tilde{u}$ is a cocartesian morphism. In our situation, this is equivalent to say that
for any morphism $u'=(id_1, f'):(1,\nfin)\to (1,\langle k \rangle)$ in $\Delta[1]\times N(\Fin)$ and any object $C'':=(1,\langle k \rangle, (Z_1, ..., Z_k))$ over $(1,\langle k \rangle)$, the composition map with $\tilde{u}$

\begin{equation}
Map^{u'}_{N_{\Delta}(\A^{\circ})}((1,\nfin,(\tilde{X}_0,..., \tilde{X}_n)),(1,\langle k \rangle, (Z_1, ..., Z_k)))\to 
Map^{u'\circ u}_{N_{\Delta}(\A^{\circ})}((0,\nfin,(X_0,..., X_1)),(1,\langle k \rangle, (Z_1, ..., Z_k)))
\end{equation}

\noindent is a weak-equivalence of simplicial sets (here we denote by $Map^{u'}_{N_{\Delta}(\A^{\circ})}(-,-)$ the directed component of $Map_{N_{\Delta}(\A^{\circ})}(-,-)$ of those maps which are defined over $u'$).

It is immediate from the definitions that the previous map is no more than the map

\begin{eqnarray}
\prod_{j\in (1,\langle k \rangle)}Map_{N_{\Delta}(\N^{\circ})}(\otimes_{i\in (f')^{-1}(\{j\})} \tilde{X}_i, Z_j)\to  \prod_{j\in (1,\langle k \rangle)}Map_{N_{\Delta}(\N^{\circ})}(\otimes_{\sigma\in (f'\circ f)^{-1}(\{j\})} F(X_{\sigma}), Z_j)\simeq
\\
\simeq \prod_{j\in (1,\langle k \rangle)}Map_{N_{\Delta}(\N^{\circ})}(\otimes_{i\in (f')^{-1}(\{j\})}(\otimes_{\alpha\in (f)^{-1}(\{i\})} F(X_{\alpha})), Z_j)
\end{eqnarray}

\noindent where the last isomorphism follows from the natural identification of the two products $\otimes_{\sigma\in (f'\circ f)^{-1}(\{j\})} F(X_{\sigma})$ and $\otimes_{i\in (f')^{-1}(\{j\})}(\otimes_{\alpha\in (f)^{-1}(\{i\})} F(X_{\alpha}))$. Finally, we can see that this previous map is the one obtained by the product of the pos-composition with the trivial cofibrations $u_i$. Since the monoidal structure is given by a Quillen bifunctor, the product of trivial cofibrations is a trivial cofibration and therefore the map between the mapping spaces is a trivial fibration and so a weak-equivalence. To conclude, the product of trivial fibrations is always a trivial fibration.
\end{proof}
\end{prop}

Finally, we can now extract the monoidal functor $F^{\otimes}$ using the Proposition 5.2.1.4 of \cite{lurie-htt}. It is also clear from the proof of the Proposition \ref{inducedmonoidafunctor} that the underlying functor of $F^{\otimes}$ is the map $\bar{F}$ described in the Proposition 5.2.4.6 in \cite{lurie-htt} which we can identify with the composition of $F$ with a fibrant replacement functor in $\N$.

\subsubsection{Strictification of Algebras and Modules}
\label{strictificationalgebras}

In some very specific cases the theory of algebras can be performed directly within the setting of model categories. In other words, it admits a strictification. An important result of \cite{schwede-shipley-algebrasandmodulesinmonoidalmodelcategories} (Theorem 4.1) is that if $\M$ is a combinatorial monoidal model category satisfying the \emph{monoid axiom} (Definition 3.3 of \cite{schwede-shipley-algebrasandmodulesinmonoidalmodelcategories}), then the category $Alg(\M)$ of strict associative algebra objects in $\M$ admits a new combinatorial model structure where:

\begin{itemize}
\item a map in $Alg(\M)$ is a weak-equivalence if and only if it is a weak-equivalence in $\M$;
\item a map in $Alg(\M)$ is a fibration if and only if it is a weak-equivalence in $\M$;
\item the forgetful functor $Alg(\M)\to \M$ is a right-Quillen map that preserves cofibrant objects.
\item this model structure in $Alg(\M)$ is simplicial if the model structure in $\M$ is.
\end{itemize}

Using this results, we can create a comparison map between the underlying $(\infty,1)$-category of $Alg(\M)$ and the $(\infty,1)$-category of algebra-objects in the underlying $(\infty,1)$-category of $\M$. More precisely, using the fact the forgetful functor $Alg(\M)\to \M$ preserves cofibrant objects, we have natural inclusions $Alg(\M)^c\subseteq Alg(\M^c)\subseteq Alg(\M)$ which preserve weak-equivalences. Passing to the localizations (in the sense of \ref{locinfinity}), the cofibrant-replacement functor produces equivalences $N(Alg(\M)^c)[W^{-1}_{c_{alg}}]\simeq N(Alg(\M^c))[W^{-1}_{c}]\simeq N(Alg(\M))[W^{-1}]$ where $W_{c_{alg}}$ denotes the class of weak-equivalences between cofibrant algebras and $W_{c}$ is the class of weak-equivalences between algebras whose underlying objects in $\M$ are cofibrant. Finally, using the fact that the localization map
$N(\M^c)\to N(\M^c)[W_c^{-1}]$ is monoidal, it provides a map $Alg(N(\M^c))\to Alg(N(\M^c)[W_c^{-1}])$ which sends weak-equivalences in $\M$ between cofibrant objects to equivalences. The universal property of the localization provides a canonical map

\begin{equation}
\label{last2}
\xymatrix{
N(Alg(\M)^c)[W^{-1}_{c_{alg}}]\simeq Alg(N(\M^c))[W_c^{-1}]\ar@{-->}[r]& Alg(N(\M^c)[W_c^{-1}])\\
Alg(N(\M^c))\ar[u] \ar[ur]&
}
\end{equation}

\noindent rendering the diagram homotopy commutative. By the Theorem $4.1.4.4$ of \cite{lurie-ha}, if $\M$ is a combinatorial monoidal model category and either $(a)$ all objects are cofibrant or $(b)$ $\M$ is left-proper, the cofibrations are generated by the cofibrations between cofibrant objects and $\M$ is symmetric and satisfies the monoid axiom, then, this canonical map is an equivalence of $(\infty,1)$-categories. In the next section we will see this result applied to the theory of differential graded algebras.\\

\begin{remark}
\label{last3}
This strictification result can be extended to a monoidal functor. More precisely, recall from \ref{tensorproductalgebras} that the category of algebras inherits a monoidal structure induced from the one in the base monoidal category. As in the Remark \ref{last1}, the functor $Alg(N(\M^c))\to Alg(N(\M^c)[W_c^{-1}])$ extends to a monoidal functor and using the monoidal localization of \ref{remarkcat} we can also promote the map in (\ref{last2}) to a monoidal functor.
\end{remark}

There is also a strictification result for bimodules over associative algebras. Given two strictly associative algebra objects $A$ and $B$ in a combinatorial monoidal model category $\M$, we can set a model structure in the classical category of bimodules in $\M$, $BiMod(A,B)(\M)$ , for which the weak-equivalences $W_{Mod}$ are given by the weak-equivalences of $\M$ \cite[Prop. 4.3.3.15]{lurie-ha} and by the Theorem 4.3.3.17 of \cite{lurie-ha} we have

\begin{equation}
N(BiMod(A,B)(\M))[W_{Mod}^{-1}]\simeq _ABMod_B(N(\M)[W^{-1}])
\end{equation}

A similar result holds for commutative algebras (Thm 4.4.4.7 of \cite{lurie-ha}) whenever the strict theory admits an appropriate model structure (as in the Prop. 4.4.4.6 of \cite{lurie-ha}). In particular, it works also for modules over commutative algebras.\\

\begin{remark}
In the general situation, there are no model structures for algebras or modules. This is exactly one of the main motivations to develop a theory of algebras and modules within the more fundamental setting of $(\infty,1)$-categories. The theory of motives is one of those important cases where model category theory does not work.
\end{remark}

\begin{remark}
Recall that an $(\infty,1)$-category is presentable iff there exists a combinatorial simplicial model category $\M$ such that $\C$ is the underlying $\infty$-category of $\M$ (which means, $\C\simeq  N_{\Delta}(\M^{\circ})$) (see Prop. A.3.7.6 of \cite{lurie-htt}). There is a similar statement for presentable monoidal $(\infty,1)$-categories, replacing the simplicial nerve by the operadic nerve (see \cite[4.1.4.9]{lurie-ha} for a sketch of proof).
\end{remark}

\subsection{Higher Algebra over a classical commutative ring $k$}
\label{complexes}

The discussion in this section will be important in the last part of our paper. Let $k$ be a (small) commutative ring and denote by $Mod(k)$ the ordinary category of small sets endowed with the structure of module over $k$. We will write $Ch(k)$ to denote the big category of (unbounded) chain complexes of small $k$-modules. This is a Grothendieck abelian category. Recall also the existence of a symmetric tensor product of complexes given by the formula
 
\begin{equation}
(E\otimes E')_n:= \bigoplus_{i+j=n} (E_i\otimes_k E_j)
\end{equation}

\noindent where $\otimes_k$ denotes the tensor product of $k
$-modules. This monoidal structure is closed, with internal-hom $\underline{Hom}_{Ch(k)}(E,E')$ given by the formula

\begin{equation}
\underline{Hom}_{Ch(k)}(E,E')_n:= \prod_i Hom_k(E_i, E_{i+n})
\end{equation}

\noindent where the differential $d_n: \underline{Hom}_{Ch(k)}(E,E')_n\to \underline{Hom}_{Ch(k)}(E,E')_{n+1}$ sends a family $\{f_i\}$ to the family $\{d\circ f_i - (-1)^nf_{i+1}\}$.\\

The category $Ch(k)$ carries a left proper combinatorial model structure \cite[Theorem 2.3.11]{hovey-modelcategories} where the weak-equivalences are the quasi-isomorphisms of complexes, the fibrations are the surjections (and so every object is fibrant). We will call it the projective model structure on complexes. The cofibrant complexes (see the Lemma 2.3.6 and the Remark 2.3.7 of \cite{hovey-modelcategories}) are the \emph{DG-projective complexes}. In particular, every cofibrant complex is a complex of projective modules (and therefore flat) and any bounded below complexes of projective $k$-modules is cofibrant. Moreover, by the Proposition 4.2.13 of loc.cit, this model structure is compatible with the tensor product of complexes and so $Ch(k)$ is a closed symmetric monoidal model category. Following \ref{monoidalink}, the proper way to encode the study of complexes of $k$-modules up to quasi-equivalences is the underlying $(\infty,1)$-category $\mathcal{D}(k)$ of the model category $Ch(k)$. This is a particular case of the Example \ref{derivedinfinitycategoryofascheme} with $X=Spec(k)$. In particular, $\derivedk$ is stable with a compact generator $k$ and 
with compact objects the perfect complexes. Moreover, because the model structure is compatible with the tensor product of complexes, $\derivedk$ acquires a symmetric monoidal structure $\derivedk^{\otimes}$ (as explained in \ref{monoidalink}).\\

\begin{remark}
This method to obtain $\derivedk$ is not the one described in \ref{notations}. This is because the projective model structure does not agree with the injective one. However, since the weak-equivalences are the same, the resulting $(\infty,1)$-categories obtained by localization are equivalent.
\end{remark}

We now review the theory of algebra objects over $k$. By definition, a strict \emph{dg-algebra} over $k$ is a strictly associative algebra-object in $Ch(k)$ with respect to the tensor product of complexes. We will denote the category of dg-algebras as $Alg(Ch(k))$. As explained in the Example \ref{classicalalgebras}, the nerve $N(Alg(Ch(k)))$ is equivalent to $Alg(N(Ch(k))$ so that the notations are coherent. Thanks to \cite[Thm 4.1]{schwede-shipley-algebrasandmodulesinmonoidalmodelcategories} the model structure in $Ch(k)$ extends to a model structure in $Alg(Ch(k))$  with fibrations and weak-equivalences given by the underlying fibrations and quasi-isomorphisms of complexes\footnote{This model structure is left proper if $k$ is a field}. This model structure satisfies the condition $(b)$ of the previous section (see \cite[8.1.4.5]{lurie-ha}). In this case, denoting its underlying $(\infty,1)$-category by $N(Alg(Ch(k))^{c})[W_{c}^{-1}]$, the strictification result provides an equivalence

\begin{equation}
\label{strictificationdga}
\xymatrix{N(Alg(Ch(k))^c)[W_{c}^{-1}]\ar[r]^>>{\sim}& Alg(\derivedk)}
\end{equation}.

\begin{remark}
\label{modelstructurecdga}
The situation for commutative algebras is not so satisfatory. In general the model structure on complexes does not extend to the category of strictly commutative algebra objects in $Ch(k)$. However, if $k$ contains the field of rational numbers $\mathbb{Q}$, the model structure extends \cite[Prop. 8.1.4.11]{lurie-ha} and the strictification result holds \cite[8.1.4.7]{lurie-ha}. Writing $CDGA_k$ to denote its the underlying $(\infty,1)$-category, the canonical map given by the universal property of the localization
\begin{equation}
CDGA_k\to CAlg(\derivedk)
\end{equation}
is an equivalence.\\

\end{remark}

The $(\infty,1)$-category $\derivedk$ carries a natural right-complete $t$-structure where $\derivedk_{\geq 0}$ is the full subcategory spanned by the complexes with zero homology in negative degree. Its heart is the category of modules over $k$ and the functor $\mathbb{H}_n:\C\to \C^{\heartsuit}$ corresponds to the classical $nth$-homology functor $H_n$. This $t$-structure is also left-complete. Indeed, this follows because $k$ is a generator in $\derivedk$ and using the the formula $H_i(X)\simeq \pi_i(Map_{\derivedk}(k,X)), \forall i\geq 0$ we see that if all the homology groups of an object $X$ are zero so is $X$. Moreover, the monoidal structure in $\derivedk$ is compatible with this $t$-struture (it follows directly from the general Kunneth formula for complexes, or alternatively, using the same methods as in \cite[8.1.1.7]{lurie-ha}). Following the discusion in \ref{monoidalt}, the left-completness implies that for any $\infty$-operad $\Opmonoidal$, we have $\tau_{\leq n}Alg_{\Op}(\derivedk)^{cn}\simeq Alg_{\Op}(\derivedk_{\geq 0}\cap \derivedk_{\leq n})$ and that Postnikov towers converge

\begin{equation}
Alg_{\Op}(\derivedk)^{cn}\simeq lim_{n} Alg_{\Op}(\derivedk_{\geq 0}\cap \derivedk_{\leq n})
\end{equation}

In particular, the heart $\derivedk^{\heartsuit}=\derivedk_{\geq 0}\cap \derivedk_{\leq 0}$ inherits a symmetric monoidal structure which we can identify with the classical tensor product of $k$-modules. In the associative (resp. commutative) case the category of algebras $\tau_{\leq 0}Alg(\derivedk)^{cn}$ (resp. $\tau_{\leq 0}CAlg_{\Ass}(\derivedk)^{cn}$) can be identified with the nerve of the classical category of associative (resp. commutative $k$-algebras). Moreover, since the map $\mathbb{H}_0:\C_{\geq 0}\to \C^{\heartsuit}$ is monoidal, it extends to a map of algebras $\mathbb{H}_0:Alg_{\Op}(\C)^{cn}\to Alg(\C^{\heartsuit})$ so that if $A$ is a connective associative (resp. commutative) algebra object in $\derivedk$, $\mathbb{H}_0(A)$ is an associative (resp. commutative) algebra in the classical sense.\\

As in the non-connective case, the theory of connective algebras admits a strictification result. More precisely, $Alg(\derivedk)^{cn}$ is equivalent to the underlying $(\infty,1)$-category $SR_k$ of a simplicial model structure in the category of simplicial associative algebras over $k$, where the with equivalences are the maps of simplicial algebras inducing a weak-equivalence between the underlying simplicial sets \cite[8.1.4.18]{lurie-ha}. 

\begin{remark}
As in \ref{modelstructurecdga}, if $k$ contains the field of rational numbers, $CAlg(\derivedk)^{cn}$ is equivalent to the underlying $(\infty,1)$-category $SCR_k$ of a simplicial model structure in the category of simplicial commutative $k$-algebras, with weak-equivalences given by the weak-equivalences between the underlying simplicial sets \cite[8.1.4.20]{lurie-ha}. In fact, the model structure for simplicial commutative algebras exists for any ring $k$ and it can be proved that $SCR_k$ is equivalent to the completion of the ordinary category of commutative $k$-algebras of the form $k[X_1,..., X_n]$, $n\geq 0$, under sifted colimits \cite[4.1.9]{lurie-structuredspaces}.
\end{remark}

\begin{remark}
The study of higher algebra over a commutative ring $k$ can be understood as a small part of the much vaster subject of higher algebra in the $(\infty,1)$-category of spectra $\Sp$. Indeed, we can understand a commutative ring $k$ as $0$-truncated connective commutative algebra object in $\Spmonoidal$ and using the same ideas as in \cite{schwede-shipley-modules} we can deduce an equivalence of stable presentable symmetric monoidal $(\infty,1)$-categories $Mod_k(Sp)^{\otimes}\simeq \derivedk^{\otimes}$ defined by sending a complex $E$ to the mapping spectrum subjacent to $Map_{\derivedk}(k,E)$ (see \cite[8.1.2.6, 8.1.2.7,8.1.2.13]{lurie-ha}). Moreover, the category of modules $Mod_{k}(\Sp)$ inherits a left-complete $t$-structure induced from the one in $\Sp$ (see \cite[8.1.1.13]{lurie-ha}) and we can easily check that the formula $E\mapsto Map_{\derivedk}(k,E)$ is compatible with the $t$-structures. In particular, this implies that for any $\infty$-operad $\Opmonoidal$, we have an equivalences $Alg_{\Op}(\derivedk)\simeq Alg_{\Op}(\Sp)_{k/}$ and $Alg_{\Op}(\derivedk)^{cn}\simeq Alg_{\Op}(Sp)^{cn}_{k/}$.
\end{remark}

\subsection{Cotangent Complexes and Square-Zero Extensions}
\label{cotangent}
In the last part of the paper we construct a functor $L_{pe}$ connecting the classical theory of theory of schemes to the noncommutative world. One of the steps (see Prop. \ref{commutativesmoothimpliesfinitetype}) requires the following noncommutative analogue of \cite[Prop. 2.2.2.4]{toen-vezzosi-hag2} and \cite[8.4.3.18]{lurie-ha}:

\begin{lemma}
\label{cotangentcompact}
Let $A$ be an object in $Alg(\derivedk)^{cn}$. The following are equivalent:

\begin{enumerate}[1)]
\item $A$ is a compact in $Alg(\derivedk)$;
\item $\mathbb{H}_0(A)$ is a finitely presented associative algebra over $k$ and the cotangent complex $\mathbb{L}_{A}$ is a compact object in $Mod_A^{\Ass}(\derivedk)$;
\end{enumerate}
\end{lemma}

In order to prove this we need to say what is the cotangent complex of a connective dg-algebra. This is a particular instance of a more general notion. Following \cite{francis-thesis} we recall how to define the \emph{cotangent complex} of an $\Op$-algebra-object in a stable symmetric monoidal $(\infty,1)$-category $\Cmonoidal$ compatible with colimits.\\

Let $\Cmonoidal$ be a stable symmetric monoidal $(\infty,1)$-category compatible with colimits. Let $\Opmonoidal$ be a $\kappa$-small coherent $\infty$-operad and let $A\in Alg_{\Op}(\C)$ be an algebra-object in $\C$. Given a module-object $M\in Mod^{\Op}_A(\C)$ and using the hypothesis that the monoidal structure is compatible with colimits we can guess that the direct sum $A\oplus M$ comes naturally equipped with the structure an $\Op$-algebra-object in $\C$ where the multiplication is determined by

\begin{equation}
(A\oplus M)\otimes (A\oplus M)\simeq (A\otimes A)\oplus (A\otimes M)\oplus (A\otimes M)\oplus (M\otimes M)\to A\oplus M
\end{equation}

\noindent where in the last step we use the multiplication $A\otimes A\to A$, the module action $A\otimes M\to M$ and the zero map $M\otimes M\to M$. This new $\Op$-algebra-object comes naturally equipped with a morphism of $\Op$-algebras $A\oplus M\to A$ which we can informally describe via the formula $(a,m)\to a$. Its fiber can be naturally identified with the module $M$. This construction should give rise to a functor

\begin{equation}\label{pretaa}
Mod_A^{\Op}(\C)\to Alg_{\Op}(\C)_{./A}
\end{equation}

In \cite{francis-thesis}-Theorem 3.4.2 the author provides a precise way to perform this construction, proving that for any stable symmetric monoidal $(\infty,1)$-category $\Cmonoidal$ compatible with colimits and any coherent $\infty$-operad $\Opmonoidal$, there is a canonical equivalence

\begin{equation}
Stab(Alg_{\Op}(\C)_{./A})\simeq Fun_{\Op}(\Op,Mod_{A}^{\Op}(\C))
\end{equation}

\noindent for any $\Op$-algebra $A$ in $\C$ (see also \cite[8.3.4.13]{lurie-ha}). In particular, if the operad only has one color, we have an equivalence between the category of modules and the stabilization. Also in this case, this equivalence recovers the functor in (\ref{pretaa}) as the delooping functor $\Omega^{\infty}$ (See Section \ref{usualspectra} for an explanation of the notations).\\

By definition, a derivation of $A$ into $M$ is the data of a morphism of $\Op$-algebras $A\to A\oplus M$ over $A$. It is an easy exercice to see that this notion recovers the classical definition using the Leibniz rule. We set $Der(A,M):= Map_{Alg_{\Op}(\C)_{./A}}(A, A\oplus M)$ to denote the space of derivations with values in $M$. The formula $M\mapsto Der(A,M)$ provides a functor $(Mod_{A}^{\Op}(\C))^{op}\to \Spaces$ which, through the Grothendieck construction, corresponds to a left fibration over $Mod_{A}^{\Op}(\C)$. By definition, the (absolute) \emph{ cotangent complex of $A$} is an object $\mathbb{L}_{A}\in Mod_{A}^{\Op}(\C)$ which makes this left fibration representable. In other words, if it has the universal property

\begin{equation}Map_{Mod_{A}^{\Op}(\C)}(\mathbb{L}_{A}, M)\simeq Map_{Alg_{\Op}(\C)_{./A}}(A, A\oplus M)\end{equation}

\noindent which allows us to understand the formula $A\mapsto \mathbb{L}_{A}$ as a left adjoint $L_A$ to the functor in (\ref{pretaa}), evaluated in $A$. In particular, if $\C$ is presentable this left adjoint exists because of the adjoint functor theorem together with the fact that (\ref{pretaa}) commutes with limits \cite[Lemma 3.1.3]{francis-thesis}. Moreover, under the equivalence between modules and the stabilization of algebras, $L_{A}$ can be identified with the suspension functor $\Sigma^{\infty}$.\\

\begin{example}
\label{noncommutativecotangentcomplex}
When applied to the example $\Cmonoidal=\derivedk^{\otimes}$ and for $\mathbb{E}_1\simeq \Ass$, this definition recovers the classical associative cotangent complex introduced by Quillen and studied in \cite{lazarev}, given by the kernel of the multiplication map $A\otimes_k A^{op}\to A$ in the $(\infty,1)$-category $ Mod_A^{\Ass}(\derivedk)$. Recall also that $Mod_A^{\Ass}(\derivedk)$ is equivalent to $_ABMod_A(\derivedk)$ which by the discussion in \ref{strictificationalgebras} is equivalent to the underlying $(\infty,1)$-category of the model category of strict $A$-bimodules in the model category of complexes $Ch(k)$. This example will play an important role in the last section of this paper.
\end{example}

\begin{remark}
\label{basechangecotangent}
The construction of cotangent complexes is well-behaved with respect to base-change. If $f:A\to A'$ is a morphism of $\Op$-algebras we can put together the functors $A\oplus-$ and $A'\oplus-$ in a diagram 

\begin{equation}
\xymatrix{
Mod_A^{\Op}(\C)\ar[r]^{A\oplus -}& Alg_{\Op}(\C)_{./A} \\
Mod_{A'}^{\Op}(\C)\ar[u]^{For}\ar[r]^{A'\oplus -}& Alg_{\Op}(\C)_{./A'}\ar[u]_{(-\times_{A'} A)}
}
\end{equation}

\noindent where $For$ is the map that considers an $A'$-module as an $A$-modules via $f$ and the map $(-\times_A' A)$ is obtained by computing the fiber product of a morphism $C\to A'$ with respect to $f$. The fact that this diagram commutes follows from the equivalence relating modules and the stabilization of algebras and from the definition of \emph{tangent bundle} studied in \cite[Section 8.3.1]{lurie-ha}. Moreover, the commutativity of this diagram implies the commutativity of the diagram associated to the left adjoints

\begin{equation}
\xymatrix{
Mod_A^{\Op}(\C)\ar[d]^{A'\otimes_A -}&\ar[l]^{L_A}\ar[d]^{f\circ -} Alg_{\Op}(\C)_{./A} \\
Mod_{A'}^{\Op}(\C)&\ar[l]^{L_{A'}} Alg_{\Op}(\C)_{./A'}
}
\end{equation}

\noindent where now $A'\otimes_A -$ is the base change with respect to $f$ and the $(f\circ-)$ is the map obtained by composing with $f$. In particular, we find that $A'\otimes_A \mathbb{L}_{A}$ is equivalent to $L_{A'}$ evaluated at $f:A\to A'$.
\end{remark}

\begin{remark}
The notion of cotangent complex has a relative version. For any $\Op$-algebra $R$, the $(\infty,1)$-category $Mod_R^{\Op}(\C)$ is again a stable symmetric monoidal $(\infty,1)$-category compatible with colimits. In particular, under the equivalence $Alg_{\Op}(Mod_R^{\Op}(\C))\simeq Alg_{\Op}(\C)_{R/.}$, for any $R$-algebra $f:R\to A$ the previous discussion provides a functor

\begin{equation}
\label{functorplus}
Mod_{A}^{\Op}(\C)\simeq Mod_A^{\Op}(Mod_R^{\Op}(\C))\to Alg_{\Op}(Mod_R^{\Op}(\C))_{./A}\simeq (Alg_{\Op}(\C)_{R/.})_{./A}
\end{equation}

\noindent sending a $A$-module $M$ to the $R$-algebra $A\oplus M$ defined over $A$. The \emph{relative cotangent complex} of $f:R\to A$ is by definition the absolute cotangent complex of $f$ as an algebra-object in $Alg_{\Op}(Mod_R^{\Op}(\C))\simeq (Alg_{\Op}(\C)_{R/.})_{./A}$. This definition recovers the absolute version when $R$ is the unit object. In what follows we will only need the absolute case.
\end{remark}

In \cite[Theorem 3.1.10]{francis-thesis} the author provides a characterization of $\mathbb{L}_{A}$ for any $\mathbb{E}_n$-algebra $A$ in a stable presentable symmetric monoidal $(\infty,1)$-category $\Cmonoidal$ such that $\C$ is generated under small colimits by the unit: $\Sigma^n(\mathbb{L}_{A})$ is the cofiber of the canonical map $Free(1)\to A$ in $Mod_A^{\mathbb{E}_n}(\C)$, where $1$ is the unit of the monoidal structure and $Free: \C\to Mod_A^{\mathbb{E}_n}(\C)$ is the left adjoint to the forgetful functor $Mod_A^{\mathbb{E}_n}(\C)\to \C$. This adjoint exists because colimits of modules are computed in $\C$ (See also \cite[Theorem 8.3.5.1]{lurie-ha}).\\

The notion of derivation can be presented using the idea of a \emph{square-zero extension}. If $d:A\to A\oplus M$ is a derivation, we fabricate a new $\Op$-algebra $\tilde{A}$ as the pullback in $Alg_{\Op}(\C)$
 
\begin{equation}
\label{estrudes}
\xymatrix{
\tilde{A}\ar[r]^f \ar[d]& A\ar[d]^{d}\\
A\ar[r]^{d_00}& A\oplus M
}
\end{equation}

\noindent where $d_0:A\to A\oplus M$ is the zero derivation $a\mapsto (a, 0)$. Since the functor $Alg_{\Op}(\C) \to \C$ preserves limits, the diagram (\ref{estrudes}) provides a pullback diagram in $\C$ and given a morphism $\ast \to A$ in $\C$, we can identify the fiber $\tilde{A}\times_A \ast$ in $\C$ with the loop $\Omega(M)$. Indeed, we have a pullback in $\C$

\begin{equation}
\label{estrudes2}
\xymatrix{
\tilde{A}\times_A \ast \ar[r]^f \ar[d]& (A\times_A \ast )\simeq \ast\ar[d]^{d}\\
\ast\simeq (A\times_A \ast) \ar[r]^{d_0}& (A\oplus M)\times_A \ast 
}
\end{equation}

\noindent and since the fiber of the canonical map $A\oplus M\to A$ can be identified with $M$, we find $\tilde{A}\times_A \ast\simeq \Omega(M)$.\\

A morphism of algebras $B\to A$ is said to be a \emph{square-zero extension of $A$ by $\Omega(M)$} if there is a derivation $d$ of $A$ with values in $M\simeq \Sigma(\Omega(M))$ such that $B\simeq \tilde{A}$. Thanks to \cite[Theorem 8.1.4.26]{lurie-ha} if $\Cmonoidal$ is a stable presentable $\mathbb{E}_k$-monoidal $(\infty,1)$-category with a compatible $t$-structutre, then the formula $(A\to A\oplus M)\mapsto (f:\tilde{A}\to A)$ establishes an equivalence between the theory of derivations and the subcategory of $Fun(\Delta[1],Alg_{\mathbb{E}_k}(\C))$ spanned by the square-zero extensions (see \cite[Section 8.4.1]{lurie-ha} for a precise formulation).\\

\begin{remark}
\label{squarezeropullback}
In the presence of a square-zero extension (\ref{estrudes}), every $\Op$-algebra $B$ induces a pullback diagram of spaces

\begin{equation}
\label{nunca1}
\xymatrix{
Map_{Alg_{\Op}(\C)}(B, \tilde{A})\ar[r]\ar[d]&\ar[d] Map_{Alg_{\Op}(\C)}(B,A)\\
Map_{Alg_{\Op}(\C)}(B,A)\ar[r]&Map_{Alg_{\Op}(\C)}(B,A\oplus M)
}
\end{equation}

Let $\phi:B\to A$ be a morphism of algebras. It follows that we can describe the fiber of the morphism
$Map_{Alg_{\Op}(\C)}(B, \tilde{A})\to Map_{Alg_{\Op}(\C)}(B,A)$ over the point corresponding to $\phi$ with the help of the cotangent complex of $B$. More precisely, we observe first that the mapping space $Map_{Alg_{\Op}(\C)_{./A}}(B,A\oplus M)$ (where $B$ is defined over $A$ via $\phi$) fits in a pullback diagram

\begin{equation}
\xymatrix{
Map_{Alg_{\Op}(\C)_{./A}}(B,A\oplus M)\ar[r]\ar[d]& Map_{Alg_{\Op}(\C)}(B,A\oplus M)\ar[d]\\
\Delta[0]\ar[r]^{\phi}& Map_{Alg_{\Op}(\C)}(B,A)
}
\end{equation}

\noindent where the right vertical map is the composition with the canonical map $A\oplus M\to A$. By tensoring with 
$(-\times_{Map_{Alg_{\Op}(\C)}(B,A)}\Delta[0])$ the diagram (\ref{nunca1}) produces a new pullback diagram

\begin{equation}
\label{nunca2}
\xymatrix{
Map_{Alg_{\Op}(\C)}(B, \tilde{A})\times_{Map_{Alg_{\Op}(\C)}(B,A)}\Delta[0]\ar[r]\ar[d]&\ar[d] Map_{Alg_{\Op}(\C)}(B,A)\times_{Map_{Alg_{\Op}(\C)}(B,A)}\Delta[0]\simeq \Delta[0]\\
Map_{Alg_{\Op}(\C)}(B,A)\times_{Map_{Alg_{\Op}(\C)}(B,A)}\Delta[0]\simeq \Delta[0]\ar[r]& Map_{Alg_{\Op}(\C)_{./A}}(B,A\oplus M)
}
\end{equation}

\noindent so that the fiber $Map_{Alg_{\Op}(\C)}(B, \tilde{A})\times_{Map_{Alg_{\Op}(\C)}(B,A)}\Delta[0]$ becomes the space of paths in $Map_{Alg_{\Op}(\C)_{./A}}(B,A\oplus M)$ between the point $\xymatrix{B\ar[r]^{\phi}& A\ar[r]^{d} &A\oplus M}$ and the point $\xymatrix{B\ar[r]^{\phi}& A\ar[r]^{d_0} &A\oplus M}$. To conclude, we can use the adjunctions of the Remark \ref{basechangecotangent} to find equivalences

\begin{equation}
Map_{Alg_{\Op}(\C)_{./A}}(B,A\oplus M)\simeq Map_{Mod_A^{\Op}(\C)}(L_A(\phi),M)\simeq Map_{Mod_A^{\Op}(\C)}(A\otimes_B \mathbb{L}_B,M)\simeq Map_{Mod_B^{\Op}(\C)}(\mathbb{L}_B,For(M))
\end{equation}

\noindent so that we find an equivalence

\begin{equation}
Map_{Alg_{\Op}(\C)}(B, \tilde{A})\times_{Map_{Alg_{\Op}(\C)}(B,A)}\Delta[0]\simeq \Omega_{0, d\circ \phi} Map_{Mod_B^{\Op}(\C)}(\mathbb{L}_B,For(M))
\end{equation}
\end{remark}

We now collect the last ingredient to prove the Lemma \ref{cotangentcompact}:

\begin{thm}(Lurie \cite[Corollary 8.4.1.28]{lurie-ha})
\label{postnikov}
Let $\Cmonoidal$ be a stable presentable symmetric monoidal $(\infty,1)$-category equipped with a compatible $t$-structure (in the sense of \ref{monoidalt}). Then for every $k\geq 0$ and any algebra $A\in Alg_{\mathbb{E}_k}(\C)^{cn}$ the morphisms in the Postnikov tower

\begin{equation}
...\to \tau_{\leq 2}A\to \tau_{\leq 1} A\to \tau_{\leq 0} A
\end{equation}

\noindent are square-zero extensions. More precisely, and following the Remark \ref{fibertruncation}, for every $n\geq 0$ the truncation map $\tau_{\leq n} A\to \tau_{\leq n-1}A$ is a square-zero extension of $\tau_{\leq n-1}A$ by a module-structure in $\mathbb{H}_n(A)[n]$. This is equivalent to the existence of a derivation $d_{n}:\tau_{\leq n-1} A\to \tau_{\leq n-1} A\oplus \mathbb{H}_n(A)[n+1]$ and a pullback diagram of algebras

\begin{equation}
\xymatrix{
\tau_{\leq n} A\ar[r]\ar[d]&\tau_{\leq n-1} A\ar[d]^{d_n}\\
\tau_{\leq n-1} A\ar[r]&\tau_{\leq n-1} A\oplus \mathbb{H}_n(A)[n+1]
}
\end{equation}

\end{thm}

We have now all the ingredients to prove the lemma.\\

\textit{Proof of the Lemma \ref{cotangentcompact}}:
We follow the same methods as in \cite[Prop. 2.2.2.4]{toen-vezzosi-hag2}. We first prove that $1)$ implies $2)$.\\

The fact that $\mathbb{H}_0(A)$ is finitely presented as an associative algebra follows from the fact that $\mathbb{H}_0$ commutes with colimits (it is a left adjoint), together with the fact that $\pi_0$ commutes with colimits in the $(\infty,1)$-category of spaces. The fact that $\mathbb{L}_{A}$ is compact follows from the universal property of the cotangent complex together with the following facts:

\begin{enumerate}[i)]
\item As explained before, the functor $(A\oplus-)$ of (\ref{pretaa}) can be identified with a delooping functor $\Omega^{\infty}$. Therefore it commutes with filtered colimits;
\item by assumption, $A$ is compact;
\end{enumerate}

We now prove that $2)$ implies $1)$. To start with, we observe that since $A$ is by assumption connective, it is enough to check that $A$ is compact in the full subcategory $Alg(\derivedk)^{cn}$ spanned by the connective objects. Indeed, recall from \ref{monoidalt} that the truncation functor $\tau_{\leq 0}$ is a right adjoint to the inclusion $Alg(\derivedk)^{cn}\subseteq Alg(\derivedk)$. We can easily check that $\tau_{\leq 0}$ commutes with filtered colimits (because the homology groups commute with filtered colimits) so that for any filtered system $\{\C_i\}_{i\in I}$ in $Alg(\derivedk)$ we have

\begin{equation}
Map_{Alg(\derivedk)}(A, colim_I C_i)\simeq Map_{Alg(\derivedk)^{cn}}(A,\tau_{\leq 0} colim_I C_i)\simeq Map_{Alg(\derivedk)^{cn}}(A, colim_I \tau_{\leq 0}C_i)
\end{equation}

\noindent so that $A$ is compact in $Alg(\derivedk)^{cn}$ if and only if is is compact in $Alg(\derivedk)$.\\

We start now by proving that $A$ is almost compact, meaning that $A$ is compact with respect to any filtered system in $Alg(\derivedk)^{cn}_{\leq n}$, for every $n\geq 0$. We proceed by induction. The case $n=0$ follows by the hypothesis. Let us suppose we know this is true for $n-1$ and prove it for $n$. Let $\{C_i\}_{i\in I}$ be a filtered system in $Alg(\derivedk)^{cn}_{\leq n}$. The discussion in \ref{complexes} together with the Theorem \ref{postnikov} implies that for each $i$, $C_i$ admits a Postnikov decomposition 

\begin{equation}
C_i= \tau_{\leq n}(C_i)\to \tau_{\leq n-1}(C_i)\to ... \to \tau_{\leq 0}(C_i)
\end{equation}

\noindent where each morphism is a square-zero extension providing a pullback diagram 

\begin{equation}
\xymatrix{
C_i=(C_i)_{\leq n}\ar[r]\ar[d]&(C_i)_{\leq n-1} \ar[d]^{d_n}\\
(C_i)_{\leq n-1}\ar[r]& (C_i)_{\leq n-1}\oplus \mathbb{H}_{n}(C_i)[n+1]
}
\end{equation}

\noindent in $Alg(\derivedk)^{cn}$ where the lower horizontal map is the zero map and right vertical map corresponds to the canonical  derivation $d_n$ associated to the square-zero extension $C_i\to \tau_{\leq n-1}C_i$. This diagram induces a pullback diagram of spaces

\begin{equation}
\xymatrix{
Map_{Alg(\derivedk)}(A,C_i)\ar[r]\ar[d]&Map_{Alg(\derivedk)}(A,\tau_{\leq n-1}(C_i)) \ar[d]\\
Map_{Alg(\derivedk)}(A,(C_i)_{\leq n-1})\ar[r]& Map_{Alg(\derivedk)}(A,\tau_{\leq n-1}(C_i)\oplus\mathbb{H}_{n}(C_i)[n+1])
}
\end{equation}

\noindent and the Remark \ref{squarezeropullback} implies that the fiber of the map 

\begin{equation}
\xymatrix{
Map_{Alg(\derivedk)}(A,C_i)\ar[r]&Map_{Alg(\derivedk)}(A,\tau_{\leq n-1}(C_i))
}
\end{equation}

\noindent over a map $u:A\to\tau_{\leq n-1}(C_i)$ is given by the space of paths in $Map_{Mod_{A}^{\Ass}}(\mathbb{L}_{A}, \mathbb{H}_{n}(C_i)[n+1])$ between the zero derivation and the point corresponding to the composition $d_n\circ u$. This reduces everything to the analysis of the diagram

\begin{equation}
\xymatrix{
colim_I \Omega_{0,d_n\circ u} Map_{Mod_A^{\Ass}(\derivedk)}(\mathbb{L}_{A},\mathbb{H}_{n}(C_i)[n+1])\ar[r]\ar[d]&\ar[d]\Omega_{0,d_n\circ u} Map_{Mod_A^{\Ass}(\derivedk)}(\mathbb{L}_{A},\mathbb{H}_{n}(colim_I C_i)[n+1])\\
colim_I Map_{Alg(\derivedk)}(A, C_i)\ar[r]\ar[d] &\ar[d]Map_{Alg(\derivedk)}(A, colim_I C_i)\\
colim_I Map_{Alg(\derivedk)}(A, \tau_{\leq n-1}(C_i))\ar[r] & Map_{Alg(\derivedk)}(A, \tau_{\leq n-1}(colim_I C_i) )
}
\end{equation}

We observe that

\begin{enumerate}[a)]
\item The left column is a fiber sequence because filtered colimits are exact in the $(\infty,1)$-category of spaces. For the same reason, there is an equivalence between the top left entry in the diagram and 

\begin{equation}\Omega_{0,d_n\circ u} colim_I Map_{Mod_A^{\Ass}(\derivedk)}(\mathbb{L}_{A},\mathbb{H}_{n}(C_i)[n+1])\end{equation}

\item The right column is also a fiber sequence. This follows from the result of \ref{postnikov} and the Remark \ref{squarezeropullback} applied to the colimit algebra $colim_I C_i$;

\item The top entry on the right is equivalent to 

\begin{equation}\Omega_{0,d_n\circ u} Map_{Mod_A^{\Ass}(\derivedk)}(\mathbb{L}_{A},colim_I\mathbb{H}_{n}( C_i)[n+1])\end{equation}. 

This is because the functor $\mathbb{H}_{n}$ is equivalent to the classical $nth$-homology functor and therefore commutes with filtered colimits.

\end{enumerate}

Finally, the induction hypothesis together with the fact that $(-)_{\leq n}$ is a left adjoint (and therefore commutes with colimits), implies that the lower horizontal arrow is an equivalence. The assumption that $\mathbb{L}_{A}$ is compact implies that the top horizontal map is also an equivalence. It follows that the middle one is also an equivalence. This proves that $A$ is almost compact in $Alg(\derivedk)^{cn}$.\\

We now complete the proof by showing that $A$ is compact. Since the $(\infty,1)$-category $Mod_{A}^{\Ass}((\derivedk))$ is equivalent to the underlying $(\infty,1)$-category of the model structure on strict $A$-bimodules in $Ch(k)$ (see \ref{noncommutativecotangentcomplex}) and the last is compactly generated in the sense of \ref{compactlygeneratedmodelcategories}, the Proposition \ref{compactlygeneratedmodelcategoriesprop} implies that $\mathbb{L}_A$ is a compact object in $Mod_{A}^{\Ass}((\derivedk))$ if and only if it is given by a finite strict cell object in the model category of bimodules. In this case, with our hypothesis that $\mathbb{L}_A$ is compact, we can find a natural number $n_0\geq 0$ such that for any object $M\in Mod_{A}^{\Ass}(\derivedk)$ concentrated in degrees strictly bigger than $n_0$ we have $\pi_0Map_{Mod_{A}^{\Ass}(\derivedk)}(\mathbb{L}_A, M)\simeq 0$. In particular, for any connective algebra $C$, the kernel of the map $C\to\tau_{\leq n_0}(C)$ is concentrated in degree $n_0+1$ and the fiber sequence of the Remark \ref{squarezeropullback} implies

\begin{equation}
\label{atlast}
\pi_0 Map_{Alg(\derivedk)}(A,C)\simeq \pi_0 Map_{Alg(\derivedk)}(A,\tau_{\leq n_0}(C)) 
\end{equation}

We now use this to show that $A$ is compact. Let $\{C_i\}_{i\in I}$ be a filtered system in $Alg(\derivedk)^{cn}$. Using the fact that $\pi_n$ commutes with filtered homotopy colimits of spaces and that $Alg(\derivedk)^{cn}$ admits all limits (it is a co-reflexive localization of $Alg(\derivedk)$), we are reduced to show that the natural map

\begin{equation}
colim_I \pi_0 Map_{Alg(\derivedk)}(A, \Omega^n C_i)\to \pi_0 Map_{Alg(\derivedk)}(A, colim_I \Omega^n C_i)
\end{equation}

\noindent is an equivalence. We show that the formula is true for any filtered system of algebras $\{U_i\}_{i\in I}$, because we have a commutative diagram

\begin{equation}
\xymatrix{
colim_I \pi_0 Map_{Alg(\derivedk)}(A, U_i)\ar[r]\ar[d]^{\sim}& \pi_0 Map_{Alg(\derivedk)}(A, colim_I U_i)\ar[d]^{\sim}\\
colim_I \pi_0 Map_{Alg(\derivedk)}(A, \tau_{\leq n_0}(U_i))\ar[r]^{\sim}& \pi_0 Map_{Alg(\derivedk)}(A, colim_I \tau_{\leq n_0}(U_i))
}
\end{equation}

\noindent where the vertical arrows are equivalences because of (\ref{atlast}) together with fact that $\tau_{\leq n_0}$ is a left adjoint, and the lower horizontal map is an equivalence because $A$ is almost compact. This concludes the proof.

\hfill $\Box$

This completes our preliminairs.

\section{Inversion of an Object in a Symmetric Monoidal $(\infty,1)$-category and the Relation with Symmetric Spectrum Objects}
\label{section4}

We finally come to the main section of our paper. In \ref{section4-1} we deal with the formal inversion of an object in a symmetric monoidal $(\infty,1)$-category. First we deal with the situation for small $(\infty,1)$-categories (Propositions \ref{prop1} and \ref{main}) and then we extend the result to the presentable setting (Prop. \ref{main3}). This method allow us to invert any object and the result is endowed with the expected universal property. In \ref{section4-2} we deal with the notion of spectrum-objects. Our main result (Cor. \ref{main5}) is that if the object we want to invert satisfies a symmetric condition then the underlying $(\infty,1)$-category of the formal inversion is nothing but the stabilization with respect to the chosen object.
Finally, in \ref{section4-3} we prove our main theorem (see \ref{maintheorem}), which ensures that the familiar construction of symmetric spectrum objects with respect to a given symmetric object $X$ together with the convolution product, is the "model category" incarnation of our $\infty$-categorical phenomenom of inverting $X$.

\subsection{Formal inversion of an object in a Symmetric Monoidal $(\infty,1)$-category}
\label{section4-1}
Let $\Cmonoidal$ be a symmetric monoidal $(\infty,1)$-category and let $X$ be an object in $\C$. We will say that $X$ is invertible with respect to the monoidal structure if there is an object $X^*$ such that $X\otimes X^*$ and $X^*\otimes X$ are both equivalent to the unit object. Since the monoidal structure is symmetric, it is enough to have one of these conditions. It is an easy observation that this condition depends only on the monoidal structure induced on the homotopy category $h(\C)$, because equivalences are exactly the isomorphisms in $h(\C)$. Alternatively, we can see that an object $X$ in $\C$ is invertible if and only if the map "multiplication by $X$" = $(X\otimes -): \C\to \C$ is an equivalence of $(\infty,1)$-categories. Indeed, if $X$ has an inverse $X^*$ then the maps $(X\otimes -)$ and $(X^* \otimes -)$ are inverses since the coherences of the monoidal structure can be used to fabricate the homotopies. Conversely, if $(X\otimes -)$ is an equivalence, the essential subjectivity provides an object $X^*$ such that $X\otimes X^*\simeq \mathit{1}$. The symmetry provides an equivalence $\mathit{1}\simeq X^*\otimes X$.\\

Our main goal is to produce from the data of $\Cmonoidal$ and $X$, a new symmetric monoidal $(\infty,1)$-category $\Cmonoidalx$ together with a monoidal map $\Cmonoidal\to \Cmonoidalx$ sending $X$ to an invertible object and universal with respect to this property. In addition, we would like this construction to hold within the world of presentable symmetric monoidal $(\infty,1)$-categories. Our steps follow the original ideas of \cite{toen-vezzosi-hag2}, where the authors studied the inversion of an element in a strictly commutative algebra object in a symmetric monoidal model category.\\ 

We start by analyzing the theory for a small symmetric monoidal $(\infty,1)$-category $\Cmonoidal$. In this case, and following the Remark \ref{remarkcat}, $\Cmonoidal$ can be identified with an object in $CAlg(\iCat)$. The objects of $Mod_{\C}(\iCat)$ can be identified with the $(\infty,1)$-categories endowed with an "action" of $\C$ and we will refer to them simply as \emph{$\Cmonoidal$-Modules}. By the Proposition \ref{algmod-alg} , $CAlg(Mod_{\Cmonoidal}(\iCat))$ is equivalent to $CAlg(\iCat)_{\Cmonoidal/}$ where the objects are the small symmetric monoidal $(\infty,1)$-categories $\Dmonoidal$ equipped with a monoidal map $\Cmonoidal\to \Dmonoidal$. We denote by $CAlg(\iCat)_{\Cmonoidal/}^X$ the full subcategory of $CAlg(\iCat)_{\Cmonoidal/}$ spanned by the algebras $\Cmonoidal\to \Dmonoidal$ whose structure map sends $X$ to an invertible object. The main observation is that the objects in $CAlg(\iCat)_{\Cmonoidal/}^X$ can be understood as local objects in $CAlg(\iCat)_{\Cmonoidal/}$ with respect to a certain set of morphisms: there is a forgetful functor 

\begin{equation}
CAlg(\iCat)_{\Cmonoidal/}\simeq CAlg(Mod_{\Cmonoidal}(\iCat))\to Mod_{\Cmonoidal}(\iCat)
\end{equation}

\noindent and since $\iCat^{\times}$ is a presentable symmetric monoidal $(\infty,1)$-category, this functor admits a left adjoint $Free_{\Cmonoidal}(\C)$ assigning to each $\Cmonoidal$-module $\D$ the free commutative $\Cmonoidal$-algebra generated by $\D$. We will denote by $\mathcal{S}_X$ the collection of morphisms in $CAlg(\iCat)_{\Cmonoidal/}$ consisting of the single morphism 

\begin{equation}
\xymatrix{
Free_{\Cmonoidal}(\C)\ar[rr]^{Free_{\Cmonoidal}(X\otimes-)}&&Free_{\Cmonoidal}(\C)
}
\end{equation}

\noindent where $\C$ is understood as a $\Cmonoidal$-module in the obvious way using the monoidal structure. We prove the following

\begin{prop}
\label{prop1}
Let $\Cmonoidal$ be a symmetric monoidal $(\infty,1)$-category. Then the full subcategory $CAlg(\iCat)_{\Cmonoidal/}^X$ coincides with the full subcategory of $CAlg(\iCat)_{\Cmonoidal/}$ spanned by the $\mathcal{S}_X$-local objects. Moreover, since $\iCat^{\times}$ is a presentable symmetric monoidal $(\infty,1)$-category, the $(\infty,1)$-categories $CAlg(\iCat)$ and $CAlg(\iCat)_{\Cmonoidal/}$ are also presentable (see Corollary 3.2.3.5 of \cite{lurie-ha}) and the results of the Proposition 5.5.4.15 in \cite{lurie-htt} follow. We deduce the existence a left adjoint $\Lmonoidal$

\begin{equation}
\xymatrix{
CAlg(\iCat)_{\Cmonoidal/}^{\mathcal{S}_X-local}=CAlg(\iCat)_{\Cmonoidal/}^X  \ar@{^{(}->}[rr] && CAlg(\iCat)_{\Cmonoidal/}  \ar@/_2pc/[ll]_{\Lmonoidal}
}
\end{equation}

In particular, the data of this adjunction provides the existence of a symmetric monoidal $(\infty,1)$-category  $\Lmonoidal(\Cmonoidal)$ equipped with a canonical monoidal map $f:\Cmonoidal\to \Lmonoidal(\Cmonoidal)$ sending $X$ to an invertible object.
\begin{proof}
The only thing to check is that both subcategories coincide. Let $\phi:\Cmonoidal\to \Dmonoidal$ be a $\C$-algebra where $X$ is sent to an invertible object. By the definition of the functor $Free_{\Cmonoidal}(\C)$ we have a commutative diagram

\begin{equation}
\label{diagramx1}
\xymatrix{
Map_{CAlg(\iCat)_{\Cmonoidal/}}(Free_{\Cmonoidal}(\C), \Dmonoidal)\ar[r] \ar[d]^{\sim}& Map_{CAlg(\iCat)_{\Cmonoidal/}}(Free_{\Cmonoidal}(\C), \Dmonoidal)\ar[d]^{\sim}\\
Map_{Mod_{\Cmonoidal}(\iCat)}(\C, \D) \ar[r] & Map_{Mod_{\Cmonoidal}(\iCat)}(\C, \D) \\
}
\end{equation}

\noindent where the lower horizontal map is described by the formula $\alpha\mapsto \alpha\circ (X\otimes -)$. Since $\phi$ is monoidal, the diagram commutes

\begin{equation}
\xymatrix{\C\ar[rr]^{(X\otimes -)} \ar[d]^{\phi}&& \C\ar[d]^{\phi}\\ \D \ar[rr]_{(\phi(X)\otimes -)}&& \D}
\end{equation}

\noindent and the lower map is in fact homotopic to the one given by the formula $\alpha\mapsto (\phi(X)\otimes-)\circ \alpha$.
Since $\phi(X)$ is invertible in $\Dmonoidal$, there exists an object $\lambda$ in $\D$ such that the maps $\phi(X)\otimes -)$ and $(\lambda\otimes -)$ are inverses and therefore the lower map in (\ref{diagramx1}), and as a consequence the top map, are isomorphisms of homotopy types.\\

Let now $\Cmonoidal\to\Dmonoidal$ be a $\Cmonoidal$-algebra, local with respect to $\mathcal{S}_X$. In particular, the map 

\begin{equation}Map_{Mod_{\Cmonoidal}(\iCat)}(\C, \D) \to Map_{Mod_{\Cmonoidal}(\iCat)}(\C, \D)\end{equation}

\noindent induced by the composition with $(X\otimes -)$ is an isomorphism of homotopy types and in particular we have $\pi_0(Map_{Mod_{\Cmonoidal}(\iCat)}(\C, \D)) \simeq \pi_0( Map_{Mod_{\Cmonoidal}(\iCat)}(\C, \D))$. We deduce the existence of a dotted arrow 
\begin{equation}
\xymatrix{\D\ar[d]^{\phi}\ar[r]^{X\otimes-}& \D\ar@{..>}[ld]^{\alpha}\\ \D}
\end{equation}
\noindent rendering the diagram of modules commutative and since $\alpha$ is a map of $\Cmonoidal$-modules and $\phi$ is monoidal we find $\phi(\mathit{1})\simeq\alpha(X\otimes \mathit{1})\simeq \phi(X)\otimes \alpha(\mathit{1})$. Using the symmetry we find that  $\alpha(\mathit{1}\otimes X)\simeq \alpha(\mathit{1})\otimes\phi(X)\simeq \mathit{1}$ which proves that $\phi(X)$ has an inverse in $\Dmonoidal$.
\end{proof}
\end{prop}

We will now study the properties of the base change along the morphism $\Cmonoidal\to \Lmonoidal(\Cmonoidal)$. In order to establish some insight, let us point out that everything fits in a commutative diagram

\begin{equation}
\xymatrix{
CAlg(\iCat)_{\Lmonoidal(\Cmonoidal)/}\simeq \ar[d]CAlg(Mod_{\Lmonoidal(\Cmonoidal)}(\iCat))\ar[r]& CAlg(Mod_{\Cmonoidal}(\iCat))  \simeq \ar[d]CAlg(\iCat)_{\Cmonoidal/}\\
Mod_{\Lmonoidal(\Cmonoidal)}(\iCat)\ar[r]^{f_*} &Mod_{\Cmonoidal}(\iCat)
}
\end{equation}

\noindent where the horizontal arrows are induced by the forgetful map given by the composition with $\Cmonoidal\to \Lmonoidal(\Cmonoidal)$ and the vertical arrows are induced by the forgetful map produced by the change of $\infty$-operads $\Trivmonoidal\to \Commmonoidal$. Since $\iCat$ with the cartesian product is a presentable symmetric monoidal $(\infty,1)$-category, there is a base change functor 

\begin{equation}\xymatrix{Mod_{\Lmonoidal(\Cmonoidal)}(\iCat)\ar[r]& Mod_{\Cmonoidal}(\iCat) \ar@/_2pc/[l]_{\Lx)}} \end{equation}

\noindent and by the general theory we have an identification of $f_*(\Lx(M))\simeq M\otimes_{\Cmonoidal} \Lmonoidal(\Cmonoidal)$ given by the tensor product in $Mod_{\Cmonoidal}(\iCat)$. This map is monoidal and therefore induces a left adjoint

\begin{equation}\xymatrix{  CAlg(\iCat)_{\Lmonoidal(\Cmonoidal)/}\ar[r]&  \ar@/_2pc/[l]_{\widetilde{\mathcal{L}}_{(\Cmonoidal,X)}} CAlg(\iCat)_{\Cmonoidal/}}\end{equation}

\noindent and the diagram

\begin{equation}
\label{diagram}
\xymatrix{
CAlg(\iCat)_{\Lmonoidal(\Cmonoidal)/}\simeq \ar[d]CAlg(Mod_{\Lmonoidal(\Cmonoidal)}(\iCat))&\ar[l]_{\widetilde{\mathcal{L}}_{(\Cmonoidal,X)}} CAlg(Mod_{\Cmonoidal}(\iCat))  \simeq \ar[d]CAlg(\iCat)_{\Cmonoidal/}\\
Mod_{\Lmonoidal(\Cmonoidal)}(\iCat) &\ar[l]^{\Lx} Mod_{\Cmonoidal}(\iCat)
}
\end{equation}

\noindent commutes. We prove the following statement, which was originally proved in \cite{toen-vezzosi-hag2} in the context of model categories:

\begin{prop}
\label{main}
Let $\Cmonoidal$ be a small symmetric monoidal $(\infty,1)$-category and $X$ be an object in $\C$. Let $f:\Cmonoidal\to \Lmonoidal(\Cmonoidal)$ be the natural map constructed above. Then
\begin{enumerate}
\item the composition map 
\begin{equation}CAlg(\iCat)_{\Lmonoidal(\Cmonoidal)/}\to CAlg(\iCat)_{\Cmonoidal/}\end{equation}
\noindent is fully faithful and its image coincides with $CAlg(\iCat)_{\Cmonoidal}^X$;
\item the forgetful functor

\begin{equation}
f_* : Mod_{\Lmonoidal(\Cmonoidal)}(\iCat)\to Mod_{\Cmonoidal}(\iCat)
\end{equation}

\noindent is fully faithful and its image coincides with the full subcategory of $Mod_{\Cmonoidal}(\iCat)$ spanned by those $\C$-modules where $X$ acts as an equivalence.
\end{enumerate}
\end{prop}

A major consequence is that the left adjoint $\widetilde{\mathcal{L}}_{(\Cmonoidal,X)}$ provided by the base change is naturally equivalent to the left adjoint $\Lmonoidal$ provided by Proposition \ref{prop1}. Moreover, since the diagram (\ref{diagram}) commutes, we have the formula $\Lx(\D)\simeq\Lmonoidal(\Dmonoidal)_{\onefin}$ for any $\Dmonoidal\in  CAlg(\iCat)_{\Cmonoidal/}$.\\

In order to prove Proposition \ref{main}, we will need some preliminary steps. We start by recalling some notation: Let $\Emonoidal$ be a symmetric monoidal $(\infty,1)$-category. A morphism of commutative algebras $A\to B$ in $\E$ is called an epimorphism (see \cite{toen-vezzosi-hag2}-Definition 1.2.6.1-1) if for any commutative $A$-algebra $C$, the mapping space $Map_{CAlg(\E)}(B,C)$ is either empty or weakly contractible. In other words, the space of dotted maps of $A$-algebras

\begin{equation}
\xymatrix{
C&\\
A\ar[r]\ar[u]& B\ar@{..>}[ul]
}
\end{equation}

\noindent rendering the diagram commutative is either empty or consisting of a unique map, up to equivalence.
We can rewrite this definition in a different way. As a result of the general theory, if $\Emonoidal$ is compatible with all small colimits, the $\infty$-category $CAlg(\E)_{A/}$ inherits a coCartesian tensor product which we denote here as $\otimes_A$. In this case it is immediate the conclusion that a map $A\to B$ is an epimorphism if and only if the canonical map $B\to B\otimes_A B$ is an equivalence. Of course, this is happens if and only if the induced colimit map $B\otimes_A B\to B$ is also an equivalence. We prove the following

\begin{prop}
\label{main2}
Let $\Emonoidal$ be a symmetric monoidal $(\infty,1)$-category compatible with all small colimits and let $f:A\to B$ a morphism of commutative algebras in $\E$. The following are equivalent:
\begin{enumerate}
\item $f$ is an epimorphism;
\item The natural map $f_*:Mod_B(\E)\to Mod_A(\E)$ is fully faithful;
\end{enumerate}
Moreover, if these equivalent conditions are satisfied, the forgetful map
\begin{equation}
CAlg(\E)_{B/}\to CAlg(\E)_{A/}
\end{equation}
is also fully faithful.
\begin{proof}
With the hypothesis that the monoidal structure is compatible with colimits, the general theory gives us a base-change functor

\begin{equation}
(-\otimes_A B):Mod_A(\E) \to Mod_B(\E) 
\end{equation}

\noindent left adjoint to the forgetful map $f_*$. In this case $f_*$ will be fully faithful if and only if the counit of the adjunction is an equivalence. If the counit is an equivalence in particular we deduce that the canonical map $ B\otimes_A B\to B$ is an equivalence and therefore $A\to B$ is an epimorphism. Conversely, if $A\to B$ is an epimorphism, for any $B$-module $M$ we have

\begin{equation}
M\otimes_A B\simeq (M\otimes_B B)\otimes_A B)\simeq M\otimes_B (B\otimes_A B)\simeq (M\otimes_B B)\simeq M
\end{equation}

It remains to prove the additional statement concerning the categories of algebras. Let us consider $u:B\to U$, $v:B\to V$ two algebras over $B$. We want to prove that the canonical map

\begin{equation}
Map_{CAlg(\E)_{B/}}(U, V) \to Map_{CAlg(\E)_{A/}}(f_*(U), f_*(V)) 
\end{equation}

\noindent is an isomorphism of homotopy types. The points in $Map_{CAlg(\E)_{A/}}(f_*(U), f_*(V)))$ can be identified with commutative diagrams

\begin{equation}
\xymatrix{
&&U\ar[dd]\\
A\ar[r]^f \ar@/^/[urr]^{u\circ f} \ar@/_/[drr]_{v\circ f}&B\ar[dr]^v \ar[ur]_u&\\
&&V
}
\end{equation}

\noindent and therefore we can rewrite $Map_{CAlg(\E)_{A/}}(f_*(U), f_*(V))$ as an homotopy pullback diagram

\begin{equation}Map_{CAlg(\E)_{A/}}(B, f_*(V))\times_{Map_{CAlg(\E)_{A/}}(A, f_*(V))} Map_{CAlg(\E)_{B/}}(U, V)\end{equation}

\noindent which by the fact $A\to B$ is an epimorphism and $Map_{CAlg(\E)_{A/}}(A, f_*(V))\simeq *$, reduces to $Map_{CAlg(\E)_{B/}}(U, V)$.
\end{proof}
\end{prop}

The following is the main ingredient in the proof of the Proposition \ref{main}.

\begin{prop}
Let $\Cmonoidal$ be a small symmetric monoidal $\infty$-category and let $X$ be an object in $\C$. Then,
the canonical map $\Cmonoidal \to \Lmonoidal(\Cmonoidal)$ is an epimorphism.
\begin{proof}
This is a direct result of the characterization of $\Lmonoidal$ as an adjoint in the Proposition \ref{prop1}. Indeed, for any algebra $\phi:\Cmonoidal\to \Dmonoidal$, either $\phi$ does not send $X$ to an invertible object and in this case 
$Map_{CAlg(\iCat)_{\Cmonoidal/}}(\Lmonoidal(\Cmonoidal), \Dmonoidal)$ is necessarily empty or, $\phi$ sends $X$ to an invertible object and we have by the universal properties

\begin{equation}
Map_{CAlg(\iCat)_{\Cmonoidal/}}(\Lmonoidal(\Cmonoidal), \Dmonoidal)\simeq Map_{CAlg(\iCat)_{\Cmonoidal/}}(\Cmonoidal, \Dmonoidal)\simeq *
\end{equation}

\end{proof}
\end{prop} 

\noindent \textit{Proof of Proposition \ref{main}:} 
By the results above we know that both maps are fully faithful. It suffices now to analyze their images.
\begin{enumerate}
\item If $\phi:\Cmonoidal\to \Dmonoidal$ is in the image, $\Dmonoidal$ is an algebra over $\Lmonoidal(\Cmonoidal)$ and there exists a monoidal factorization 

\begin{equation}
\xymatrix{
\Cmonoidal\ar[r]^{\phi}\ar[d]& \Dmonoidal\\
\Lmonoidal(\Cmonoidal)\ar@{..>}[ur]&
}
\end{equation}

\noindent and therefore $X$ is sent to an invertible object. Conversely, if $\phi:\Cmonoidal\to \Dmonoidal$ sends $X$ to an invertible object, $\phi:\Cmonoidal\to \Dmonoidal$ is local with respect to $Free_{\Cmonoidal}(X\otimes-):Free_{\Cmonoidal}(\C)\to Free_{\Cmonoidal}(\C)$ and therefore the adjunction morphisms
of the Proposition \ref{prop1} fit in a commutative diagram

\begin{equation}
\xymatrix{
\Cmonoidal \ar[rr]^{\phi} \ar[d]&& \Dmonoidal \ar[d]^{\sim}\\
\Lmonoidal(\Cmonoidal)\ar[rr]^{\Lmonoidal(\phi)} && \Lmonoidal(\Dmonoidal)
}
\end{equation}

\noindent where the right vertical map is an equivalence and we deduce the existence of a monoidal map presenting $\Dmonoidal$ as a $\Lmonoidal(\Cmonoidal)$-algebra, therefore being in the image of $f_*$.

\item Again, it remains to prove the assertion about the image. If $M$ is a $\Cmonoidal$-module in the image, by definition, its module structure is obtained by the composition $\Cmonoidal\times M\to \Lmonoidal(\Cmonoidal)\times M\to M$ and therefore the action of $X$ on $M$ is invertible. Conversely, let $M$ be a $\Cmonoidal$-module where $X$ acts as a equivalence. We want to show that $M$ is in the image of the forgetful functor. Since we know it is fully faithful, this is equivalent to show that the unit map of the adjunction 

\begin{equation}M\to f_* \Lx (M)\simeq M\otimes_{\Cmonoidal} \Lmonoidal(\Cmonoidal)\end{equation}

\noindent is an equivalence. To prove this we will need a reasonable description of $Free_{\Cmonoidal}(M)$ - the free $\Cmonoidal$ algebra generated by $M$. Following the Construction $3.1.3.7$ and the Example $3.1.3.12$ of \cite{lurie-ha} we know that the underlying $\Cmonoidal$-module $Free_{\Cmonoidal}(M)_{\onefin}$ can be described as a coprodut

\begin{equation}
\coprod_{n\geq 0}Sym^n(M)_{\Cmonoidal}
\end{equation}

\noindent where $Sym^n(M)_{\Cmonoidal}$ is a colimit diagram in $Mod_{\Cmonoidal}(\iCat)$ which can be informally described as $M^{\otimes_{\Cmonoidal}}/\Sigma_n$ where $\otimes_{\Cmonoidal}$ refers to the natural symmetric monoidal structure in $Mod_{\Cmonoidal}(\iCat)$. Let us proceed.

\begin{itemize}
\item The general machinery tells us that $Free_{\Cmonoidal}(M)$ exists in our case and by construction it comes naturally equipped with a canonical monoidal map $\phi:\Cmonoidal\to Free_{\Cmonoidal}(M)$. We remark that the multiplication map $(\phi(X)\otimes-): Free_{\Cmonoidal}(M)_{\onefin}\to Free_{\Cmonoidal}(M)_{\onefin}$ can be identified with the image $Free_{\Cmonoidal}(X\otimes-)_{\onefin}$ of the multiplication map $(X\otimes-):M\to M$. Since this last one is an equivalence (by the assumption), we conclude that $Free_{\Cmonoidal}(M)$ is in fact a $\Cmonoidal$ algebra where $X$ is sent to an invertible object. This means that it is in fact a $\Lmonoidal(\Cmonoidal)$-algebra and therefore $Free_{\Cmonoidal}(M)_{\onefin}$ is in fact a $\Lmonoidal(\Cmonoidal)$-module, which means that the unit map

\begin{equation}
Free_{\Cmonoidal}(M)_{\onefin}\to f_*( \Lx(Free_{\Cmonoidal}(M)_{\onefin})) \simeq Free_{\Cmonoidal}(M)_{\onefin}\otimes_{\Cmonoidal} \Lmonoidal(\Cmonoidal)
\end{equation}

\noindent is an equivalence. 

\item We observe now that we have a canonical map $M\to Free_{\Cmonoidal}(M)_{\onefin}$ because $Sym^1(M)=M$ and that this map is obviously fully faithful. The unit of the natural transformation associated to the base-change gives us a commutative diagram

\begin{equation}
\xymatrix{
M\ar[r]\ar[d]& M\otimes_{\Cmonoidal} \Lmonoidal(\Cmonoidal) \ar[d]\\
Free_{\Cmonoidal}(M)_{\onefin}\ar[r]^(.3){\sim}& Free_{\Cmonoidal}(M)_{\onefin}\otimes_{\Cmonoidal} \Lmonoidal(\Cmonoidal)
}
\end{equation}

\noindent where the lower arrow is an equivalence from the discussion in the previous item. Since the monoidal structure is compatible with coproducts and using the identification $Sym^n(M)_{\Cmonoidal}\simeq M^{\otimes_{\Cmonoidal}^n}/\Sigma_n$, we have 

\begin{equation}
Free_{\Cmonoidal}(M)_{\onefin}\otimes_{\Cmonoidal} \Lmonoidal(\Cmonoidal)\simeq \coprod [(M^{\otimes_{\Cmonoidal}^n})^\otimes_{\Cmonoidal} \Lmonoidal(\Cmonoidal)]/\Sigma_n
\end{equation}

\noindent and finally, using the fact $\Cmonoidal\to \Lmonoidal(\Cmonoidal)$ is an epimorphism, we have 

\begin{equation} (\Lmonoidal(\Cmonoidal))^{\otimes_{\Cmonoidal}^n}\simeq \Lmonoidal(\Cmonoidal)\end{equation}

\noindent for any $n\geq 0$. We find an equivalence

\begin{equation}
Free_{\Cmonoidal}(M)_{\onefin}\otimes_{\Cmonoidal} \Lmonoidal(\Cmonoidal)\simeq Free_{\Cmonoidal}(M\otimes_{\Cmonoidal} \Lmonoidal(\Cmonoidal))_{\onefin}
\end{equation}

The first diagram becomes

\begin{equation}
\xymatrix{
M\ar[r]\ar[d]& M\otimes_{\Cmonoidal} \Lmonoidal(\Cmonoidal)\ar[d]\\
Free_{\Cmonoidal}(M)_{\onefin}= \coprod_{n\geq 0}Sym^n(M)_{\Cmonoidal}\ar[r]^{\sim}& \coprod_{n \geq 0} Sym^n(M\otimes_{\Cmonoidal}\Lmonoidal(\Cmonoidal))_{\Cmonoidal}
}
\end{equation}

\noindent where both vertical maps are now the canonical inclusions in the coproduct. Therefore, since $\iCat$ has disjoint coproduts (because coproducts can be computed as homotopy coproducts in the combinatorial model category of marked simplicial sets and here coproducts are disjoint), we conclude that the canonical map $M\to M\otimes_{\Cmonoidal}\Lmonoidal(\Cmonoidal)$ is also an equivalence.

\end{itemize}

This concludes the proof.
\end{enumerate}

\begin{remark}
Let $X$ and $Y$ be two objects in a symmetric monoidal $(\infty,1)$-category $\C$. If $X$ and $Y$ are equivalent, then a monoidal map $\Cmonoidal\to \Dmonoidal$ sends $X$ to an invertible object if and only if it sends $Y$ to an invertible object. In this case we have $\Lmonoidal\simeq \mathcal{L}_{(\Cmonoidal,Y)}^{\otimes}$.
\end{remark}

\begin{remark}
\label{invertingtwoobjects}
Let $\Cmonoidal$ be a symmetric monoidal $(\infty,1)$-category. Let $X$ and $Y$ be two objects in $\C$ and let $X\otimes Y$ denote their product with respect to the monoidal structure. Since the monoidal structure is symmetric, it is an easy observation that $X\otimes Y$ is an invertible object if and only if $X$ and $Y$ are both invertible. Therefore, we can identify the full subcategory $CAlg(\iCat)_{\Cmonoidal/}^{X\otimes Y}$ with the full subcategory  $CAlg(\iCat)_{\Cmonoidal/}^{X,Y}$ spanned by the algebra objects $\Cmonoidal\to\Dmonoidal$ sending both $X$ and $Y$ to invertible objects. As a consequence, we can provide a relative version of our methods and by the universal properties the diagram

\begin{equation}
\xymatrix{
\ar@/_2pc/[d]_{base-change} CAlg(\iCat)_{\Lmonoidal(\Cmonoidal)/}=CAlg(\iCat)_{\Cmonoidal/}^{X}  \ar@{^{(}->}[r]& CAlg(\iCat)_{\Cmonoidal/}  \ar@/_2pc/[l]_{base-change} \ar@/^2pc/[d]^{base-change}\\
CAlg(\iCat)_{\mathcal{L}^{\otimes}_{(\Cmonoidal, X\otimes Y)}(\Cmonoidal)/}=CAlg(\iCat)_{\Cmonoidal/}^{X,Y} \ar@{^{(}->}[u] \ar@{^{(}->}[r]& \ar@/^2pc/[l]^{base-change}CAlg(\iCat)_{\Cmonoidal/}^{Y}= CAlg(\iCat)_{\mathcal{L}^{\otimes}_{(\Cmonoidal,Y)}(\Cmonoidal)/} \ar@{^{(}->}[u]
}
\end{equation}

\noindent has to commute. 
\end{remark}

\begin{remark}
\label{workswithspaces}
The results of \ref{prop1} and \ref{main} also hold if we restrict our attention to symmetric monoidal $(\infty,1)$-categories that are $\infty$-groupoids. More precisely, if $\Cmonoidal$ is an object in $CAlg(\Spaces)$ and $X$ is an object in $\C$, the inclusion

\begin{equation}
\xymatrix{
CAlg(\Spaces)_{\Cmonoidal/.}^{X}\ar@{^{(}->}[r]& CAlg(\Spaces)_{\Cmonoidal/}
}
\end{equation}

\noindent admits a left adjoint $\mathcal{L}^{spaces, \otimes}_{\Cmonoidal,X}$. This follows from the same arguments as in \ref{prop1}, using the fact that $\Spaces$ is presentable. Moreover, as in \ref{main}, we can identify $CAlg(\Spaces)_{\Cmonoidal/.}^{X}$ with the $(\infty,1)$-category of commutative $\mathcal{L}^{spaces, \otimes}_{\Cmonoidal,X}(\Cmonoidal)$-algebras.\\

Recall now that the existence of a a fully-faithful inclusion $i:\Spaces\subseteq \iCat$. This inclusion is monoidal with respect to the cartesian structures and produces an inclusion $i:CAlg(\Spaces)\subseteq CAlg(\iCat)$. Therefore, for every symmetric monoidal $\infty$-groupoid $\Cmonoidal$ together with the choice of an object $X\in \C$, we have a commutative diagram

\begin{equation}
\xymatrix{
CAlg(\iCat)_{i(\Cmonoidal)/.}^{X}\ar@{^{(}->}[r]&  CAlg(\iCat)_{i(\Cmonoidal)/.}\\
CAlg(\Spaces)_{\Cmonoidal/.}^{X} \ar@{^{(}->}[u] \ar@{^{(}->}[r]&  CAlg(\Spaces)_{\Cmonoidal/.}\ar@{^{(}->}[u]
}
\end{equation}

\noindent from which, using the universal property of the adjuntion in $\ref{prop1}$, we can deduce the existence of a canonical monoidal map of symmetric monoidal $(\infty,1)$-categories

\begin{equation}
\mathcal{L}^{\otimes}_{i(\Cmonoidal),X}(i(\Cmonoidal))\to i(\mathcal{L}^{spaces, \otimes}_{\Cmonoidal,X}(\Cmonoidal))
\end{equation}

Later on (see the Remark \ref{workswithspaces2}) we will see that under an extra assumption on $X$ this comparison map is an equivalence.
\end{remark}

Our goal now is to extend our construction to the setting of presentable symmetric monoidal $\infty$-categories. The starting observation is that, if $\Cmonoidal$ is a small symmetric monoidal $(\infty,1)$-category the inversion of an object $X$ can now be rewritten by means of a pushout square in $CAlg(\iCat)$: Since $\iCat$ is a symmetric monoidal $(\infty,1)$-category compatible with all colimits, the forgetful functor

\begin{equation}
CAlg(\iCat)\to \iCat 
\end{equation}

\noindent admits a left adjoint $free^{\otimes}$ which assigns to an $\infty$-category $\D$, the \emph{free symmetric monoidal $(\infty,1)$-category generated by} $\D$. An object in $\C$ can be interpreted as a monoidal map $free^{\otimes}(\Delta[0])\to \Cmonoidal$ where $free^{\otimes}(\Delta[0])$ is the free symmetric monoidal category generated by one object $*$. By the universal property of $\mathcal{L}_{(\Cmonoidal,*)}^{\otimes}(free^{\otimes}(\Delta[0]))$, a monoidal map $\Cmonoidal\to \Dmonoidal$ sends $X$ to an invertible object if and only if it factors as a commutative diagram

\begin{equation}
\label{diagram2}
\xymatrix{
free^{\otimes}(\Delta[0])\ar[d]^X\ar[r]& \mathcal{L}_{(\Cmonoidal,*)}^{\otimes}(free^{\otimes}(\Delta[0]))\ar[d]\\
\Cmonoidal\ar[r] &\Dmonoidal
}
\end{equation}

\noindent and by the combination of the universal properties, the pushout in $CAlg(\iCat)$

\begin{equation}
\Cmonoidal\coprod_{free^{\otimes}(\Delta[0])}\mathcal{L}_{(\Cmonoidal,*)}^{\otimes}(free^{\otimes}(\Delta[0]))
\end{equation}
 
\noindent is canonically equivalent to $\Lmonoidal(\Cmonoidal)$. The existence of this pushout is ensured by the fact that $\iCat^{\times}$ is compatible with all colimits. \\

We will use this pushout-version to construct the presentable theory. By the tools described in the section \ref{section3-2}, if $\Cmonoidal$ is a presentable symmetric monoidal $(\infty,1)$-category (not necessarily small) and $X$ is an object in $\C$, the universal \emph{monoidal} property of presheaves ensures that any diagram like (\ref{diagram2}) factors as

\begin{equation}
\label{diagram3}
\xymatrix{
free^{\otimes}(\Delta[0])\ar[d]^j\ar[r]& \mathcal{L}_{(\Cmonoidal,*)}^{\otimes}(free^{\otimes}(\Delta[0]))\ar[d]^{j'}\\
\mathcal{P}(free^{\otimes}(\Delta[0]))^{\otimes}\ar@{-->}[d]\ar[r]& \mathcal{P}(\mathcal{L}_{(\Cmonoidal,*)}^{\otimes}(free^{\otimes}(\Delta[0])))^{\otimes}\ar@{-->}[d]\\
\Cmonoidal\ar[r] &\Dmonoidal
}
\end{equation}

\noindent where $\mathcal{P}^{\otimes}(-)$ is the natural extension of the symmetric monoidal structure to presheaves, the vertical maps $j$ and $j'$ are the respective Yoneda embeddings (which are monoidal maps) and the dotted arrows are given by colimit-preserving monoidal maps obtained as left Kan extensions. 

\begin{defn}
\label{defformal}
Let $\Cmonoidal$ be a presentable symmetric monoidal $(\infty,1)$-category and let $X$ be an object in $\C$. The \emph{formal inversion} of $X$ in $\Cmonoidal$ is the new presentable symmetric monoidal $(\infty,1)$-category $\Cmonoidalx$ defined by pushout

\begin{equation}
\label{formulainverse}
\Cmonoidalx:=\Cmonoidal \coprod_{\mathcal{P}(free^{\otimes}(\Delta[0]))^{\otimes}}\mathcal{P}(\mathcal{L}_{(\Cmonoidal,*)}^{\otimes}(free^{\otimes}(\Delta[0])))^{\otimes}
\end{equation}

\noindent in $CAlg(\Prl)$
\end{defn}

\begin{remark}
Let $\Cmonoidal$ be a small symmetric monoidal $(\infty,1)$-category and let $X$ be an object in $\C$. Again using the tools described in the section \ref{section3-2}, the monoidal structure in $\C$ extends to a monoidal structure in $\mathcal{P}(\C)$ and it makes it a presentable symmetric monoidal $(\infty,1)$-category. It is automatic by the universal properties that the inversion $\mathcal{P}(\C)^{\otimes}[X^{-1}]$ in the setting of presentable $(\infty,1)$-categories is canonically equivalent to $\mathcal{P}(\Lmonoidal(\Cmonoidal))^{\otimes}$.
\end{remark}

As in the \emph{small} context, we analyze the base change with respect to this map. Since $(\Prl)^{\otimes}$ is   compatible with all small colimits, all the machinery related to algebras and modules can be applied. The composition with the canonical map $\Cmonoidal\to \Cmonoidalx$ produces a forgetful functor

\begin{equation}
Mod_{ \Cmonoidalx}(\Prl)\to Mod_{\Cmonoidal}(\Prl)
\end{equation}

\noindent and the base-change functor $\Lpr$ exists, is monoidal and therefore induces an adjunction

\begin{equation}\xymatrix{  CAlg(\Prl)_{\Cmonoidalx/}\ar[r]&  \ar@/_2pc/[l]_{\Lpmonoidal} CAlg(\Prl)_{\Cmonoidal/}}\end{equation}

Our main result is the following:

\begin{prop}
\label{main3}
Let $\Cmonoidal$ be a presentable symmetric monoidal $(\infty,1)$-category. Then
\begin{enumerate}
\item the canonical map

\begin{equation}
CAlg(\Prl)_{\Cmonoidalx/}\to CAlg(\Prl)_{\Cmonoidal/}
\end{equation}
\noindent is fully faithful and its essential image consists of full subcategory spanned by the algebras $\Cmonoidal\to \Dmonoidal$ sending $X$ to an invertible object; In particular we have a canonical equivalence $\Lpmonoidal(\Cmonoidal)\simeq \Cmonoidalx$
\item The canonical map

\begin{equation}
Mod_{ \Cmonoidalx}(\Prl)\to Mod_{\Cmonoidal}(\Prl)
\end{equation}
\noindent is fully faithful and its essential image consists of full subcategory spanned by the presentable  $(\infty,1)$-categories equipped with an action of $\C$ where $X$ acts as an equivalence.
\end{enumerate}

\begin{proof}
Since $(\Prl)^{\otimes}$ is a closed symmetric monoidal $(\infty,1)$-category (see the discussion in the section \ref{section3-4}), it is compatible with all colimits and so the results of the Proposition \ref{main2} can be applied. We prove that $\Cmonoidal\to \Cmonoidalx$  is an epimorphism. Indeed, if $\phi:\Cmonoidal\to \Dmonoidal$ does not send $X$ to an invertible object, by the universal property of the $\Cmonoidalx$ as a pushout, the mapping space $Map_{CAlg(\iCat)_{\Cmonoidal/}}(\Cmonoidalx, \Dmonoidal)$ is empty. Otherwise if $\phi$ sends $X$ to an invertible object, by the universal property of the pushout we have

\begin{equation}Map_{CAlg(\Prl)_{\Cmonoidal/}}(\Cmonoidalx, \Dmonoidal)\simeq Map_{CAlg(\Prl)}(\Cmonoidalx, \Dmonoidal)\end{equation}

\noindent and the last is given by the homotopy pullback of

\begin{equation}
\xymatrix{
&\ar[d]Map_{CAlg(\Prl)}(\mathcal{P}^{\otimes}(\mathcal{L}_{(\Cmonoidal,*)}^{\otimes}(free^{\otimes}(\Delta[0]))), \Dmonoidal)\\
Map_{CAlg(\Prl)}(\Cmonoidal, \Dmonoidal)\ar[r]&Map_{CAlg(\Prl)}(\mathcal{P}^{\otimes}(free^{\otimes}(\Delta[0])), \Dmonoidal)
}
\end{equation}

\noindent which, by the universal property of $\mathcal{P}^{\otimes}(-)$ is equivalent to

\begin{equation} Map_{CAlg(\Prl)}(\Cmonoidal, \Dmonoidal)\times_{Map_{CAlg(\iCat)}(free^{\otimes}(\Delta[0]), \Dmonoidal)} Map_{CAlg(\iCat)}( \mathcal{L}_{(\Cmonoidal,*)}^{\otimes}(free^{\otimes}(\Delta[0])), \Dmonoidal)
\end{equation}

\noindent and we use the fact that $free^{\otimes}(\Delta[0])\to  \mathcal{L}_{(\Cmonoidal,*)}^{\otimes}(free^{\otimes}(\Delta[0]))$ is an epimorphism to conclude the proof.

It remains now to discuss the images.
\begin{enumerate}
\item It is clear by the universal property of the pushout defining $\Cmonoidalx$;
\item If $M$ is in the image, the action of $X$ is clearly invertible. Let $M$ be a $\Cmonoidal$-module with an invertible action of $X$. By repeating exactly the same arguments as in the proof of Prop. \ref{main3} we arrive to a commutative diagram in $\Prl$

\begin{equation}
\xymatrix{
M\ar[r]\ar[d]& M\otimes_{\Cmonoidal} \Lmonoidal(\Cmonoidal)\ar[d]\\
Free_{\Cmonoidal}(M)_{\onefin}= \coprod_{n\geq 0}Sym^n(M)_{\Cmonoidal}\ar[r]^{\sim}& \coprod_{n \geq 0} Sym^n(M\otimes_{\Cmonoidal}\Lmonoidal(\Cmonoidal))_{\Cmonoidal}
}
\end{equation}

\noindent where the vertical maps are the canonical inclusions in the colimit and $Sym^n(-)_{\Cmonoidal}$ is now a colimit in $Mod_{\Cmonoidal}(\Prl)$. We recall now that coproducts in $\Prl$ are computed as products in $\mathcal{P}r^{R}$. Let $u:A\to B$ and $v:X\to Y$ be colimit preserving maps between presentable $(\infty,1)$-categories and assume the coprodut map $u\coprod v:A\coprod X\to B\coprod Y$ is an equivalence. The coproduct $A\coprod X$ is canonically equivalent to the product $A\times X$ and we have commutative diagrams

\begin{equation}
\xymatrix{
A\ar[r]^u\ar[d]_i &B\ar[d]_j\\
A\coprod X\ar[r]^{\sim}_{u\coprod v}&B\coprod Y
}
\end{equation}

\noindent and

\begin{equation}
\label{pinta}
\xymatrix{
A &\ar[l]^{\bar{u}}B\\
A\coprod X=A\times X\ar[u]_p & \ar[l]^{\overline{u\coprod v}}B\coprod Y=B\times Y \ar[u]_q 
}
\end{equation}

\noindent with $i$ and $j$ the canonical inclusions and $p$ and $q$ the projections. The maps in the second diagram are right adjoints to the maps in the first, with $\bar{u\coprod v}\simeq \bar{u}\times \bar{v}$ and therefore $u\coprod v$ and $\bar{u}\times \bar{v}$ are inverses. Since the projections are essentially surjective, the inclusions $i$ and $j$ are fully faithful and we conclude that $u$ has to be fully faithful and $\bar{u}$ is essentially surjective. To conclude the proof is it enough to check that $u$ is essentially surjective or, equivalently (because $u$ is fully faithful), that $\bar{u}$ is fully-faithful. This is the same as saying that for any diagram as in (\ref{pinta}) with $\bar{u}\times \bar{v}$ fully faithful, $\bar{u}$ is necessarily fully faithful. This is true because $Y$ is presentable and therefore has a final object $e$ and since $\bar{v}$ commutes with limits, for any objects $b_0, b_1\in Obj(\A)$ we have 

\begin{eqnarray}
Map_{B}(b_0, b_1)\simeq Map_{B}(b_0, b_1)\times Map_{Y}(e,e)\simeq Map_{A}(\bar{u}(b_0)\bar{u}(b_1))\times Map_{X}(\bar{v}(e),\bar{v}(e))\simeq \\
 \simeq Map_{A}(\bar{u}(b_0)\bar{u}(b_1))
\end{eqnarray}

\end{enumerate}
\end{proof}
\end{prop}

\begin{remark}
The considerations in the Remark \ref{invertingtwoobjects} work, mutatis mutandis, in the presentable setting.
\end{remark}

\subsection{Connection with ordinary Spectra and Stabilization}
\label{section4-2}
In the previous section we proved the existence of a formal inversion of an object $X$ in a symmetric monoidal $(\infty,1)$-category. Our goal for this section is to compare our formal inversion to the more familiar notion of  (ordinary) spectrum-objects.

\subsubsection{Stabilization}
\label{section4-2-1}
 Let $\C$ be an $(\infty,1)$-category and let $G:\C\to \C$ be a functor with a right adjoint $U:\C\to \C$. We define the \emph{stabilization of $\C$ with respect to $(G,U)$ } as the limit in $\iCatbig$ 

\begin{equation}
\xymatrix{Stab_{(G,U)}(\C):=&...\ar[r]^U& \C \ar[r]^U &\C \ar[r]^U &\C  }
\end{equation}

We will refer to the objects of $Stab_{(G,U)}(\C)$ as \emph{spectrum objects in $\C$ with respect to $(G,U)$}. As a limit, we have a canonical functor "evaluation at level $0$" which we will denote as $\Omega^{\infty}_{\C}: Stab_{(G,U)}(\C)\to \C$.

\begin{remark}
\label{stabilizationpresentablecase}
Let $\C$ is a presentable $(\infty,1)$-category together with a colimit preserving functor $G:\C\to \C$. By the Adjoint Functor Theorem we deduce the existence a right adjoint $U$ to $G$. Using the equivalence  $\Prl \simeq (\mathcal{P}r^{R})^{op}$, and the fact that both inclusions $\Prl, \mathcal{P}r^{R} \subseteq \iCatbig$ preserve limits, we conclude that $Stab_{(G,U)}(\C)$ is equivalent to the colimit  of 

\begin{equation}
\xymatrix{\C \ar[r]^G &\C \ar[r]^G &\C \ar[r]^G &...}
\end{equation}

\end{remark}

\begin{example}
\label{usualspectra}
The construction of spectrum objects provides a method to stabilize an $\infty$-category: Let $\C$ be an $(\infty,1)$-category with final object $*$. If $\C$ admits finite limits and colimits we can construct a pair of adjoint functors $\Sigma_{\C}:\C_{*/}\to \C_{*/}$ and $\Omega_{\C}:\C_{*/}\to \C_{*/}$ defined by the formula

\begin{equation}
\Sigma_{\C}(X):= *\coprod_X * 
\end{equation}
\noindent and 
\begin{equation}
\Omega_{\C}(X):= *\times_X *
\end{equation}
\noindent and by the Proposition 1.4.2.24 of \cite{lurie-ha} we can define the \emph{stabilization of $\C$} as the $\infty$-category 

\begin{equation}Stab(\C):=Stab_{(\Sigma_{\C},\Omega_{\C})}(\C_{*/})\end{equation}.

By the Corollary 1.4.2.17 of \cite{lurie-ha}, $Stab(\C)$ is a stable $\infty$-category and by the Corollary 1.4.2.23 of loc.cit, the functor $\Omega^{\infty}:Stab(\C)\to \C$ has a universal property: for any stable 
$(\infty,1)$-category $\D$, the composition with $\Omega^{\infty}$ induces an equivalence

\begin{equation}
Fun'(\D, Stab(\C))\to Fun'(\D, \C)
\end{equation}

\noindent between the full subcategories of functors preserving finite limits. Suppose now that $\C$ is presentable. Since  $\Omega_{\C}$ by definition commutes with all limits and $\Prr$ is closed under limits, $Stab(\C)$ will also be presentable and $\Omega^{\infty}$ will also commute with all limits. Therefore, by the Adjoint Functor Theorem it will admit a left adjoint $\Sigma^{\infty}:\C\to Stab(\C)$. Using the equivalence $\Prl\simeq (\Prr)^{op}$ we find (see Corollary 1.4.4.5 of \cite{lurie-ha}) that $\Sigma^{\infty}$ is characterized by the following universal property: for every stable presentable $(\infty,1)$-category $\D$, the composition with $\Sigma^{\infty}$ induces an equivalence

\begin{equation}
Fun^{L}(Stab(\C), \D)\to Fun^{L}(\C,\D)
\end{equation}

\end{example}

Our goal for the rest of this section is to compare this notion of stabilization to something more familiar. Let us start with some precisions about the notion of limit in $\iCat$. By the Theorem 4.2.4.1 of \cite{lurie-htt}, the stabilization $Stab_{(G,U)}(\C)$ can be computed as an homotopy limit for the tower

\begin{equation}
\xymatrix{...\ar[r]^U& \C^{\natural} \ar[r]^U &\C^{\natural} \ar[r]^U &\C^{\natural}  }
\end{equation}

\noindent in the simplicial model category $\ssetsmarked$ of (big) marked simplicial sets of the Proposition 3.1.3.7 in \cite{lurie-htt} (as a marked simplicial set $\C^{\natural}$ is the notation for the pair $(\C, W)$ where $W$ is the collection of all edges in $\C$ which are equivalences). By the Theorem 3.1.5.1, the cofibrant-fibrant objects in $\ssetsmarked$ are exactly the objects of the form $\C^{\natural}$ with $\C$ a quasi-category and, forgetting the marked edges provides a right-Quillen equivalence from $\ssetsmarked$ to $\ssets$ with the Joyal model structure. Therefore, to obtain a model for the homotopy limit in $\ssetsmarked$ we can instead compute the homotopy limit in $\ssets$ (with the Joyal's structure).

Let now us recall some important results about homotopy limits in model categories. All the following results can be deduced using the Reedy/injective model structures (see \cite{hovey-modelcategories} or the Appendix section of \cite{lurie-htt}) to study diagrams in the underlying model category. The first result is that for a pullback diagram
\begin{equation}
\xymatrix{
&X\ar[d]_f\\
Y\ar[r]_g&Z
}
\end{equation}

\noindent to be an homotopy pullback, it is enough to have $Z$ fibrant and both $f$ and $g$ fibrations. In fact, these conditions can be a bit weakened, and it is enough to have either $(i)$ the three objects are fibrant and one of the maps is a fibration; $(ii)$ if the model category is right-proper, $Z$ is fibrant and one of the maps is a fibration (this last one applies for instance in the model category of simplicial sets with the standard model structure). 
Secondly, we recall another important fact related to the homotopy limits of towers (again, this can be deduced using the Reedy structure). For the homotopy limit of a tower

\begin{equation}
\xymatrix{...\ar[r]^{T_3}& X_2 \ar[r]^{T_2} &X_1 \ar[r]^{T_1} &X_0 }
\end{equation}

\noindent to be given directly by the associated strict limit, it suffices to have the object $X_0$ fibrant and all the maps $T_i$ given by fibrations. In fact, these towers are exactly the fibrant-objects for the Reedy structure and therefore we can replace any tower by a weak-equivalent one in these good conditions. The following result provides a strict model for the homotopy limit of a tower:

\begin{lemma}
\label{lemmahomtowers}
Let $\M$ be a simplicial model category and let $T:\mathbb{N}^{op}\to \M$ be tower in $\M$

\begin{equation}
\xymatrix{...\ar[r]^{T_3}& X_2 \ar[r]^{T_2} &X_1 \ar[r]^{T_1} &X_0 }
\end{equation}

\noindent with each $X_n$ a fibrant object of $\M$. In this case, the homotopy limit $holim_{(\mathbb{N}^{op})} T_n$ is weak-equivalent to the strict pullback of the diagram

\begin{equation}
\xymatrix{
& \prod_n X_n^{\Delta[1]}\ar[d]\\
\prod_n X_n\ar[r]& \prod_n X_n\times X_n
}
\end{equation}

\noindent where the vertical arrow is the fibration\footnote{it is a fibration because of the simplicial assumption} induced by the composition with the cofibration $\partial \Delta[1]\to \Delta[1]$ and the horizontal map is the product of the compositions $\prod_n X_n \to X_n\times X_{n+1}\to X_n\times X_n$
where the last map is the product $Id_{X_n}\times T_n$. Notice that every vertice of the diagram is fibrant.
\begin{proof}
See \cite{jardine-simplicialhomotopytheory}-VI-Lemma 1.12.
\end{proof}
\end{lemma}

Back to our situation, we conclude that the homotopy limit of 

\begin{equation}
\xymatrix{...\ar[r]^U& \C^{\natural} \ar[r]^U &\C^{\natural} \ar[r]^U &\C^{\natural}  }
\end{equation}

\noindent is given by the explicit strict pullback in $\ssetsmarked$ 

\begin{equation}
\xymatrix{
& \prod_n (\C^{\natural})^{\Delta[1]^{\sharp}}\ar[d]\\
\prod_n \C^{\natural}\ar[r]& \prod_n \C^{\natural}\times \C^{\natural}
}
\end{equation}

\noindent where $\Delta[1]^{\sharp}$ is the notation for the simplicial set $\Delta[1]$ with all the edges marked and $(\C^{\natural})^{\Delta[1]^{\sharp}}$ is the coaction of $\Delta[1]^{\sharp}$ on $\C^{\natural}$. In fact, it can be identified with the marked simplicial set $Fun'(\Delta[1], \C)^{\natural}$ where $Fun'(\Delta[1], \C)$ corresponds to  the full-subcategory of $Fun(\Delta[1], \C)$ spanned by the maps $\Delta[1]\to \C$ which are equivalences in $\C$.\\

Let us move further. Consider now a combinatorial simplicial model category $\M$ and let $G:\M\to \M$ be a left simplicial Quillen functor with a right adjoint $U$. Using the technique of the Proposition 5.2.4.6 in \cite{lurie-htt}, from the adjunction data we can extract an endo-adjunction of the underlying $(\infty,1)$-category of $\M$

\begin{equation}
\xymatrix{
N_{\Delta}(\M^{\circ}) \ar@<+.7ex>[r]^{\bar{G}} & N_{\Delta}(\M^{\circ})\ar@<+.7ex>[l]^{\bar{U}}
}
\end{equation}

\noindent where the object $\bar{U}$ can be identified with the composition $Q\circ U$ with $Q$ a simplicial\footnote{(see for instance the Proposition 6.3 of \cite{rezk-schwede-shipley} for the existence of simplicial factorizations in a simplicial cofibrantly generated model category)} cofibrant-replacement functor in $\M$, which we shall fix once and for all. We can consider the stabilization $Stab_{(\bar{G},\bar{U})}(N_{\Delta}(\M^{\circ}))$ given by the homotopy limit

\begin{equation}
\xymatrix{...\ar[r]^{\bar{U}}& N_{\Delta}(\M^{\circ})^{\natural} \ar[r]^{\bar{U}} &N_{\Delta}(\M^{\circ})^{\natural} \ar[r]^{\bar{U}} &N_{\Delta}(\M^{\circ})^{\natural}  }
\end{equation}

\noindent which we now know, is weak-equivalent to the strict pullback of 

\begin{equation}
\label{diaa}
\xymatrix{
& \prod_n Fun'(\Delta[1], N_{\Delta}(\M^{\circ}))^{\natural} \ar[d]\\
\prod_n N_{\Delta}(\M^{\circ})^{\natural} \ar[r]& \prod_n N_{\Delta}(\M^{\circ})^{\natural} \times N_{\Delta}(\M^{\circ})^{\natural} 
}
\end{equation}

\noindent and we know that its underlying simplicial set can be computed as a pullback in $\ssets$ by ignoring all the markings. Moreover, by the Proposition 4.2.4.4 of \cite{lurie-htt}, we have an equivalence of $(\infty,1)$-categories between 

\begin{equation}
\xymatrix{N_{\Delta}((\M^{I})^{\circ})\ar[r]^{\sim}&N_{\Delta}(\M^{\circ})^{\Delta[1]}}
\end{equation}

\noindent where $I$ is the categorical interval and $\M^I$ denotes the category of morphisms in $\M$ endowed with the projective model structure (its cofibrant-fibrant objects are the arrows $f:A\to B$ in $\M$ with both $A$ and $B$ cofibrant-fibrant and $f$ a cofibration in $\M$). Moreover, the equivalence above restricts to a new one between the simplicial nerve of $(\M^{I})^{\circ}_{triv}$ (the full simplicial subcategory of $(\M^{I})^{\circ}$ spanned by the arrows $f:A\to B$ which have $A$ and $B$ cofibrant-fibrant and $f$ a trivial cofibration) and $Fun'(\Delta[1], N_{\Delta}(\M^{\circ}))$. Using this equivalence, we find an equivalence of diagrams

\begin{equation}
\label{diab}
\xymatrix{
& N_{\Delta}((\M^{I})^{\circ}_{triv}) \ar[d]\ar[rr]^{\sim}&&Fun'(\Delta[1], N_{\Delta}(\M^{\circ}))\ar[d]\\
&\prod_n N_{\Delta}(\M^{\circ}) \times N_{\Delta}(\M^{\circ}) \ar[rr]^{id}&&\prod_n N_{\Delta}(\M^{\circ})\times N_{\Delta}(\M^{\circ})\\
\prod_n N_{\Delta}(\M^{\circ})\ar[ru]\ar[rr]^{id}&&\prod_n N_{\Delta}(\M^{\circ})\ar[ru]&
}
\end{equation}

The homotopy pullbacks of both diagrams are weak-equivalent but since the vertical map on the left diagram is no longer a fibration, the associated strict pullback is no longer a model for the homotopy pullback. We continue: the simplicial nerve functor $N_{\Delta}$ is a right-Quillen functor from the category of simplicial categories with the model structure of \cite{bergner-simplicialcategories} to the category of simplicial sets with the Joyal's structure. Therefore, it commutes with homotopy limits and so, the simplicial set underlying the pullback of the previous diagram is in fact given by the simplicial nerve of the homotopy pullback
of 

\begin{equation}
\label{pullbackquasispectra}
\xymatrix{
& \prod_n (\M^I)^{\circ}_{triv}\ar[d]\\
\prod_n \M^{\circ} \ar[r]& \prod_n \M^{\circ}\times\M^{\circ}
}
\end{equation}

\noindent in the model category of simplicial categories.\\

Let us now progress in another direction. We continue with $\M$ a model category together with $G:\M\to \M$ a Quillen left endofuctor with a right adjoint $U$. We recall the construction of a category $Sp^{\mathbb{N}}(\M, G)$ of spectrum objects in $\M$ with respect to $(G,U)$: its objects are the sequences $X= (X_0, X_1,...)$ together with data of morphisms in $\M$,  $\sigma_i:G(X_i)\to X_{i+1}$ (by the adjunction, this is equivalent to the data of morphisms $\bar{\sigma}_i:X_i\to U(X_{i+1})$). A morphism $X\to Y$ is a collection of morphisms in $\M$, $f_i:X_i\to Y_i$, compatible with the structure maps $\sigma_i$. If $\M$ is a cofibrantly generated model category (see Section 2.1 of \cite{hovey-modelcategories}) we can equipped $Sp^{\mathbb{N}}(\M, G)$ with a \emph{stable model structure}. First we define the projective model structure: the weak equivalences are the maps $X\to Y$ which are levelwise weak-equivalences in $\M$ and the fibrations are the levelwise fibrations. The cofibrations are defined by obvious  left-lifting properties. By the Theorem 1.14 of \cite{hovey-spectraandsymmetricspectra} these form a model structure which is again cofibrantly generated and by the Proposition 1.15 of loc. cit,  the cofibrant-fibrant objects are the sequences $(X_0, X_1,...)$ where every $X_i$ is fibrant-cofibrant in $\M$, and the canonical maps $G(X_i)\to G(X_{i+1})$ are cofibrations. We shall write $Sp^{\mathbb{N}}(\M, G)_{proj}$ to denote this model structure. The stable model structure, denoted as $Sp^{\mathbb{N}}(\M, G)_{stable}$, is obtained as a Bousfield localization of the projective structure so that the new fibrant-cofibrant objects are the \emph{$U$-spectra}, meaning, the sequences $(X_0, X_1,...)$ which are fibrant-cofibrant for the projective model structure and such that for every $i$, the adjoint of the structure map $\sigma_i$, $X_i\to U(X_{i+1})$ is a weak-equivalence. (See Theorem 3.4 of \cite{hovey-spectraandsymmetricspectra}).

By the Theorem 5.7 of \cite{hovey-spectraandsymmetricspectra}, this construction also works if we assume $\M$ to be a combinatorial simplicial model category and $G$ to be a left simplicial Quillen functor \footnote{The reader is left with the easy exercise of checking that the following conditions are equivalent for a Quillen adjunction $(G,U)$ between simplicial model categories: $(i)$ $G$ is enriched; $(ii)$ $G$ is compatible with the simplicial action, meaning that for any simplicial set $K$ and any object $X$ we have $G(K\otimes X)\simeq K\otimes G(X)$; $(iii)$ $U$ is compatible with the coaction, meaning that any for any simplicial set $K$ and object $Y$ we have $U(Y^K)\simeq U(Y)^K$; $(iv)$ $U$ is enriched.}. In this case, $Sp^{\mathbb{N}}(\M, G)$ (both with the stable and the projective structures) is again a combinatorial simplicial model category with mapping spaces given by the pullback

\begin{equation}
\label{formulamappingspacesspectrumobjects}
\xymatrix{
&\prod_n Map_{\M}(X_i, Y_i)\ar[d]\\
\prod_n Map_{\M}(X_i, Y_i)\ar[r]&\prod_n Map_{\M}(X_i, U(Y_{i+1}))
}
\end{equation}

\noindent where
\begin{itemize}
\item the horizontal map is the product of the maps 

\begin{equation}
Map_{\M}(X_i, Y_i)\to Map_{\M}(X_i, U(Y_{i+1}))
\end{equation}

\noindent induced by the composition with the adjoint $\bar{\sigma}_i:Y_i\to U(Y_{i+1})$;

\item 
The vertical map is the product of the compositions

\begin{equation}
Map_{\M}(X_{i+1}, Y_{i+1})\to Map_{\M}(U(X_{i+1}),U(Y_{i+1}))\to Map_{\M}(X_i, U(Y_{i+1}))
\end{equation}

\noindent where the first map is induced by $U$ and the second map is the composition with $X_i\to U(X_{i+1})$.
\end{itemize}

Its points correspond to the collections $f=\{f_i\}_{i\in \mathbb{N}}$ for which the diagrams

\begin{equation}
\xymatrix{
X_i\ar[d]^{f_i}\ar[r]& U(X_{i+1})\ar[d]^{U(f_{i+1})}\\
Y_i \ar[r] & U(Y_{i+1})
}
\end{equation}

\noindent commute.\\

By the Proposition \ref{dwyer-kan}, the underlying $(\infty,1)$-categories of $Sp^{\mathbb{N}}(\M, G)_{proj}$ and $Sp^{\mathbb{N}}(\M, G)_{stable}$ are given, respectively by the simplicial nerves $N_{\Delta}((Sp^{\mathbb{N}}(\M, G)_{proj})^{\circ})$ and $N_{\Delta}((Sp^{\mathbb{N}}(\M, G)_{stable})^{\circ})$ and by construction the last appears as the full reflexive subcategory of the first, spanned by the $U$-spectrum objects.\\

Up to this point we have two different notions of spectrum-objects. Of course they are related. To understand the relation we observe first that $Sp^{\mathbb{N}}(\M, G)$ fits in a strict pullback diagram of simplicial categories

\begin{equation}
\label{diagramspk0}
\xymatrix{
Sp^{\mathbb{N}}(\M, G)\ar[r]\ar[d]& \prod_n (\M^I)\ar[d]\\
\prod_n \M \ar[r]& \prod_n \M\times\M
}
\end{equation}

\noindent where the top horizontal map is the product of all maps of the form $(X_i)_{i\in \mathbb{N}}\mapsto (X_i\to U(X_{i+1}))$ and the vertical-left map sends a spectrum-object to its underlying sequence of objects.
The right-vertical map sends a morphism in $\M$ to its respective source and target and the lower-horizontal map is the product of the compositions $(X_i)_{i\in \mathbb{N}}\mapsto (X_i, X_{i+1})\mapsto (X_1, U(X_{i+1}))$. All the maps in this diagram are compatible with the simplicial enrichment. We fabricate a new diagram which culminates in (\ref{pullbackquasispectra}). 

\begin{equation}
\xymatrix{
&&&&\prod_n (\M^I)^{\circ}_{triv} \ar@{^{(}->}[d]\\
&&&&\prod_n (\M^I)^{\circ} \ar@{^{(}->}[d]\\
Sp^{\mathbb{N}}(\M, G)^{\circ}_{stable}\ar@{-->}[rrrruu]_f\ar@{^{(}->}[r]\ar@{-->}[rd]_{a'}&Sp^{\mathbb{N}}(\M, G)^{\circ}_{proj}\ar@{-->}[rrru]_e\ar@{^{(}->}[r]\ar@{-->}[d]_a&Sp^{\mathbb{N}}(\M, G)\ar[r]_x\ar[d]_y& \prod_n (\M^I)\ar[d]_z\ar[r]^b&\ar[d]\prod_n (\M^I)\\
&\prod_n \M^{\circ} \ar@{^{(}->}[r]\ar@{-->}[rrd]_c&\prod_n \M\ar[r]_w& \prod_n \M\times\M& \prod_n \M^{\circ}\times\M^{\circ}\\
&&&\prod_n \M^{\circ}\times\M^{fib}\ar@{^{(}->}[u]\ar[ru]_d&
}
\end{equation}

\noindent where the maps

\begin{enumerate}
\item $x$, $y$, $z$, $w$ are the maps in the diagram (\ref{diagramspk0});
\item $a$ is the restriction of the projection $Sp^{\mathbb{N}}(\M, G)\to \prod_n \M$ (it is well-defined because the cofibrant-fibrant objects in $Sp^{\mathbb{N}}(\M, G)$ are supported on sequences of cofibrant-fibrant objects in $\M$) 
\item $a'$ is the composition of $a$ with the canonical inclusion;
\item $b$ is the product of the compositions

\begin{equation}\xymatrix{\M^I\ar[r]^{Q} &\M^I\times \M^I\ar[r]&\M^I}\end{equation}

\noindent where $\Q$ is the machine associated to our chosen simplicial functorial factorization of the form "(cofibration, trivial fibration)" (
sending a morphism $f:A\to B$ in $\M$ to the pair $(u:A\to X, v:X\to Y)$ with $u$ a cofibration and $v$ a trivial fibration) and the second arrow is the projection in the first coordinate.

\item $c$ is induced by composition of $w$ with the canonical inclusion. Given a sequence of cofibrant-fibrant objects $(X_i)_{i\in \mathbb{N}}$, we have $w((X_i)_{i\in \mathbb{N}})= (X_i, U(X_i+1))_{i\in \mathbb{N}}$ with $X_i$ fibrant-cofibrant and $U(X_{i+1})$ fibrant (because $U$ is a right-Quillen functor). Therefore, the composition factors through $\prod_n \M^{\circ}\times\M^{fib}$ and $c$ is well-defined;
\item To obtain $d$, we consider first the composition

\begin{equation}
\xymatrix{
\M^{\circ}\times \M \ar[r]& \M^{\circ}\times \M^I\ar[r]^(.4){id\times Q}&\M^{\circ}\times (\M^I\times \M^I)   \ar[r]&\M^{\circ}\times \M^I\ar[r]& \M^{\circ}\times \M
}
\end{equation}

\noindent where the first arrow sends $(X,Y)\mapsto (X, \emptyset\to X)$, the third arrow is induced by the projection of $\M^I\times \M^I\to \M^I$ on the first coordinate and the last arrow is induced by taking the source. All together, this composition is sending a pair $(X,Y)$ to the pair $(X, Q(Y))$ with $Q$ a cofibrant-replacement of $Y$ using the same factorization device of the item $(4)$. In particular, if $Y$ is already fibrant, $Q(Y)$ will be cofibrant-fibrant and we have a dotted arrow

\begin{equation}
\xymatrix{
\M^{\circ}\times \M\ar[r]&\M^{\circ}\times \M\\
\M^{\circ}\times \M^{fib}\ar@{^{(}->}[u]\ar@{-->}[r]& \M^{\circ}\times \M^{\circ}\ar@{^{(}->}[u]
}
\end{equation}

\noindent rendering the diagram commutative.

By definition, $d$ is the product of all these dotted maps;
\item $e$ is the map induced by composing $b\circ x$ with the canonical inclusion and it is well-defined for the reasons given also in $(2)$;
\item $f$ is deduced from $e$ by restricting to the $U$-spectra objects: If $(X_i)_{i\in \mathbb{N}}$ is a $U$-spectra, the canonical maps $X_i\to U(X_{i+1})$ are weak-equivalences and therefore, when we perform the factorization  encoded in the composition $b\circ x$, the first map is necessarily a trivial cofibration and therefore $f$ factors through $\prod_n (\M^I)^{\circ}_{triv}$.
\end{enumerate}

Finally, the fact that everything commutes is obvious from the definition of factorization system. All together, we found a commutative diagram

\begin{equation}
\label{pinta}
\xymatrix{
Sp^{\mathbb{N}}(\M, G)^{\circ}_{stable}\ar[d]\ar[r]& \prod_n (\M^I)^{\circ}_{triv}\ar[d]\\
\prod_n \M^{\circ} \ar[r]& \prod_n \M^{\circ}\times\M^{\circ}
}
\end{equation}

In summary, the up horizontal map sends a $U$-spectra $X=(X_i)_{i\in \mathbb{N}}$ to the list of trivial cofibrations $(X_i\to Q(U(X_{i+1})))_{i\in \mathbb{N}}$ and the left-vertical map sends $X$ to its underlying sequence of cofibrant-fibrant objects. By considering the simplicial nerve of the diagram above and using the equivalence of diagrams in (\ref{diab}), we obtain, using the universal property of the strict pullback, a map

\begin{equation}
\phi:N_{\Delta}((Sp^{\mathbb{N}}(\M, G)_{stable})^{\circ})\to Stab_{(\bar{G},\bar{U})}(N_{\Delta}(\M^{\circ}))
\end{equation}

\noindent where we identify $Stab_{(\bar{G},\bar{U})}(N_{\Delta}(\M^{\circ}))$ with the strict pullback of the diagram (\ref{diaa}).\\

The following result clarifies this already long story:

\begin{prop}
\label{main4}
Let $\M$ be a combinatorial simplicial model category and let $G:\M\to \M$ be a left simplicial Quillen functor with a right adjoint $U$. Let $Sp^{\mathbb{N}}(\M, G)^{stable}$ denote the combinatorial simplicial model category of  \cite{hovey-spectraandsymmetricspectra} equipped the stable model structure. Then, the canonical map induced by the previous commutative diagram

\begin{equation}
\phi:N_{\Delta}((Sp^{\mathbb{N}}(\M, G)_{stable})^{\circ})\to Stab_{(\bar{G},\bar{U})}(N_{\Delta}(\M^{\circ}))
\end{equation}

\noindent is an equivalence of $(\infty,1)$-categories.

\begin{proof}

We will prove this by checking the map is essentially surjective and fully-faithful. We start with the essential surjectivity. For that we can restrict ourselves to study of the map induced between the maximal $\infty$-groupoids (Kan-complexes) on both sides . 

\begin{equation}
N_{\Delta}((Sp^{\mathbb{N}}(\M, G)_{stable})^{\circ})^{\simeq}\to Stab_{(\bar{G},\bar{U})}(N_{\Delta}(\M^{\circ}))^{\simeq}
\end{equation}

To conclude the essential surjectivity it suffices to check that the map induced between the $\pi_0$'s 

\begin{equation}
\pi_0(N_{\Delta}((Sp^{\mathbb{N}}(\M, G)_{stable})^{\circ})^{\simeq})\to \pi_0(Stab_{(\bar{G},\bar{U})}(N_{\Delta}(\M^{\circ}))^{\simeq})
\end{equation}

\noindent is surjective. We start by analyzing the right-side. First, the operation $(-)^{\simeq}$ commutes with homotopy limits. To see this, notice that both the $(\infty,1)$-category of homotopy types $\Spaces$ and the $(\infty,1)$-category of small $(\infty,1)$-categories $\iCat$ are presentable. The combinatorial simplicial model category of simplicial sets with the Quillen structure is a strict model for the first and $\ssetsmarked$ models the second. By combining the Theorem 3.1.5.1 and the Proposition 5.2.4.6 of \cite{lurie-htt}, the inclusion $\Spaces \subseteq \iCat$ is in fact a Bousfield (a.k.a reflexive) localization and its the left adjoint can be understood (by its universal property) as the process of inverting all the morphisms. By combining the Proposition 3.3.2.5 and the Corollaries 3.3.4.3 and 3.3.4.6 of \cite{lurie-htt}, we deduce that the inclusion $\Spaces\subseteq \iCat$ commutes with colimits. Since $\Spaces$ and $\iCat$ are presentable, by the Adjoint Functor Theorem (see Corollary 5.5.2.9 of \cite{lurie-htt}), the inclusion $\Spaces\subseteq \iCat$ admits a right adjoint which, by its universal property can be identified with the operation $(-)^{\simeq}$.
 An immediate application of this fact is that $\pi_0(Stab_{(\bar{G},\bar{U})}(N_{\Delta}(\M^{\circ}))^{\simeq})$ is in bijection with the $\pi_0$ of the homotopy limit of the tower of Kan-complexes

\begin{equation}
\xymatrix{...\ar[r]^{\bar{U}}& N_{\Delta}(\M^{\circ})^{\simeq} \ar[r]^{\bar{U}} &N_{\Delta}(\M^{\circ})^{\simeq} \ar[r]^{\bar{U}} &N_{\Delta}(\M^{\circ})^{\simeq}  }
\end{equation}

Using the Reedy structure (on $\ssets$ with the Quillen structure), we can find a morphism of towers

\begin{equation}
\xymatrix{...\ar[r]^{\bar{U}}& \ar[d]N_{\Delta}(\M^{\circ})^{\simeq} \ar[r]^{\bar{U}} &\ar[d]N_{\Delta}(\M^{\circ})^{\simeq} \ar[r]^{\bar{U}} &\ar[d]N_{\Delta}(\M^{\circ})^{\simeq}\\
...\ar[r]&T_2\ar[r]&T_1\ar[r]&T_0 }
\end{equation}

\noindent where the vertical maps are weak-equivalences of simplicial sets for the Quillen structure, every object is again a Kan-complex but this time the maps in the lower tower are fibrations. By the nature of the weak-equivalences, this morphism of diagrams becomes an isomorphism at the level of the $\pi_0$'s

\begin{equation}
\xymatrix{...\ar[r]& \pi_0(N_{\Delta}(\M^{\circ})^{\simeq}) \ar[r]^{\pi_0(\bar{U})} \ar[d]^{\sim}&\ar[d]^{\sim}\pi_0(N_{\Delta}(\M^{\circ})^{\simeq}) \ar[r]^{\pi_0(\bar{U})} &\ar[d]^{\sim}\pi_0(N_{\Delta}(\M^{\circ})^{\simeq})\\
...\ar[r]&\pi_0(T_2)\ar[r]&\pi_0(T_1)\ar[r]&\pi_0(T_0) }
\end{equation}

\noindent and therefore the limits $lim_{\mathbb{N}^{op}} \pi_0(N_{\Delta}(\M^{\circ})^{\simeq})$ and $lim_{\mathbb{N}^{op}} \pi_0(T_i)$ are isomorphic. Finally, using the Milnor's exact sequence associated to a tower of fibrations together with the fact that fibrations of simplicial sets are surjective (see Proposition VI-2.15 and Proposition VI-2.12-2  in \cite{jardine-simplicialhomotopytheory}) we deduce an isomorphism

\begin{equation}
\pi_0(lim_{\mathbb{N}^{op}} T_i)\simeq lim_{\mathbb{N}^{op}} \pi_0(T_i)
\end{equation}

\noindent and by combining everything we have

\begin{equation}
\pi_0(Stab_{(\bar{G},\bar{U})}(N_{\Delta}(\M^{\circ}))^{\simeq})\simeq lim_{\mathbb{N}^{op}} \pi_0(N_{\Delta}(\M^{\circ})^{\simeq})
\end{equation}

\noindent where set on the right is the strict limit of the tower of sets

\begin{equation}
\xymatrix{...\ar[r]& \pi_0(N_{\Delta}(\M^{\circ})^{\simeq}) \ar[r]^{\pi_0(\bar{U})} &\pi_0(N_{\Delta}(\M^{\circ})^{\simeq}) \ar[r]^{\pi_0(\bar{U})} &\pi_0(N_{\Delta}(\M^{\circ})^{\simeq})
}
\end{equation}

\noindent and since $\bar{U}$ can be identified with $Q\circ U$, the elements of the last can be presented as sequences $([X_i])_{i\in \mathbb{N}}$ with each $[X_i]$ an equivalence class of an object $X_i$ in $N_{\Delta}(\M^{\circ})$, satisfying $[QU(X_{i+1})]=[X_i]$, which is the same as stating the existence of an equivalence in $N_{\Delta}(\M^{\circ})$ between $X_i$ and $QU(X_{i+1})$. Since we are dealing with cofibrant-fibrant objects, we can find an actual homotopy equivalence $X_i\to QU(X_{i+1})$ and by choosing a representative for each $[X_i]$ together with composition maps $X_i\to QU(X_{i+1})\to U(X_{i+1})$ we retrieve a $U$-spectra. This proves that the map is essentially surjective. \\

It remains to prove $\phi$ is fully-faithful. Given two $U$-spectrum objects $X=(X_i)_{i\in \mathbb{N}}$ and $Y=(Y_i)_{i\in \mathbb{N}}$, the mapping space in $N_{\Delta}((Sp^{\mathbb{N}}(\M, G)_{stable})^{\circ})$ between $X$ and $Y$ is given by the pullback\footnote{see the formula (\ref{formulamappingspacesspectrumobjects})} of the diagram

\begin{equation}
\label{diaA}
\xymatrix{
&\prod_n Map_{\M}(X_i, Y_i)\ar[d]\\
\prod_n Map_{\M}(X_i, Y_i)\ar[r]&\prod_n Map_{\M}(X_i, U(Y_{i+1}))
}
\end{equation}

All vertices in this diagram are given by Kan-complexes (because $\M$ is a simplicial model category, each $Y_i$ and $X_i$ is cofibrant-fibrant and $U$ is right-Quillen) and the vertical map is a fibration. Indeed, it can be identified with product of the compositions

\begin{equation}
Map_{\M}(X_{i+1}, Y_{i+1})\to Map_{\M}(G(X_i), Y_{i+1}))\simeq Map_{\M}(X_i, U(Y_{i+1}))
\end{equation}

\noindent where the last isomorphism follows from the adjunction data and the first map is the fibration  induced by the composition with structure maps $G(X_i)\to X_{i+1}$ of $X$ (which are cofibrations because $X$ is a $U$-spectra). Therefore, the pullback square is an homotopy pullback.

At the same time, because of the equivalence of diagrams (\ref{diab}) the mapping spaces in $Stab_{(\bar{G},\bar{U})}(N_{\Delta}(\M^{\circ}))$ between the image of $X$ and the image of $Y$ can obtained\footnote{The mapping spaces in the homotopy pullback are the homotopy pullback of the mapping spaces} as the homotopy pullback of 

\begin{equation}
\label{diaB}
\xymatrix{
&\prod_n Map_{\M}(X_i, Y_i)\ar[d]^U\\
&\prod_n Map_{\M}(U(X_i), U(Y_i))\ar[d]^Q\\
&\prod_n Map_{\M}(QU(X_{i}), QU(Y_{i}))\ar[d]\\
\prod_n Map_{\M}(X_i, Y_i)\ar[r]&\prod_n Map_{\M}(X_i, QU(Y_{i+1}))
}
\end{equation}

To conclude the proof it suffices to produce a weak-equivalence between the formulas. Indeed, we produce a map from the diagram (\ref{diaB}) to the diagram (\ref{diaA}), using the identity maps in the outer vertices and in the corner we use the product of the maps induced by the composition with the canonical map $QU(Y_{i+1})\to U(Y_{i+1})$.
 
\begin{equation}
Map_{\M}(X_i, QU(Y_{i+1}))\to  Map_{\M}(X_i, U(Y_{i+1}))
\end{equation}

Of course, this map is a trivial fibration: $\M$ is a simplicial model category, $X_i$ is cofibrant and $QU(Y_{i+1})\to U(Y_{i+1})$ is a trivial fibration.

\end{proof}
\end{prop}

In the situation of the Proposition \ref{main4}, with $\M$ a combinatorial simplicial model category and $G$ a left-simplicial Quillen functor, we know that $Sp^{\mathbb{N}}(\M,G)^{\circ}_{stable}$ is again combinatorial and simplicial and so, both the underlying $(\infty,1)$-categories $N_{\Delta}(\M^{\circ})$ and $N_{\Delta}(Sp^{\mathbb{N}}(\M,G)^{\circ}_{stable})$ are presentable (see the Proposition A.3.7.6 of \cite{lurie-htt}). Finally, using the Remark \ref{stabilizationpresentablecase} we deduce the existence of canonical equivalence between $N_{\Delta}(Sp^{\mathbb{N}}(\M,G)^{\circ}_{stable})$ and the colimit of the sequence

\begin{equation}
\xymatrix{ N_{\Delta}(\M^{\circ}) \ar[r]^{\bar{G}} &N_{\Delta}(\M^{\circ}) \ar[r]^{\bar{G}} &...  }
\end{equation}

\subsubsection{Stabilization and Symmetric Monoidal Structures}
Let us proceed. Our goal now is to compare the construction of spectra with the formal inversion $\C[X]^{\otimes}$. The idea of a relation between the two comes from the following classical theorem:

\begin{thm}(see Theorem 4.3 of \cite{voevodsky-icm})\\
Let $\C$ be a symmetric monoidal category with tensor product $\otimes$ and unit $\mathbf{1}$. Let $X$ be an object in $\C$. Let $Stab_X(\C)$ denote the colimit of the sequence 

\begin{equation}
\xymatrix{...\ar[r]^{X\otimes-} &\C\ar[r]^{ X\otimes-} &\C\ar[r]^{X\otimes-} &\C\ar[r]^{X\otimes-} &...}
\end{equation}

\noindent in $\Cat$ (up to equivalence). Then, if the action of the cyclic permutation on $X\otimes X\otimes X$ becomes an identity map $\C$ after tensoring with $X$ an appropriate amount of times (which is the same as saying it is the identity map in $Stab_X(\C)$) the category $Stab_X(\C)$ admits a canonical symmetric monoidal structure and the canonical functor $\C\to Stab_X(\C)$ is monoidal, sends $X$ to an invertible object and is universal with respect to this property.

\begin{proof}

We can identify the colimit of the sequence with the category of pairs $(A,n)$ where $A$ is an object in $\C$ and $n$ an integer. The hom-sets are given by the formula

\begin{equation}
Hom_{Stab_X(\C)}((A,n),(B,m))= colim_{(k>-n,-m)}Hom_{\C}(X^{n+k}\otimes A,X^{m+k}\otimes B) 
\end{equation}
The composition is the obvious one. There is a natural wannabe symmetric monoidal structure on $Stab_X(\C)$, namely, the one given by the formula $(A,n)\wedge (B,m):= (A\otimes B, n+m)$. When we try to define this operation on the level of morphisms, we find the need for our hypothesis on $X$: Let $[f]:(Z,n)\to (Y,m)$ and $[g]:(A,a)\to (B,b)$ be two maps in $Stab_X(\C)$. Let $f:X^{\alpha+ n}(Z)\to X^{\alpha+ m}(Y)$ and $g:X^{\gamma+ a}(A)\to X^{\gamma+ b}(B)$ be representatives for $[f]$ and $[g]$. Their product has to be a map in $Stab_X(\C)$ represented by some map in $\C$, $X^{n+a+ k}(Z\otimes A)\to X^{m+ b+ k}(Y\otimes B)$. In order to define this map from the data of $f$ and $g$ we have to make a choice of which copies of $X$ should be kept together with $Z$ and which should be kept with $A$. These choices will differ by some permutation of the factors of $X$, namely, for each two choices there will be a commutative diagram

\begin{equation}
\xymatrix{
X^{n+a+ \alpha+ \gamma}(Z\otimes A)\ar[d]^{\text{Use Choice 1}} \ar[rr]_{\exists \sigma\in \Sigma_{n+a+ \alpha+ \gamma}} && X^{n+a+ \alpha+ \gamma}(Z\otimes A) \ar[d]^{\text{Use Choice 2}} \\
X^{m+ b+ \alpha+\gamma}(Y\otimes B) \ar[rr]_{\exists \sigma'\in \Sigma_{m+b+ \alpha+ \gamma}}&& X^{m+ b+ \alpha+\gamma}(Y\otimes B)
}
\end{equation}

The reason why we cannot adopt one choice once and for all, is because if we choose different representatives for $f$ and $g$, for instance, $id_X\otimes f$ and $id_X\otimes g$, we will need a permutation of factors to make the second result equivalent to the one given by our first choice. Therefore, in order to have a well-define product map, it is sufficient to ask for the  different permutations of the $p$-fold product $X^p$ to become equal after tensoring with the identity of $X$ an appropriate amount of times. In other words, they should become an identity map. For this, it is sufficient to ask for the action of the cyclic permutation $(123)$ on $X^3$ to become the identity. This is because any permutation of $p$-factors can be built from permutations of $3$-factors, by composition.\\

It is now an exercise to check that this operation, together with the object $(1,0)$ and the natural associators and commutators induced from $\C$, endow $Stab_X(\C)$ with the structure of a symmetric monoidal category. Moreover, one can also check that the object $(X,0)$ becomes invertible, with inverse given by $(\mathbf{1},-1)$.\\

The fact that $Stab_X(\C)$ when endowed with this symmetric monoidal structure is universal with respect to the inversion of $X$ comes from fact that any monoidal functor $f:\C\to \D$ sending $X$ to an invertible element produces a morphism of diagrams (in the homotopy category of (small) categories)
\begin{equation}
\xymatrix{
...\ar[rr]^{X\otimes-} &&\C\ar[rr]^{X\otimes-}\ar[d]^f &&\C\ar[rr]^{X\otimes-}\ar[d]^f &&\C\ar[rr]^{X\otimes-}\ar[d]^f &&...\\
...\ar[rr]^{f(X)\otimes-} &&\D\ar[rr]^{f(X)\otimes-} &&\D\ar[rr]^{ f(X)\otimes-} &&\D\ar[rr]^{f(X)\otimes-} &&...
}
\end{equation}

\noindent together with an associated colimit map $Stab_X(\C)\to Stab_X(\D)$. To conclude the proof we need two observations: let $\D$ be a symmetric monoidal category and let $U$  be an invertible object in $\D$, then we the following facts:

\begin{enumerate}
\item $U$ automatically satisfies the cocyle condition. This follows from a more general fact. If $U$ is an invertible object in $\C$, we can prove that the group of automorphisms of $U$ in $\C$ is necessarily abelian. This follows from the existence of an isomorphism $U\simeq U\otimes U^*\otimes U$ and the fact that any map $f:U\to U$ can either be written as $f\otimes id_X\otimes id_X$ or $id_X\otimes id_X\otimes f$. Given two maps $f$ and $g$ we can write 

\begin{equation}g\circ f= (g\otimes id_X\otimes id_X)\circ (id_X\otimes id_X\otimes f) = (f\otimes id_X\otimes id_X)\circ (id_X\otimes id_X\otimes g)= f\circ g\end{equation}

The fact that $U$ satisfies the cocycle condition is an immediate consequence, because the actions of the transpositions $(i,i+1)$ and $(i+1, i+2)$ have to commute and we have the identity $((i,i+1)\circ (i+1, i+2))^3=id$.

\item the functor $U\otimes -: \D\to \D$ is an equivalence of categories with inverse given by multiplication with $U^*$, the inverse of $U$ in $\D$. In this case, multiplications by the powers of $U$ and $U^*$ make $\D$ a cocone over the stabilizing diagram. It is an easy observation that the canonical colimit $Stab_U(\D)\to\D$ (which can be described by the formula $(A,n)\mapsto (U^*)^n\otimes A$) is an equivalence. Moreover, since $U$ satisfies the cocycle condition (following the previous item), $Stab_U(\D)$ comes naturally equipped with a symmetric monoidal structure and we can check that the colimit map is monoidal. Under these circumstances, any monoidal functor $f:\C\to \D$  with $f(X)$ invertible, gives a canonical colimit map $Stab_X(\C)\to Stab_{f(X)}(\D)\simeq \D$. It is an observation that this map is monoidal under our hypothesis on $X$. This implies the universal property.

\end{enumerate}

\end{proof}
\end{thm}

\begin{remark}
The condition on $X$ appearing in the previous result is trivially satisfied if the action of the cyclic permutation $(X\otimes X\otimes X)^{(1,2,3)}$ is already an identity map in $\C$. For instance, this particular situation holds when $\C$ is the pointed $\mathbb{A}^1$-homotopy category and $X$ is $\mathbb{P}^1$ (See Theorem 4.3 and Lemma 4.4 of \cite{voevodsky-icm}).
\end{remark}

Our goal now is to find an analogue for the previous theorem in the context of symmetric monoidal $(\infty,1)$-categories.

\begin{defn}
\label{symmetric}
Let $\Cmonoidal$ be a symmetric monoidal $(\infty,1)$-category and let $X$ be an object in $\C$. We say that $X$ is \emph{symmetric} if there is a $2$-equivalence in $\C$ between the cyclic permutation $\sigma:(X\otimes X\otimes X)^{(1,2,3)}$ and the identity map of $X\otimes X\otimes X$. In other words, we demand the existence of a $2$-cell in $\C$

\begin{equation}
\xymatrix{ 
X\otimes X\otimes X \ar[r]^{\sigma} \ar[d]_{id} & X\otimes X\otimes X\\
  X\otimes X\otimes X \ar[ur]_{id}|*{}="M" \ar@{=>}[u];"M"                          }
\end{equation} 

\noindent providing an homotopy between the cyclic permutation and the identity. In fact, since $\C$ is a quasi-category, this condition is equivalent to ask for $\sigma$ to be equal to the identity in $h(\C)$.
\end{defn}

This notion of symmetry is well behaved under equivalences. Moreover, it is immediate that monoidal functors map symmetric objects to symmetric objects.
 
\begin{remark}
\label{remarksymmetric}
Let $\V$ be a symmetric monoidal model category with a cofibrant unit $1$. Recall that a unit interval $I$ is a cylinder object for the unit of the monoidal structure $I:=C(1)$, together with a map $I\otimes I\to I$ such that the diagrams

\begin{equation}
\xymatrix{
1\otimes I\simeq I \ar[r]^{\pi}\ar[d]_{\partial_0\otimes Id_I}&1\ar[d]_{\partial_0}\\
I\otimes I\ar[r]& I
}
\end{equation}

\begin{equation}
\xymatrix{
I\otimes 1 \simeq I\ar[r]^{\pi}\ar[d]_{Id_I\otimes \partial_0}&1\ar[d]_{\partial_0}\\
I\otimes I\ar[r]& I
}
\end{equation}

\noindent and

\begin{equation}
\xymatrix{
I\otimes 1\simeq I\ar[rd]_{Id_I}\ar[d]_{\partial_1\otimes Id_I}&\\
I\otimes I\ar[r]& I
}
\end{equation}

\noindent commute, where $\partial_0,\partial_1:1\to I$ and $\pi:I\to 1$ are the maps providing $I$ with a structure of cylinder object. 

Recall also that two maps $f,g:A\to B$ are said to be homotopic with respect to a unit interval $I$ if there is a map $H:A\otimes I\to B$ rendering the diagram commutative

\begin{equation}
\xymatrix{
A\simeq A\otimes 1\ar[dr]_{Id_A\otimes \partial_0}\ar@/^/[drr]^{f} &&\\
&A\otimes I\ar[r]^H& B\\
A\simeq A\otimes 1\ar[ru]^{id_A\otimes \partial_1}\ar@/_/[urr]_{g}&&
}
\end{equation}

In \cite{hovey-spectraandsymmetricspectra}-Def.9.2, the author defines an object $X$ to be symmetric if it is cofibrant and if there is a \emph{unit interval} $I$, together with an homotopy

\begin{equation}
H:X\otimes X\otimes X \otimes I\to X\otimes X\otimes X
\end{equation}

\noindent between the cyclic permutation $\sigma$ and the identity map. 
We observe that if an object $X$ is symmetric in the sense of \cite{hovey-spectraandsymmetricspectra} then it is symmetric as an object in the underlying symmetric monoidal $(\infty,1)$-category of $\V$ in the sense of the Definition \ref{symmetric}. Indeed, since $\V$ is a symmetric monoidal model category with a cofibrant unit, the full subcategory $\V^c$ of cofibrant objects is closed under the tensor product and therefore inherits a monoidal structure, which moreover preserves the weak-equivalences. In the Section \ref{link2} we used this fact to define the underlying symmetric monoidal $(\infty,1)$-category of $\V$, $N((\V^c)^{\otimes})[W_c^{-1}]$ (see Section \ref{link2} for the notations). Its underlying $(\infty,1)$-category is $N(\V^c)[W^{-1}]$ and its homotopy category is the classical localization in $\Cat$. Moreover, it comes canonically equipped with a monoidal functor $L:N((\V^c)^{\otimes})\to N((\V^c)^{\otimes})[W_c^{-1}]$. Now, if $X$ is symmetric in $\V$ in the sense of \cite{hovey-spectraandsymmetricspectra}, the homotopy $H$ forces $\sigma$ to become the identity in $h(N(\V^c)[W^{-1}])$ (because the classical localization functor is monoidal and the map $I\to 1$ is a weak-equivalence). The conclusion now follows from the commutativity of the diagram induced by the unit of the adjunction $(h,N)$

\begin{equation}
\xymatrix{
N(\V^c)\ar[r]\ar[d]_{\sim}& N(\V^c)[W^{-1}]\ar[d]\\
N(h(N(\V^c)))\ar[r]& N(h(N(\V^c)[W^{-1}]))
}
\end{equation}

\noindent and the fact that the both horizontal arrows are monoidal and therefore send the cyclic permutation of the monoidal structure in $\V$ to the cyclic permutation associated to the monoidal structure in  $N((\V^c)^{\otimes})[W_c^{-1}]$.
\end{remark}

We now come to the generalization of the Theorem 9.3 of \cite{hovey-spectraandsymmetricspectra}. The following results relate our formal inversion of an object to the construction of spectrum objects.\\

\begin{remark}
\label{remarkmain411}
Let $\Cmonoidal$ be a small monoidal $(\infty,1)$-category and let $\overline{M}$ be an object in $Mod_{\Cmonoidal}(\iCat)$ (which we will understand as a left-module). Since $\iCat$ admits classifying object for endomorphisms (given by the categories of endofunctors), the data of $\overline{M}$ is equivalent to the data an $(\infty,1)$-category $M:=\overline{M}(\textbf{m})$ together with a monoidal functor $T^{\otimes}:\Cmonoidal\to End(M)^{\otimes}$ where the last is endowed with the strictly associative monoidal structure induced by the composition of functors. 

If $X$ is an object in $\C$, the endofunctor $T(X):M\to M$ corresponds to the action of $X$ in $M$ by means of the operation $\C\times M\to M$ encoded in the module-structure. We will call it the \emph{multiplication by $X$}.

Notice that if the monoidal structure $\Cmonoidal$ is symmetric, the map $T(X)$ extends to a morphism of module-objects $\overline{T}(X):\overline{M}\to \overline{M}$. The standard coherences $\beta_{T(X),Y}$ that make T(X) a "$\C$-linear" map are given by the image through $T$ of the twisting equivalences $\tau_{X,Y}:X\otimes Y\to Y\otimes X$ in $\C$. More precisely, we have commutative squares in $\C$

\begin{equation}
\xymatrix{
T(X\otimes Y)\ar[r]^-{\sim} \ar[d]^{\tau_{X,Y}}& T(X)\simeq T(Y)\ar[d]^{\beta_{T(X),Y}}\\
T(Y\otimes X)\ar[r]^-{\sim}& T(Y)\circ T(X)
}
\end{equation}

\noindent given by the fact $T$ is monoidal.

\end{remark}

The following is our key result:

\begin{prop}
\label{main41}
Let $\Cmonoidal$ be a small symmetric monoidal $(\infty,1)$-category and $X$ be a symmetric object in $\C$. Then, for any $\Cmonoidal$-module $\overline{M}$, the colimit of the diagram of $\Cmonoidal$-modules 
\begin{equation}
\xymatrix{\overline{Stab}_X(\overline{M}):= colimit_{Mod_{\Cmonoidal}(\iCat)}(...\ar[r]& \overline{M} \ar[r]^{\overline{T}(X)}& \overline{M} \ar[r]^{\overline{T}(X)}& \overline{M} \ar[r]^{\overline{T}
(X)}& ...)}
\end{equation}

\noindent is a $\Cmonoidal$-module where the multiplication by $X$ is an equivalence.

\begin{proof}
Since $\iCat^{\times}$ is compatible with all small colimits, the Corollary 3.4.4.6 of \cite{lurie-ha} \footnote{Since we are working the commutative setting, we could also refer to the Corollary 4.2.3.5 of \cite{lurie-ha}} implies that $\overline{Stab}_X(\overline{M})$ exists as an object in $Mod_{\Cmonoidal}(\iCat)$ and it also that it can be computed in $\iCat$ by means of the forgetful functor. This means that the colimit module $\overline{Stab}_X(\overline{M})$ corresponds to a $\C$-module structure on the $(\infty,1)$-category

\begin{equation}
Stab_X(M):=\xymatrix{colimit_{\iCat}(...\ar[r]& M\ar[r]^{T(X)}& M \ar[r]^{T(X)}& M \ar[r]^{T(X)}& ...)}
\end{equation}

To be more precise, this diagram is an object in $Fun(N(\mathbb{Z})$, $\iCat)$ and our sketch misses all the faces

\begin{equation}
\xymatrix{ 
M \ar[r]^{T(X)}  & M \\ 
M\ar[u]^{T(X)} \ar[ru]_{T(X)\circ T(X)}|*{}="M" \ar@{=>}[u];"M"^\alpha   }
\end{equation}

\noindent providing the compositions. Moreover, since $T$ is a monoidal functor, we can find a $2$-cell providing an homotopy between the composition $T(X)\circ T(X)$ and the multiplication map $T(X\otimes X)$.\\

Since $\iCat^{\otimes}$ is compatible with colimits, the $\C$-action $\C\times Stab_X(M)\to Stab_X(M)$ is given by the canonical map induced between the colimits of the two rows

\begin{equation}
\xymatrix{
...\ar[rr]&& \C\times M\ar[d]\ar[rr]^{id\times T(X)}&& \ar@{=>}[lld] \C\times M\ar[d] \ar[rr]^{id\times T(X)}&& \ar@{=>}[lld] \C\times M\ar[d] \ar[rr]^{id\times T(X)}&& ...\\
...\ar[rr]&& M\ar[rr]_{T(X)}&& M \ar[rr]_{T(X)}&& M \ar[rr]_{T(X)}&& ...
}
\end{equation}

\noindent where the vertical maps correspond to the action of $\C$ on $M$ and the faces correspond to the coherences that make $T(X)$ a map of modules, which as we saw in the previous remark are given by the twisting equivalences provided by the symmetry of $\Cmonoidal$.\\

We can now prove the statement concerning the multiplication by $X$. Of course, it is a particular case of the previous diagram, replacing the vertical arrows by $T(X)$ and keeping the coherences

\begin{equation}
\xymatrix{
...\ar[rr]&& M\ar[d]_{ T(X)}\ar[rr]^{ T(X)}&& \ar@{=>}[lld]^{\beta_{T(X),Y}} M\ar[d]\ar[rr]^{ T(X)}&& \ar@{=>}[lld]^{\beta_{T(X),Y}}  M\ar[d]^{ T(X)} \ar[rr]^{ T(X)}&& ...\\
...\ar[rr]&& M\ar[rr]_{T(X)}&& M \ar[rr]_{T(X)}&& M \ar[rr]_{T(X)}&& ...
}
\end{equation}

\noindent which in this case are induced by the twisting permutation equivalence $\tau_{X,X}:X\otimes X\to X\otimes X$.

The crucial observation is that the horizontal composition of the natural transformations  $\beta_{T(X),Y}\circ \beta_{T(X),Y}$ can be identified with the natural transformation $T(\sigma)$ induced by the cyclic permutation $\sigma:X\otimes X\otimes X\to X\otimes X \otimes X$. To see this, let us rewrite the previous diagram keeping track of the different copies of $X$

\begin{equation}
\xymatrix{
 M\ar[d]_{ T(X_1)}\ar[rr]^{ T(X_2)}&& \ar@{=>}[lld]^{\beta_{T(X_1),X_2}} M\ar[d]^{T(X_1)}\ar[rr]^{ T(X_3)}&& \ar@{=>}[lld]^{\beta_{T(X_1),X_3}}  M\ar[d]^{ T(X_1)} \\
 M\ar[rr]_{T(X_2)}&& M \ar[rr]_{T(X_3)}&& M
}
\end{equation}

Since the composition in $End(M)$ is strictly associative, we find a commutative diagram in $\C$

\begin{equation}
\xymatrix{
T((X_1\otimes X_2) \otimes X_3)\ar[rr]^-{\sim}\ar[d]^{\tau_{X_1,X_2}\otimes Id_3}&& T(X_1)\circ (T(X_2)\circ T(X_3))= (T(X_1)\circ T(X_2))\circ T(X_3)\ar[d]^{\beta_{T(X_1), X_2}\circ T(Id_3)}\\
T((X_2\otimes X_1)\otimes X_3)\ar[rr]^-{\sim}\ar[d]^{T(\gamma_{X_2,X_1,X_3})}&&(T(X_2)\circ T(X_1))\circ T(X_3)\ar@{=}[d]\\
T(X_2\otimes (X_1\otimes X_3)\ar[rr]^-{\sim} \ar[d]^{T(Id_2\otimes\tau_{X_1,X_3})}&&T(X_2)(\circ T(X_1)\circ T(X_3))\ar[d]^{T(Id_2)\circ \beta_{T(X_1),X_3}}\\
T(X_2\otimes (X_3\otimes X_1))\ar[rr]^-{\sim}&& T(X_2)\circ (T(X_3)\circ T(X_1))
}
\end{equation}

\noindent where the horizontal arrows are the natural equivalences that make $T$ monoidal and the map $\gamma_{X_2,X_1,X_3}$ is the associativity restrain in $\C$. The relation with the cyclic permutation $\sigma$ follows from the definition of $\sigma$: it is constructed as the composition $(Id_2\otimes\tau_{X_1,X_3})\circ \gamma_{X_2,X_1,X_3}\circ (\tau_{X_1,X_2}\otimes Id_3)$ in $\C$.

Finally, since $X$ is symmetric, there is a $2$-cell in $\C$ providing an homotopy between $\sigma$ and the identity of $X\otimes X\otimes X$. T sends it to a $3$-cell in $\iCat$ providing an homotopy between the natural transformation $T(\sigma)$ and the identity. By cofinality, the colimit map induced by the previous diagram is in fact equivalent to the colimit map induced by

\begin{equation}
 \xymatrix{ 
...\ar[rr]^{T(X)}&& M \ar[rr]^{T(X)} \ar[d]_{T(X))} && M \ar[rr]^{T(X)} \ar[d] \ar@{=>}[lld]^{Id} && \ar@{=>}[lld]^{Id}M \ar[d]^{T(X)} \ar[rr]^{T(X)}&& ...\\
... \ar[rr]_{T(X)}&& M \ar[rr]_{T(X)} && M \ar[rr]_{T(X)} && M \ar[rr]_{T(X)} &&...}
\end{equation}

\noindent and therefore, the induced colimit map $Stab_X(M)\to Stab_X(M)$ is an equivalence.

\end{proof}
\end{prop}

\begin{remark}
A similar argument shows that the same result holds if $X$ is $n$-symmetric, meaning that, there exists a number $n\in \mathbb{N}$ such that $\tau^n$ is equal to the identity map in $h(\C)$.
\end{remark}

\begin{remark}
The Remark \ref{remarkmain411} and the Proposition \ref{main41} applies mutatis-mutandis in the presentable setting. This is true because of the Proposition \ref{presentablehaveendomorphismobject} - $\Prl$ admits classifying objects for endomorphisms. If $M$ is a presentable $(\infty,1)$-category, $End^L(M)$ is a classifying object for endomorphisms of $M$, with the strictly associative monoidal structure given by the composition of functors.
\end{remark}

We can finally establish the connection between the adjoint $\Lpr$ and the notion of spectra.

\begin{cor}
\label{main5}
Let $\Cmonoidal$ be a presentable symmetric monoidal $(\infty,1)$-category and let $X$ be a symmetric object in $\C$. Given a $\Cmonoidal$-module $M$, $Stab_X(M)$ is a $\Cmonoidal$-module where $X$ acts as an equivalence and therefore the adjunction of Proposition \ref{main3} provides a map of $\Cmonoidal$-modules

\begin{equation}\Lpr(M)\to Stab_X(M)\end{equation}

This map is an equivalence. In particular, the underlying $\infty$-category of the formal inversion $\Cmonoidalx$ is equivalent to the stabilization $Stab_X(\C)$. 

\begin{proof}
The map can be obtained as a composition

\begin{equation}
\Lpr(M)\to \Lpr(Stab_X(M))\to Stab_X(M)
\end{equation}

\noindent where the first arrow is the image of the canonical map $M\to Stab_X(M)$ by the adjunction $\Lpr$  and the second arrow is the counit of the adjunction. In fact, with our hypothesis and because of the previous Proposition, the action of $X$ is invertible in $Stab_X(M)$ and therefore, by the Proposition \ref{main3} the second arrow is an equivalence It remains to prove that the first map is an equivalence. But now, since $Stab_X(M)$ is a colimit and $\Lpr$ is a left adjoint and therefore commutes with colimits, we have a commutative diagram

\begin{equation}
\xymatrix{
\Lpr(M)\ar[dr]\ar[r]& \Lpr(Stab_X(M))\\ 
& Stab_X(\Lpr(M))\ar[u]^{\sim}
}
\end{equation}

\noindent where the diagonal arrow is the colimit map induced by the stabilization of $\Lpr(M)$. It is enough now to observe that if $M$ is a $\Cmonoidal$-module where the action of $X$ is already invertible, then the canonical map $M\to Stab_X(M)$ is an equivalence of modules. The \emph{2 out of 3} argument concludes the proof.

\end{proof}
\end{cor}

In particular

\begin{cor}
\label{corolariodacaca}
Let $\Cmonoidal$ be a stable presentable symmetric monoidal $(\infty,1)$-category and let $X$ be a symmetric object in $\C$. Then $\Cmonoidalx$ is again a stable presentable symmetric monoidal $(\infty,1)$-category.
\begin{proof}
If $\Cmonoidal$ is stable presentable, the multiplication by $X$ is an exact functor. Moreover, since $X$ is symmetric, the previous corollary provides an equivalence $\C[X^{-1}]\simeq Stab_X(\C)$ where the last is a colimit in $\Prl$. Moreover, since the whole diagram is in $\Prl_{Stb}$ and the last  has all colimits and the inclusion $\Prl_{Stb}\subseteq \Prl$ commutes with them\footnote{To see this we can use the equivalence between $\Prl_{Stb}$ and $Mod_{\Sp}(\Prl)$ \cite[6.3.2.18]{lurie-ha} and the identification of  the inclusion $\Prl_{Stb}\subseteq \Prl$ with the forgetful functor $Mod_{\Sp}(\Prl)\to \Prl$. Now we use the fact that $\Prlmonoidal$ is compatible with colimits (its has internal-hom objects) and therefore colimits of modules are computed in $\Prl$ using the forgetful functor \cite[3.4.4.6]{lurie-ha}. }, we find that $\C[X^{-1}]$ is stable. Moreover, since by construction $\Cmonoidalx$ is a presentable symmetric monoidal  $(\infty,1)$-category, we conclude it is a stable presentable symmetric monoidal $(\infty,1)$-category.
\end{proof}
\end{cor}

\begin{remark}
\label{workswithspaces2}
Let $\Cmonoidal$ be a small symmetric monoidal $\infty$-groupoid and let $X$ be a symmetric object in $\C$. Then,  using the same arguments as in the proof of the previous corollary together with the fact that the $(\infty,1)$-category of spaces $\Spaces$ admits classifying objects for endomorphisms, we deduce that the underlying $(\infty,1)$-category of the formal inversion $\mathcal{L}^{spaces,\otimes}_{\Cmonoidal,X}(\Cmonoidal)$ of the Remark \ref{workswithspaces} is equivalent to the stabilization $Stab^{spaces}_X(\C)$ obtained as the colimit in $\Spaces$ of the diagram induced by the multiplication by $X$. Moreover, since the inclusion $\Spaces\subseteq \iCat$ admits a right adjoint (the "maximal $\infty$-groupoid"), it preserves colimits and we see that the comparison map of \ref{workswithspaces} is an equivalence

\begin{equation}
\mathcal{L}^{\otimes}_{i(\Cmonoidal),X}(i(\Cmonoidal))_{\onefin}\simeq Stab_{X}(i(\C))\simeq i(Stab^{spaces}_X(\C))\simeq \mathcal{L}^{spaces,\otimes}_{\Cmonoidal,X}(\Cmonoidal)_{\onefin}
\end{equation}

where $Stab_{X}(i(\C))$ is the stabilization in $\iCat$.

\end{remark}

\begin{example}
\label{examplemonoidalstructurespectra}
In \cite{lurie-ha} the author introduces the $(\infty,1)$-category of spectra $\Sp$ as the stabilization of the $(\infty,1)$-category of spaces. More precisely, following the notations of the Example \ref{usualspectra} it is given by 

\begin{equation}\Sp:=Sp^{\mathbb{N}}_{(\Sigma_{\Spaces},\Omega_{\Spaces})}(\Spaces_{*/})\end{equation} 

where $\Spaces$ denotes the $(\infty,1)$-category of spaces. By the Propositions and 1.4.3.6 and 1.4.4.4 of \cite{lurie-ha} this $(\infty,1)$-category is presentable and stable and by the Proposition 6.3.2.18 of \cite{lurie-ha} it admits a natural presentable stable symmetric monoidal structure $\Spmonoidal$ which can be described by means of a universal property: it is an initial object in $CAlg(\Prlstable)$. The unit of this monoidal structure is the sphere-spectrum.

Our corollary \ref{main5} provides an alternative characterization of this symmetric monoidal structure. We start with $\Spacesp$ the $(\infty,1)$-category of pointed spaces. Recall that this $(\infty,1)$-category is presentable and admits a monoidal structure given by the so-called \emph{smash product} of pointed spaces. (see the Remark 6.3.2.14 of \cite{lurie-ha} and the section \ref{smashproducts} below). We will denote it as $\Spacessmash$. According to the Proposition 6.3.2.11 of \cite{lurie-ha}, $\Spacessmash$ has an universal property amongst the presentable pointed symmetric monoidal $(\infty,1)$-categories: it is a initial one. The unit of this monoidal structure is the pointed space $S^0=*\coprod*$. We will see below (Corollary \ref{pointedmodelmonoidal} and Remark \ref{smashnovoigualsmashvelho}) that $\Spacessmash$ is the underlying symmetric monoidal $(\infty,1)$-category of the combinatorial simplicial model category of pointed simplicial sets $\ssets_*$ equipped with the classical smash product of spaces. Since $S^1$ is symmetric in $\ssets_*$ with respect to this classical smash (see the Lemma 6.6.2 of \cite{hovey-modelcategories}), by the Remark \ref{remarksymmetric} it will also be symmetric in $\Spacessmash$. Our inversion $\Spacessmash[(S^1)^{-1}]$ provides a new presentable symmetric monoidal $(\infty,1)$-category and because of the symmetry of $S^1$, the fact that $(S^1\wedge-)$ can be identified with $\Sigma_{\Spaces}$ and the Corollary \ref{main5}, we conclude that the underlying $(\infty,1)$-category of $\Spacessmash[(S^1)^{-1}]$ is the stabilization defining $\Sp$ and therefore that $\Spacessmash[(S^1)^{-1}]$ is a presentable stable symmetric monoidal $(\infty,1)$-category. By the universal property of $\Spmonoidal$ there is a unique (up to a contractible space of choices) monoidal map
 
\begin{equation}\Spmonoidal\to \Spacessmash[(S^1)^{-1}]\end{equation}
 
At the same time, since every stable presentable $(\infty,1)$-category is pointed, the universal property of $\Spacessmash$ ensures the existence of a canonical morphism

\begin{equation}
\Spacessmash\to \Spmonoidal
\end{equation}

\noindent which is also unique up to a contractible space of choices. This morphism is just the canonical stabilization morphism and it sends $S^1$ to the sphere-spectrum in $\Sp$ and therefore the universal property of the localization provides a factorization

\begin{equation}
\Spacessmash[(S^1)^{-1}]\to \Spmonoidal
\end{equation}
\noindent which is unique up to homotopy. By combining the two universal properties we find that these two maps are in fact inverses up to homotopy

\end{example}

\begin{remark}
\label{monoidalstabilization}
The technique of inverting an object provides a way to define the monoidal stabilization of a pointed presentable symmetric monoidal $(\infty,1)$-category $\Cmonoidal$. It follows from the Proposition 6.3.2.11 of \cite{lurie-ha} that for any such $\Cmonoidal$, there is an essentially unique (base-point preserving and colimit preserving) monoidal map $f:\Spacessmash\to \Cmonoidal$. Let $f(S^1)$ denote the image of the topological circle through this map. The (presentable) universal property of inverting an object provides an homotopy commutative diagram of commutative algebra objects in $\Prl$

\begin{equation}
\xymatrix{
\Spmonoidal\simeq \Spacessmash[(S^1)^{-1}]\ar@{-->}[d]&\Spacessmash\ar[l]\ar[d]^f\\
\Cmonoidal[f(S^1)^{-1}]& \Cmonoidal\ar[l]
}
\end{equation}
 
The monoidal map $\Spacessmash\to \Spmonoidal$ produces a forgetful functor

\begin{equation}
CAlg(\Prl)_{\Spmonoidal/}\to CAlg(\Prl)_{\Spacessmash/}
\end{equation}

\noindent which by the Proposition \ref{main3} is fully faithful and admits a left adjoint given by the base-change formula $\Cmonoidal\mapsto \Spmonoidal\otimes_{\Spacessmash} \Cmonoidal$. The combination of the universal property of the adjunction and the universal property of inverting an object ensures the existence of an equivalence of pointed symmetric monoidal $(\infty,1)$-categories 

\begin{equation}\Cmonoidal[f(S^1)^{-1}]\simeq \Spmonoidal\otimes_{\Spacessmash}\Cmonoidal\end{equation}

Finally, by the combination of the Proposition \ref{main3} with the Example 6.3.1.22 of \cite{lurie-ha} we deduce that the underlying $(\infty,1)$-category of $\Cmonoidal[f(S^1)^{-1}]$ is the stabilization $Stab(\C)$. 

Moreover, we deduce also that if $\Cmonoidal$ is a stable presentable symmetric monoidal $(\infty,1)$-category and $X$ is any object in $\C$, in order to conclude that the inversion $\Cmonoidalx$  is stable presentable it is enough to show that $\C[X^{-1}]$ is pointed, thus extending the result \ref{corolariodacaca}. Indeed, by the previous discussion, $\C$  is stable if and only if $f(S^1)$ is invertible. Since the inversion functor $\Cmonoidal\to \Cmonoidalx$ is monoidal, the image of $f(S^1)$ in $\C[X^{-1}]$ will again by invertible. Finally, if $\C[X^{-1}]$ is pointed, the image of $f(S^1)$ will necessarily correspond to the image of $S^1$ in $\C[X^{-1}]$, which therefore will be invertible, and so, by the previous discussion, $\C[X^{-1}]$ will be stable.

\end{remark}

\subsection{Connection with the Symmetric Spectrum objects of Hovey}
\label{section4-3}

We recall from \cite{hovey-spectraandsymmetricspectra} the construction of symmetric spectrum objects: Let $\V$ be a combinatorial simplicial symmetric monoidal model category and let $\M$ be a combinatorial simplicial $\V$-model category. Following the Theorem 7.11 of \cite{hovey-spectraandsymmetricspectra}, for any object $X$ in $\V$ we can produce a new combinatorial simplicial  $\V$-model category $Sp^{\Sigma}(\M,X)$ of spectrum objects in $\M$ endowed with the \emph{stable model structure} and where $X$ acts by an equivalence. In particular, by considering $\V$ as a $\V$-model category (using the monoidal structure) the new model category $Sp^{\Sigma}(\V,X)$ inherits the structure of a combinatorial simplicial symmetric monoidal model category and there is left simplicial Quillen monoidal map $\V\to Sp^{\Sigma}(\V,X)$ sending $X$ to an invertible object. 

This general construction sends an arbitrary combinatorial simplicial $\V$-model category to a combinatorial simplicial $\V$-model category where the action of $X$ is invertible. In fact, by the Theorem 7.11 of \cite{hovey-spectraandsymmetricspectra} $Sp^{\Sigma}(\M,X)$ is a combinatorial simplicial $Sp^{\Sigma}(\V,X)$-model category. This a first sign of the fundamental role of the construction of symmetric spectrum objects as an adjoint in the spirit of Section \ref{section4-1}.
We have canonical simplicial left Quillen maps

\begin{equation}
\xymatrix{Sp^{\Sigma}(\V,X)\ar[r]^(.4){\sim} &Sp^{\mathbb{N}}(Sp^{\Sigma}(\V,X), X)\ar[r]^{\sim}& Sp^{\Sigma}(Sp^{\mathbb{N}}(\V,X), X)& \ar[l] Sp^{\mathbb{N}}(\V,X)}
\end{equation}
\noindent but in general the last map is not an equivalence. By the Theorem 8.1 of \cite{hovey-spectraandsymmetricspectra} for the last map to be an equivalence we only need $Sp^{\mathbb{N}}(\V,X)$ to be a $\V$-model category where $X$ as an equivalence. This is exactly the functionality of the symmetric condition on $X$ (see Theorems 9.1 and 9.3 in \cite{hovey-spectraandsymmetricspectra}).

We state our main result

\begin{thm}
\label{maintheorem}
Let $\V$ be a combinatorial simplicial symmetric monoidal model category whose unit is cofibrant and let $X$ be a symmetric object in $\V$ in the sense of the Def. 9.2 in \cite{hovey-spectraandsymmetricspectra} (therefore, cofibrant). Let $Sp^{\Sigma}(\V,X)$ denote the combinatorial simplicial symmetric monoidal model category provided by Theorem 7.11 of \cite{hovey-spectraandsymmetricspectra}, equipped the convolution product. Let $\Cmonoidal:=N((\V^c)^{\otimes})[W_c^{-1}]$, respectively $Sp^{\Sigma}_X(\C)^{\otimes}:=N((Sp^{\Sigma}(\V,X)^c)^{\otimes})[W_c^{-1}]$ denote the underlying symmetric monoidal $(\infty,1)$-category of $\V$, resp. $Sp^{\Sigma}(\V,X)$ (see Section \ref{link2} for the details). By the Corollary 4.1.3.16 of \cite{lurie-ha} we have monoidal equivalences $\Cmonoidal\simeq \N_{\Delta}^{\otimes}((\V^{\circ})^{\otimes})$ and $Sp^{\Sigma}_X(\C)^{\otimes}\simeq N_{\Delta}^{\otimes}((Sp^{\Sigma}(\V,X)^{\circ})^{\otimes})$  and therefore both $\Cmonoidal$ and $Sp^{\Sigma}_X(\C)^{\otimes}$ are presentable symmetric monoidal $(\infty,1)$-categories. Moreover, the left-Quillen monoidal map $\V\to Sp^{\Sigma}(\V,X)$ induces a monoidal functor $\Cmonoidal\to Sp^{\Sigma}_X(\C)^{\otimes}$ (see the Prop. \ref{inducedmonoidafunctor}) which sends $X$ to an invertible object, endowing $Sp^{\Sigma}_X(\C)^{\otimes}$ with the structure of object in $CAlg(\Prl)_{\Cmonoidal/}^X$. In this case, the adjunction of the Prop.\ref{main3} provides a monoidal map

\begin{equation}
\Cmonoidalx\simeq \Lpmonoidal(\Cmonoidal)\to Sp^{\Sigma}_X(\C)^{\otimes}
\end{equation}

We claim that this map is an equivalence of presentable symmetric monoidal $(\infty,1)$-categories.
\begin{proof}
By the remark \ref{remarksymmetric} if $X$ is symmetric in the sense of \cite{hovey-spectraandsymmetricspectra} then it is symmetric in $\Cmonoidal$ in the sense of the Definition \ref{symmetric}.

By definition, the map is obtained as a composition 

\begin{equation}
\xymatrix{\Lpmonoidal(\Cmonoidal)\ar[r]& \Lpmonoidal(Sp^{\Sigma}_X(\C)^{\otimes})\ar[r] & Sp^{\Sigma}_X(\C)^{\otimes}}
\end{equation}

\noindent where the last arrow is the counit of the adjunction of Proposition \ref{main3}. To prove that this map is an equivalence it is enough to verify that the map between the underlying $(\infty,1)$-categories

\begin{equation}\Lpr(\C)\to Sp^{\Sigma}_X(\C)\end{equation}

\noindent is an equivalence. But now, by the combination of the Corollary \ref{main5} with the main result of the Corollary 9.4 in \cite{hovey-spectraandsymmetricspectra}, we find a commutative diagram of equivalences

\begin{equation}
\xymatrix{
\Lpr(\C)\ar[d]^{\sim} \ar[r]&Sp^{\Sigma}_X(\C)= N_{\Delta}(Sp^{\Sigma}(\V,X)^{\circ})\ar[d]^{\sim}\\
Stab_X(\C)\simeq  N_{\Delta}(Sp^{\mathbb{N}}(\V,X)^{\circ})\ar[r]^(.4){\sim}& Stab_X(N_{\Delta}(Sp^{\Sigma}(\V,X)^{\circ}))\simeq N_{\Delta}(Sp^{\mathbb{N}}(Sp^{\Sigma}(\V,X), X)^{\circ})
}
\end{equation}
\noindent where the left vertical map is an equivalence because $X$ is symmetric in $\Cmonoidal$; the equivalence $Stab_X(\C)\simeq  N_{\Delta}(Sp^{\mathbb{N}}(\V,X)^{\circ})$ follows from the Proposition \ref{main4} with $G=(X\otimes -)$ (it is a left Quillen functor because $X$ is cofibrant), and the fact that $\C$ is presentable; the right vertical map is an equivalence because $X$ is already invertible in $N_{\Delta}(Sp^{\Sigma}(\V,X)^{\circ})$  and because a Quillen equivalence between combinatorial model categories induces an equivalence between the underlying $(\infty,1)$-categories (see Lemma 1.3.4.21 of \cite{lurie-ha}). This same last argument, together with the Corollary 9.4 of \cite{hovey-spectraandsymmetricspectra}, justifies the fact that the lower horizontal map is an equivalence.

\end{proof}
\end{thm}

\begin{remark}
\label{withoutassumption}
In the proof of Theorem \ref{maintheorem}, we used the condition on $X$ twice. The first using the result of \cite{hovey-spectraandsymmetricspectra} and the second with the Proposition \ref{main4}. We believe the use of this condition is not necessary. Indeed, everything comes down to prove an analogue of Proposition \ref{main4} for the construction of symmetric spectrum objects, replacing the natural numbers by some more complicated partially ordered set. If such a result is possible, then the construction of symmetric spectra in the combinatorial simplicial case can be presented as a colimit of a diagram of simplicial categories. In this case, the Proposition \ref{main41} would follow immediately even without the condition on $X$. We will not pursue this topic here since it won't be necessary for our goals.
\end{remark}

\begin{example}
The combination of the Theorem \ref{maintheorem} together with the Remark \ref{remarksymmetric} and the Example \ref{examplemonoidalstructurespectra} provides a canonical equivalence of presentable symmetric monoidal presentable $(\infty,1)$-categories $\Spmonoidal\simeq N_{\Delta}^{\otimes}(Sp^{\Sigma}(\ssetsp,S^1))$.
\end{example}

\section{Universal Characterization of the Motivic Stable Homotopy Theory of Schemes}
\label{section5}

Let $\uniU\in \uniV\in \uniW$ be universes. In the following sections, we shall write $\sch$ to denote the $\uniV$-small category of smooth separated $\uniU$-small schemes of finite type over a Noetherian $\uniU$-scheme $S$.

\subsection{$\mathbb{A}^1$-Homotopy Theory of Schemes}

The main idea in the subject is to "do homotopy theory with schemes" in more or less the same way we do with spaces, by thinking of the affine line $\mathbb{A}^1$ as an "interval".  Of course, this cannot be done directly inside the category of schemes for it does not have all colimits. In \cite{voevodsky-morel}, the authors constructed a \emph{place} to realize this idea. The construction proceeds as follows: start from the category of schemes and add formally all the colimits. Then make sure that the following two principles hold: 

\begin{enumerate}[I)]
\item the line $\mathbb{A}^1$ becomes contractible;
\item if $X$ is a scheme and $U$ and $V$ are two open subschemes whose union equals $X$ in the category of schemes then make sure that their union continues to be $X$ in the new place;
\end{enumerate}

The original construction in \cite{voevodsky-morel} was performed using the techniques of model category theory and this \emph{place} is the homotopy category of a model category $\M_{\mathbb{A}^1}$. During the last years their methods were revisited and reformulated in many different ways. In \cite{dugger-universalhomotopytheories}, the author presents a "universal" characterization of the original construction using the theory of Bousfield localizations for model categories\footnote{see \cite{hirschhorn}} together with a universal characterization of the theory of simplicial presheaves, within model categories. The construction of \cite{dugger-universalhomotopytheories} can be summarized by the expression

\begin{equation}\M_{\mathbb{A}^1}=L_{\mathbb{A}^1}L_{HyperNis}((SPsh(\sch)))\end{equation}

\noindent where $SPsh(-)$ stands for simplicial presheaves with the projective model structure, $L_{HyperNis}$ corresponds the Bousfield localization with respect to the collection of the hypercovers associated to the \emph{Nisnevich topology} (see below) and $L_{\mathbb{A}^1}$ corresponds to the Bousfield localization with respect to the collection of all projection maps $X\times \mathbb{A}^1\to X$.\\

 It is clear today that model categories should not be taken as fundamental objects, but rather, we should focus on their associated $(\infty,1)$-categories. In this section, we use the insights of \cite{dugger-universalhomotopytheories} to perform the construction of an $(\infty,1)$-category $\hsch$ directly within the setting of $\infty$-categories. By the construction, it will have a universal property and using the link described in Section \ref{section1-2} and the theory developed by J.Lurie in \cite{lurie-htt} relating Bousfield localizations to localizations of $\infty$-categories, we will be able to prove that $\hsch$ is equivalent to the $\infty$-category underlying the $\mathbb{A}^1$ model category of Morel-Voevodsky.

The construction of $\hsch$ proceeds as follows. We start from the category of smooth schemes of finite type over $S$ - $\sch$ and consider it as a trivial $\uniV$-small $(\infty,1)$-category $N(\sch)$. Together with the \emph{Nisnevich topology}  (\cite{nisnevich}), it acquires the structures of an \emph{$\infty$-site} (see Definition 6.2.2.1 of \cite{lurie-htt}). By definition (see Def. 1.2 of \cite{voevodsky-morel}) the Nisnevich topology is the topology generated by the pre-topology whose covering families of an $S$-scheme $X$ are the collections of étale morphisms $\{f_i:U_i\to X\}_{i\in I}$ such that for any $x\in X$ there exists an $i\in I$ and $u_i\in U_i$ such that $f_i$ induces an isomorphism between the residual fields $k(x)\simeq k(u_i)$. Recall from \cite{voevodsky-morel} (Def. 1.3) that an elementary Nisnevich square is a commutative square of schemes

\begin{equation}
\label{nissquare}
\xymatrix{
p^{-1}(U)\ar[r]\ar[d]&V\ar[d]^p\\
U\ar[r]^i&X
}
\end{equation}

such that 

\begin{enumerate}[a)]
\item $i:U\hookrightarrow X$ is an open immersion of schemes;
\item $p:V\to X$ is an étale map;
\item the square (\ref{nissquare}) is a pullback. In particular $p^{-1}(U)\to V$ is also an open immersion.
\item the canonical projection $p^{-1}(X-U)\to X-U$ is an isomorphism where we consider the closed subsets $Z:=X-U$ and $p^{-1}{Z}$ both equipped with the reduced structures of closed subschemes;
\end{enumerate}

and from this we can easily deduce that
\begin{enumerate}[e)]
\item the square 

\begin{equation}
\xymatrix{
V\ar[d]_p&\ar[d] \ar[l] p^{-1}(Z)\\
X&\ar[l]Z:=X-U
}
\end{equation}

is a pullback with both $Z$ and $p^{-1}(Z)$ equipped with the reduced structures;
\item the square (\ref{nissquare}) is a pushout.
\end{enumerate}

The crucial fact is that each family $(V\to X, U\to X)$ as above forms a Nisnevich covering and the families of this form provide a \emph{basis} for the Nisnevich topology (see the Proposition 1.4 of \cite{voevodsky-morel}). We consider the very big $(\infty,1)$-category $\mathcal{P}^{big}(N(\sch)):=Fun(N(\sch)^{op},\widehat{\mathcal{S}})$ of presheaves of (big) homotopy types over $N(\sch)$ (See Section 5.1 of \cite{lurie-htt}) which has the expected universal property (Thm. 5.1.5.6 of \cite{lurie-htt}): it is the free completion of $N(\sch)$ with $\uniV$-small colimits (in the sense of $\infty$-categories). Using the Proposition 4.2.4.4  of \cite{lurie-htt} we can immediately identify $\mathcal{P}^{big}(N(\sch))$ with the underlying $\infty$-category of the model category of simplicial presheaves on $\sch$ endowed with the projective model structure. The results of \cite{lurie-htt} provide an $\infty$-analogue for the classical Yoneda's embedding, meaning that we have a fully faithful map of $\infty$-categories $j:N(\sch)\to \mathcal{P}^{big}(N(\sch))$ and as usual we will identify a scheme $X$ with its image $j(X)$. We now restrict to those objects in $\mathcal{P}^{big}(N(\sch))$ which are sheaves with respect to the Nisnevich topology. Because the Nisnevich squares form a basis for the Nisnevich topology, an object $F\in \mathcal{P}^{big}(N(\sch))$ is a sheaf iff it maps Nisnevich squares to pullback squares. In particular, every representable $j(X)$ is a sheaf (because Nisnevich squares are pushouts). Following the results of \cite{lurie-htt}, the inclusion of the full subcategory $Sh_{Nis}^{big}(\sch)\subseteq \mathcal{P}^{big}(N(\sch))$ admits a left adjoint (which is known to be exact - Lemma 6.2.2.7 of  \cite{lurie-htt}) and provides a canonical example of an $\infty$-topos (See Definition 6.1.0.4 of \cite{lurie-htt}). More importantly to our needs, this is an example of a presentable localization of a presentable $(\infty,1)$-category and we can make use of the results discuss in Section \ref{section3-5}. Notice also that the fact that any Nisnevich square is a pushout square, implies that any representable $j(X)$ is a Nisnevich sheaf. \\
 
\begin{remark}
\label{affineisenough}
In the case $S$ is affine given by a ring $k$, the category of smooth schemes $\sch$ can be replaced by the category of \emph{affine} smooth schemes of finite type over $k$, $\aff$, and the resulting $(\infty,1)$-categories $Sh_{Nis}^{big}(\sch)$ and $Sh_{Nis}^{big}(\aff)$ are equivalent. This follows because we can identify $\sch$ with a full subcategory of $\mathcal{P}^{big}(\aff)$ using the map sending a smooth scheme $X$ to the representable functor $Y\in \aff \mapsto Hom_{\sch}(Y,X)$, and this identification is compatible with the Nisnevich topologies. For more details see \cite{MR1693330}.
\end{remark}

Next step, we consider the \emph{hypercompletion} of the $\infty$-topos $Sh_{Nis}^{big}(\sch)$ (see Section 6.5.2 of \cite{lurie-htt}). By construction, it is a presentable localization of $Sh_{Nis}^{big}(\sch)$ and by the Corollary 6.5.3.13 of \cite{lurie-htt} it coincides with $Sh_{Nis}^{big}(\sch)^{hyp}$: the localization of $\mathcal{P}^{big}(N(\sch))$ spanned by the objects which are local with respect to the class of \emph{Nisnevich hypercovers}.

Finally, we reach the last step: We will restrict ourselves to those sheaves in $Sh_{Nis}^{big}(\sch)^{hyp}$ satisfying $\mathbb{A}^1$-invariance, meaning those sheaves $F$ such that for any scheme $X$, the canonical map induced by the projection $F(X)\to F(X\times \mathbb{A}^1)$ is an equivalence. More precisely, we consider the localization of $Sh_{Nis}^{big}(\sch)^{hyp}$ with respect to the class of all projection maps $\{X\times \mathbb{A}^1\to X\}_{X\in Obj(\sch)}$. We will write $\hsch$ for the result of this localization 
and write $l_{\mathbb{A}^1}:Sh_{Nis}^{big}(\sch)^{hyp}\to \hsch$ for the localization functor. Notice that $\hsch$ is very big, presentable with respect to the universe $\uniV$. It is also clear from the construction that $\hsch$ comes naturally equipped with a universal characterization: 

\begin{thm}
\label{universalpropertymotivichomotopy}
Let $\sch$ be the category of smooth schemes of finite type over a base noetherian scheme $S$ and let $L:N(\sch)\to \hsch$ denote the composition of localizations

\begin{equation}
N(\sch)\to \mathcal{P}^{big}(N(\sch))\to Sh_{Nis}^{big}(\sch)\to Sh_{Nis}^{big}(\sch)^{hyp}\to \hsch
\end{equation}

Then, for any $(\infty,1)$-category $\D$ with all $\uniV$-small colimits, the map induced by composition with $L$ 

\begin{equation}
Fun^L(\hsch, \D)\to Fun(N(\sch),\D)
\end{equation}
 
\noindent is fully faithful and its essential image is the full subcategory of $Fun(N(\sch),\D)$ spanned by those functors satisfying \emph{Nisnevich descent} and \emph{$\mathbb{A}^1$-invariance}. The left-side denotes the full subcategory of $Fun(\hsch, \D)$ spanned by the colimit preserving maps.
\begin{proof}
The proof follows from the combination of the universal property of presheaves with the universal properties of each localization in the construction and from the fact that for the Nisnevich topology in $\sch$, descent is equivalent to hyperdescent (see \cite[Prop. 5.9]{MR2593670} or \cite[3- 1.16]{voevodsky-morel}) and therefore the localization $Sh_{Nis}^{big}(\sch)\to Sh_{Nis}^{big}(\sch)^{hyp}$ is an equivalence.
\end{proof}
\end{thm}

Our goal now is to provide the evidence that $\hsch$ really is the underlying $(\infty,1)$-category of the $\mathbb{A}^1$- model category of Morel-Voevodsky. In fact, we already saw that our first step coincides with the first step in the construction of $\M_{\mathbb{A}^1}$ - simplicial presheaves are a model for $\infty$-presheaves. It remains to prove that our localizations produce the same results of the Bousfield localizations. But of course, this follows from the general results in the appendix of \cite{lurie-htt} applied to the model category $\M:=SPsh(\sch)$. (See our introductory survey \ref{combinatorialmodelcategories}).

\begin{remark}
It is important to remark that the sequence of functors in the Theorem \ref{universalpropertymotivichomotopy} can be promoted to a sequence of monoidal functors with respect to the cartesian monoidal structures

\begin{equation}
N(\sch)^{\times}\to \mathcal{P}^{big}(N(\sch))^{\times}\to Sh_{Nis}^{big}(\sch)^{\times}\simeq (Sh_{Nis}^{big}(\sch)^{hyp})^{\times}\to \hsch^{\times}
\end{equation}

The first is the Yoneda map which we know commutes with limits. The second map is the sheafification functor which we also know is exact. The last functor is a monoidal localization because of the definition of the $\mathbb{A}^1$-equivalences. These localized monoidal structures are cartesian because of the existence of fully faithful right adjoints.
\end{remark}

\subsection{The monoidal structure in $\hschp$}
\label{smashproducts}

Let $\hsch$ be the $(\infty,1)$-category introduced in the last section. Since it is presentable it admits a final object $*$ and the $(\infty,1)$-category of pointed objects $\hschp$ is also presentable (see Proposition 5.5.2.10 of \cite{lurie-htt}. In this case, since the forgetful functor $\hschp\to \hsch$ commutes with limits, by the Adjoint Functor Theorem (see the Corollary 5.5.2.9 of \cite{lurie-htt}) it admits a left adjoint $()_+:\hsch\to \hschp$ which we can identify with the formula $X\mapsto X_+:=X\coprod*$. In order to follow the stabilization methods of Morel-Voevodsky we need to explain how the cartesian product in $\hsch$ extends to a symmetric monoidal structure in $\hschp$ and how the pointing morphism becomes monoidal.

This problems fits in a more general setting. Recall that the $(\infty,1)$-category of spaces $\Spaces$ is the unit for the symmetric monoidal structure $\Prlmonoidal$. In \cite{lurie-ha}-Prop. 6.3.2.11  the author proves that the pointing morphism $-\coprod \ast:\Spaces\to \Spacesp$ endows $\Spacesp$ with the structure of an idempotent object in $\Prlmonoidal$ and proves that its associated local objects are exactly the pointed presentable $(\infty,1)$-categories. It follows from the general theory of idempotents  that the product functor $\C\mapsto \C\otimes \Spacesp$ is a left adjoint to the inclusion functor $\Prl_{Pt}\subseteq \Prl$. Also from the general theory, this left adjoint is monoidal. The final ingredient is that for any presentable $(\infty,1)$-category $\C$ there is an equivalence of $(\infty,1)$-categories $\C_*\simeq \C\otimes \Spacesp$ (see the Example 6.3.1.20 of \cite{lurie-ha}) and via this equivalence, the pointing map $\C\to \C_*$ is equivalent to the product map $id_\C\otimes ()_+:\C\otimes \Spaces\to \C\otimes\Spacesp$ where $()_+$ denotes the pointing map of spaces.  Altogether, we have the following result

\begin{cor}(Lurie)
\label{corollarylurie}
The formula $\C\mapsto \C_*$ defines a monoidal left adjoint to the inclusion $\Prl_{Pt}\subseteq \Prl$ and therefore induces a left adjoint to the inclusion $CAlg(\Prl_{Pt})\subseteq CAlg(\Prl)$. In other words, for any presentable symmetric monoidal $(\infty,1)$-category $\Cmonoidal$, there exists a pointed presentable symmetric monoidal $(\infty,1)$-category $\C_*^{\wedge(\otimes)}$ whose underlying $(\infty,1)$-category is $\C_*$, together with a monoidal functor $\Cmonoidal \to \C_*^{\wedge(\otimes)}$ extending the pointing map $\C\to \C_*$, and satisfying the following universal property:\\

$(*)$ for any  pointed presentable symmetric monoidal $(\infty,1)$-category $\Dmonoidal$, the composition

\begin{equation}
Fun^{\otimes, L}(\C_*^{\wedge(\otimes)}, \Dmonoidal)\to Fun^{\otimes,L}(\Cmonoidal, \Dmonoidal)
\end{equation}

\noindent is an equivalence. 
\end{cor}

\begin{remark}
\label{pointedextension}
In the situation of the previous corollary, given a functor $F:\C\to \D$ with $\D$ being pointed, its canonical extension $\tilde{F}$ 

\begin{equation}
\xymatrix{
\C\ar[r]\ar[d]& \D\\
\C_*\ar@{-->}[ru]^{\tilde{F}}&
}
\end{equation}

\noindent is naturally identified with the formula $(u:*\to X)\mapsto cofiber (u)\in \D$.
\end{remark}

The symmetric monoidal structure $\C_*^{\wedge(\otimes)}$ will be called the \emph{smash product induced by $\Cmonoidal$}. Of course, if $\Cmonoidal$ is already pointed we have an equivalence $\C_*^{\wedge(\otimes)}\simeq \Cmonoidal$. In the particular case when $\Cmonoidal$ is Cartesian, we will use the notation $\C_*^{\wedge}:=\C_*^{\wedge(\times)}$.\\

Let us now progress in a different direction. Let $\M$ be a combinatorial simplicial model category. Assume also that $\M$ is cartesian closed and that its final object $*$ is cofibrant. This makes $\M$ a symmetric monoidal model category with respect to the cartesian product and we have an monoidal equivalence

\begin{equation} N_{\Delta}^{\otimes}((\M^{\circ})^{\times})\simeq N_{\Delta}(\M^{\circ})^{\times}\end{equation}

Moreover, because the product is a Quillen bifunctor, $N_{\Delta}(\M^{\circ})^{\times}$ is a presentable symmetric monoidal $(\infty,1)$-category and therefore we equip $N_{\Delta}(\M^{\circ})_{*}$ with a canonical presentable symmetric monoidal structure $N_{\Delta}(\M^{\circ})_{*}^{\wedge}$ for which the pointing map becomes monoidal

\begin{equation}
N_{\Delta}(\M^{\circ})^{\times}\to  N_{\Delta}(\M^{\circ})_{*}^{\wedge}
\end{equation}

Independently of this, we can consider the natural model structure in $\M_*$ (see the Remark 1.1.8 in \cite{hovey-modelcategories}). Again, it is combinatorial and simplicial and comes canonically equipped with a left-Quillen functor $\M\to \M_*$ defined by the formula $X\mapsto X\coprod *$. Moreover, it acquires the structure of a symmetric monoidal model category via the usual definition of the smash product, given by the formula

\begin{equation}
(X,x)\wedge (Y,y):= \frac{(X,x)\times (Y,y)}{(X,x)\vee (Y,y)}
\end{equation}

It is well-known that this formula defines a symmetric monoidal structure with unit given by $(*)_+$ and by the Proposition 4.2.9 of \cite{hovey-modelcategories} it is compatible with the model structure in $\M_*$. Let $N_{\Delta}^{\otimes}(((\M_{*})^{\circ})^{\wedge_{usual}})$ be its underlying symmetric monoidal $(\infty,1)$-category. The same result also tells us that the left-Quillen map $\M\to \M_*$ is monoidal. By the Proposition \ref{inducedmonoidafunctor}, it induces a monoidal map between the underlying symmetric monoidal $(\infty,1)$-categories.

\begin{equation}
f^{\otimes}: N_{\Delta}(\M^{\circ})^{\times}\to  N_{\Delta}^{\otimes}(((\M_{*})^{\circ})^{\wedge_{usual}})
\end{equation}

Of course, $N_{\Delta}^{\otimes}(((\M_{*})^{\circ})^{\wedge_{usual}})$ is a pointed presentable symmetric monoidal $(\infty,1)$-category and by the universal property defining the smash product we obtain a monoidal map

\begin{equation}
N_{\Delta}(\M^{\circ})_{*}^{\wedge(\otimes)}\to N_{\Delta}^{\otimes}(((\M_{*})^{\circ})^{\wedge_{usual}})
\end{equation}

\begin{cor}
\label{pointedmodelmonoidal}
The upper map is an equivalence of presentable symmetric monoidal $(\infty,1)$-categories.
\begin{proof}
Since the map is monoidal, the proof is reduced to the observation that the underlying map  

\begin{equation}
f=f^{\otimes}_{\onefin}:N_{\Delta}(\M^{\circ})_{*}\to N_{\Delta}(((\M_{*})^{\circ})
\end{equation}

\noindent is an equivalence. To prove this, we observe first that since $*$ is cofibrant, we have an equality of simplicial sets $N_{\Delta}(\M^{\circ})_{*}= N_{\Delta}((\M^{\circ})_*)$. Secondly, we observe that the cofibrant-fibrant objects in $(\M_{*})$ are exactly the pairs $(X,*\to X)$ with $X$ cofibrant-fibrant in $\M$ and $*\to X$ a cofibration. This means there is a natural inclusion of $(\infty,1)$-categories $i: N_{\Delta}(((\M_{*})^{\circ})\subseteq N_{\Delta}((\M^{\circ})_*)$. It follows from the definition of the model structure in $\M_*$ that this inclusion is essentially surjective: if $(X,*\to X)$ is an object in $N_{\Delta}((\M^{\circ})_*)$, we consider the factorization of $*\to X$ through a cofibration followed by a trivial fibration in $\M$,

\begin{equation}*\to X'\simeq X\end{equation} 

Of course, $(X',*\to X')$ is an object in  $N_{\Delta}(((\M_{*})^{\circ})$ and it is equivalent to $(X,*\to X)$ in $N_{\Delta}((\M^{\circ})_*)$.

Finally, we notice that the composition $i\circ f: N_{\Delta}(\M^{\circ})\to N_{\Delta}((\M_*)^{\circ})\subseteq  N_{\Delta}((\M^{\circ})_*)$ yields the result of the canonical pointing map $N_{\Delta}(\M^{\circ})\to N_{\Delta}(\M^{\circ})_{*}$. Indeed, the pointing map is characterized by the universal property of the homotopy pushout, and since $*\to X$ in  $N_{\Delta}((\M_*)^{\circ})$ is a cofibration and $X$ is also cofibrant, the coproduct $X\coprod *$ is an homotopy coproduct. The result now follows from the universal property of the pointing map.
\end{proof}
\end{cor}

\begin{remark}
\label{smashnovoigualsmashvelho}
If $\M=\ssets$ is the model category of simplicial sets with the cartesian product, it satisfies the above conditions and we find a monoidal equivalence between $\Spacessmash$ and the underlying symmetric monoidal $(\infty,1)$-category of $\ssets_*$ endowed with the classical smash product of pointed spaces.\\
\end{remark}

\begin{remark}
\label{smashsimplicialpresheaves}
If $\C$ is a simplicial category, the left-Quillen adjunction $\ssets\to \ssets_*$ extends to a left Quillen adjunction $SPsh(\C)\to SPsh_*(\C)$, where $SPsh_*(\C)$ corresponds to the category of presheaves of pointed simplicial sets over $\C$, endowed with the projective model structure (see \cite{lurie-htt}-Appendix). It follows that $SPsh(\C)$ has all the good properties which intervene in the proof of the Corollary \ref{pointedmodelmonoidal} and we find a monoidal equivalence $N_{\Delta}(SPsh(\C)^{\circ})_*^{\wedge}\to N_{\Delta}^{\otimes}((SPsh_*(\C)^{\circ})^{\wedge_{usual}})$ where the last is the underlying symmetric monoidal $(\infty,1)$-category associated to the smash product in $SPsh(\C)_*$.
\end{remark}
 
The Corollary \ref{pointedmodelmonoidal} implies that

\begin{cor}
\label{newsmashgivesclassicalsmash}
Let $\hsch^{\times}$ be the presentable symmetric monoidal $(\infty,1)$-category underlying the model category $\M$ encoding the $\mathbb{A}^1$-homotopy theory of Morel-Voevodsky together with the cartesian product. Let $\mathcal{M}_*$ be its pointed version with the smash product given by the Lemma 2.13 of \cite{voevodsky-morel}. Then, the canonical map induced by the universal property of the smash product 

\begin{equation}
\hschp^{\wedge}\to N_{\Delta}^{\otimes}(((\mathcal{M}_*)^{\circ})^{\wedge})
\end{equation}

\noindent is an equivalence of presentable symmetric monoidal $(\infty,1)$-categories.
\end{cor}

In other words and as expected, $\hschp^{\wedge}$ is the underlying symmetric monoidal $(\infty,1)$-category of the classical construction.

\subsection{The Stable Motivic Theory}

As in the classical setting, we may now consider a \emph{stabilized} version of the theory. In fact, two stabilizations are possible - one with respect to the \emph{topological circle} $S^1:= \Delta[1]/\partial \Delta[1]$ (pointed by the image of $\partial \Delta[1]$) and another one with respect to the \emph{algebraic circle} defined as 
$\mathbb{G}_m:=\mathbb{A}^1-\{0\}$. The \emph{motivic stabilization } of the theory is by definition, the stabilization with respect to the product $\mathbb{G}_m\wedge S^1$ which we know (Lemma 4.1 of \cite{voevodsky-icm}) is equivalent to $(\mathbb{P}^1,\infty)$ in $\hschp$. \footnote{It is very easy to deduce this equivalence: take the Nisnevich covering of $(\mathbb{P}^1,1)$ given by two copies of $\mathbb{A}^1$ both pointed at $1$, together with the maps $\mathbb{A}^1\to \mathbb{P}^1$ sending $x\mapsto (1:x)$, respectively, $x\mapsto (x:1)$. Their intersection is  $\mathbb{A}^1-\{0\}$ (also pointed at $1$). The result follows because $\mathbb{A}^1$ is contractible in $\hschp$}

\begin{defn}(Definition 5.7-\cite{voevodsky-icm})
Let $S$ be a base scheme. The \emph{stable motivic $\mathbb{A}^1$ $\infty$-category over $S$} is the underlying $(\infty,1)$-category of the presentable symmetric monoidal $\infty$-category $\stmonoidal$ defined by the formula

\begin{equation}
\stmonoidal:=\hschp^{\wedge}[(\mathbb{P}^1,\infty)^{-1}]
\end{equation}

\noindent as in the Definition \ref{defformal}.
\end{defn}

The standard way to define the stable motivic theory is to consider the combinatorial simplicial symmetric monoidal model category $Sp^{\Sigma}((\M_{\mathbb{A}^1})_{*}, (\mathbb{P}^1,\infty))$ where $(\M_{\mathbb{A}^1})_{*}$ is equipped with the smash product. By the Lemma 4.4 of \cite{voevodsky-icm} together with the Remark \ref{symmetric}, we know that $(\mathbb{P}^1,\infty)$ is symmetric and consequently the Theorem \ref{maintheorem} ensures that $\stmonoidal$ recovers the classical definition. In addition, since we have an equivalence $(\mathbb{P}^1,\infty)\simeq \mathbb{G}_m\wedge S^1$, the universal properties provide canonical monoidal equivalences of presentable symmetric monoidal $(\infty,1)$-categories

\begin{equation}
\stmonoidal\simeq (\hschp^{\wedge})[(\mathbb{G}_m\wedge S^1)^{-1}]\simeq (\hschp^{\wedge})[((\mathbb{P}^1,\infty)\wedge S^1)^{-1}]\simeq ((\hschp^{\wedge})[(S^1)^{-1}])[(\mathbb{P}^1,\infty)^{-1}]
\end{equation}

Since $S^1$ is symmetric in $\Spacessmash$ (see \cite{hovey-modelcategories} Lemma 6.6.2 together with the Remark \ref{remarksymmetric}) it is also symmetric in $\hschp^{\wedge}$ (because it is given by the image of the unique colimit preserving monoidal map $\Spacessmash\to \hschp^{\wedge}$). In this case, we can use the Proposition \ref{main5} to deduce that the underlying $\infty$-category of $(\hschp^{\wedge})[(S^1)^{-1}]$ is equivalent to the stable $\infty$-category $Stab(\hsch)$. Plus, since  $(\hschp^{\wedge})[(S^1)^{-1}]$ is presentable by construction, the monoidal structure is compatible with colimits, thus exact, separately on each variable. We conclude that it is a stable presentable symmetric monoidal $(\infty,1)$-category. 

Finally, because $(\mathbb{P}^1,\infty)$ is symmetric, the Corollary \ref{corolariodacaca} tells us that  $\stmonoidal$ is  a \emph{stable presentable symmetric monoidal $\infty$-category}. In particular its homotopy category is triangulated and inherits a canonical symmetric monoidal structure.

\begin{cor}
\label{universalpropertymotives}
Let $S$ be a base scheme and $\sch$ denote the category of smooth schemes of finite type over $S$, together with the cartesian product. The composition of monoidal functors

\begin{equation}
\theta^{\otimes}:N(\sch)^{\times}\to \mathcal{P}^{big}(N(\sch))^{\times}\to \hsch^{\times}\to \hschp^{\wedge}\to \hschp^{\wedge}[(S^1)^{-1}]\to \stmonoidal
\end{equation}

\noindent satisfies the following universal property: for any pointed presentable symmetric monoidal $(\infty,1)$-category $\Dmonoidal$, the composition map

\begin{equation}Fun^{\otimes,L}(\stmonoidal, \Dmonoidal)\to Fun^{\otimes}(N(\sch)^{\times}, \Dmonoidal)\end{equation}

\noindent is fully faithful and its image consists of those monoidal functors $N(\sch)^{\times}\to \Dmonoidal$ whose underlying functor satisfy Nisnevich descent, $\mathbb{A}^1$-invariance and such that the cofiber of the image of the point at $\infty$, $\xymatrix{S\ar[r]^{\infty}& \mathbb{P}^1}$ is an invertible object in $\Dmonoidal$. Moreover, any pointed presentable symmetric monoidal $(\infty,1)$-category $\Dmonoidal$ admitting such a monoidal functor is necessarily stable.

\begin{proof}
Here, $N(\sch)$ denotes the standard way to interpret an ordinary category as a trivial $(\infty,1)$-category using the nerve. The Yoneda map $j:N(\sch)\to \mathcal{P}^{big}(N(\sch))$ extends to a monoidal map because of the monoidal universal property of presheaves (consult our introductory section on Higher Algebra). By the Proposition 2.15 pg. 74 in \cite{voevodsky-morel}, the localization functor $\mathcal{P}^{big}(N(\sch))\to \hsch$ is monoidal with respect to the cartesian structure and therefore extends to a monoidal left adjoint to the inclusion $\hsch^{\times}\subseteq \mathcal{P}^{big}(N(\sch))^{\times}$. The result now follows from the discussion above, together with the Corollaries \ref{corollarylurie} and \ref{newsmashgivesclassicalsmash}, the Corollary \ref{corolariodacaca}, the Theorem \ref{maintheorem} and Remark \ref{pointedextension}. \\

The last assertion follows from the Remark \ref{monoidalstabilization}, together with the fact that $\mathbb{P}^1$ mod out by the point at infinity is the tensor product of $S^1$ and $G_m$, so that, since we are dealing with monoidal functors, the conditions defining the image of the composition map force the image $S^1$ to be tensor invertible in $\Dmonoidal$. 
\end{proof}
\end{cor}

\begin{remark}
\label{remarksingularities}
Thanks to the results of \cite{riou-spanierwhitehead} and to our discussion in the Prop. \ref{cg}, if $k$ is a field admitting resolutions of singularities the collection of dualizable objects in $\stmonoidalk$ forms a family of compact generators (in the sense of \ref{cg}). The same for the collection of projective varieties. In particular the $(\infty,1)$-category $\stk$ is compactly generated.
\end{remark}

\subsection{Description using Spectral Presheaves}
\label{usingpresheavesofspectra1}

In this section we give an alternative description of the symmetric monoidal $(\infty,1)$-category $\stmonoidalk$ using presheaves of spectra. 

\begin{remark}
\label{enrichedyoneda}(Spectral Yoneda's Lemma) Recall (for instance, see the discussion in \ref{stableinfinitycategories}) that any stable $(\infty,1)$-category has a natural enrichment over spectra, determined by means of a universal property. In this remark we explain how to use this universal property to deduce an enriched version of the Yoneda's lemma for spectral presheaves. More precisely, if $\C$ is a small $(\infty,1)$-category, we consider the composition of the Yoneda's embedding with the pointing map followed by stabilization $\Sigma^{\infty}_+\circ j: \C\hookrightarrow \mathcal{P}(\C)\to \mathcal{P}(\C)_{\ast}\to Stab(\mathcal{P}(\C))\simeq Fun(\C^{op}, \Sp)$ (because the stabilization is a limit).  Now, given  an object $X$ in $\C$, we can use the Yoneda's lemma for $\mathcal{P}(\C)$ to construct a natural equivalence of functors $Map_{Fun(\C^{op}, Sp)}(\Sigma^{\infty}_+\circ j(X), -)\to \Omega^{\infty}\circ ev_X$, where $ev_X : Fun(\C^{op}, \Sp) \to \Sp$ is the evaluation map at $X$. This is possible because the delooping of presheaves is computed objectwise. To conclude, since  the composition with $\Omega^\infty$ induces an equivalence $Exc_{\ast}(\C, \Sp)\simeq Exc_{\ast}(\C, \Spaces)$, we can lift the previous natural equivalence to a new one

\begin{equation}
Map^{Sp}_{Fun(\C^{op}, Sp)}(\Sigma^{\infty}_+\circ j(X), -)\to  ev_X
\end{equation}

\noindent which, when evaluated at $F$ gives us the Yoneda's formula we seek. This holds for any universe: if $\C$ is only $\uniV$-small for some universe $\uniV$ we apply the same arguments as above to the $\uniV$-small $(\infty,1)$-category of spectra obtained from the stabilization of the $\uniV$-small $(\infty,1)$-category of spaces.\\
\end{remark}

Now, we start from the $(\infty,1)$-category $N(\sch)$ and consider the very big $(\infty,1)$-category $Fun(N(\sch)^{op}, \widehat{\Sp})$\footnote{Here $\widehat{\Sp}$ denotes the big $(\infty,1)$-category of spectra, obtained from the stabilization of the big $(\infty,1)$-category of spaces $\widehat{\Spaces}$} which is canonically equivalent to $Stab(\mathcal{P}^{big}(N(\sch))_{\ast})$. Using the Remark \ref{monoidalstabilization} we obtain a canonical monoidal structure $Fun(N(\sch)^{op}, \widehat{\Sp})^{\otimes}$ defined by the inversion $\mathcal{P}^{big}(N(\sch))_{\ast}^{\wedge(\otimes)}[(S^1)^{-1}]^{\otimes}$ where $\mathcal{P}^{big}(N(\sch))_{\ast}^{\wedge(\otimes)}$ is the canonical monoidal smash structure given by the Prop. \ref{corollarylurie} extending the monoidal structure $\mathcal{P}^{big}(N(\sch))^{\otimes}$ of \ref{monoidalstructurepresheaves}. 

We proceed as before and perform the localization with respect to the Nisnevich topology and $\mathbb{A}^1$. Extra care is needed, for the class of maps with respect to which we need to localize is not the same as for presheaves of spaces. In order to describe these two classes we recall first that $Fun(N(\sch)^{op},  \widehat{\Sp})$ is a stable presentable $(\infty,1)$-category and by the discussion in \ref{stableinfinitycategories}, for any $G\in  Fun(N(\sch)^{op},  \widehat{\Sp})$ we have a mapping spectrum functor $Map^{Sp}(G, -): Fun(N(\sch)^{op},  \widehat{\Sp})\to  \widehat{\Sp}$ which when composed with $\Omega^{\infty}$ recovers the mapping space functor in $Fun(N(\sch)^{op},  \widehat{\Sp})$. Moreover, because of the universal property that defines it  and because the composition $\Omega^{\infty}Map^{Sp}(G,-)$ commutes with all limits, we conclude that $Map^{Sp}(G,-)$ also commutes with all limits. In particular, by the Adjoint functor theorem, it admits a left adjoint which we shall denote as $\delta_G:  \widehat{\Sp}\to Fun(N(\sch)^{op},  \widehat{\Sp})$ and for any $K\in  \widehat{\Sp}$ and $F\in Fun(N(\sch)^{op},  \widehat{\Sp})$ we have 

\begin{equation}
Map_{\Sp}(K, Map^{Sp}(G,F))\simeq Map_{Fun(N(\sch)^{op},  \widehat{\Sp})}(\delta_G(K), F)
\end{equation}

We can now use this to define the class of maps that generate the Nisnevich localization. Namely, for any Nisnevich square

\begin{equation}
\xymatrix{
W\ar[r]\ar[d]& \ar[d]V\\
U\ar[r]& X
}
\end{equation}

\noindent we consider its image through $\Sigma^{\infty}_{+}\circ j$

\begin{equation}
\xymatrix{
 \Sigma^{\infty}_{+}\circ j(W)\ar[r]\ar[d]& \ar[d] \Sigma^{\infty}_{+}\circ j(V)\\
 \Sigma^{\infty}_{+}\circ j(U)\ar[r]&  \Sigma^{\infty}_{+}\circ j(X)
}
\end{equation}

\noindent in $Fun(N(\sch)^{op},  \widehat{\Sp})$.  Every commutative square like this produces a commutative diagram of natural transformations

\begin{equation}
\xymatrix{
\delta_{ \Sigma^{\infty}_{+}\circ j(W)}\ar[r]\ar[d]& \ar[d]\delta_{ \Sigma^{\infty}_{+}\circ j(V)}\\
 \delta_{\Sigma^{\infty}_{+}\circ j(U)}\ar[r]& \delta_{ \Sigma^{\infty}_{+}\circ j(X)}
}
\end{equation}

\noindent and in particular for each $K\in  \widehat{\Sp}$, a commutative diagram in  $Fun(N(\sch)^{op},  \widehat{\Sp})$

\begin{equation}
\xymatrix{
\delta_{ \Sigma^{\infty}_{+}\circ j(W)}(K)\ar[r]\ar[d]& \ar[d]\delta_{ \Sigma^{\infty}_{+}\circ j(V)}(K)\\
 \delta_{\Sigma^{\infty}_{+}\circ j(U)}(K)\ar[r]& \delta_{ \Sigma^{\infty}_{+}\circ j(X)}(K)
}
\end{equation}

Finally, we localize with respect to the class of maps

\begin{equation}
 \delta_{\Sigma^{\infty}_{+}\circ j(U)}(K)\coprod_{\delta_{ \Sigma^{\infty}_{+}\circ j(W)}(K)}\delta_{ \Sigma^{\infty}_{+}\circ j(V)}(K)\to  \delta_{ \Sigma^{\infty}_{+}\circ j(X)}(K)
\end{equation}

\noindent given by the universal property of the pushout, with $K$ in $( \widehat{\Sp})^{\omega}$ \footnote{Here $( \widehat{\Sp})^{\omega}$ denotes the full subcategory of $ \widehat{\Sp}$ spanned by the compact objects. Recall that $ \widehat{\Sp}\simeq Ind(( \widehat{\Sp})^{\omega})$.} and $W$,$V$,$U$ and $X$ part of a Nisnevich square. Finally, the fact that this class satisfies the required properties follows directly from the definition of the functors $\delta_{ \Sigma^{\infty}_{+}\circ j(-)}$ as left adjoints to $Map^{Sp}$ and from the enriched version of the Yoneda's lemma \ref{enrichedyoneda}.\\

For the $\mathbb{A}^1$ localization, we localize with respect to the class of all induced maps 

\begin{equation}
\delta_{ \Sigma^{\infty}_{+}\circ j(X\times \mathbb{A}^1)}(K)\to \delta_{ \Sigma^{\infty}_{+}\circ j(X)}(K)
\end{equation}

\noindent  with $X$ in $N(\sch)^{op}$ and $K \in ( \widehat{\Sp})^{\omega}$.

We observe that these localizations are monoidal. This  follows because for any two objects $G$ and $G'$ in $Fun(N(\sch)^{op},  \widehat{\Sp})$, we have

\begin{eqnarray}
Map_{Fun(N(\sch)^{op},  \widehat{\Sp})}(\delta_{G\otimes G'}(K), F)\simeq Map_{ \widehat{\Sp}}(K, Map^{Sp}(G\otimes G',F))\simeq\\
\simeq Map_{ \widehat{\Sp}}(K, Map^{Sp}(G ,\underline{Hom}(G',F))\simeq Map_{Fun(N(\sch)^{op},  \widehat{\Sp})}(\delta_{G}(K), \underline{Hom}(G',F))\simeq\\
\simeq Map_{Fun(N(\sch)^{op},  \widehat{\Sp})}(\delta_{G}(K)\otimes G',F)) 
\end{eqnarray}

\noindent where $\underline{Hom}$ denotes the internal-hom in $Fun(N(\sch)^{op},  \widehat{\Sp})$\footnote{which exists because $Fun(N(\sch)^{op},  \widehat{\Sp})$ is a stable presentable symmetric monoidal $(\infty,1)$-category}. In particular, the previous chain of equivalences implies $\delta_{G\otimes G'}(K)\simeq \delta_{G}(K)\otimes G'$. 

Finally, we denote the result of both these localizations as $Fun^{Nis}_{\mathbb{A}^1}(N(\sch)^{op},  \widehat{\Sp})^{\otimes}$.  To conclude, we invert $\mathbb{G}_m$ and obtain a new stable presentable symmetric monoidal $(\infty,1)$-category which by the universal properties involved, is canonically monoidal equivalent to $\stmonoidalk$.

\subsection{Stable Presentable Symmetric Monoidal $(\infty,1)$-Categories of Motives over a Scheme}

One important result in the subject of motives (see \cite{rodingsostaer}-Theorem 1.1) tells us that the homotopy category of modules over the \emph{motivic Eilenberg-Maclane spetrum} $M\mathbb{Z}$ in 

\begin{equation}h(\st)\simeq h(Sp^{\Sigma}((\M_{\mathbb{A}^1})_{*}, (\mathbb{P}^1,\infty)))\end{equation}

\noindent is (monoidal) equivalent to the triangulated category of motives constructed by Voevodsky in \cite{voevodsky-triangulatedmotives} (whenever $S$ is field of characteristic zero).\\
 
This brings the study of motives to the realm of abstract homotopy theory: it is encoded in the homotopy theory of module objects in $Sp^{\Sigma}((\M_{\mathbb{A}^1})_{*}, (\mathbb{P}^1,\infty))$. However, this is exactly one of those situations where the theory of strictly commutative algebra-objects and their associated theories of modules do not have satisfactory model structures in the sense of the Section \ref{link2}. Therefore, this is exactly one of those situations where the techniques of higher algebra are crucial: they provide a direct access to the $(\infty,1)$-category of commutative algebra objects $CAlg(\st)$ where we can recognize the $K$-theory ring spectrum $\mathbb{K}$ (see for instance \cite{gepner-algebraiccobordismalgebraicKtheory} and \cite{1010.3944}) and the motivic Eilenberg-MacLane spectrum $H\mathbb{Z}$. Moreover, we also have direct access to the theory of modules $\Mot:=Mod_{H\mathbb{Z}}(\st)$. In addition, since $\stmonoidal$ is a presentable symmetric monoidal $(\infty,1)$-category, $\Mot$ is also presentable and inherits a natural symmetric monoidal structure $\Motmonoidal$. Plus, since $\stmonoidal$ is stable, the $(\infty,1)$-category $\Motmonoidal$ is also stable because an $\infty$-category of modules-objects over an algebra in a stable symmetric monoidal $\infty$-category, is stable. Therefore, the homotopy category $h(\Mot)$ carries a canonical triangulated structure and by our results and the main result of \cite{rodingsostaer} it is equivalent to the triangulated category of motives of Voevodsky.\\

We should now be able to reproduce the results of \cite{ayoub1, ayoub2} and \cite{cisinski-tcmm} in this new setting. More precisely, the assignment $S\mapsto \Motmonoidal$ should be properly understood as an $\infty$-functor with values in stable presentable symmetric monoidal $(\infty,1)$-categories and we should study its descent properties and verify the \emph{six-operations} (which have recently been well-understood and reformulated in the setting of symmetric monoidal $(\infty,1)$-categories \cite{1211.5948,1211.5294}).

\section{Noncommutative Motives - Motivic Stable Homotopy Theory of Noncommutative Spaces over a Ring}
\label{section6}

Our goal in this section is to formulate a motivic stable homotopy theory for the noncommutative spaces of Kontsevich \cite{kontsevich3, kontsevich2, kontsevich1}. This new construction will be canonically related to the classical theory for schemes by means of the universal property proved in the previous chapter of this paper. We start with a small survey of the main notions and results concerning the Morita theory of dg-categories and its relation to the notion of finite type introduced by Toën-Vaquié in \cite{toen-vaquie}. In the second part we review how a classical scheme gives birth to a dg-category $L_{qcoh}(X)$ with a compact generator and compact objects given by the perfect complexes of quasi-coherent sheaves. It follows that $L_{pe}(X)$ - the full sub-dg-category of $L_{qcoh}(X)$ spanned by the compact objects - is enough to recover the whole $L_{qcoh}(X)$. In \cite{kontsevich1}, Kontsevich proposes the dg-categories of the form $L_{pe}(X)$ as the natural objects of noncommutative geometry. We recall his notions of smoothness and properness for dg-categories and how they relate to the notion of finite type. The last is the appropriate notion of smoothness while the notion of Kontsevich should be understood as "formal smoothness". Following this, we define the $(\infty,1)$-category of smooth noncommutative spaces $\nc$ as the opposite of the $(\infty,1)$-category of dg-categories of finite type and explain how the formula $X\mapsto L_{pe}(X)$ can be arranged as a monoidal functor from schemes to $\nc$. Finally, we perform the construction of a motivic stable homotopy theory for these new noncommutative spaces. After the work in the previous chapter, our task is reduced to find the appropriate analogue for the Nisnevich topology in the noncommutative setting.

\subsection{ Preliminaries on Dg-categories}

\subsubsection{The Homotopy Theory of dg-categories over a ring}
\label{dg1}

For a first contact with the subject we recommend the highly pedagogical expositions in \cite{sedano, ondgcategories}. In order to make our statements precise, we will need to work with three universes $\uniU\in \uniV\in \uniW$ (which we will assume big enough to fit our purposes). The reader is invited to verify that none of the definitions and constructions depends on the choice of the universes. The $\uniU$-small objects will be called \emph{small}, the $\uniV$-small \emph{big} and the $\uniW$-small, \emph{very big}. We redirect the reader to the Section \ref{notations} for our notations and 
conventions. \\

Let $\uniU$ be our fixed base universe. For the rest of this section we fix a small commutative ring $k$ and following the discussion in \ref{complexes} we denote by $Ch(k)$ the big category of (unbounded) complexes of small $k$-modules. By definition, a small dg-category $T$ is a small category enriched over $Ch(k)$. In other words, $T$ consists of: a small collection of objects $Ob(T)$, for every pair of objects $x,y\in Ob(T)$ a complex of small $k$-modules $T(x,y)$ and composition maps $T(x,y)\otimes T(y,z)\to T(x,z)$ satisfying the standard coherences of a composition. To every small dg-category $T$ we can associated a classical small category $[T]$ with the same objects of $T$ and morphisms given by the zero-homology groups $H^0(T(x,y))$. It is called the \emph{homotopy category of $T$}.\\

Let $\Cat_{Ch(k)}$ the big category of all small dg-categories together with the $Ch(k)$-enriched functors (dg-functors for short).

\begin{remark}
We can of course give sense to these notions within any universe. In our context we will denote by $\Cat_{Ch(k)}^{big}$ the very big category of big dg-categories. We will allow ourselves not to mention the universes whenever the context applies for any universe.
\end{remark}

Let us provide some important simple examples:

\begin{example}
\label{dgalgebrasdgcats}
Given a $k$-dg-algebra $A$, we denote by $A_{dg}$ the dg-category with a single object and $A$ as endomorphisms, with composition given by the multiplicative structure of $A$. This assignment provides a fully-faithful functor $(-)_{dg}:Alg(Ch(k)) \hookrightarrow \Cat_{Ch(k)}$. With this, we a have a commutative diagram of fully-faithful functors between ordinary categories

\begin{equation}
\xymatrix{
k-algebras \ar@{^{(}->}[r] \ar@{^{(}->}[d]&\ar@{^{(}->}[d] k-additive Cats\\
Alg(Ch(k))\ar@{^{(}->}[r]^{(-)_{dg}}&\Cat_{Ch(k)}
}
\end{equation}

\noindent where the horizontal maps understand an algebra as a category with one object and the vertical maps understand an object as concentrated in degree zero.
\end{example}

\begin{example} 
The category of complexes $Ch(k)$ has a natural structure of dg-category with the enrichement given by the internal-hom described in \ref{complexes}. We will write $Ch_{dg}(k)$ to denote $Ch(k)$ together with this enrichement.
\end{example}

\begin{construction}
Every model category $\M$ with a compatible $Ch(k)$-enrichment (see Def. 4.2.6 of \cite{hovey-modelcategories}) provides a new dg-category $Int(\M)$: the full dg-category of $\M$ spanned by the cofibrant-fibrant objects (usually called the \emph{underlying dg-category of $\M$}). Also in this case, the homotopy category $h(Int(\M))$ recovers the usual homotopy category of $M$. We will call $Int(\M)$ the \emph{underlying dg-category of $\M$}. This machine is crucial to the foundational development of the theory. The dg-category $Ch_{dg}(k)$ of the previous item is a canonical example of this situation since its model structure is compatible with the $Ch(k)$-enrichment. In this case, $Int(Ch_{dg}(k))$ is the full dg-subcategory of $Ch_{dg}(k)$ spanned by the cofibrant complexes (all objects are fibrant). \\
\end{construction}

If $T$ and $T'$ are dg-categories, we can form a third dg-category $T\otimes T'$ where the objects are the pairs $(x,y)$ with $x$ an object in $T$ and $y$ in $T'$ and the mapping complex $T\otimes T'((x,y), (x',y'))$ is given by the tensor product of complexes in $Ch(k)$, $T(x,x')\otimes T(y,y')$. This formula endows $Cat_{Ch(k)}$ with a symmetric monoidal structure with unit $\mathit{1}_{k}$ given by the dg-category with a single object and $k$ as its complex of endomorphisms. We can use the general arguments of \cite{kelly-enriched} to deduce the existence of an internal-hom functor. In other words, given $T$ and $T'$ there is a new dg-category $\underline{Hom}(T,T')$ and a natural isomorphism
\begin{equation}
Hom_{Cat_{Ch(k)}}(T'', \underline{Hom}(T,T'))\simeq Hom_{Cat_{Ch(k)}}(T''\otimes T, T')
\end{equation}

In particular the objects of $\underline{Hom}(T,T')$ are the $Ch(k)$-enriched functors and the morphisms can be identified with the $Ch(k)$-natural transformations.\\

\begin{construction}
If $T$ is a dg-category, the \emph{dg-category of $T$-dg-modules} is defined by the formula

\begin{equation}Ch(k)^T:=\underline{Hom}(T, Ch_{dg}(k))\end{equation}

An object $E\in Ch(k)^T$ can be naturally identified with a formula that assigns to each object $x\in T$ a complex $E(x)$ and for each pair of objects $x,y$ in $T$, a map of complexes $T(x,y)\otimes E(x)\to E(y)$ compatible with the composition in $T$.

The dg-category of dg-modules carries a natural $Ch(k)$-model structure induced from the one in $Ch(k)$, with weak-equivalences and fibrations determined objectwise. More generally, for any $Ch(k)$-model category $\M$, the \emph{dg-category of $T$-modules with values in $\M$} is defined by the formula $\M^T:=\underline{Hom}(T, \M)$. If $\M$ is a cofibrantly generated model category then we can equip $\M^T$ with the projective model structure and we can check that this is again compatible with the dg-enrichment. This construction can be made functorial in $T$. More precisely, if $f:T\to T'$ is a dg-functor, we have a canonical restriction functor $f^*:\M^{T'}\to \M^T$ defined by sending a $T'$-module $F$ to the composition $F\circ f$. This functor admits a left adjoint $f_!$ and the pair $(f_!, f^*)$ forms a Quillen adjunction compatible with the $Ch(k)$-enrichment. 
\end{construction}

\begin{remark}
\label{bimodulesandcoisas}
If $T$ is the dg-category of the form $A_{dg}$ as in the example \ref{dgalgebrasdgcats}, $Ch(k)^T$ can be naturally identified with the category of left $A$-modules in $Ch(k)$. In particular, when endowed with the projective model structure, the dg-category $Int(Ch(k)^T)$ is a dg-enhancement of the classical derived category of $A$, in the sense that the $[Int(Ch(k)^T)]\simeq D(A)$. Another important case is when $T$ is associated to the product of two dg-algebras $A\otimes B^{op}$. In this case the category of $T$-modules is naturally isomorphic to $BiMod(A,B)(Ch(k))$.
\end{remark}

In practice we are not interested in the strict study of dg-categories but rather on what results from the study of complexes up to quasi-isomorphisms. A \emph{Dwyer-Kan} equivalence of dg-categories is an (homotopic) fully faithful $Ch(k)$-enriched functor $f:T\to T'$ (ie, such that the induced maps $T(x,y)\to T'(f(x),f(y))$ are weak-equivalences in $Ch(k)$) such that the functor induced between the homotopy categories $[T]\to [T']$ is essentially surjective. Of course, if $T\to T'$ is a Dwyer-Kan equivalence, the induced functor $[T]\to [T']$ is an equivalence of ordinary categories. It is the content of \cite{tabuada-quillen} that $\Cat_{Ch(k)}$ admits a (non left-proper) cofibrantly generated model structure to study these weak-equivalences. Moreover, the model structure is combinatorial because $Cat_{Ch(k)}$ is known to be presentable (see \cite{MR1870617}). The fibrations are the maps $T\to T'$ such that the induced applications $T(x,y)\to T'(f(x),f(y))$ are surjections (meaning fibrations in $Ch(k))$ and the map induced between the associated categories has the lifting property for isomorphisms. Therefore, every object is fibrant and the cofibrant objects, which are more difficult to describe, are necessarily enriched over cofibrant complexes (see Prop. 2.3 in \cite{Toen-homotopytheorydgcatsandderivedmoritaequivalences}). We will address to this model structure as the "standard one".\\

\begin{remark}
\label{modulesbehavewellunderequivalence}
The theory of modules over dg-categories is well-behaved with respect to the Dwyer-Kan equivalences. By the Proposition 3.2 of \cite{Toen-homotopytheorydgcatsandderivedmoritaequivalences}, if $f:T\to T'$ is a Dwyer-Kan equivalence of dg-categories, then the adjunction $f^*:\M^{T'}\to \M^T$  is a Quillen equivalence if one of the following situations hold: $(i)$ $T$ and $T'$ are locally cofibrant (meaning: enriched over cofibrant complexes); $(ii)$ the product (using the $Ch(k)$-action) of a cofibrant object $A$ in $\M$ with a weak-equivalence of complexes $C\to D$ is a weak-equivalence in $\M$. In particular the second condition holds if $\M= Ch(k)$. Moreover, by the Proposition 3.3 of \cite{Toen-homotopytheorydgcatsandderivedmoritaequivalences}, if $T$ is locally cofibrant then the evaluation functors $ev_x:\M^T\to \M$ sending $F\mapsto F(x)$, preserve fibrations, cofibrations and weak-equivalences. In particular, $Int(\M^T)$ is made of objectwise cofibrant-fibrant objects in $\M$.\\
\end{remark}

The information of this homotopical study is properly encoded in a new big\footnote{because $\Cat_{Ch(k)}$ is big, its nerve is a a big simplicial set and the localization is obtained as a cofibrant-fibrant replacement in the model category of big marked simplicial sets} $(\infty,1)$-category

\begin{equation}
\dg:= N(\Cat_{Ch(k)})[W_{DK}^{-1}]
\end{equation}

\noindent where $W_{DK}$ denote the collection of all Dwyer-Kan equivalences. Because the homotopy category $h(\dg)$ recovers the ordinary localization in $\Cat$, the objects of $\dg$ can be again identified with the small dg-categories. Notice that in this situation we cannot apply the Proposition \ref{dwyer-kan} because $\Cat_{Ch(k)}$ is not a simplicial model category. 

\begin{remark}
It is important to remark that the inclusion of very big categories $Cat_{Ch(k)} \subseteq Cat_{Ch(k)}^{big}$ is compatible with the model structure of \cite{tabuada-quillen} and we have a fully-faithful map of very big $(\infty,1)$-categories $\dg\subseteq \dg^{big}$ (where the last is defined by the same formula using the theory for $\uniV$-small simplicial sets).
\end{remark}

\begin{remark}
The inclusion $(-)_{dg}:Alg(Ch(k)) \hookrightarrow Cat_{Ch(k)}$ sends weak-equivalences of dg-algebras to Dwyer-Kan equivalences. The localization gives us a canonical map of $(\infty,1)$-categories $Alg(\derivedk)\to \dg$ (see \ref{complexes}). As pointed to me by B. Toën, this map is not fully-faithfull anymore. To see this, we observe that the new map $Alg(\derivedk)\to \dg$ factors as $Alg(\derivedk)\to \dg_{\ast}\to \dg$ where $\dg_{\ast}$ denotes the $(\infty,1)$-category of pointed objects in $\dg$. The first map in the factorization is fully-faithful but the second one is not. Indeed, if $A$ and $B$ are two dg-algebras, the mapping space $Map_{\dg_{\ast}}(A_{dg}, B_{dg})$ can be obtained as the homotopy quotient of $Map_{\dg}(A,B)$ by the action of the simplicial group of units in $B$. 
\end{remark}

\begin{remark}
The theory of \emph{$A_{\infty}$-categories} of \cite{kontsevich-soibelman} provides an equivalent approach to the homotopy theory of dg-categories.
\end{remark}

We now collect some fundamental results concerning the inner structure of the $(\infty,1)$-category $\dg$.

\begin{enumerate}
\item \emph{ Existence of Limits and Colimits}: This results from the fact that the model structure in $Cat_{Ch(k)}$ is combinatorial, together with the Proposition A.3.7.6 \cite{lurie-htt} and the main result of \cite{dugger-combinatorial}.

\item \emph{Symmetric Monoidal Structure:} Notice that $Cat_{Ch(k)}$ is \emph{not} a symmetric monoidal model category in the sense of the Definition 4.2.6 in \cite{hovey-modelcategories}. For instance, the product of the two cofibrant objects is not necessarily cofibrant (Exercise 14 of \cite{sedano}). Therefore, we cannot apply directly the abstract-machinery reviewed in the Section \ref{link2} to deduce the existence of a monoidal structure in $\dg$. Luckily, we can overcome this problem and extend the monoidal structure to $\dg$ even under these bad circumstances.\\

We observe first that the product of dg-categories whose hom-complexes are cofibrant in $Ch(k)$ (also called \emph{locally cofibrant}) is again a dg-category with cofibrant hom-complexes. This follows from the fact that $Ch(k)$ is a symmetric monoidal model category and so the product of cofibrant complexes is again a cofibrant complex. Second, we notice that the product of weak-equivalences between dg-categories with cofibrant hom-complexes is again a weak-equivalence. It is enough to check that for any triple of locally-cofibrant dg-categories $T$, $T'$ and $S$ and for any weak-equivalence $T\to T'$, the product $T\otimes S\to T'\otimes S$ is again a weak-equivalence. The fact that the map between the homotopy categories is essentially surjective is immediate. Everything comes down to prove that if $M$ is a cofibrant complex of $k$-modules and $N\to P$ is quasi-isomorphism between cofibrant complexes then $M\otimes N\to M\otimes P$ is also a quasi-isomorphisms of complexes. Again this follows because $Ch(k)$ is a symmetric monoidal model category.

The first conclusion of this discussion is that the full-subcategory $Cat_{Ch(k)}^{loc-cof}$ of $Cat_{Ch(k)}$ spanned by the locally-cofibrant dg-categories, is closed under tensor products and contains the unit of $Cat_{Ch(k)}$ and therefore inherits a symmetric monoidal structure

We have inclusions

\begin{equation}Cat_{Ch(k)}^{cof}\subseteq Cat_{Ch(k)}^{loc-cof}\subseteq Cat_{Ch(k)}\end{equation}

mapping weak-equivalences to weak-equivalences. Since in $Cat_{Ch(k)}$ we can choose a functorial cofibrant-replacement $Q$ preserving the objects (see \cite{tabuada-quillen} for details) and together with the inclusions in the previous diagram, $Q$ produces equivalences of $(\infty,1)$-categories 

\begin{equation}N(Cat_{Ch(k)}^c)[W_{c}^{-1}]\simeq \dgloccof\simeq \dg\end{equation}

where we set $\dgloccof:=N(Cat_{Ch(k)}^{loc-cof})[W_{loc-cof}^{-1}]$. \\

Finally, since the symmetric monoidal structure in $N(Cat_{Ch(k)}^{loc-cof})$ preserves weak-equivalences, by the discussion in \ref{link2}, the localization $\dgloccof$ inherits a natural symmetric monoidal structure  $\dgloccofmonoidal\to N(\Fin)$ obtained as the monoidal localization of $Cat_{Ch(k)}^{loc-cof}$ seen as a trivial simplicial coloured operad (followed by the Grothendieck construction). We can now use the equivalence $\dgloccof\simeq \dg$ to give sense to the product of two arbitrary dg-categories $T\otimes^{\mathbb{L}}T'\simeq Q(T)\otimes T'$ . This recovers the famous formula for the derived tensor product. 

\begin{remark}
\label{dgalgebrasdgcatsmonoidal}
The category of strict $k$-dg-algebras inherits a symmetric monoidal structure, obtained by tensoring the underlying complexes. It follows that $(-)_{dg}: Alg(Ch(k))\hookrightarrow Cat_{Ch(k)}$ is monoidal. Moreover, since the product of cofibrant dg-algebras remains a dg-algebra with a cofibrant underlying complex (as proved in \cite{schwede-shipley-algebrasandmodulesinmonoidalmodelcategories}), we can use the Remark \ref{last3} and the discussion in \ref{complexes} to deduce that the induced inclusion 
$(-)_{dg}: Alg(\derivedk)\subseteq \dgloccof\simeq \dg$ is a monoidal functor.
\end{remark}

\begin{notation}
Following the previous remark, we will sometimes abuse the notation and identify a dg-algebra $A$ with its associated dg-category $A_{dg}$.
\end{notation}

\item \emph{The Mapping spaces in $\dg$}: The first important technical result of \cite{Toen-homotopytheorydgcatsandderivedmoritaequivalences} is the characterization of the mapping spaces $Map_{\dg}(T,T')$. The description uses the monoidal structure introduced in the previous item: from the input of two dg-categories $T$ and $T'$, we consider the $Ch(k)$-model category of $T\otimes^{\mathbb{L}} (T')^{op}:= Q(T)\otimes (T')^{op}$-dg-modules. Again, this homotopy theory is properly encoded in the $(\infty,1)$-category

\begin{equation}(T,T')_{\infty}:=N(Ch(k)^{T\otimes^{\mathbb{L}} (T')^{op}})[W_{qis}^{-1}]\end{equation}

\noindent inside which we can isolate the full subcategory spanned by the \emph{right quasi-representable objects}. \footnote{By definition, these are the $T\otimes (T')^{op}$- modules $F$ such that for any $x\in T$, there is an object $f_x \in T'$ and an isomorphism between $F(x,-)$ and $T'(-,f_x)$ in the homotopy category of $(T')^{op}$-modules} which we denote here as $rrep(T,T')_{\infty}$. By the Theorem 4.2 of \cite{Toen-homotopytheorydgcatsandderivedmoritaequivalences}, there is an explicit isomorphism of homotopy types
\begin{equation}
Map_{\dg}(T,T')\simeq rrep(T,T')_{\infty}^{\simeq}
\end{equation}

where $rrep(T,T')_{\infty}^{\simeq}$ denotes the $\infty$-groupoid of equivalences in $rrep(T,T')_{\infty}$.

\begin{remark}
\label{nc1nerveequivalencesequalsmaximalgrupoid}
The original formulation of this theorem in \cite{Toen-homotopytheorydgcatsandderivedmoritaequivalences} uses another presentation of the Kan-complex $rrep(T,T')^{\simeq}$. Let $\M$ be a model category with weak-equivalences $W$. In one direction, we can consider the subcategory $W$ of $\M$ consisting of all the objects in $\M$ together with the weak-equivalences between them. The inclusion $W\subseteq \M$ sends weak-equivalences to weak-equivalences and by using the nerve we have a natural homotopy commutative diagram of $(\infty,1)$-categories

\begin{equation}
\xymatrix{
N(W)\ar[r]\ar[d]& N(M)\ar[d]\\
N(W)[W^{-1}]\ar[r]&N(M)[W^{-1}]
}
\end{equation}

The $(\infty,1)$-category $N(W)[W^{-1}]$ is a Kan-complex (because every arrow is invertible)\footnote{See for instance the Proposition 1.2.5.1 of \cite{lurie-htt}} and therefore, the map $N(W)[W^{-1}]\to N(M)[W^{-1}]$ factors as $N(W)[W^{-1}]\to N(M)[W^{-1}]^{\simeq}\subseteq N(M)[W^{-1}]$. This map is a weak-equivalence of simplicial sets for the Quillen structure because $M$ is a model category. It results from the foundational works of Dwyer and Kan (see the Proposition 4.3 of \cite{dwyer-kan-simpliciallocalizationofcategories} ) that the canonical map $N(W)\to N(W)[W^{-1}]$ is also weak-equivalences of simplicial sets for the standard model structure.
\end{remark}

\item \emph{The monoidal structure in $\dg$ is closed:} More precisely, by the Theorem 6.1 of \cite{Toen-homotopytheorydgcatsandderivedmoritaequivalences}, for any three small dg-categories $A$, $B$ and $C$, there exists a new small dg-category $\mathbb{R}\underline{Hom}(B,C)$ (in the same universe of $B$ and $C$ - see Proposition 4.11 of \cite{Toen-homotopytheorydgcatsandderivedmoritaequivalences}) and functorial isomorphisms of homotopy types

\begin{equation}
Map_{\dg}(A\otimes^{\mathbb{L}}B, C)\simeq Map_{\dg}(A, \mathbb{R}\underline{Hom}(B,C))
\end{equation}

Furthermore, $\mathbb{R}\underline{Hom}(B,C)$ is naturally equivalent in $\dg$ to the full essentially small sub-dg-category $Int(B\otimes^L C^{op}-Mod))_{rrep}$ of $Int(B\otimes^L C^{op}-Mod))$ spanned by the right quasi-representable modules.

An immediate implication of this result is that the derived tensor product is compatible with colimits on each variable separately. \\

\item \emph{Existence of dg-localizations}: The description of the mapping spaces in $\dg$ allows us to prove the existence of a localization process inside the dg-world. By the Corollary 8.7 of \cite{Toen-homotopytheorydgcatsandderivedmoritaequivalences}, given a dg-category $T\in \dg$ together with a class of morphisms $S$ in $[T]$ we can formally construct a new dg-category $L_ST$ together with a map $T\to L_ST$ in $\dg$, such that for any dg-category $T'$ the composition map 

\begin{equation}
Map_{\dg}(L_ST, T')\to Map^S_{\dg}(T,T')
\end{equation}

is a weak-equivalence. Here $Map^S_{\dg}(T,T')$ denotes the full simplicial set of $Map_{\dg}(T,T')$ given by the union of all connected components corresponding to morphisms in $h(\dg)$ sending $S$ to isomorphisms in $[T']$. 
Another way to formulate this is to say that for each pair $(T,S)$, the functor $\dg\to \Spaces$ sending $T'\mapsto Map^S_{\dg}(T,T')$ is co-representable.\\

\begin{remark}
\label{localizationdg}
This localization allows us to prove a dg-analogue of a fundamental result of Quillen for model categories: for a $Ch(k)$-model category $\M$, the dg-localization of $\M$ with respect to its weak-equivalences is equivalent to $Int(\M)$ (see \cite{sedano}). This motivates the terminology for the last.
\end{remark}

\end{enumerate}

\subsubsection{Morita Theory of dg-categories}
\label{morita}

Let $T$ be a small dg-category. The enriched version of the Yoneda's Lemma allows us understand $T$ as a full sub-dg-category of $Ch(k)^{T^{op}}$.

\begin{equation}
h:T\to Ch(k)^{T^{op}}
\end{equation}

This big dg-category admits a compatible model structure induced from the one in $Ch(k)$. It is well known that this model structure is stable so that its homotopy category inherits a canonical triangulated structure where the exact triangles are the image of the homotopy fibration sequences through the localization map. It is an important remark that for each $x \in T$, the representable $h_x$ is a cofibrant $T^{op}$-module (it follows again from the Yoneda's lemma and the fact that fibrations are defined as levelwise surjections). By setting $\widehat{T}:=Int(Ch(k)^{T^{op}})$, there is a factorization
\begin{equation}
T\to \widehat{T}\subseteq Ch(k)^{T^{op}}
\end{equation}

\noindent and from now will use the letter $h$ to denote the first map in this factorization. Of course, when $T=\mathit{1}_{k}$ we have $(\mathit{1}_{k})^{op}-Mod= Ch(k)$ and therefore $\widehat{\mathit{1}_{k}}=Int(Ch(k))$. \\

\begin{remark}
\label{omega1}
It is important to notice that if $T$ is a small locally-cofibrant dg-category, then $\widehat{T}$ will not be locally-cofibrant in general. However, we will see in the Remark \ref{emerald1} that it is possible to provide an alternative construction of the assignement $T\mapsto \widehat{T}$ that preserves the condition of being locally-cofibrant.
\end{remark}

It is also important to remark that the passage $T\mapsto \widehat{T}$ is (pseudo) functorial. If $f:T\to T'$ is a dg-functor, we have a natural restriction map

\begin{equation}
Ch(k)^{(T')^{op}}\to Ch(k)^{T^{op}} 
\end{equation}

\noindent induced by the composition with $f^{op}$. As a limit preserving map between presentable, it admits a left adjoint and this adjunction is compatible with the model structures. Since all objects are fibrant, the left adjoint restricts to a well-define map $\widehat{T}\to \widehat{T'}$.\\

We finally come to the subject of (derived) Morita theory. Classically, it can described as the study of algebras up to their (derived) categories of modules. The version for small dg-categories implements the same principle: it is the study of small dg-categories $T$ up to their derived dg-categories of modules $\widehat{T}$. It generalizes the classical theory for algebras for when $T$ is a dg-category coming from an algebra $A$ the homotopy category of $\widehat{T}$ recovers the derived category of $A$. We will see that the following three constructions are equivalent:

\begin{enumerate}[a)]
\item the localization of $Cat_{Ch(k)}$ with respect to the class of dg-functors $T\to T'$ for which the induced map
$\widehat{T}\to \widehat{T'}$ is a weak-equivalence of dg-categories;
\item the full subcategory of $\dg$ spanned by the \emph{idempotent complete} dg-categories;
\item the (non-full) subcategory of $\dg^{loc-cof,big}$ spanned by the dg-categories of the form $\widehat{T}$ (with $T$ small), together with those morphisms that preserve colimits and compact objects.
\end{enumerate}

The link is made by the notion of a compact object. It is well known that the model category of complexes is combinatorial and compactly generated so that we can apply the results of our discussion in \ref{compactlygeneratedmodelcategories}. It is immediate that the same will hold for the projective structure in model category $Ch(k)^{T}$, for any small dg-category $T$. Following this, we denote by $\widehat{T}_c$ the full sub-dg-category of $\widehat{T}$ spanned by those cofibrant modules which are homotopicaly finitely presented in the model category $Ch(k)^{T^{op}}$. Again by the general machinery described in \ref{compactlygeneratedmodelcategories}, they can be constructed as retracts of strict finite cell-objects and  correspond to the compact objects in the underlying $(\infty,1)$-category of $Ch(k)^{T^{op}}$ and with this in mind we will refer to them as compact. 

It follows from the definitions and from the enriched version of the Yoneda Lemma that for any object $x$ in $T$, the representable dg-module $h(x)$ is compact. In particular, $h$ factors as $T\to (\widehat{T})_c\subseteq \widehat{T}$. At the level of the homotopy categories, this produces a sequence of inclusions $[T]\subseteq [(\widehat{T})_c]\subseteq [\widehat{T}]$ and the fact that $Ch(k)^{T^{op}}$ is stable model category implies two important things: $(i)$ the category $[\widehat{T}]$ has a triangulated structure; $(ii)$ because homotopy pushouts are homotopy pullbacks, a dg-module $F$ is compact if and only if $Map(F,-)$ commutes with arbitrary coproducts. With this we identify the subcategory $[(\widehat{T})_c]\subseteq [\widehat{T}]$ with $[\widehat{T}]_c\subseteq [\widehat{T}]$ - the full triangulated subcategory spanned by the compact objects in the sense of Neeman (see \ref{senseofNeeman}). In particular when $T=1_k$, the objects in $\widehat{T}_c$ are are exactly the perfect complexes of $k$-modules.

\begin{remark}
Because any compact module is a strict finite $I$-cell (\ref{compactlygeneratedmodelcategories}) the dg-category $(\widehat{T})_c$ is essentially small and can be considered as an object in $\dg$. For the same reason, $(\widehat{T})_c$ is stable under shifts, retracts and pushouts.
\end{remark}

Recall now that a small dg-category is said to be \emph{idempotent complete} (or \emph{triangulated}) if the dg-functor

\begin{equation}
T\to (\widehat{T})_c
\end{equation}

\noindent is a weak-equivalence of dg-categories. The first reason why idempotent dg-categories are relevant to Morita theory is because for any small dg-category $T$, the restriction along $T\to (\widehat{T})_c$

\begin{equation}
\widehat{((\widehat{T})_c)}\to \widehat{T}
\end{equation}

\noindent is a weak-equivalence of dg-categories (Lemma 7.5-(1) in \cite{Toen-homotopytheorydgcatsandderivedmoritaequivalences}). In other words, at the level of modules we cannot distinguish between $T$ and $(\widehat{T})_c$. It follows that a dg-functor $f:T\to T'$ induces a weak-equivalence $\widehat{T}\to \widehat{T'}$ if and only if its restriction to compact objects \footnote{This restriction is well-defined because every compact dg-module in $\widehat{T}$ can be constructed from representables using finite colimits and retracts. The conclusion follows because the map $\widehat{T}\to \widehat{T'}$ sends representables to representables, representables are compact and compact objects stable under finite colimits and retracts. (See \cite{Toen-homotopytheorydgcatsandderivedmoritaequivalences} for more details )} $(\widehat{T})_c\to (\widehat{T'})_c$ is a weak-equivalence. We will denote by $\dg^{idem}$ the full subcategory of $\dg$ spanned by those dg-categories which are idempotent complete.

\begin{prop}
\label{omega2}
The formula $T\mapsto (\widehat{T})_c $ provides a left adjoint to the inclusion $\dg^{idem} \subseteq \dg$. 
\begin{proof}
In order to prove this result we construct a cocartesian fibration $(\infty,1)$-categories $p:\M\to \Delta[1]$ with $p^{-1}(\{0\})=\dg$ and $p^{-1}(\{1\})=\dg^{idem}$. We consider the full $(\infty,1)$-category $\M$ of the product $\dg\times \Delta[1]$ spanned by the pairs $(T,0)$ where $T$ can be any small dg-category and the pairs of the form $(T,1)$ only accept dg-categories $T$ which are triangulated. By construction, there is a canonical projection $p:\M\to \Delta[1]$ whose fiber over $0$ is $\dg$ and over $1$ is $\dg^{idem}$. We are reduced to check that $p$ is a cocartesian fibration. Notice a map in $\M$ over the morphism $0\to 1$ in $\Delta[1]$ consists in the data of a morphism $T\to T'$ in $\M$ where $T$ is any dg-category and $T'$ is a triangulated one. To say that $p$ is cocartesian is equivalent to say that for any dg-category $T$, there is a new triangulated dg-category $T'$ together with a morphism $T\to T'$ having the following universal property: for any triangulated dg-category $D$, the composition map

\begin{equation}
Map_{\dg}(T',D)\to Map_{\dg}(T, D)
\end{equation}

\noindent is a weak-equivalence of spaces. We set $T':= (\widehat{T})_c$ and $T\to T'$ the yoneda's map. Since any triangulated dg-category $D$ is equivalent in $\dg$ to $\widehat{D}_c$, we are reduced to prove the composition map

\begin{equation}Map_{\dg}((\widehat{T})_c, \widehat{D}_c)\to Map_{\dg}(T, \widehat{D}_c)\end{equation}

\noindent is a weak-equivalence. Using the internal-hom, we are reduced to prove that the natural map
\begin{equation}
\mathbb{R}\underline{Hom}((\widehat{T})_c, \widehat{D}_c)\to \mathbb{R}\underline{Hom}(T, \widehat{D}_c)
\end{equation}
\noindent is an equivalence in $\dg$. This is the content of the Theorem 7.2-(2) in \cite{Toen-homotopytheorydgcatsandderivedmoritaequivalences}.

\end{proof}
\end{prop}

The existence of this left adjoint, which we will denote as $(\widehat{-})_c$, makes $\dg^{idem}$ a reflexive localization of $\dg$. In particular the idempotent dg-categories can be described as local objects with respect to the class of maps in $\dg$ whose image through the composition of the inclusion with $(\widehat{-})_c$ is an equivalence. This establishes $\dg^{idem}$ as the second approach in the list.

\begin{remark}
\label{omega3}
The equivalence $\dgloccof \simeq \dg$ restricts to an equivalence $\dg^{idem, loc-cof}\subseteq \dg^{idem}$. Using the Remark \ref{omega1} we see that the left adjoint of the Prop. \ref{omega2} restricts to a left adjoint to the inclusion $\dg^{idem, loc-cof}\subseteq \dgloccof$: if $T$ is small and locally-cofibrant, we can find an equivalent way to define the formula $T\mapsto \widehat{T}$, this time preserving the condition of being locally-cofibrant so that the full subcategory $(\widehat{T})_c$ is locally cofibrant.
\end{remark}

We now explain the first approach. Recall that dg-functor $T\to T'$ in $Cat_{Ch(k)}$ is called a \emph{Morita equivalence} if the induced map $\widehat{T}\to \widehat{T'}$ is a weak-equivalence of dg-categories.

\begin{cor}
\label{moritatheorysmall}
Let $W_{Mor}$ denote the collection of Morita equivalences between small dg-categories.  Then there is a canonical isomorphism in the homotopy category of $(\infty,1)$-categories 

\begin{equation}\dg^{idem}\simeq N(Cat_{Ch(k)})[W_{Mor}^{-1}]\end{equation} 

\begin{proof}
This follows from the universal property of the localization in the setting of $(\infty,1)$-categories, together with the observations that: $(i)$ every weak-equivalence of dg-categories is in $W_{Mor}$ and $(ii)$ a dg-functor $f$ is in $W_{Mor}$ if and only if its image through the composition of localizations 
$N(Cat_{Ch(k)})\to \dg\to \dg^{idem}$ is an equivalence.
\end{proof}
\end{cor}

Both sides of this equivalence admit natural symmetric monoidal structures and the equivalence preserves them. More precisely, for the first side we have

\begin{prop}
\label{epavala}
The reflexive localization $\dg^{idem,loc-cof}\subseteq \dgloccof$ is compatible with the monoidal structure in $\dgloccof^{\otimes}$. In other words, there is a natural symmetric monoidal structure in $\dg^{idem,loc-cof}$ for which the left adjoint is monoidal. Informally, it is given by the formula $T\otimes^{idem} T':= (\widehat{T\otimes^{\mathbb{L}}T'})_c$. In particular, the unit is the idempotent completion of the dg-category with a single object with $k$ as endomorphisms.
\begin{proof}
It is enough to check that if $f:T\to T'$ is a morphism in $\dg$ such that $(\widehat{T})_c\to\widehat{ T'}_c$ is an equivalence, then for any dg-category $C\in \dg$, the product $f\otimes^{\mathbb{L}}Id_C: T\otimes^{\mathbb{L}}C\to T'\otimes^{\mathbb{L}}C$ will also by sent to an equivalence in $\dg^{idem}$. This follows directly from the Lemma 7.5-(1) in \cite{Toen-homotopytheorydgcatsandderivedmoritaequivalences}.
\end{proof}
\end{prop}

\begin{remark}
The combination of the argument in the previous proof, with the Theorem 7.2-(2) of \cite{Toen-homotopytheorydgcatsandderivedmoritaequivalences} implies that for any idempotent dg-category $Z$ and any dg-category $T$, the internal-hom $\mathbb{R}\underline{Hom}(T,Z)$ is again idempotent. In particular, it provides an internal-hom for the monoidal structure in $\dg^{idem}$.\\
\end{remark}

To find a monoidal structure in the second localization $N(Cat_{Ch(k)})[W_{Mor}^{-1}]$ it suffices to verify that the tensor product of Morita equivalences in $Cat_{Ch(k)}$ is again a Morita equivalence. However, and as for the Dwyer-Kan equivalences, this is not true. It happens that everything is well-behaved if we restrict to locally cofibrant dg-categories (see the Proposition 2.22 in \cite{tabuada-cisinski} together with the fact that any cofibrant complex is flat). The problem is solved by considering the monoidal localization of the trivial simplicial coloured operad associated to the well-defined monoidal structure in $Cat_{Ch(k)}^{loc-cof}$. The fact that the equivalence in the Corollary \ref{moritatheorysmall} is monoidal follows immediately from the universal property of the monoidal localization.\\

To compare these two approaches with the third one, it is convenient to have a description of the mapping spaces in $\dg^{idem}$. Being a full subcategory of $\dg$ and using again the Theorem 7.2-(2) of \cite{Toen-homotopytheorydgcatsandderivedmoritaequivalences} we find equivalences

\begin{eqnarray}
Map_{\dg^{idem}}((\widehat{T})_{c}, (\widehat{T'})_c)\simeq Map_{\dg}((\widehat{T})_{c}, (\widehat{T'})_c)\simeq Map_{\dg}(\mathit{1}_k,\mathbb{R}\underline{Hom}((\widehat{T})_{c}, (\widehat{T'})_c))\\
\simeq Map_{\dg}(\mathit{1}_k,\mathbb{R}\underline{Hom}(T, (\widehat{T'})_c))
\end{eqnarray}

\noindent and the internal-hom $\mathbb{R}\underline{Hom}(T, (\widehat{T'})_c)$ is given by the full sub-dg-category of $Int(T\otimes^{\mathbb{L}} ((\widehat{T'})_c)^{op}-modules)$ spanned by the right-representable. In this particular case, the last can be described as the full sub-dg-category of $\widehat{T^{op}\otimes^{\mathbb{L}} T'}$ spanned by those modules $E$ which for every $x\in T,$ the module $E(x,-): (T')^{op}\to Ch(k)$ is compact. These are called \emph{pseudo-perfect} (over $T$ relatively to $T'$). Following \cite{toen-vaquie} we will write $\widehat{T^{op}\otimes^{\mathbb{L}} T'}_{pspe}$ for the sub-dg-category of $\widehat{T^{op}\otimes^{\mathbb{L}} T'}$ spanned by the pseudo-perfect dg-modules over $T$ relative to $T'$. In the next section we will review how the interplay between the notion of pseudo-perfect and compact is essential to express the geometrical behavior of dg-categories. Using this description we have

\begin{equation}
Map_{\dg}(\mathit{1}_k,\mathbb{R}\underline{Hom}(T, \widehat{T'}_c))\simeq Map_{\dg}(\mathit{1}_k,\widehat{T^{op}\otimes^{\mathbb{L}} T'}_{pspe})\simeq
rrep(\mathit{1}_k, \widehat{T^{op}\otimes^{\mathbb{L}} T'}_{pspe})
\end{equation}

\noindent where the last is our notation for the maximal $\infty$-groupoid of the full subcategory spanned by right-representable in the underlying $(\infty,1)$-category of all $\mathit{1}_k\otimes^{\mathbb{L}} (\widehat{T^{op}\otimes^{\mathbb{L}} T'}_{pspe})^{op}$-modules. We can easily check this is equivalent to the maximal $\infty$-groupoid of $pspe(T,T)_{\infty}\subseteq (T,T')_{\infty}$ - the full subcategory spanned by the pseudo-perfect modules.\\

We now come to the third approach. Let $\dgc$ denote the (non full) subcategory of $\dg^{big}$ spanned by the dg-categories of the form $\widehat{T}$ for some small dg-category $T$, together with those morphisms $\widehat{T}\to \widehat{T'}$ whose map induced between the homotopy categories $[\widehat{T}]\to[\widehat{T'}]$ commutes with arbitrary sums\footnote{This notion is well-defined because the map induced between the homotopy categories is unique up to isomorphism of functors.}. Notice that each map in $\dgc$ corresponds to a unique (up to quasi-isomorphism) $(\widehat{T}\otimes \widehat{T'}^{op})$-dg-module. Let $\mathbb{R}\underline{Hom}_c(\widehat{T}, \widehat{T'})$ be the full sub-dg-category of $\mathbb{R}\underline{Hom}_c(\widehat{T}, \widehat{T'})$ spanned by those modules which induce a sum preserving map $[\widehat{T}]\to [\widehat{T'}]$. Then, by the Theorem 7.2-(1) in \cite{Toen-homotopytheorydgcatsandderivedmoritaequivalences}, for any small dg-category $T'$ the composition with the Yoneda's embedding $h:T\to \widehat{T}$  

\begin{equation}
\mathbb{R}\underline{Hom}_c(\widehat{T},\widehat{T'})\to \mathbb{R}\underline{Hom}(T,\widehat{T'})
\end{equation}

\noindent is an isomorphism in the homotopy category of dg-categories. It follows from the description of the internal-hom as right representable modules, that the last is equivalent to the dg-category $\widehat{T^{op}\otimes^{\mathbb{L}}(T')}$. One corollary of this result (see \cite{Toen-homotopytheorydgcatsandderivedmoritaequivalences}) is the description of the mapping spaces  $Map_{\dgc}(\widehat{T}, \widehat{T'})$ as the maximal $(\infty,1)$-groupoid in $(T, T')_{\infty}$. Another corollary is the existence of a functor $\widehat{(-)}:\dg\to \dgc$ sending a small dg-category to its category of dg-modules. For an explicit description, we consider the canonical projection $\dg^{big}\times \Delta[1]\to \Delta[1]$ and the full subcategory $\M$ spanned by the vertices $(i,T)$ where if $i=0$, $T$ is small and if $i=1$, $T$ is of the form $\widehat{T_0}$ for some small dg-category $T_0$ and the only admissible maps $(1,T)\to (1,T')$ are the ones in $\dgc$. The fact that this fibration is cocartesian follows again from the theorem.\\

To formalize the third approach, we will restrict our attention to a subcategory of $\dgc$. As we just saw, a map $f:\widehat{T}\to \widehat{T'}$ in $\dgc$ corresponds to the data of a (uniquely determined) $T\otimes^{\mathbb{L}} (T')^{op}$-module $E_f$. We will say that $f$ preserves compact objects if for every object $x \in T$, the $(T')^{op}$-module $E_f(x,-)$ is compact. According to our terminology, this is the same as saying that $E_f$ is pseudo-perfect over $T$ relatively to $(T')^{op}$. With this, we denote by $\dgcc$ the (non-full) subcategory of $\dgc$ containing all the objects together with those maps that preserve compact objects. It follows from the definitions that the mapping spaces $Map_{\dgcc}(\widehat{T}, \widehat{T'})$ are given by the maximal $\infty$-groupoids inside $pspe(T,T')_{\infty}$. It is now easy to see that the canonical map $\widehat{(-)}: \dg\to \dgc$ factors through $\dgcc$. The following proposition establishes $\dgcc$ as a third approach to Morita theory

\begin{prop}
The composition $\dg^{idem}\hookrightarrow \dg \to \dgcc$ is an equivalence of $(\infty,1)$-categories. An inverse is given by the formula sending a dg-category $\widehat{T}$ to the full subcategory $(\widehat{T})_c$ spanned by the compact objects.

\begin{proof}
By the definition of $\dgcc$ the map is essentially surjective. It is fully-faithful because the mapping spaces in $\dgcc$ are by definition, the same as in $\dg^{idem}$, corresponding both to the $\infty$-groupoid of pseudo-perfect modules.
\end{proof}
\end{prop}

\begin{remark}
\label{emerald1}
Notice that if $\widehat{T}$ is a locally-cofibrant dg-category, then so is $\widehat{T}_c$. In this case, the equivalence $(-)_c$ restricts to an equivalence $\mathcal{D}g^{cc, loc-cof}(k)\to \dg^{idem,loc-cof}$. By choosing an inverse to this functor we solve the problem posed in the Remark \ref{omega1} of finding a model for the formula $T\mapsto \widehat{T}$ that preserves the hypothesis of being locally-cofibrant.  
\end{remark}

To complete the comparison between the second and third approaches, we regard the existence of a symmetric monoidal structure in $\mathcal{D}g^{cc, loc-cof}(k)$ which makes $(-)_c:\mathcal{D}g^{cc, loc-cof}(k)\to \dg^{idem,loc-cof}$ a monoidal functor.

\begin{prop}
The $(\infty,1)$-category $\mathcal{D}g^{cc, loc-cof}(k)$ is the underlying $\infty$-category of a symmetric monoidal structure $\mathcal{D}g^{cc, loc-cof}(k)^{\otimes}$. Given two objects $\widehat{T},\widehat{T'}\in \dgcc$, their monoidal product can be informally described by the formula $\widehat{T}\otimes \widehat{T}= \widehat{T\otimes^{\mathbb{L}} T'}$, where $\otimes^{\mathbb{L}}$ denotes the monoidal structure in $\dgloccof$.
\end{prop}

\begin{proof}
The proof of this proposition requires two steps. The first concerns the construction of an $(\infty,1)$-category $\mathcal{D}g^{cc, loc-cof}(k)^{\otimes}$ equipped with a map to $N(\Fin)$. The second step is the prove that this map is a cocartesian fibration. For the first, we start with $\mathcal{D}g^{loc-cof,big}(k)^{\otimes}\to N(\Fin)$ the symmetric monoidal structure in the $(\infty,1)$-category of the big locally-cofibrant dg-categories (as constructed in the section \ref{dg1}). By construction, its objects can be identified with the pairs $(\nfin, (T_1,..., T_n))$ with $\nfin\in N(\Fin)$ and $(T_1,..., T_n)$ a finite sequence of dg-categories. By the cocartesian property, maps $(\nfin,(T_1,.., T_n))\to (\mfin,(Q_1,..., Q_m))$ over $f:\nfin\to \mfin$ correspond to families of edges in $\dg^{loc-cof, big}$

\begin{equation}
\bigotimes_{j\in f^{-1}(\{i\})}T_j\to Q_i
\end{equation}

\noindent with $1\leq i\leq m$, where $\otimes$ denotes the tensor product in $\mathcal{D}g^{loc-cof,big}(k)^{\otimes}$. Given small dg-categories $T$, $T'$, $Q$, we will stay that an object in $\mathbb{R}\underline{Hom}(\widehat{T}\otimes^{\mathbb{L}}\widehat{T'}, \widehat{Q})$ is \emph{multi-continuous} if its image through the canonical adjunction is in $\mathbb{R}\underline{Hom}_c(\widehat{T},\mathbb{R}\underline{Hom}_c(\widehat{T'}, \widehat{Q}))$. 

With this, we consider the (non full)subcategory $\mathcal{D}g^{c, loc-cof}(k)^{\otimes}\subseteq \mathcal{D}g^{loc-cof,big}(k)^{\otimes}$ spanned by the pairs $(\nfin,(T_1,..., T_n))$ where each $T_i$ is an object in $\dgc$ together with those morphisms $(\nfin,(T_1,.., T_n))\to (\mfin,(Q_1,..., Q_m))$ corresponding to the edges

\begin{equation}
\bigotimes_{j\in f^{-1}(\{i\})}T_j\to Q_i
\end{equation}

\noindent which are multi-continuous. It follows that the composition $\mathcal{D}g^{c, loc-cof}(k)^{\otimes}\subseteq \mathcal{D}g^{loc-cof,big}(k)^{\otimes}\to N(\Fin)$ is a cocartesian fibration: a cocartesian lifting for a morphism $f:\nfin\to \mfin$ at a sequence $(\nfin,(T_1=\widehat{t_1},..., T_n=\widehat{t_n}))$ is given by the edge corresponding to the family of canonical maps

\begin{equation}
u_i: \bigotimes_{j\in f^{-1}(\{i\})}T_j\to Q_i:=\widehat{\otimes^{\mathbb{L}}_{j\in f^{-1}(\{i\})} t_j} 
\end{equation}

\noindent obtained from the identity of $\widehat{\otimes^{\mathbb{L}}_{j\in f^{-1}(\{i\})} t_j}$  using the canonical equivalences

\begin{equation}
\mathbb{R}\underline{Hom}_{multi-continuous}(\widehat{T}\otimes \widehat{T'},\widehat{A}):= \mathbb{R}\underline{Hom}_c(\widehat{T}, \mathbb{R}\underline{Hom}_c(\widehat{T},\widehat{A}))\simeq \mathbb{R}\underline{Hom}_c(\widehat{T}, \mathbb{R}\underline{Hom}(T,\widehat{A}))\simeq
\end{equation}

\begin{equation}
 \simeq \mathbb{R}\underline{Hom}(T, \mathbb{R}\underline{Hom}(T,\widehat{A}))\simeq \mathbb{R}\underline{Hom}(T\otimes T', \widehat{A})\simeq \mathbb{R}\underline{Hom}_c(\widehat{T\otimes T'}, \widehat{A})
\end{equation}

The same equivalences imply the cocartesian property of the family $(u_i)$.\\

With this, we are reduced to prove that this monoidal structure restricts to the (non-full) subcategory $\mathcal{D}g^{cc, loc-cof}(k)\subset\mathcal{D}g^{c, loc-cof}(k)$. For this purpose, it suffices to check that the same canonical morphisms

\begin{equation}
u_i: \bigotimes_{j\in f^{-1}(\{i\})}T_j\to Q_i:=\widehat{\otimes^{\mathbb{L}}_{j\in f^{-1}(\{i\})} t_j}
\end{equation}

\noindent preserve compact objects on each variable, which follows using an analogue of the above argument for multi-continuous maps.\\

\end{proof}

By inspection of the proof it is obvious that the map $(-)_c:\mathcal{D}g^{cc, loc-cof}(k)\to \dg^{idem,loc-cof}$ is compatible with the monoidal structures.\\

In summary, we have three equivalent ways to encode Morita theory. 

\begin{equation}
\label{versionsmorita}
N(Cat_{Ch(k)}^{loc-cof})[W_{Mor}^{-1}]^{\otimes }\simeq(\dg^{idem,loc-cof})^{\otimes}\simeq \mathcal{D}g^{cc, loc-cof}(k)^{\otimes}
\end{equation}

\begin{convention}
\label{omega5}
For the future sections, and for the sake of simplicity, we will omit the fact that these monoidal structures are defined for locally-cofibrant dg-categories and that to make sense of this monoidal product for arbitrary dg-categories, we need to perform cofibrant-replacements to fall in the locally-cofibrant context.
\end{convention}

Furthermore, in \cite{tabuada-invariantsadditifs} the author proves the existence of a combinatorial compactly generated model structure in $Cat_{Ch(k)}$ with weak-equivalences the Morita equivalences and the same cofibrations as for the Dwyer-Kan model structure. It follows that the three $(\infty,1)$-categories are presentable. In particular, they have all limits and colimits and we can compute them as homotopy limits and homotopy colimits in $Cat_{Ch(k)}$ with respect to this Morita model structure. In particular, we can prove that $\dg^{idem}$ has a zero object $\ast$ (the dg-category with one object and one morphism) and that, more generally, finite sums are equivalent to finite products (denoted by $\oplus$). The last follows because for any two small dg-categories $T$ and $T'$  we have a canonical equivalence of big dg-categories $\underline{Hom}(T\coprod T', Ch_{dg}(k))\simeq \underline{Hom}(T, Ch_{dg}(k))\times \underline{Hom}(T, Ch_{dg}(k))$ compatible with the natural model structures for dg-modules. We can now use this equivalence to find that $\widehat{(T\coprod T')}_{c}\simeq \widehat{T}_{c}\times \widehat{T'}_c$.\\

\subsubsection{Dg-categories with a compact generator}
\label{dgcategorieswithcompactobjects}

A dg-category $\widehat{T}\in \dgc$ is said to have a compact generator if the triangulated category $[\widehat{T}]$ has a compact generator in the sense of Neeman (see the Remark \ref{senseofNeeman}). We will say that a small dg-category $T$ has a \emph{compact generator} if the triangulated category $[\widehat{T}]$ admits a compact generator in the previous sense. It follows that $T$ has a compact generator if and only if its idempotent completion $\widehat{T}_c$ has a compact generator (of course, this follows from the equivalence $\widehat{(\widehat{T})_c}\simeq \widehat{T}$).\\ 

Let $Perf$ be the composition

\begin{equation}\xymatrix{Alg(\derivedk)\ar[r]^{(-)_{dg}}& \dg\ar[r]^{(\widehat{-})_c} &\dg^{idem}}\end{equation}

Using the same methods as in \cite{schwede-shipley-modules}, it can be proved that $T$ has a compact generator if and only if it is in the essential image of $Perf$. For the "only if" direction we consider the dg-algebra $B$ given by the opposite algebra of endomorphisms of the compact generator in $\widehat{T}$. For the "if" direction, if $T\simeq Perf(B)$ then $B$, seen as a dg-module over itself, is a compact generator.

\begin{remark}
\label{productdgcategorieswithcompactgeneratorhascompactgenerator}
Let $T, T'\in \dg^{idem}$ be idempotent complete dg-categories having a compact generator. Then their tensor product in $\dg^{idem}$ has a compact generator. This follows because the functor $Perf$ is monoidal (see \ref{dgalgebrasdgcatsmonoidal} and \ref{epavala}).
\end{remark}

\subsubsection{Dg-categories of Finite Type}
\label{finitetype}

In this section we discuss the notion of dg-category of finite type studied by Toën-Vaquie in \cite{toen-vaquie}. In the next section they will give body to our noncommutative spaces.\\
 
It follows from the results of \cite{tabuada-invariantsadditifs} that the Morita model structure is combinatorial, compactly generated and satisfies the general conditions of the Proposition 2.2 in \cite{toen-vaquie}. Following the discussion in \ref{compactlygeneratedmodelcategories}, we can identify the compact objects in $\dg^{idem}$ with the retracts of finite cell objects and we have a canonical equivalence $\dg^{idem}\simeq Ind( (\dg^{idem})^{\omega})$. At the same time in \cite{tabuada-cisinski}-Theorem 4.3, the authors prove that an object $T\in \dg^{idem}$ is compact if and only if its internal-hom functor $\mathbb{R}\underline{Hom}(T,-)$ in $\dg^{idem,\otimes}$ commutes with filtered colimits. An immediate corollary of this is that the product of compact objects in $\dg^{idem,\otimes}$ is again compact so that the subcategory $(\dg^{idem})^{\omega}$ inherits a symmetric monoidal structure. \\

Following \cite{toen-vaquie}, we say that an idempotent dg-category $T$ is of \emph{finite type} if it is equivalent in $\dg^{idem}$ to a dg-category of the form $Perf(B)$ for some dg-algebra $B$ which is compact as an object in the $(\infty,1)$-category $Alg(\derivedk)$ \footnote{$Alg(\derivedk)$ is the underlying $(\infty,1)$-category of a compactly generated model structure in the category of strictly associative dg-algebras and again by the discussion in \ref{compactlygeneratedmodelcategories} we can identify its compact objects with the retracts of finite cell strict dg-algebras}. In particular a dg-category of finite type has a compact generator.

In \cite{toen-vaquie}-Lemma 2.11, the authors prove that a dg-category of the form $Perf(B)$ is compact in $\dg^{idem}$ if and only if $B$ is compact in the $(\infty,1)$-category $Alg(\derivedk)$. In fact, an object in $\dg^{idem}$ is compact if and only if it is of finite type:

\begin{prop}(Toën-Vaquié)
\label{dgideminddgft}
Let $\dg^{ft}$ denote the full subcategory of $\dg^{idem}$ spanned by the dg-categories of finite type. Then, the inclusion $\dg^{ft}\subseteq (\dg^{idem})^{\omega}$ is an equivalence.
\begin{proof}
By the discussion in \ref{dgcategorieswithcompactobjects} it suffices to prove that any compact dg-category $T\in \dg^{idem}$ has a compact generator. Indeed, we can always write $T$ as a filtered colimit of its subcategories generated by compact objects. Since $T$ is compact it is equivalent to one of these subcategories and therefore the triangulated category $[T]$ is compactly generated by a finite family of objects $\{x_1,..., x_n\}$ (in the sense of Neeman - see the Remark \ref{senseofNeeman}). Since $T$ is idempotent complete, it admits finite sums and therefore the finite direct product $\oplus x_i$ is a compact generator.
\end{proof}
\end{prop}

With this, we have a canonical equivalence $\dg^{idem}\simeq Ind(\dg^{ft})$. It follows that $\dg^{ft}$ is closed under finite direct sums, pushouts and contains the zero object.

\begin{remark}
\label{nc2dgftdgidempreservesproducts}
It follows from the Yoneda's lemma that the inclusion $\dg^{ft}\subseteq \dg^{idem}$ commutes with arbitrary limits whenever they exist in $\dg^{ft}$.
\end{remark}

\subsection{Dg-categories vs stable $(\infty,1)$-categories}
\label{linkdgspectrastable}

This section is merely expository and sketches the conjectural relation between the theory of dg-categories and the theory of stable $(\infty,1)$-categories. We aim to somehow justify our choice to work with dg-categories. I learned this vision from B. Toën.\\

For any commutative ring $k$, $\derivedk^{\otimes}$ is a stable presentable symmetric monoidal $(\infty,1)$-category. In this case, the universal property of $\Spmonoidal$ ensures the existence of a (unique up to a contractible space of choices) monoidal colimit preserving map

\begin{equation}
f:\Spmonoidal\to  \derivedk^{\otimes}
\end{equation}

\noindent sending the sphere spectrum to the ring $k$ seen as complex concentrated in degree zero. This is a morphism of commutative algebras in $\Prlmonoidal$ and therefore produces a base-change adjunction

\begin{equation}
\xymatrix{
\Prl_{Stb}\simeq Mod_{\Spmonoidal}(\Prl)\ar@<.7ex>[rr]^{(\derivedk\otimes_{\Sp}-)}&& Mod_{\derivedk^{\otimes}}(\Prl) \ar@<.7ex>[ll]^{f^{*}}
}
\end{equation}

\noindent with $f^*$ the forgetful map given by the composition with $f$. Notice that the objects in the left side are stable $(\infty,1)$-categories because the adjunction is defined over the forgetful functors to $\Prl$. By definition, a $k$-linear stable $(\infty,1)$-category is an object in $Mod_{\derivedk^{\otimes}}(\Prl)$. \\

At the same time, there is a canonical way to assign an $(\infty,1)$-category to a dg-category. More precisely, given a small dg-category $T$ we can apply the Dold-Kan construction to the positive truncations of the complexes of morphisms in $T$ to get mapping spaces. Since the Dold-Kan functor is right-lax monoidal (via the Alexander-Whitney map), this construction provides a new simplicial category which happens to be enriched over Kan-complexes. By takings its simplicial nerve we obtain an $(\infty,1)$-category $N_{dg}(T)$. The details of this mechanism can be found in \cite[Section 1.3.1]{lurie-ha}. Moreover, the assignement $T\mapsto N_{dg}(T)$ provides a right Quillen functor between the model category of dg-categories categories with the Dwyer-Kan model equivalences and the model category of simplicial sets with the Joyal's model structure \cite[1.3.1.20]{lurie-ha}. Following the discussion in \ref{combinatorialmodelcategories} and since these model structures are combinatorial, this assignement provides a functor between the $(\infty,1)$-categories

\begin{equation}N_{dg}:\dg\to \iCat\end{equation}

By the properties of the Dold-Kan correspondence, $N_{dg}$ preserves the notion of "homotopy category"\footnote{Recall that $\pi_n(DK(A))\simeq H_n(A)$, where $DK$ denotes the Dold-Kan map}. Moreover, using the arguments in (\ref{combinatorialmodelcategories}), the combinatorial property implies that $N_{dg}$ has a left adjoint and therefore preserves limits. In particular, for a bigger universe we also have a well-defined map

\begin{equation}N_{dg}^{big}:\dg^{big}\to \iCat^{big}\end{equation}

Following \cite{toen-azumaya} we have the notion of a locally presentable dg-category. By definition, these are big dg-categories that can be obtained as accessible reflexive localizations of big dg-categories of the form $\widehat{T_0}$ for some small dg-category $T_0$. Alternatively, we can describe them as the dg-categories of cofibrant-fibrant objects of a Bousfield localization of the left proper combinatorial model category $Ch(k)^{T_0}$ for some small dg-category $T_0$. Together with the colimit preserving maps, they form a (non-full) subcategory $\dglp$ of $\dg^{big}$. In particular, the $(\infty,1)$-category $\dgcc$ introduced in the previous section has a non-full embedding in $\dglp$. As explained in the proof of \cite[Lemma 2.3]{toen-azumaya} a big dg-category having all colimits is locally presentable if and only if $N_{dg}^{big}(T)$ is in $\Prl$. In particular, as the notions of colimit are compatible, the restriction

\begin{equation}N_{dg}^{L}:\dglp\to \Prl\end{equation}

\noindent is well-defined.

For a dg-category of the form $\widehat{T_0}$, the $(\infty,1)$-category $N_{dg}^L(\widehat{T_0})$ can be identified with the 
underlying $(\infty,1)$-category of the combinatorial model category $Ch(k)^{T_0}$ which is compactly generated. In particular, since $Ch(k)^{T_0}$ is stable (in the sense of model categories), we find that $N_{dg}^L(\widehat{T_0})$ is a stable compactly generated $(\infty,1)$-category\footnote{In the condition of having all limits and colimits, the property of being stable depends only on the fact the suspension functor is invertible at the level of the homotopy category}. In fact, a dg-category $T\in \dglp$ is in $\dgcc$ if and only if $N_{dg}^{L}(\widehat{T_0})$ is compactly generated. More generally, we can identify the functor $N_{dg}^{L}$ with the map sending a Bousfield localization of $Ch(k)^{T_0}$ to its underlying $(\infty,1)$-category. In particular we find that $N_{dg}^{L}$ factors through the full subcategory of $\Prl$ spanned by the stable presentable $(\infty,1)$-categories $\Prl_{Stb}$. In particular, $N_{dg}^L$ restricts to $\dgcc\to \Prl_{\omega,Stb}\subseteq \Prl_{\omega}$.\\

\begin{remark}
\label{nervedgconservative} 
It follows from the fact that  $Ch(k)^{T_0}$ is a stable  model category and from the properties of the Dold-Kan construction that $N_{dg}^L$ is conservative, for it preserves the notion of homotopy category and using stability, we see that it also reflects fully-faithfulness. More generally, the restriction of $N_{dg}$ to dg-categories satisfying stability is conservative. 
\end{remark}

\begin{remark}
\label{limitspresentabledg}
By the previous remark, since $N_{dg}^{L}$ and more generally $N_{dg}$ (restricted to big stable dg-categories) are conservative, and both  commute with limits and the non-full inclusion $\Prl\subseteq \iCat^{big}$ (respectively the inclusion of big stable dg-categories inside all big dg-categories) preserves limits, we find that $\dglp$ also has all small limits and that the inclusion $\dglp \subseteq \dg^{big}$ also preserves them. 
\end{remark}

We now come to the conjectural relation between the Morita theory of dg-categories and the theory of stable presentable $(\infty,1)$-categories: the map $N_{dg}^{L}:\dglp\to \Prl_{Stb}$ is expected to factor through the forgetful functor $f^{*}:Mod_{\derivedk^{\otimes}}(\Prl)\to \Prl_{Stb}$

\begin{equation}
\xymatrix{
\dglp\ar@{-->}[r]^{\theta} &Mod_{\derivedk^{\otimes}}(\Prl)\ar[r]^{f^{*}} & \Prl_{Stb}
}
\end{equation}

\noindent and this factorization $\theta$ is expected to be an equivalence of $(\infty,1)$-categories. In this case, the restriction

\begin{equation}
\xymatrix{\dgcc\ar[r]^{\sim}& Mod_{\derivedk^{\otimes}}(\Prl_{\omega})}
\end{equation}

\noindent will provide a link between the Morita theory of dg-categories (as described in the previous section) and the Morita theory of stable $\infty$-categories studied in \cite{tabuada-gepner}. The following diagram is an attemptive to schematize this landscape

\begin{equation}
\xymatrix{
&&&\infty(SpectralCats/Morita)\ar[d]_{\alpha}^{\sim}\ar[ld]_{\beta}^{\sim}&\\
&\ar[dl]_{\sim}Mod_{\Spmonoidal}(\Prl)& \ar@{_{(}->}[l]_(.3){nonfull}\Prl_{\omega,Stb}\ar[r]_{\gamma}^{\sim}&Cat_{\infty}^{\mathcal{E}x,idem}\ar@{^{(}->}[r]&\iCatstable\\
\Prl_{Stb}& \ar[l]Mod_{\derivedk^{\otimes}}(\Prl)\ar[u]^{f^*}& \ar@{_{(}->}[l]_{nonfull}Mod_{\derivedk^{\otimes}}(\Prl_{\omega})\ar[u]^{f^*}&&&\\
&\ar[ul]^{N_{dg}^{L}}\dglp \ar@{-->}[u]_{\theta}^{\sim}&\ar@{_{(}->}[l]_{nonfull}\dgcc \ar@{-->}[u]_{\theta}^{\sim}\ar[r]_w^{\sim}&\dg^{idem}\ar@{^{(}->}[r]&\dg\\
&&&N(Cat_{Ch(k)})[W_{Mor}^{-1}]\ar[ul]_u^{\sim}\ar[u]_v^{\sim}&
}
\end{equation}

Here $\infty(Spectral/Morita)$ (resp. $N(Cat_{Ch(k)})[W_{Mor}^{-1}]$) denotes the $(\infty,1)$-category associated to the Morita model structure on the small spectral categories (resp. small dg-categories). The map $\beta$ is defined by sending a spectral category $\C$ to the stable $(\infty,1)$-category associated to the stable model category of $\C$-modules in spectra (ie, functors from $\C$ to the model category of spectra, together with the projective structure). The map $\gamma$ is the equivalence discussed in \ref{stableinfinitycategories} obtained by taking the full-subcategory of compact objects. The map $\alpha$ is the composition $\gamma\circ \beta$ and the fact that it is an equivalence is due to the Theorem 4.23 of \cite{tabuada-gepner}. The map $\theta$ is the conjectural equivalence and the maps $u,v$ and $w$ are the dg-analogues of $\alpha$,$\beta$ and $\gamma$, and the fact that they are equivalences results from the main results of \cite{tabuada-quillen,Toen-homotopytheorydgcatsandderivedmoritaequivalences, toen-vaquie} as indicated in the previous section. It is also important to remark that $\theta$ should respect the natural monoidal structures.\\

We hope this discussion clarifies the decision to work with dg-categories. For a quasi-compact and quasi-separated  scheme $X$ over $k$ we shall have $\theta(L_{qcoh}(X))\simeq \mathcal{D}(X)$ where $L_{qcoh}(X)$ is the derived dg-category of $X$ (see the next section) and $\mathcal{D}(X)$ is the stable presentable symmetric monoidal derived $(\infty,1)$-category of $X$ as in \cite[Def. 1.3.5.8]{lurie-ha}.

\subsection{Dg-Categories and Noncommutative Geometry}

\subsubsection{From Schemes to dg-algebras (over a ring $k$)- Part I}
Let $k$ be a ring. Given a quasi-compact and separated $k$-scheme $(X,\mathcal{O}_X)$ we consider $Qcoh(X)\subseteq \mathcal{O}_X-Mod$ the subcategory of quasi-coherent sheaves on $X$. Under some general conditions, the natural tensor product in $\mathcal{O}_X-Mod$ is closed for quasi-coherent sheaves (see the Prop. 9.1.1 of \cite{ega1}-Chap. 1). It results from a theorem of Deligne (see \cite{har66}-Appendix, Prop 2.2) that $Qcoh(X)$ is a \emph{Grothendieck abelian category}\footnote{More generally (see \cite{conrad}-Lemma 2.1.7) for any scheme $X$ there is an infinite cardinal $\kappa$ such that $qcoh(X)$ is $\kappa$-presentable.} so that we can apply to $C(Qcoh(X))$ the Theorem 2.2 of \cite{hovey-modelstructureonsheaves} which tells us that the category of unbounded complexes on a Grothendieck abelian category can be equipped with a model structure, with cofibrations given by the monomorphisms and the weak-equivalences the quasi-isomorphisms of complexes. By the Proposition 2.12 of loc.cit, every fibrant-object is a complex of injectives and every bounded above complex of injectives is fibrant. Since $X$ is defined over $k$, $\mathcal{O}_X$ is a sheaf of $k$-algebras and each $\mathcal{O}_X$-module is naturally a sheaf of $k$-modules. This induces a canonical action of $Ch(k)$ on $C(Qcoh(X))$ compatible with the model structure. By definition, the \emph{dg-derived category of $X$} is the dg-category $L_{qcoh(X)}:=Int(C(Qcoh(X)))$. Following the Remark \ref{localizationdg}, its associated homotopy category $[L_{qcoh(X)}]$ is canonically equivalent to the classical derived category of quasi-coherent sheaves on $X$. 

\begin{remark}
In general, for any quasi-compact scheme $X$, the correct derived dg-category to consider is full subcategory of the derived category of $\mathcal{O}_X$-modules with quasi-coherent cohomology. When $X$ is separated, this agrees with $L_{qcoh(X)}$.
\end{remark}

It is compactly generated (in the sense of Neeman \cite{Neeman-triangulatedcategories}) and thanks the results of \cite{thomasonalgebraic} we know that its compact objects are perfect complexes of quasi-coherent sheaves. We write $L_{pe}(X)$ for the full subcategory of $L_{qcoh(X)}$ spanned by the perfect complexes and the general theory gives us a canonical equivalence $\widehat{L_{pe}(X)}\simeq L_{qcoh(X)}$. By construction $L_{pe}(X)$ is an idempotent dg-category and we will understand it as the natural noncommutative incarnation of the scheme $X$. The philosophical importance of the following result is evident

\begin{thm}(Bondal-Van den Bergh \cite{bondal-vandenbergh}-Thm 3.1.1)
Let $X$ be a quasi-compact and quasi-separated scheme over a ring $k$. Then $L_{pe}(X)$ has a compact generator.
\end{thm}

Together with the preceeding discussion, this result implies that for any quasi-compact and quasi-separated scheme $X$ over $k$, the dg-category of perfect complexes $L_{pe}(X)$ is of the form $Perf(B)$ for some dg-algebra $B$.

\subsubsection{Smooth and Proper Dg-categories}
\label{nc1smoothandproper}

The geometric notions of smoothness and properness can be adapted to the world of dg-categories, in a way compatible with $L_{pe}(-)$. Recall that a dg-category $T$ is said to be \emph{locally perfect} it is enriched over perfect complexes of $k$-modules. $T$ is said to be \emph{proper} if it is locally perfect and it has a compact generator. We say that $T$ is \emph{smooth} if the object in $\widehat{T\otimes^{\mathbb{L}} T^{op}}$ defined the formula $(x,y)\mapsto T(x,y)$, is compact. Finally, we say that a dg-category is \emph{saturated} if it is smooth, proper and idempotent complete. It can easily be checked (\cite{toen-vaquie}-Lemma 2.6-(2)) that a dg-category $T$ is proper (resp. smooth) if and only if its idempotent completion is proper (resp. smooth). Of course, these notions are also invariant under the operation $(-)^{op}$. This implies that a dg-category of the form $Perf(B)$ is proper (resp. smooth) if and only $B$ is perfect as a complex of $k$-modules (resp. the $B^{op}\otimes^{\mathbb{L}} B$-module defined by the formula $(\bullet,\bullet)\mapsto B$ is compact).\\

\begin{example}
\label{smoothalgebrassmoothdgcats}
This notion of smoothness is compatible with the classical geometrical notion: a morphism $Spec(A)\to Spec(k)$ is smooth (meaning, $A$ is regular over $k$) if and only if the dg-category $Perf(A)$ is smooth. This is proved using a famous theoreom of J.P.Serre \cite[IV-37, Thm 9]{serre-algebrelocal}: a commutative ring is regular if and only if it is of finite global homological dimension. More generally, and thanks to the Lemma 3.27 in \cite{toen-vaquie} we have a machine to produce smooth and proper dg-categories: for any scheme $X$ smooth and proper over a ring $k$, the dg-category $L_{pe}(X)$ is smooth and proper.
\end{example}

In \ref{morita} we explained how the notion of pseudo-perfectness can be used to describe the mapping spaces in $\dg^{idem}$. Notice now that an object $E\in \widehat{T}$ is pseudo-perfect (over $T$ relatively to $\mathit{1_k}$) if it has values in compact complexes of $k$-modules. The distinction between being compact and pseudo-perfect is the key to understand the notions of smooth and proper as the following results from \cite{toen-vaquie} suggest:

\begin{itemize}
\item \cite{toen-vaquie}-Lemma 2.8-(1): A dg-category $T$ is locally perfect if and only if for any dg-category $T'$, we have an an inclusion of subcategories

\begin{equation}(\widehat{T\otimes^{\mathbb{L}} T'})_c\subseteq (\widehat{T\otimes^{\mathbb{L}} T'})_{pspe}\end{equation}

\item \cite{toen-vaquie}-Lemma 2.8-(2): A dg-category $T$ is smooth if and only if for any dg-category $T'$, we have an an inclusion of subcategories

\begin{equation}(\widehat{T\otimes^{\mathbb{L}} T'})_{pspe}\subseteq (\widehat{T\otimes^{\mathbb{L}} T'})_c\end{equation}

\item \cite{toen-vaquie}-Lemma 2.8-(3): From the two previous items, a dg-category $T$ is smooth and proper iff it has a compact generator and for any dg-category $T'$, the subcategories of $\widehat{T\otimes^{\mathbb{L}} T'}$ spanned by compact, respectively pseudo-perfect modules, coincide. 

\end{itemize}

\begin{remark}
Recall from \ref{morita} that the mapping spaces $Map_{\dg^{idem}}(T,T')$ are given by the maximal $\infty$-groupoids in $pspe(T,T')_{\infty}$ - the full subcategory of $(T,T')_{\infty}$ spanned by the pseudo-perfect modules. It follows that if $T$ is smooth and proper, we can identify $pspe(T,T')_{\infty}$ with the full subcategory $(T,T')_{\infty}^{\omega}$ spanned by the compact modules. 
\end{remark}

The notions of smooth and proper are related to the notion of finite type:

\begin{itemize}
\item \cite{toen-vaquie}-Corollary 2.13:  Any smooth and proper dg-category is of finite type;
\item \cite{toen-vaquie}-Proposition 2.14: Any dg-category of finite type is smooth.
\end{itemize}

To conclude this section we recall another important characterization of smoothness and properness given by the following result due to B. Toën 

\begin{prop}(see \cite{sedano}-Lectures on dg-categories)
\label{nc1smoothandproperaredualizable}
An object $T \in \dg^{idem}$ is smooth and proper if and only if it is dualizable with respect to the symmetric monoidal structure $\dg^{idem,\otimes}$. In particular, the dual of a dg-category $\widehat{T}_c\in \dg^{idem}$ is the opposite $(\widehat{T^{op}})_c$.
\end{prop}

\subsubsection{From Schemes to Noncommutative Spaces (over a ring $k$) - Part II}
\label{section633}

Following \cite{toen-vaquie}, the notion of finite type should be understood as the correct notion of smoothness for noncommutative spaces, while the smooth dg-categories should only be understood as "formally smooth"' noncommutative spaces. Finally, we are ready to introduce our smooth noncommutative geometric objects.

\begin{defn}
Let $k$ be a ring. We define the $(\infty,1)$-category of \emph{smooth noncommutative spaces over $k$}- $\nck$ - to be the \emph{opposite} of $\dg^{ft}$. It has a natural symmetric monoidal structure $\nckmonoidal$ induced from the one in $\dg^{ft,\otimes}$ , with unit object given by $L_{pe}(k)$.
\end{defn}

\begin{notation}
We will denote our smooth noncommutative spaces using caligraphic letters $\X$,$\UU$, $\V$, $\W$, etc. For a smooth noncommutative space $\X\in \nc$ we will denote by $T_{\X}$ its associated dg-category of finite type and by $A_{\X}$ a compact dg-algebra such that $T_{\X}\simeq Perf(A_{\X})$.
\end{notation}

We will say that a smooth noncommutative space $\X$ is \emph{smooth and proper} if its associated dg-category $T_{\X}$ is smooth and proper. We will let $\nck^{sp}$ denote the full subcategory of $\nck$ spanned by the smooth and proper noncommutative spaces. Since the smooth and proper dg-categories correspond to the dualizable objects in $\dg^{ft}$, the subcategory $\nck^{sp}$ is closed under tensor products.\\

It follows immediately from the properties of $\dg^{ft}$ that $\nck$ admits pullbacks, together with finite direct sums and a zero object. Moreover, the tensor product commutes with limits. In particular, if $\X$ and $\Y$ are two smooth noncommutative spaces, the mapping space $Map_{\nck}(\X,\Y)$ is given by the $\infty$-groupoid $pspe(A_{\Y},A_{\X})_{\infty}^{\simeq}$ of pseudo-perfect $A_{\Y}\otimes^\mathbb{L}A_{\X}^{op}$-dg-modules and equivalences between them.\\

We now explain how the formula $X\mapsto L_{pe}(X)$ can be properly arranged as a functor. We define it for the smooth affine schemes of finite type over $k$, whose ordinary category we denote by $AffSm^{ft}(k)$. Recall that the full subcategory of $0$-truncated objects in $Alg(\derivedk)^{cn}$ is equivalent to the nerve of the category of classical associative rings. In particular, we can identify the nerve of the category of commutative smooth $k$-algebras of finite type $N(SmCommAlg_k)\simeq \aff^{op}$ with a full subcategory of $Alg(\derivedk)^{cn}$. Let $L$ denote the composition

\begin{equation}
\xymatrix{
N(SmCommAlg_k)\ar@{^{(}->}[r]& Alg(\derivedk)^{cn} \ar@{^{(}->}[r]& Alg(\derivedk)  \ar[r]^{Perf}&  \dg^{idem}
}
\end{equation}

The following is a key result:

\begin{prop}
\label{commutativesmoothimpliesfinitetype}
Let $A$ be a classical commutative smooth $k$-algebra of finite type. Then, $L(A)$ is a dg-category of finite type. In other words, $L$ provides a well-defined functor $N(SmCommAlg_k)\to \dg^{ft}$.
\begin{proof}
If $A$ is smooth as a classical commutative $k$-algebra it is smooth as a dg-category (Example \ref{smoothalgebrassmoothdgcats}) which by definition means it is compact as a $A\otimes_k A^{op}$-dg-module. Following the Remark \ref{bimodulesandcoisas} the category of $A\otimes_k A^{op}$-dg-modules can be naturally identified with the category of $A$-bimodules $BiMod(A,A)(Ch(k))$. Using the strictification results of \ref{strictificationalgebras} the underlying $(\infty,1)$-category of $BiMod(A,A)(Ch(k)$) is equivalent to $_ABMod_{A}(\derivedk)\simeq Mod_{A}^{\Ass}(\derivedk)$.

Of course, if $A$ is compact in $Mod_{A}^{\Ass}(\derivedk)$ and since $A\otimes_k A^{op}$ is also compact (it is a generator), the kernel of the multiplication map $I\to A\otimes_k A^{op}\to A$ will also be compact. Following the Example \ref{noncommutativecotangentcomplex} we can now identify $I$ with the relative cotangent complex $\mathbb{L}_{A/k}\in Mod_{A}^{\Ass}(\derivedk) $. The Lemma \ref{cotangentcompact} completes the proof.
\end{proof}
\end{prop}

Using this, we define $L_{pe}$ as the opposite of $L$

\begin{equation}L_{pe}:\aff\to \nck\end{equation}

To conclude this section we observe that $L_{pe}$ can be promoted to a monoidal functor 

\begin{equation}
L_{pe}^{\otimes}:\aff^{\times}\to\nck^{\otimes}
\end{equation}

\noindent where $\aff^{\times}$ is the cartesian structure in $\aff$ which corresponds to the coproduct of classical commutative smooth $k$-algebras which is, well-known, given by the classical tensor product over $k$.\\ 

It follows from \ref{tensorproductalgebras} and the fact that the tensor product in $\derivedk$ is compatible with the $t$-structure, that the composition $Alg(\derivedk)^{cn}\to Alg(\derivedk)$ is monoidal. Moreover, the functor $Perf$ is monoidal because it is the composition of monoidal functors - \ref{dgalgebrasdgcatsmonoidal} and \ref{epavala}. We are left to check that the inclusion $N(SmCommAlg_k)\to Alg(\derivedk)^{cn}$ is monoidal. In other words, that for a commutative smooth $k$-algebra of finite type over $k$, the classical tensor product agrees with the derived tensor product. But this is true since smooth $k$-algebras are flat over $k$.

\subsection{The Motivic $\mathbb{A}^1$-Homotopy Theory of Kontsevich's Noncommutative Spaces over a ring $k$}

We will now use our main results to fabricate a motivic $\mathbb{A}^1$-homotopy theory for smooth noncommutative spaces over a ring $k$. In this section we proceed in analogy with the construction of the motivic stable homotopy for schemes as described in the previous chapter of this work. Recall from the Remark \ref{affineisenough} that these constructions only depend on the category of \emph{affine} smooth schemes of finite type over $k$.\\

\begin{remark}
\label{extendingtononaffineschemes}
There is a natural way to extend the functor $L_{pe}$ to non-affine schemes. To do this, we observe that the classical category of schemes can be identified with a full subcategory of $\mathcal{P}^{big}(\aff)$, by the identification of a scheme with its "functor of points". The universal property of (big) presheaves provides a colimit preserving map

\begin{equation}
\xymatrix{
\aff \ar[d]\ar[r]^{L_{pe}}& \nck\ar[d]\\
\mathcal{P}^{big}(\aff)\ar@{-->}[r]& \mathcal{P}^{big}(\nck)
}
\end{equation}

The Lemma 3.27 in \cite{toen-vaquie} implies that the image through this map of any smooth and proper scheme $X$ over $k$ is representable in  $\mathcal{P}^{big}(\nck)$. This should remain true without the properness condition.
\end{remark}

 To start with, we need to introduce an appropriate analogue for the Nisnevich topology, for the interval $\mathbb{A}^1$ and for the projective space $\mathbb{P}^1$. For the last two we have natural choices - $L_{pe}(\mathbb{A}^1)$ and $L_{pe}(\mathbb{P}^1)$: the first is a dg-category of finite type because $\mathbb{A}^1$ is smooth affine over $k$; the second, $L_{pe}(\mathbb{P}^1)$, is of finite type because the canonical morphism $\mathbb{P}^1\to Spec(k)$ is smooth and proper (see \ref{smoothalgebrassmoothdgcats}). The analogue of the Nisnevich topology requires a more careful discussion.

\subsubsection{The noncommutative version of the Nisnevich Topology}

To obtain our noncommutative analogue for the Nisnevich topology we isolate the formal properties of the commutative squares in $\nck$

\begin{equation}
\xymatrix{
L_{pe}(p^{-1}(U))\ar[r]\ar[d]& L_{pe}(V)\ar[d]\\
L_{pe}(U)\ar[r]& L_{pe}(X)
}
\end{equation}

\noindent induced by the Nisnevich squares of schemes. Following the list of properties given in Section \ref{section5}, we start with the notion of an open embedding. For that we need some preparations. Recall that an \emph{exact sequence} in $\dg^{idem}$  is the data of a commutative square  

\begin{equation}
\label{pku}
\xymatrix{
A\ar[r]^f \ar[d]& \ar[d]^g B\\
\ast \ar[r]& C
}
\end{equation}

\noindent where $\ast$ is the zero object in $\dg^{idem}$, such that $f$ fully-faithful and the diagram is a pushout. Since $\dg^{idem}$ is a reflexive localization of $\dg$, this pushout $C$ is canonically equivalent to the idempotent completion of the pushout $B/A$ computed in $\dg$. Of course, using the equivalence (\ref{versionsmorita}), the previous diagram is an exact sequence if and only if the diagram

\begin{equation}
\label{exactsequence}
\xymatrix{
\widehat{A}\ar[r]^{\widehat{f}} \ar[d]& \ar[d]^{\widehat{g}} \widehat{B}\\
\ast \ar[r]& \widehat{C}
}
\end{equation}

\noindent is an exact sequence in $\dgcc$ in the same sense. Thanks to the works of B.Keller in \cite{keller-exact}, we know that this notion of exact sequence extends the notion given by Verdier \cite{Verdier}. 

\begin{prop}(B. Keller \cite{keller-exact})
\label{kellerexact}
The following conditions are equivalent:
\begin{enumerate}
\item a diagram as above is an exact sequence;
\item the functor $\widehat{f}$ induces an equivalence of $[\widehat{A}]$ with a triangulated subcategory of the triangulated category $[\widehat{B}]$ and $\widehat{g}$ exhibits the homotopy category $[\widehat{C}]$ as the Verdier quotient $[\widehat{B}]/[\widehat{A}]$;
\item the functor $f$ induces an equivalence of $[A]$ with a triangulated subcategory of the triangulated category $[B]$ and the canonical map from the Verdier quotient $[B]/[A]\hookrightarrow [C]$ is cofinal (see our discussion in \ref{exactsequencepresentable}).
\end{enumerate}
\end{prop}

\begin{remark}
\label{exactdgiffexactstable}
Following the discussion in \ref{linkdgspectrastable}, we can use the functor $N_{dg}^L: \dgcc \to \Prl_{\omega,Stb}$ to relate exact sequences of dg-categories in the above sense to exact sequences of stable presentable $(\infty,1)$-categories in the sense of \ref{exactsequencepresentable}. As explained in \ref{nervedgconservative} $N_{dg}^L$ is conservative, preserves fully-faithfulness and preserves the notion of "homotopy category" (see the Remark \cite[1.3.1.11]{lurie-ha}). This result, together with the Propositions \ref{kellerexact} and \ref{gepnerexact} implies that a sequence of dg-categories $\widehat{A}\to \widehat{B}\to \widehat{C}$ in $\dgcc$ is exact in the sense discussed in this section if and only if its image $N_{dg}^L(\widehat{A})\to N_{dg}^L(\widehat{B})\to N_{dg}^L(\widehat{C})$ in $\Prl_{\omega,Stb}$ is exact in the sense discussed in \ref{exactsequencepresentable}. It follows also that $\widehat{A}$ has a compact generator if and only if $h(N_{dg}^L(\widehat{A}))$ has a compact generator.
\end{remark}

\begin{remark}
\label{exactsequencesarestrict2}
It is common to find in the literature the terminology of \emph{strict exact sequence} to denote an exact sequence (\ref{pku}) in $\dg^{idem}$ which, appart from being a pushout square, is also a pullback in $\dg^{idem}$. It follows again from the results of \cite{keller-exact} that in terms of the associated homotopy triangulated categories this corresponds to the additional condition that $[\widehat{A}]$ is thick in $[\widehat{B}]$. It follows however that when working in $\dg^{idem}$ this terminology is unnecessary because every exact sequence is strict. This follows from the properties of the functor $N_{dg}^L$ together with the Corollary \ref{exactsequencesarestrict}.
\end{remark}

Let us now come back to the definition of open immersion. Thanks to the results of Thomason in \cite[Section 5]{thomasonalgebraic} and to the work of B.Keller in \cite{keller-exact}, we know that for a quasi-compact and quasi-separated scheme $X$ with a quasi-compact open embedding $i:U\hookrightarrow X$, the restriction map $i^*:L_{qcoh}(X)\to L_{qcoh}(U)$ fits in a strict exact sequence in $\dgcc$

\begin{equation}
\xymatrix{
L_{qcoh}(X)_{X-U}\ar[d]\ar[r]&\ar[d]^{i^*} L_{qcoh}(X)\\
\ast \ar[r]& L_{qcoh}(U)
}
\end{equation}

\noindent where $L_{qcoh}(X)_{X-U}$ is by definition the kernel of the restriction $i^*$. It is also well-known that this kernel has a compact generator (see for instance the Lecture notes by M. Schlichting in \cite{sedano}). Of course, using the equivalence $\dg^{idem}\simeq \dgcc$, we can reformulate this in terms of an exact sequence in $\dg^{idem}$

\begin{equation}
\xymatrix{
(L_{qcoh}(X)_{X-U})_c \ar[d]\ar[r]&\ar[d]^{i^*} L_{pe}(X)\\
\ast \ar[r]& L_{pe}(U)
}
\end{equation}

\noindent where $(L_{qcoh}(X)_{X-U})_c$ has a compact generator. This motivates the following definition:

\begin{defn}
\label{defopenimmersion}
Let $f:\UU\to \X$ be a morphisms of smooth noncommutative spaces over $k$. We say that $f$ is an \emph{open immersion} if there exists a dg-category with a compact generator $K_{\X-\UU}\in \dg^{idem}$ (see \ref{dgcategorieswithcompactobjects}) together with a fully-faithful map $K_{\X-\UU}\hookrightarrow T_{\X}$ such that opposite of $f$ in $\dg^{ft}$ fits in a exact sequence in $\dg^{idem}$:

\begin{equation}
\xymatrix{
K_{\X-\UU} \ar[r] \ar[d]&T_{\X}\ar[d]\\
\ast \ar[r]&T_{\UU}
}
\end{equation}

It follows from the Remark \ref{exactsequencesarestrict2} that this diagram is also a pullback square.
\end{defn}

\begin{defn}
\label{defncnisnevich}
We will say that a commutative diagram in $\nck$

\begin{equation}
\xymatrix{
\W\ar[r]\ar[d]& \V\ar[d]\\
\UU\ar[r]&\X
}
\end{equation}

\noindent is a \emph{Nisnevich square of smooth noncommutative spaces} if the following conditions hold:

\begin{enumerate}
\item The maps $\UU\to \X$ and $\W\to \V$ are open immersions;
\item The associated map $T_{\X}\to T_{\V}$ sends the compact generator of $K_{\X-\UU}\subseteq T_{\X}$ to the compact generator of $K_{\V-\W}\subseteq T_{\V}$ and induces an equivalence $K_{\X-\UU}\simeq K_{\V-\W}$; 
\item The diagram is a pushout.
\end{enumerate}
\end{defn}

\begin{convention}
\label{emptyisnisnevich}
We will adopt the convention that if $\X$ is a smooth noncommutative space whose underlying dg-category $T_{\X}$ is a zero object of $\dg^{ft}$, the empty collection forms a Nisnevich square of $\X$. 
\end{convention}

Using the duality between smooth noncommutative spaces and dg-categories, a Nisnevich square corresponds to the data of a commutative diagram in $\dg^{idem}$

\begin{equation}
\label{nisnevichdata}
\xymatrix{
K_{\X-\UU}\ar[r]\ar[d]& \ast \ar[d]\\
T_{\X}\ar[r] \ar[d]& T_{\UU}\ar[d]\\
T_{\V}\ar[r]&T_{\W}\\
K_{\V-\W}\ar[u]\ar[r] & \ar[u]\ast
}
\end{equation}

\noindent where:

\begin{enumerate}[1)]
\item all $T_{\X}$, $T_{\UU}$, $T_{\V}$ and $T_{\W}$ are of finite type;
\item Both $K_{\X-\UU}$ and $K_{\V-\W}$ belong to $\dg^{idem}$, have a compact generator and the maps $K_{\X-\UU}\to T_{\X}$ and $K_{\V-\W}\to T_{\V}$ are fully-faithful;
\item The associated map $T_{\X}\to T_{\V}$ sends the compact generator of $K_{\X-\UU}\subseteq T_{\X}$ to the compact generator of $K_{\V-\W}\subseteq T_{\V}$ and induces an equivalence $K_{\X-\UU}\simeq K_{\V-\W}$;
\item the upper and lower squares are pushouts and pullbacks in $\dg^{idem}$ (see \ref{exactsequencesarestrict2}) and the middle square is a pullback in $\dg^{ft}$ and therefore in $\dg^{idem}$ (see the Remark \ref{nc2dgftdgidempreservesproducts}).
\end{enumerate}

These conditions also imply that the middle square is a pushout in $\dg^{idem}$. Indeed, because the exterior diagrams are pushouts we can write $T_{\W}\simeq T_{\V}\coprod_{K_{\V-\W}}\ast$ and $T_{\UU}\simeq T_{\X}\coprod_{K_{\X-\UU}}\ast$. Together with the fact that $K_{\X-\UU}$ and $K_{\V-\W}$ are equivalent, we have

\begin{equation}
T_{\V}\coprod_{T_{\X}}T_{\UU}\simeq  T_{\V}\coprod_{T_{\X}}(T_{\X}\coprod_{K_{\X-\UU}}\ast)\simeq T_{\V}\coprod_{K_{\X-\UU}}\ast\simeq T_{\V}\coprod_{K_{\V-\W}}\ast\simeq T_{\W}
\end{equation}

\begin{cor}
\label{nisnevichpullback}
Every Nisnevich square in $\nck$ is a pullback.
\end{cor}

\begin{remark}
\label{strictexactsequencesgivenisnevich}
Let $\UU\to \X$ be an open immersion of smooth noncommutative spaces. If the associated dg-category $K_{\X-\UU}\in \dg^{idem}$ is of finite type we can then see it as the dg-category $T_{\Z}=K_{\X-\UU}$ dual to a smooth noncommutative space $\Z$. Of course, since the zero map $T_{\Z}\to \ast$ is a quotient of $T_{\Z}$ by itself, its dual $\ast\to \Z$ is an open immersion. Moreover, since the diagram $K_{\X-\UU}\hookrightarrow T_{\X}\to T_{\U}$ is also a fiber sequence (see \ref{exactsequencesarestrict2}),  the square of smooth noncommutative spaces

\begin{equation}
\xymatrix{
\UU\ar[r]\ar[d]&\X\ar[d]\\
\ast \ar[r]& \Z
}
\end{equation}

\noindent is a pushout and therefore, Nisnevich.
\end{remark}

\begin{example}
\label{exceptional1}
The notions of \emph{semi-orthogonal decomposition} and \emph{exceptional collection} for triangulated categories (see \cite{bondal-orlov-semi}) have an immediate translation to the setting of dg-categories in terms of \emph{split short exact sequence} in $\dg^{idem}$. Recall that an exact sequence in $\dg^{idem}$

\begin{equation}
\xymatrix{
I\ar[d]\ar[r]^f& T\ar[d]^g\\ 
\ast\ar[r]& I'
}
\end{equation}

\noindent is said to \emph{split} if the functor $f$ (resp. $g$ ) admits a right adjoint $j$ (resp. fully-faithful right adjoint $i$). Following the Remark \ref{strictexactsequencesgivenisnevich} if $\X$ is a smooth noncommutative space, every semi-ortogonal decomposition of the associated dg-category $T_{\X}$ given by dg-categories $I$, $I'$ of finite type provides the data dual to a Nisnevich square

\begin{equation}
\xymatrix{
I\ar[r]\ar[d]&\ar[d] \ast\\
T_{\X}\ar[r]& I'
}
\end{equation}

\end{example}

\begin{example}
\label{exceptional2}
The previous example will be particularly important to us in the case $\X=L_{pe}(\mathbb{P}^1)$. Thanks to the results of \cite{beilinsonprojective} we know that $\mathbb{P}^n$ admits an exceptional collection generated by the twisting sheaves $\langle \mathcal{O},..., \mathcal{O}(-n)\rangle$. By the previous example, the diagram in $\dg^{idem}$ associated to the split exact sequence 

\begin{equation}
\xymatrix{
Perf(k)\ar[r]\ar[d]&\ar[d] \ast\\
L_{pe}(\mathbb{P}^1)\ar[r]& Perf(k)
}
\end{equation}

\noindent provides the data of a Nisnevich square.
\end{example}

We now prove that our Nisnevich squares are compatible with the monoidal product of smooth noncommutative spaces. For that we will need the following preliminary result

\begin{lemma}
\label{fiberproductdg}
Let \item Let

\begin{equation}
\xymatrix{
\W\ar[r]\ar[d]& \V\ar[d]\\
\UU\ar[r]&\X
}
\end{equation}

\noindent be a Nisnevich square of smooth noncommutative spaces and let 

\begin{equation}
\label{flora}
\xymatrix{
\widehat{T_{\X}}\ar[r]\ar[d]& \widehat{T_{\UU}}\ar[d]\\
\widehat{T_{\V}}\ar[r]&\widehat{T_{\W}}
}
\end{equation}

\noindent be its associated pullback diagram in $\dgcc$. Then the image of (\ref{flora}) through the (non-full) inclusion $\dgcc\to \dgc$ remains a pullback diagram.
\begin{proof}
As discussed in \ref{limitspresentabledg}, the dg-category $\dglp$ has all limits and the (non-full) inclusion $\dglp\subseteq \dg^{big}$ preserves them. By definition, $\dgc$ is the full subcategory of $\dglp$ spanned by the locally presentable dg-categories of the form $\widehat{T}$ for some small dg-category $T$. Therefore, we are reduced to show that the (non-full) inclusion $\dgcc\subseteq \dglp$ preserves the pullback diagrams (\ref{flora}) associated to Nisnevich squares. This statement is the dg-analogue of the Proposition \ref{bondalcontext}.\\

Following the discussion in \ref{linkdgspectrastable}, the functor $N_{dg}^L$ provides a commutative square

\begin{equation}
\label{ultimaweapon}
\xymatrix{
\dg^{idem}\simeq \dg^{cc}\ar@{^{(}->}[rr]^(0.7){non-full} \ar[d]^{N_{dg}^L}&& \dg^{lp}\ar[d]^{N_{dg}^L}\\
\Prl_{\omega,Stb}\ar@{^{(}->}[rr]^{non-full}&&\Prl_{Stb}
}
\end{equation}

\noindent and as explained in \ref{nervedgconservative} $N_{dg}^L$ is conservative, preserves fully-faithfulness and preserves the notion of "homotopy category" (see the Remark \cite[1.3.1.11]{lurie-ha}). It follows, as explained in  the Remark \ref{exactdgiffexactstable} that $N_{dg}^L$ preserves the notions of exact sequence. It follows also that $\widehat{A}$ has a compact generator if and only if $h(N_{dg}^L(\widehat{A}))$ has a compact generator.

Consider now the pullback diagram (\ref{flora}) associated to a Nisnevich covering and let $\widehat{K}\simeq \widehat{K_{\X-\UU}}\simeq \widehat{K_{\V-\W}}$ be the dg-category (with a compact generator) in $\dgcc$ associated to the open immersions. We find a diagram in $\Prl_{\omega,Stb}$ 

\begin{equation}
\xymatrix{
&N_{dg}^L(\widehat{T_{\X}})\ar[r]\ar[d]& N_{dg}^L(\widehat{T_{\UU}})\ar[d]\\
N_{dg}^L(\widehat{K})\ar@{^{(}->}[r]&N_{dg}^L(\widehat{T_{\V}})\ar[r]& N_{dg}^L(\widehat{T_{\W}})
}
\end{equation}

Since $N_{dg}^L$ commutes with limits, this diagram remains a pullback in $\Prl_{\omega,Stb}$ and we find ourselves facing the conditions of the Proposition \ref{bondalcontext} so that the diagram remains a pullback after the inclusion in $\Prl_{Stb}$. Finally, since $N_{dg}^L$ is conservative, the commutativity of (\ref{ultimaweapon}) implies that (\ref{flora}) remains a pullback in $\dg^{lp}$. This concludes the proof.
\end{proof}
\end{lemma}

We can now state the main result:

\begin{prop}
\label{nisnevichncsquaresmonoidal}
\begin{enumerate}[1)]
\item Let $\UU\to \X$ be an open immersion of smooth noncommutative spaces. Then, for any smooth noncommutative space $\Y$, the product map 

\begin{equation}
\UU\otimes \Y \to \X\otimes \Y 
\end{equation}

\noindent is also an open immersion;
\item Let

\begin{equation}
\xymatrix{
\W\ar[r]\ar[d]& \V\ar[d]\\
\UU\ar[r]&\X
}
\end{equation}

\noindent be a Nisnevich square of smooth noncommutative spaces. Then, for any smooth noncommutative space $\Y$, the square

\begin{equation}
\label{L7}
\xymatrix{
\W\otimes \Y \ar[r]\ar[d]& \V\otimes \Y\ar[d]\\
\UU\otimes \Y\ar[r]&\X \otimes \Y
}
\end{equation}

\noindent remains a Nisnevich square.
\end{enumerate}
\begin{proof}

To prove $1)$, let 

\begin{equation}
\xymatrix{
K_{\X-\UU}\ar[r] \ar[d]&T_{\X}\ar[d]\\
\ast \ar[r] &T_{\U}
}
\end{equation}

\noindent be the data in $\dg^{idem}$ corresponding to the open immersion. We are reduced to prove that by tensoring with $T_{\Y}$ (in $\dg^{idem}$) the diagram

\begin{equation}
\label{L6}
\xymatrix{
K_{\X-\UU}\otimes T_{\Y} \ar[r] \ar[d]&T_{\X}\otimes T_{\Y}\ar[d]\\
\ast\otimes T_{\Y} \ar[r] &T_{\U}\otimes T_{\Y}
}
\end{equation}

\noindent remains the data of an open immersion. Observe first that since the monoidal structure in $\dg^{idem}$ is compatible with colimits, $\ast\otimes T_{\Y}$ is again a zero object. To complete the proof it suffices to check that $(i)$ $K_{\X-\UU}\otimes T_{\Y}$ remains a dg-category having a compact generator; $(ii)$ the map $K_{\X-\UU}\otimes T_{\Y}\to T_{\X}\otimes T_{\Y}$ remains fully-faithful and $(iii)$ the diagram (\ref{L6}) is a pushout. The first assertion follows from the Remark \ref{productdgcategorieswithcompactgeneratorhascompactgenerator}. The second is obvious by the definition of fully-faithful and the construction of tensor products. The third follows from the Proposition 1.6.3 in \cite{drinfeld1}.\\

Let us now prove $2)$. It follows from $1)$ that both $\W\otimes \Y\to \V\otimes \Y$ and $\UU\otimes \Y \to \X\otimes \Y$ remain open immersions, corresponding the quotients by the subcategories $K_{\X-\UU}\otimes T_{\Y}$ and $K_{\V-\W}\otimes T_{\Y}$.  Since the map $K_{\X-\UU}\to K_{\V-\W}$ is an equivalence, the tensor product with the identity of $T_{\Y}$

\begin{equation}K_{\X-\UU}\otimes T_{\Y}\to K_{\V-\W}\otimes T_{\Y}\end{equation} 

\noindent remains an equivalence. We are now left to prove that the diagram (\ref{L7}) remains a pushout. This is equivalent to prove that associated diagram of dg-categories

\begin{equation}
\xymatrix{
T_{\X}\otimes T_{\Y}\ar[r]\ar[d]& \ar[d] T_{\UU}\otimes T_{\Y}\\
T_{\V}\otimes T_{\Y}\ar[r]& T_{\W}\otimes T_{\Y}
}
\end{equation}

\noindent remains a pullback in $\dg^{ft}$. Since all the dg-categories in this diagram are of finite type we can find dg-algebras $T_{\X}=Perf(A_{\X})$, $T_{\V}=Perf(A_{\V})$, $T_{\UU}=Perf(A_{\UU})$, $T_{\W}=Perf( A_{\W})$ and $T_{\Y}=Perf(A_{\Y})$. It follows that the previous diagram is a pullback if and only if the diagram 

\begin{equation}
\xymatrix{
\widehat{A_{\X}\otimes A_{\Y}}\ar[r]\ar[d]& \ar[d] \widehat{A_{\UU}\otimes A_{\Y}}\\
\widehat{A_{\V}\otimes A_{\Y}}\ar[r]& \widehat{A_{\W}\otimes A_{\Y}}
}
\end{equation}

\noindent is a pullback in $\dgcc$. By the hypothesis, the diagram
\begin{equation}
\xymatrix{
\widehat{A_{\X}}\ar[r]\ar[d]& \ar[d] \widehat{A_{\UU}}\\
\widehat{A_{\V}}\ar[r]& \widehat{A_{\W}}
}
\end{equation}

\noindent is a pullback and so, thanks to \cite[Theorem 7.2-1)]{Toen-homotopytheorydgcatsandderivedmoritaequivalences} and to the Lemma \ref{fiberproductdg}, we have equivalences

\begin{equation}
\widehat{A_{\X}\otimes A_{\Y}}\simeq  \mathbb{R}\underline{Hom}_c(\widehat{A_{\Y}},\widehat{A_{\X}})\simeq
\mathbb{R}\underline{Hom}_c(\widehat{A_{\Y}},\widehat{A_{\V}}\times_{\widehat{A_{\W}}} \widehat{A_{\UU}})\simeq 
\end{equation}

\begin{equation}
\mathbb{R}\underline{Hom}_c(\widehat{A_{\Y}},\widehat{A_{\V}})\times_{\mathbb{R}\underline{Hom}_c(\widehat{A_{\Y}},\widehat{A_{\W}})} \mathbb{R}\underline{Hom}_c(\widehat{A_{\Y}},\widehat{A_{\UU}})\simeq \widehat{A_{\V}\otimes A_{\Y}}\times_{\widehat{A_{\W}\otimes A_{\Y}}} \widehat{A_{\UU}\otimes A_{\Y}}
\end{equation}

\end{proof}
\end{prop}

To conclude this section we prove that our notion of Nisnevich squares of smooth noncommutative spaces is compatible with the classical notion for schemes.

\begin{prop}
\label{classicalnisnevichgoestoncnisnevich}
If $X$ is an affine smooth scheme of finite type over $k$ and 

\begin{equation}
\label{L1}
\xymatrix{
p^{-1}(U)\ar[r]\ar[d]& V\ar[d]^p\\
U\ar[r]^i&X
}
\end{equation}

\noindent is a Nisnevich square in $\aff$, then the induced diagram in $\nck$

\begin{equation}
\label{L2}
\xymatrix{
L_{pe}(p^{-1}(U))\ar[r]\ar[d]& L_{pe}(V)\ar[d]\\
L_{pe}(U)\ar[r]& L_{pe}(X)
}
\end{equation}

\noindent is a Nisnevich square of smooth noncommutative spaces. 
\begin{proof}

Indeed, it is immediate that both maps $L_{pe}(p^{-1}(U))\to L_{pe}(V)$ and $L_{pe}(U)\to L_{pe}(X)$ are open immersions of smooth noncommutative spaces. This is exactly the example that motivated the definition. They correspond to the quotient maps in $\dg^{idem}$

\begin{equation}
\label{jklm}
\xymatrix{
L_{pe}(X)\ar[r]& L_{pe}(X)/ L_{pe}(X)_{X-U}&\text{and}& L_{pe}(V)\ar[r]& L_{pe}(V)/ L_{pe}(V)_{V-p^{-1}(U)}
}
\end{equation}

\noindent We are left to check that:

\begin{enumerate}[1)]
\item The square in $\dg^{idem}$

\begin{equation}
\label{L5}
\xymatrix{
L_{pe}(X)\ar[r]\ar[d]& L_{pe}(U)\ar[d]\\
L_{pe}(V)\ar[r]& L_{pe}(p^{-1}(U))
}
\end{equation}

\noindent is a pullback;

\item the map $L_{pe}(X)\to L_{pe}(V)$ in $\dg^{idem}$ induces an equivalence $L_{pe}(X)_{X-U}\simeq L_{pe}(V)_{V-p^{-1}(U)}$;
\end{enumerate}

The fact that (\ref{L5}) is a pullback follows from the fact that perfect complexes satisfy descent for the étale topology (which is a refinement of the Nisnevich topology). This result was originally proven by Hirschowitz and Simpson in \cite{simpson-descente}. See also \cite{toen-vezzosi-hag2} for further details.\\

The assertion $2)$ follows from $1)$ together with the fact that both $L_{pe}(X)_{X-U}$ and $L_{pe}(V)_{V-p^{-1}(U)}$ are by definition, the kernels of the quotient maps (\ref{jklm}).
\end{proof}
\end{prop}

\begin{remark}
In fact, it can be proved that if a pullback diagram like (\ref{L1}) induces a Nisnevich square of smooth noncommutative spaces then it is a Nisnevich square in the classical sense. This can be deduced using the equivalence $L_{qcoh}(X)_{X-U}\simeq L_{qcoh}(V)_{p^{-1}(X-U)}$ together with the equivalences $L_{qcoh}(X)_{X-U}\simeq L_{qcoh}(\widehat{X}_{X-U})$ and $L_{qcoh}(V)_{p^{-1}(X-U)}\simeq L_{qcoh}(\widehat{V}_{p^{-1}(X-U)})$ where $\widehat{X}_{X-U}$, respectively, $\widehat{V}_{p^{-1}(X-U)}$, denotes the formal completion of $X$ (resp. $V$) at the closed subset $X-U$ (resp. $p^{-1}(X-U)$ (see \cite[Prop. 7.1.3]{gaitsgory-indschemes}). In particular this shows that the new notion of Nisnevich square is not really a weaker form of the original notion.
\end{remark}

\subsubsection{The Motivic Stable Homotopy Theory of Noncommutative Spaces}
\label{comparisonmap}

Now that we have an analogue for the Nisnevich topology in the noncommutative setting, compatible with the classical notion for schemes, we can finally conclude our task. We apply the same formula that produces the classical theory. We start with $\nck^{\otimes}$ and consider its free cocompletion $\mathcal{P}^{big}(\nck)$ together with the natural unique monoidal product extending the monoidal operation in $\nck$, compatible with colimits on each variable and making the inclusion $j:\nck \to \mathcal{P}^{big}(\nck)$  monoidal (see \ref{monoidalstructurepresheaves}). In particular, $j(L_{pe}(k))$ is the unit object. Next step, consider the localization $\mathcal{P}^{big}_{Nis}(\nck)$ of $\mathcal{P}^{big}(\nck)$ along the set of all edges $j(\UU)\coprod_{j(\W)} j(\V)\to j(\X)$ running over all the Nisnevich squares of smooth noncommutative space

\begin{equation}
\xymatrix{
\W\ar[r]\ar[d]& \V\ar[d]\\
\UU\ar[r]&\X
}
\end{equation}

The theory of localization for presentable $(\infty,1)$-categories \cite[5.5.4.15]{lurie-htt} implies that $\mathcal{P}^{big}_{Nis}(\nck)$ is an accessible reflexive localization of $\mathcal{P}^{big}(\nck)$. The same result, together with the fact that the Nisnevich squares are pushouts squares, implies that every representable $j(\X)$ is in $\mathcal{P}^{big}_{Nis}(\nck)$. Moreover, and thanks to the Proposition \ref{nisnevichncsquaresmonoidal}, we deduce that this localization is monoidal. Finally, and in analogy with the commutative case, we consider the localization

\begin{equation}l_{\mathbb{A}^1}^{nc}:\mathcal{P}^{big}_{Nis}(\nck)\to \hnck\end{equation}
 
\noindent taken with respect to the set of all maps

\begin{equation}
\xymatrix{
j(Id_{\X})\otimes j(L_{pe}(p)):j(\X)\otimes j( L_{pe}(\mathbb{A}_k^1))\ar[r]& j(\X)\otimes j( L_{pe}(Spec(k)))
}
\end{equation}

\noindent with $\X$ running over $\nck$. Here, $p:\mathbb{A}^1\to Spec(k)$ is the canonical projection and the tensor product is computed in $\mathcal{P}^{big}_{Nis}(\nck)$. \footnote{Of course, since $j$ is monoidal and the representable objects are Nisnevich local, this is the same as localizing with respect to the class of all maps $\xymatrix{j(\X\otimes  L_{pe}(\mathbb{A}_k^1)\ar[r]& \X\otimes  L_{pe}(Spec(k)))}$.} Again, this is an accessible reflective localization of $\mathcal{P}^{big}_{Nis}(\nck)$ and it follows immediately from the definition of the localizing set that it is monoidal. Altogether, we have a sequence of monoidal localizations

\begin{equation}
\label{formulanewton}
\xymatrix{
\nck^{\otimes}\ar[r]^{j}& \mathcal{P}^{big}(\nck)^{\otimes}\ar[r]& \mathcal{P}^{big}_{Nis}(\nck)^{\otimes} \ar[r]& \hnck^{\otimes}}
\end{equation}

\noindent and by construction, $\hnck$ is a presentable symmetric monoidal $(\infty,1)$-category and has a final object which we can identify with the image of the zero object of $\nck$ through the yoneda's map. Again, in analogy with the classical situation, we consider the universal pointing map 

\begin{equation}
()_{+}^{nc}:\hnck^{\otimes}\to \hnck_{*}^{\wedge(\otimes)}
\end{equation}

\noindent which is an equivalence because of our Convention \ref{emptyisnisnevich}: when we localize with respect to the Nisnevich topology with \ref{emptyisnisnevich} the $(\infty,1)$-category $\hnck$ becomes pointed.\\

Finally, the compatibility between the classical and the new Nisnevich squares \footnote{Recall that the collection of classical Nisnevich squares forms a basis for the Nisnevich topology} and the respective $\mathbb{A}^1$ and $L_{pe}(\mathbb{A}^1)$-localizations, we deduce the existence of uniquely determined monoidal colimit preserving functors that make the diagram homotopy commutative

\begin{equation}
\label{newtonoutravez}
\xymatrix{
\aff^{\times} \ar[d]^{j^{\otimes}}\ar[r]^{L_{pe}}& \nckmonoidal \ar[d]^{j^{\otimes}}\\
\mathcal{P}^{big}(\aff)^{\times}\ar@{-->}[r]^{(L_{pe})_{!}} \ar[d]& \mathcal{P}^{big}(\nck)^{\otimes}\ar[d]\\
Sh_{Nis}^{big}(\aff)^{\times}\ar@{-->}[r]\ar[d]^{l_{\mathbb{A}^1}^{\times}}& \ar[d]^{l_{\mathbb{A}^1}^{nc,\otimes}}\mathcal{P}_{Nis}^{big}(\nck)^{\otimes}\\
\hsch^{\times}\ar@{-->}[r] \ar[d]^{()_{+}}&\hnck^{\otimes}\\ 
\mathcal{H}(k)_{*}^{\wedge}\ar@{-->}[ru]^{\psi^{\otimes}}& 
}
\end{equation}

If we proceed according to the classical construction, the next step would be to stabilize the theory, first with respect to $S^1$ (the ordinary stabilization) and then with respect to the Tate circle. It happens that the inner properties of the noncommutative world make both these steps unnacessary.

\begin{prop}
\label{alreadystable}
The presentable pointed symmetric monoidal $(\infty,1)$-category $\hnck^{\otimes}$ is stable. Moreover, the Tate circle $\psi(\mathbb{G}_m)$ is already an invertible object.
\end{prop}

Recall that in $\mathcal{H}(k)_{*}^{\wedge}$ we have an equivalence $(\mathbb{P}^1,\infty)\simeq S^1\wedge \mathbb{G}_m$ with $\mathbb{G}_m$ pointed at $1$. Since the functor $\psi^{\otimes}$ is monoidal and commutes with colimits, we also have $\psi((\mathbb{P}^1,\infty))\simeq S^1\wedge \psi(\mathbb{G}_m)$. In particular, the  Proposition \ref{alreadystable} will follow immediately from the following lemma (using the Remark \ref{monoidalstabilization}).

\begin{lemma}
The object $\psi((\mathbb{P}^1,\infty))\in \hnck^{\otimes}$ is invertible.
\begin{proof}
By definition, we have 

\begin{equation}(\mathbb{P}^1,\infty):=cofiber_{\mathcal{H}(k)}[l_{\mathbb{A}^1}(\infty:Spec(k)\to \mathbb{P}^1)]\end{equation}

\noindent where $\infty:Spec(k)\to \mathbb{P}^1$ is the point at infinity. By diagram chasing, the fact that $L_{pe}(\mathbb{P}^1)$ is a dg-category of finite type, and the fact that all the relevant maps commute with colimits we find 

\begin{equation}
\psi((\mathbb{P}^1,\infty))\simeq  l_{\mathbb{A}^1}^{nc}(cofiber_{\mathcal{P}_{Nis}(\nck)}[j(L_{pe}(\infty))])
\end{equation}

\noindent We claim that the last cofiber is the unit for the monoidal structure in $\mathcal{P}_{Nis}^{big}(\nck)$, which, because the Yoneda's functor is monoidal, corresponds to $j(Perf(k))$. To see this, we observe first that map $L_{pe}(\infty):Perf(k)=L_{pe}(k)\to L_{pe}(\mathbb{P}^1)$ in $\nck$ corresponds in fact to the pullback map $L_{pe}(\mathbb{P}^1)\to Perf(k)$ along $\infty$ in $\dg^{idem}$. Recall the existence of an exceptional collection in $L_{pe}(\mathbb{P}^1)$ generated by the sheaves $\mathcal{O}$ and $\mathcal{O}(-1)$.  Since the pullback preserves structural sheaves, the map $Perf(k)\to L_{pe}(\mathbb{P}^1)$ in $\nck$ fits in the Nisnevich square of the Example \ref{exceptional2}

\begin{equation}
\xymatrix{
Perf(k) \ar[d]\ar[r]& L_{pe}(\mathbb{P}^1)\ar[d]\\
\ast \ar[r]& Perf(k)
}
\end{equation}

\noindent dual to the split exact sequence provided by the exceptional collection. Finally, since in $\mathcal{P}_{Nis}^{big}(\nck)$ every Nisnevich square is forced to become a pushout, we have

\begin{equation}
cofiber_{\mathcal{P}_{Nis}(\nck)}[j(Perf(k)\to L_{pe}(\mathbb{P}^1))]\simeq Perf(k)
\end{equation}

\noindent which concludes the proof.
\end{proof}
\end{lemma}

It also follows that we have a canonical equivalence 

\begin{equation}\hnck^{\otimes}[(\psi(\mathbb{P}^1,\infty)^{-1})]\simeq \hnck^{\otimes}\end{equation}

\noindent and for this reason, we reset the notations to match the classical one 

\begin{equation}
\stncmonoidalk:=\hnck^{\otimes}
\end{equation}

\begin{remark}
\label{splitstablesums}
Let $\C$ be an $(\infty,1)$-category with a zero object $0$.  Recall that a \emph{split exact sequence} in $\C$ is the data of a pushout square 

\begin{equation}
\xymatrix{
A\ar[r]^i\ar[d] & B \ar[d]^p\\
0\ar[r]& C
}
\end{equation}

\noindent such that there exist maps $u: C\to B$  and $v:B\to A$ in $\C$  with $p\circ u\sim id_C$ and $v\circ i\sim id_A$. We can now see that if $\C$ is a stable $(\infty,1)$-category, the data of a split exact sequence provides an equivalence $B\simeq A\oplus C$. To see this, it is enough to check that $i\circ v + u\circ p\sim id_B$, or, equivalently, that $(id_B -  u\circ p )\sim i\circ v$. To see the last, we use the fact that $\C$ is stable and therefore the previous square is also a pullback. In this case, since $p\circ (id_B -  u\circ p)\sim (p - p\circ u\circ p)\sim (p - Id_C\circ p)\sim 0$,  we can find a factorization $\delta$(unique up to a contractible space)

\begin{equation}
\xymatrix{
B\ar@/_/[ddr]\ar@/^/[rrd]^{id_B -  u\circ p}\ar@{-->}[dr]^{\delta}&&\\
&A\ar[r]^i \ar[d] & B \ar[d]^p\\
&0\ar[r]& C
}
\end{equation}

\noindent with  $(id_B -  u\circ p) \sim i\circ \delta$. But, since $v\circ i\sim id_A$, we have $\delta \sim v\circ i\circ \delta \sim v\circ ( id_B -  u\circ p)\sim v$.
\end{remark}

\begin{remark}
\label{exceptionalsplittosums}
Let $\X$ be smooth noncommutative space whose associated dg-category $T_{\X}$ admits an exceptional collection generated by $n+1$ elements. By the Remarks \ref{exceptional1} and \ref{exceptional2}, this provides to the data of $n$ different Nisnevich coverings. These are sent to split exact sequences in $\stnck$ which we now know is stable. Using the Remark \ref{splitstablesums} we find that the image of $\X$ in $\stnck$ decomposes as a direct sum of $n+1$ copies of the unit $1= l_{\mathbb{A}^1}^{nc}(Perf(k))$. In particular the smooth noncommutative space $L_{pe}(\mathbb{P}^n)$ becomes equivalent to the direct sum $\underbrace{1\oplus ...\oplus 1}_{n+1}$ in $\stncmonoidalk$.
\end{remark}

Finally, our universal property for inverting an object in a presentable symmetric monoidal $(\infty,1)$-category ensures the existence of a unique monoidal colimit map $\mathcal{L}^{\otimes}$ extending the diagram (\ref{newtonoutravez}) to

\begin{equation}
\xymatrix{
\aff^{\times}\ar[r]^{L_{pe}^{\otimes}}\ar[d]& \nck^{\otimes}\ar[d]\\
\stmonoidalk\ar@{-->}[r]^{\mathcal{L}^{\otimes}}& \stncmonoidalk
}
\end{equation}

\noindent relating the classical stable homotopy theory of schemes with our new theory.\\

\begin{remark}
\label{usingpresheavesofspectra2}
Using the same arguments of  \ref{usingpresheavesofspectra1} we can describe the symmetric monoidal $(\infty,1)$-category $\stncmonoidalk$ using presheaves of spectra. More precisely, we can start from the $(\infty,1)$-category of smooth noncommutative spaces $\nck$ and consider the very big $(\infty,1)$-category $Fun(\nck^{op},  \widehat{\Sp})$. Using the equivalence $Fun(\nck^{op},  \widehat{\Sp})\simeq Stab(\mathcal{P}^{big}(\nck)_{\ast})$ together with the Remark \ref{monoidalstabilization} we obtain a canonical monoidal structure $Fun(\nck^{op},  \widehat{\Sp})^{\otimes}$ defined by the inversion $\mathcal{P}^{big}(\nck)_{\ast}^{\wedge(\otimes)}[(S^1)^{-1}]^{\otimes}$. 
We proceed, and perform the localizations with respect to the noncommutative version of the Nisnevich topology and $L_{pe}(\mathbb{A}^1)$. More precisely, and using the same notations as in \ref{usingpresheavesofspectra1} we localize with respect to the class of all canonical maps

\begin{equation}
 \delta_{\Sigma^{\infty}_{+}\circ j(\UU)}(K)\coprod_{\delta_{ \Sigma^{\infty}_{+}\circ j(\W)}(K)}\delta_{ \Sigma^{\infty}_{+}\circ j(\V)}(K)\to  \delta_{ \Sigma^{\infty}_{+}\circ j(\X)}(K)
\end{equation}

\noindent  with $K$ in $( \widehat{\Sp})^{\omega}$ and $\W$,$\V$,$\UU$ and $\X$ part of a Nisnevich square of noncommutative smooth spaces. For the $\mathbb{A}^1$ localization, we localize with respect to the class of all induced maps 

\begin{equation}
\delta_{ \Sigma^{\infty}_{+}\circ j(\X\otimes L_{pe}(\mathbb{A}^1))}(K)\to \delta_{ \Sigma^{\infty}_{+}\circ j(\X)}(K)
\end{equation}

\noindent  with $\X$ in $\nck$ and $K \in ( \widehat{\Sp})^{\omega}$. By the same argument, these are monoidal reflexive localizations. We denote the result as $Fun_{Nis, L_{pe}(\mathbb{A}^1)}(\nck^{op},  \widehat{\Sp})^{\otimes}$. It is a stable presentable symmetric monoidal $(\infty,1)$-category and by the Prop. \ref{alreadystable} and the universal properties involved, it is canonically monoidal equivalent to $\stncmonoidalk$.

Using this equivalence and the definition of $\nck$, we can identify an object $F\in \stncmonoidalk$ with a functor $\dg^{ft}\to  \widehat{\Sp}$ satisfying $L_{pe}(\mathbb{A}^1)$-invariance, having a descent property with respect to the Nisnevich squares and because of the convention \ref{emptyisnisnevich}, satisfying $F(0)=\ast$.
\end{remark}

\subsection{Future Works}
The previous section concludes our goals for this paper. In a second part of this work, we investigate the following insight/conjecture of Kontsevich in \cite{kontsevich1}:

\begin{conjecture}(Kontsevich)
Let $\X$ and $\Y$ be smooth noncommutative spaces over $k$ and assume $\Y$ is dualizable. Then, there is a natural equivalence of spectra
\begin{equation}
Map_{\stnc}(\X, \Y)\simeq K(T_{\X}\otimes^{\mathbb{L}}T_{\Y}^{\circ})
\end{equation}
\end{conjecture}

\noindent where $T_{\Y}^{\circ}$ is the dual of the dg-category $T_{\Y}$ with respect to the symmetric monoidal structure in $\nck$. G. Tabuada proved this conjecture is true in his approach (see \cite{tabuada-higherktheory}). It should also be true in our new setting and in this case this will immediately force the right adjoint $\mathcal{M}^{\otimes}$ to the map $\mathcal{L}^{\otimes}$ (which exists due to the adjoint functor theorem and is lax-monoidal because of formal abstract-nonsense - see the discussion in \ref{monoidaladjoint}) to send the unit of the monoidal structure in $\stncmonoidalk$ to the object $\mathbb{KH}\in \stk$ representing homotopy invariant algebraic $K$-theory, endowing it with a natural structure of object in $CAlg(\stk)$. In this case we can easily see that  $\mathcal{L}^{\otimes}$ factors through a new monoidal functor

\begin{equation}
\xymatrix{
\stmonoidalk \ar[r]^{\mathcal{L}^{\otimes}}\ar[d]_{(-\otimes_{1_{\stk}}\mathbb{KH})}& \stncmonoidalk\\
Mod_{\mathbb{KH}}(\stk)^{\otimes}\ar@{-->}[ru]_{\mathcal{L}^{\otimes}_{\mathbb{KH}}}& 
}
\end{equation}

\noindent where $1_{\stk}$ is the unit of the monoidal structure and the vertical map is the base-change with respect to the unit $1_{\stk}\to \mathbb{KH}$. Our main goal is to investigate to what extent this new factor is fully-faithfull.

\bibliographystyle{abbrv}	
\bibliography{biblio}

\end{document}